\documentclass[11pt,twoside]{article}
\usepackage{amsmath}
\usepackage{amssymb}
\usepackage{amsthm}
\usepackage{mathrsfs}
\usepackage{simplewick}
\usepackage{times}
\usepackage{color}
\usepackage{hep}
\usepackage{leftidx}
\usepackage{graphicx}
\usepackage{float}
\usepackage{enumerate}
\usepackage{titletoc}
\usepackage{epstopdf}
\usepackage{ulem}
\usepackage{appendix}
\usepackage[hidelinks]{hyperref}
\usepackage{enumitem}

\allowdisplaybreaks

\DeclareSymbolFont{EulerExtension}{U}{euex}{m}{n}
\DeclareMathSymbol{\euintop}{\mathop}{EulerExtension}{"52}
\DeclareMathSymbol{\euointop}{\mathop}{EulerExtension}{"48}

\def\g{\gamma}
\def\al{\alpha}
\def\be{\beta}
\def\de{\delta}

\def\o{\omega}

\def\na{\nabla}

\def\Si{\Sigma}
\def\th{\theta}

\def\ve{\varepsilon}
\def\ve0{\varepsilon_{0}}
\def\vp{\varphi}

\def \no{\nonumber}

\def\D{\mathscr{D}}

\def\f{\frac}

\def\Lies{\slashed{\mathcal{L}}}

\def\Lr{\mathring{L}}

\def\p{\partial}
\def\pis{\slashed{\pi}}

\def\dps{\displaystyle}
\def\dl{\delta}

\begin{document}
\footskip=0pt
\footnotesep=2pt
\let\oldsection\section
\renewcommand\section{\setcounter{equation}{0}\oldsection}
\renewcommand\thesection{\arabic{section}}
\renewcommand\theequation{\thesection.\arabic{equation}}
\newtheorem{claim}{\noindent Claim}[section]
\newtheorem{theorem}{\noindent Theorem}[section]
\newtheorem{lemma}{\noindent Lemma}[section]
\newtheorem{proposition}{\noindent Proposition}[section]
\newtheorem{definition}{\noindent Definition}[section]
\newtheorem{remark}{\noindent Remark}[section]
\newtheorem{corollary}{\noindent Corollary}[section]
\newtheorem{example}{\noindent Example}[section]

\title{Blowup of smooth solutions to 2D quasilinear wave equations
with higher-order null conditions and short pulse initial data}
\author{Ding
	Bingbing$^{1}$, \quad Xin Zhouping$^{2}$, \quad Yin
	Huicheng$^{1,}$\footnote{Ding Bingbing (13851929236@163.com, bbding@njnu.edu.cn) and Yin Huicheng (huicheng@nju.edu.cn, 05407@njnu.edu.cn)
are supported by the NSFC (No.12331007, No.12671259). Xin Zhouping (zpxin@ims.cuhk.edu.hk)
is partially supported by the Zheng Ge Ru Foundation, Hong Kong RGC Earmarked Research Grants
CUHK-14301421, CUHK-14300819, CUHK-14302819, CUHK-14301023, Basic and Applied Basic Research Foundations of Guangdong Province 2020131515310002, and the Key Projects of National Natural Science Foundation of China (No.12131010, No.11931013).}\vspace{0.5cm}\\
	\small 1.  School of Mathematical Sciences and IMS, Nanjing Normal University,  Nanjing, 210023, China.
	\\
	\vspace{0.5cm}
	\small 2. Institute of Mathematical Sciences, The Chinese University of Hong Kong, Shatin, NT, Hong Kong.}
\date{}
\maketitle


\begin{abstract}
This paper investigates the 2D quasilinear wave equation
$\dps\sum_{\alpha,\beta=0}^{2}
g^{\alpha\beta}(\partial\phi)
\partial_{\alpha\beta}^{2}\phi=0$ with the short pulse initial data
$(\phi, \p_t\phi)|_{t=1}=(\delta^{2-\varepsilon_0}\phi_0(\f{r-1}{\delta},\omega),
\delta^{1-\varepsilon_0}\phi_1(\f{r-1}{\delta},\omega))$, where $0<\varepsilon_0<1$ is a fixed
constant and $\delta>0$ is a small parameter. In our previous work [Ding Bingbing, Xin Zhouping, Yin Huicheng,
Sci. China Math., Vol. 69 (2026), No. 3,
667-736], we have established the global existence of a smooth solution $\phi$ under
the $k$-th order null condition  ($k\ge 2$) when the $(k+1)$-th order null condition fails, provided that $0<\varepsilon_0<\varepsilon_k^*=\frac{k}{k+1}$.
As a natural continuation, the present paper demonstrates that
for $\varepsilon_0\ge \varepsilon_k^*$, the smooth solution  $\phi$ blows up in finite time.
Consequently, the threshold
$\varepsilon_0=\varepsilon_k^*$  is shown to be the critical exponent that precisely characterizes the global existence and finite-time blowup
for smooth short pulse solutions to the aforementioned 2D quasilinear wave equation, under the assumption that the
$k$-th order null condition holds while the
$(k+1)$-th order null condition does not.  A key novelty of our approach lies in deriving sharp orders of
$\delta$  for almost all quantities in both energy and pointwise estimates. This is achieved by utilizing the higher-order null structure in conjunction with a new good unknown recently introduced by us 
and the transport equation for the difference between  this good unknown and its corrector.
\end{abstract}
\vskip 0.2 true cm
{\bf Keywords:} {Quasilinear wave equation, short pulse initial data, higher-order null condition,

\qquad\quad  critical exponent, good unknown, corrector}

\tableofcontents
\section{Introduction}\label{in}

\subsection{Main result}

Consider the 2D quasilinear wave equation
\begin{equation}\label{quasi}
\sum_{\alpha,\beta=0}^{2}
g^{\alpha\beta}(\partial\phi)
\partial_{\alpha\beta}^{2}\phi=0,
\end{equation}
where $(x^0,x^1,x^2)=(t,x^1,x^2)\in [1,\infty)\times\mathbb R^2$,
$\partial=(\partial_0,\partial_1,\partial_2)=(\p_{x^0}, \p_{x^1}, \p_{x^2})$, and
$g^{\alpha\beta}(\partial\phi)=g^{\beta\al}(\p\phi)$ are smooth functions of
$\partial\phi$. Moreover, for small $\partial\phi$,
it is assumed that
\begin{equation}\label{gh}
\begin{aligned}
g^{\alpha\beta}(\partial\phi)
=&\,m^{\alpha\beta}
+
\sum_{\gamma_1,\cdots,\gamma_k=0}^{2}
g^{\alpha\beta,\gamma_1,\cdots,\gamma_k}
\partial_{\gamma_1}\phi\cdots
\partial_{\gamma_k}\phi
\\
&+
\sum_{\gamma_1,\cdots,\gamma_{k+1}=0}^{2}
h^{\alpha\beta,\gamma_1,\cdots,\gamma_{k+1}}
\partial_{\gamma_1}\phi\cdots
\partial_{\gamma_{k+1}}\phi
+
O(|\partial\phi|^{k+2}),
\end{aligned}
\end{equation}
where $k\ge 2$ is a fixed integer, $m^{\alpha\beta}$ denotes the Minkowski metric (i.e., $m^{00}=-1$,
$m^{ii}=1$ for $1\le i\le 2$, $m^{\al\beta}=0$ for $\al\neq\beta$),
the constants $g^{\al\beta,
\gamma_1,\cdots,\gamma_k}$ and $h^{\alpha\beta,\gamma_1,\cdots,\gamma_{k+1}}$
are symmetric for the permutations of
$(\al,\beta)$, $(\gamma_1,\cdots,\gamma_k)$ and  $(\gamma_1,\cdots,\gamma_{k+1})$, respectively.
Note that the equation \eqref{quasi} with \eqref{gh} includes many interesting physical models, such as the
2D relativistic membrane equation $\p_t\big(\f{\p_t\phi}{\sqrt{1-(\p_t\phi)^2+|\na_x\phi|^2}}\big)
-div\big(\f{\nabla_x\phi}{\sqrt{1-(\p_t\phi)^2+|\na_x\phi|^2}}\big)=0$ with $\na_x\phi=(\p_1\phi,\p_2\phi)$,
the 2D nonlinear membrane equation $\p_t^2\phi-div\big(\f{\nabla_x\phi}{\sqrt{1+|\na_x\phi|^2}}\big)=0$ and others (see \cite{Ding4}
for more details).

We impose the following short pulse initial data for \eqref{quasi}:
\begin{align}\label{id}
\phi|_{t=1}=\delta^{2-\varepsilon_0}\phi_0(\f{r-1}{\delta},\omega),\
\p_t\phi|_{t=1}=\delta^{1-\varepsilon_0}\phi_1(\f{r-1}{\delta},\omega),
\end{align}
where $\delta>0$ is a small parameter, $0<\varepsilon_0<1$ is a fixed constant, $r=|x|=\sqrt{(x^1)^2+(x^2)^2}$,
$\o=(\o_1,\o_2)=\dps\f{x}{r}\in\mathbb S$, $(\phi_0,\phi_1)(\sigma,\o)$ are smooth functions
supported in $(-1,0)$ for the variable $\sigma$.
The short pulse initial data \eqref{id} with $\varepsilon_0=\f12$ were first
introduced by D. Christodoulou \cite{Ch2}
for studying the formation of black holes in general relativity
(see also \cite{K-R}). It is worth noting that although the amplitude of the initial data in \eqref{id} is small,
their derivatives are large due to the small concentration scale $\delta$.
This differs from the usual small data solution problem for the quasilinear wave equations,
where higher-order derivatives of the initial data are small
(see \cite{Ali1}-\cite{Ali2}, \cite{H} and \cite{Li-Zhou}).

As in \cite{Ding4} (also see \cite{Chen-Xie}, \cite{L-Y}, \cite{MPY}-\cite{M-Y} and \cite{Wang}-\cite{Wang-1}),
we impose the following outgoing constraint condition:
\begin{align}\label{Y-0}
(\p_t+\p_r)\phi|_{t=1}=O(\delta^{2-\varepsilon_0}),
\end{align}
which can be easily realized by choosing $\phi_1(\sigma,\omega)
=-\partial_\sigma\phi_0(\sigma,\omega)$.
In addition, the $k$-th
null condition $(k\ge 2, k\in\mathbb N)$ for \eqref{quasi} is defined as  (see \cite{Ding4})
\begin{equation}\label{null}
\sum_{0\le \al,\beta,\gamma_1, ..., \gamma_k\le 2}g^{\al\beta,\gamma_1,\cdots,\gamma_k}
\omega_\al\omega_\beta\omega_{\gamma_1}\cdots{\omega_{\gamma_k}}\equiv 0\quad \text{for $\omega_0=-1$ and
$(\omega_1, \omega_2)\in \mathbb S$}.
\end{equation}

It was shown in \cite{Ding4} that under \eqref{gh}-\eqref{null}, the smooth short pulse solution
$\phi$ exists globally when $0<\varepsilon_0<\varepsilon_k^*=\frac{k}{k+1}$.
The purpose of the present paper is to investigate the optimality of the exponent
$\varepsilon_k^*$. Namely, when the $k$-th order null condition \eqref{null} holds, but the
$(k+1)$-th order null condition fails,
i.e.,
\begin{equation}\label{notnull}
\sum_{\alpha,\beta,\gamma_1,\cdots,\gamma_{k+1}=0}^{2}
h^{\alpha\beta,\gamma_1,\cdots,\gamma_{k+1}}
\omega_\alpha\omega_\beta
\omega_{\gamma_1}\cdots
\omega_{\gamma_{k+1}}
\not\equiv 0,
\end{equation}
we prove that the smooth solution $\phi$ blows up in finite time when $\phi_1(s,\o)\not\equiv 0$
and $\varepsilon_k^*\le\varepsilon_0<1$.

Define
\begin{equation}\label{XI}
\Xi
=\max_{\sigma,\omega}
\big\{
(-1)^k
\sum_{\alpha,\beta,\gamma_1,\cdots,\gamma_{k+1}=0}^{2}
h^{\alpha\beta,\gamma_1,\cdots,\gamma_{k+1}}
\omega_\alpha\omega_\beta
\omega_{\gamma_1}\cdots
\omega_{\gamma_{k+1}}
\phi_1^k(\sigma,\omega)
\partial_\sigma\phi_1(\sigma,\omega)
\big\}.
\end{equation}
Note that
$0<\Xi<\infty$ holds due to  $\phi_1\not\equiv0$ and $\phi_1\in C_0^{\infty}$. Indeed,  let
\begin{equation}\label{YCao-5}
\begin{aligned}
\mathcal H(\sigma,\omega)
=(-1)^k\phi_1^{k+1}(\sigma,\omega)
\sum_{\alpha,\beta,\gamma_1,\cdots,\gamma_{k+1}=0}^{2}
h^{\alpha\beta,\gamma_1,\cdots,\gamma_{k+1}}
\omega_\alpha\omega_\beta
\omega_{\gamma_1}\cdots\omega_{\gamma_{k+1}}.
\end{aligned}
\end{equation}
By \eqref{notnull} and $\phi_1\not\equiv0$, one has
$\mathcal H\not\equiv0$
and $\operatorname{supp}\mathcal H
\subset(-1,0)\times\mathbb S$.
Therefore, there exists a point $(\sigma_0,\omega_0)$ such that
$\partial_\sigma\mathcal H(\sigma_0,\omega_0)>0$,
which yields $\Xi>0$.

Without loss of generality and for simplicity, suppose that in \eqref{quasi},
\begin{equation}\label{g00}
g^{00}(\p\phi)=-1.
\end{equation}
In addition, in order to treat the blowup problem of \eqref{quasi} with \eqref{id} for the critical exponent 
$\varepsilon_0=\varepsilon_k^*$,
we will impose the following quantitative condition
\begin{equation}\label{critical-compression}
(k+1)\Xi>k-1,
\end{equation}
which means that there is a point $(\sigma_0,\omega_0)\in(-1,0)\times\mathbb S$
such that $\partial_\sigma\mathcal H(\sigma_0,\omega_0)>k-1$.
This can be easily realized by the suitable choice of $\phi_1$.

Our main result is stated as follows.
\begin{theorem}\label{maintheorem}
Assume that $k\ge2$, $\phi_1\not\equiv0$ and \eqref{gh}--\eqref{g00} hold.
Let $T_\delta$ denote the lifespan of the smooth solution $\phi$.
Then we have
\begin{enumerate}[label=(\roman*)]
\item when $\varepsilon_k^*<\varepsilon_0<1$, for small $\delta>0$, the solution $\phi$ develops a singularity in finite time, and
\begin{equation}\label{lifespan}
\limsup_{\delta\to0^+}
\delta^{(k+1)(\varepsilon_k^*-\varepsilon_0)}(T_\delta-1)
\le\frac{2}{(k+1)\Xi}.
\end{equation}
\item when $\varepsilon_0=\varepsilon_k^*$ and \eqref{critical-compression}
holds, for small $\delta>0$,  the solution $\phi$ will blow up in finite time and
\begin{equation}\label{critical-lifespan}
\limsup_{\delta\to0^+}T_\delta
\le S_{\mathrm{crit}}
:=\big(1-\frac{k-1}{(k+1)\Xi}\big)^{-\frac{2}{k-1}}.
\end{equation}
\end{enumerate}
\end{theorem}

\begin{remark}\label{Rem.2}
Collecting the main result in \cite{Ding4} and
Theorem~\ref{maintheorem} in the present paper, we precisely characterize the global existence range
$0<\varepsilon_0<\varepsilon_k^*$ and finite-time blowup range
$\varepsilon_k^*\le\varepsilon_0<1$.
\end{remark}

\begin{remark}\label{Rem.3}
Consider the small data solution problem for the 2D quasilinear wave equation
\begin{equation}\label{quasi-0}
\text{$\sum_{\al,\beta=0}^2g^{\al\beta}(\p v)\p_{\al\beta}^2 v=0$ with $(v,\p_t v)(1,x)=\dl(v_0(x),v_1(x))$},
\end{equation}
where $(v_0, v_1)(x)\in C_0^{\infty}(\mathbb R^2)$ and $\dl>0$ is small.
Assume that for small $\p v$,
\begin{equation}\label{H-01}
\dps g^{\al\beta}(\p v)=m^{\al\beta}+\sum_{\kappa=0}^{2} g^{\al\beta,\kappa}\p_\kappa v+\sum_{\kappa,\varsigma=0}^2h^{\al\beta,\kappa\varsigma}\p_\kappa v\p_\varsigma v+O(|\p v|^3),
\end{equation}
where $m^{\alpha\beta}$ denotes the Minkowski metric, $g^{\al\beta,\kappa}$
and $h^{\al\beta,\kappa\varsigma}$ are constants.
It is shown in \cite{Ali3} that the problem \eqref{quasi-0} admits a global smooth solution $v$ satisfying $|\partial v|\le C\delta(1+t)^{-1/2}$ when both the first and the second null conditions hold (i.e.,
$\dps\sum_{\al,\beta,\kappa=0}^2g^{\al\beta,\kappa}\omega_\al\omega_\beta\omega_\kappa\equiv0$ and $\dps\sum_{\al,\beta,\kappa,\varsigma=0}^2 h^{\al\beta,\kappa\varsigma}\omega_\al\omega_\beta\omega_\kappa\omega_\varsigma\equiv0$
for $\omega_0=-1$ and $(\omega_1,\omega_2)\in\mathbb S$, respectively).
Otherwise, if either the first or the second null condition fails,
under the generic non-degenerate assumption on $(v_0(x),v_1(x))$, then $v$ blows up in finite time (see \cite{Ali1} and \cite{Ali2}).

For the short pulse problem considered here, no additional generic non-degeneracy assumption is needed for $\varepsilon_k^* \leq \varepsilon_0 < 1$ except the usual condition $\phi_1 \not\equiv 0$.
\end{remark}

\begin{remark}\label{Yin.4}
If the inverse foliation density $\mu$ (see its definition in \eqref{DXY-1} below), which measures the compression of the characteristics, has a positive lower bound,
through establishing uniform bounds for the relevant energies and geometric quantities,
then one can conclude that the equation in \eqref{quasi-0} equipped with short pulse data
admits a global smooth solution $v$ (see \cite{DLY}-\cite{Ding4}, \cite{MPY} and \cite{Wang}).
In addition, the authors of \cite{An-1, An-2} apply the Lorentzian geometric tools from \cite{Ch1}
to study the local well-posedness or ill-posedness of low regularity solutions to the elastic wave system.
On the other hand, based on the Lorentzian geometric methods, the authors of \cite{Luo-Yu-1, Luo-Yu-2}
have shown the structural stability in the regime of two families of rarefaction waves
for multi-dimensional compressible isentropic and irrotational Euler equations.
\end{remark}

\subsection{Sketch of the proof}

We now briefly describe the main ingredients of the proof of Theorem  \ref{maintheorem}.
Motivated by the geometric methods of D. Christodoulou \cite{Ch1},
we introduce the inverse foliation density
\begin{align}\label{DXY-1}
\mu=-\f{1}{\sum_{\al=0}^2g^{0\al}\p_{\al}u}
\end{align}
associated with equation \eqref{quasi},
which measures the squeezing of the outgoing characteristic surfaces. Here the optical function
$u$ solves, for $t\ge t_0=1+2\delta$,
\begin{equation}\label{YHCC-60}
\displaystyle\sum_{\alpha,\beta=0}^2 g^{\al\be}(\p\phi)\p_{\al}u\p_{\be}u=0\quad\text{with $u(t_0,x)=t_0-r$ and $\p_0u>0$.}
\end{equation}
If $\mu\to 0+$ can be shown for $t\le T_\delta$,
then the blowup of $\p^2\phi$ occurs  for problem \eqref{quasi}.
We next give the detailed illustrations.

{\bf $\bullet$} In order to derive the energy estimates for problem \eqref{quasi} and the transport equation of $\mu$,
as in \cite{Ch1, Luk, Sp} and so on, the null frame $\{\mathring L, \mathring{\underline L}, X\}$ is introduced for the Lorentz
metric $(g_{\alpha\beta})$ (see Section \ref{Section 2} below for details).

{\bf $\bullet$} Set
\begin{equation*}
\begin{split}
&\text{$\vp_{\kappa}=\p_{\kappa}\phi$ and
$G_{\alpha\beta}^{\kappa}= \partial_{\vp_{\kappa}} g_{\alpha\beta}$ for $\kappa=0,1,2$},\\
&\text{$G_{UV}^{\kappa} = \dps\sum_{\alpha,\beta=0}^2G_{\alpha\beta}^{\kappa} U^{\alpha} V^{\beta}$ for any vector fields
$U = \dps\sum_{\alpha=0}^2U^{\alpha} \partial_{\alpha}$ and $V = \dps\sum_{\alpha=0}^2V^{\alpha} \partial_{\alpha}$.}
\end{split}	
\end{equation*}
Let $\slashed g_{XX}=g(X, X)$ and $\varrho={t}-u$. Define $\D$ as the Levi-Civita connection of the metric $g=(g_{\alpha\beta})$ and $\chi_{XX}=g(\mathscr{D}_X \mathring{L}, X)$ as the second fundamental form.

Throughout the whole paper, we will use the following crucial
``good unknown"
\begin{equation}\label{FG-Y2}
\check{\mathscr{X}}_{XX}=\chi_{XX}- \frac{1}{2}\dps\sum_{\kappa=0}^2G_{XX}^\kappa\mathring{L}\vp_\kappa-\f{1}{\varrho}\slashed g_{XX},
\end{equation}
which was first introduced in \cite{Ding5}.
By estimating $\mathring{L}\bigl(\varrho^2\operatorname{tr}\check{\mathscr X}\bigr)$ with
$\operatorname{tr}\check{\mathscr X}=\slashed g^{XX}\check{\mathscr{X}}_{XX}$ and integrating along the integral curves
of $\mathring{L}$, we can obtain that (see Lemma \ref{YHCC-555} and Proposition \ref{prop:lower-order-Linfty-estimates}
below)
\begin{equation}\label{YHCCC-0}
\check{\mathscr{X}}_{XX}=O(\delta^{k(1-\varepsilon_0)}).
\end{equation}

By the basic property of $\check{\mathscr{X}}_{XX}$ in \eqref{YHCCC-0}
and the new decomposition of
$G_{\mathring L\mathring L}^{\gamma}\slashed{\triangle}\varphi_{\gamma}$
(see \eqref{Gtrivp} in Lemma~\ref{lem:null-transport}), through studying the delicate
structure of $\mu\Box_gZ^{m}\varphi_\kappa$ and deriving the related energy estimates for \eqref{quasi},
we prove that near the outermost characteristic cone, it holds
\begin{equation}\label{bootstrap-assumptions-Y}
\begin{split}
&\delta^{l+m\varepsilon_k} \| Z^{\alpha} \vp_{\gamma} \|_{L^{\infty}}+\|\slashed\triangle\vp_{\gamma}\|_{L^{\infty}}
\leq C\delta^{1-\varepsilon_0},\quad \gamma=0, 1, 2,
\end{split}
\end{equation}
where $\varepsilon_k=(k+1)(\ve0-\varepsilon_k^*)$, $Z \in \{\varrho\mathring{L}, T, R\}$ with $T=\mu(-\dps\sum_{\alpha=0}^2g^{\al 0}\p_\al-\mathring L)$ and
$R=\slashed\Pi\Omega$ being the projection of $\Omega=x^1\p_2-x^2\p_1$, and $l$ (resp. $m$) denotes the number of $T$ (resp. $\varrho\mathring{L}$) contained in $Z^{\alpha}$.

\vskip 0.2 true cm

{\bf $\bullet$}
Based on \eqref{bootstrap-assumptions-Y}{,}
together with delicate estimates on some related geometric quantities
(see Proposition \ref{prop:lower-order-Linfty-estimates}), we obtain (see Proposition \ref{Prop.5.1})
\begin{equation}\label{blowup-0}
\mathring L\mu
=\varrho^{-(k+1)/2}\mathscr M(s,u,\vartheta)
+O(\delta^{1-\varepsilon_0-\varepsilon_k})\sim -\delta^{-\varepsilon_k},
\end{equation}
where
$$
\mathscr{M}(s,u,\vartheta)=(-1)^k\frac{k+1}{2}\bigl(\varrho^{(k+1)/2}h^{\alpha\beta,\gamma_1\cdots\gamma_{k+1}}
\omega_\alpha\omega_\beta\omega_{\gamma_1}\cdots\omega_{\gamma_{k+1}}\vp_0^k T\varphi_0\bigr)(s,u,\vartheta).
$$
By \eqref{blowup-0}, it can be derived that
\begin{equation}\label{mathM-0}
\begin{aligned}
\mu&=1-\frac{k+1}{2}\Xi\delta^{-\varepsilon_k}(s-t_0)+o(1)\quad
&&\text{for $\varepsilon_k^*<\varepsilon_0<1$ and $s-t_0=O(\delta^{\varepsilon_k})$},\\
\mu&=1-\frac{(k+1)\Xi}{k-1}\bigl(1-s^{-(k-1)/2}\bigr)+o(1)\quad
&&\text{for $\varepsilon_0=\varepsilon_k^*$ and $s$ in any fixed bounded interval},
\end{aligned}
\end{equation}
where the uniform error estimate $o(1)$ in \eqref{mathM-0} will be established
in Sections~\ref{YCao-3}-\ref{ert}.
From \eqref{mathM-0}  together with $\Xi>0$, Theorem \ref{maintheorem} is shown.

\vskip 0.2 true cm

{\bf $\bullet$} To prove the crucial estimates in \eqref{bootstrap-assumptions-Y}, we shall
give the following bootstrap assumptions for $\vp_{\kappa}$ ($\kappa=0,1,2$), $t_0\leq t \leq T_{\delta}$ and
$|\alpha|\le N$ with a suitably large integer $N$:
\begin{equation}\label{bootstrap-assumptions-0}
\begin{split}
&\delta^{l+m\varepsilon_k} \| Z^{\alpha} \vp_{\gamma} \|_{L^{\infty}}+\|\slashed\triangle\vp_{\gamma}\|_{L^{\infty}}+\delta^{\varepsilon_k}\|\slashed d\mathring L\vp_\gamma\|_{L^\infty}+\|\slashed d\vp_\gamma\|_{L^\infty} \leq M \delta^{1-\varepsilon_0},\\
&\delta^{l+m\varepsilon_k} \|Z^{\alpha}\phi\|_{L^\infty}+\|\slashed d\phi\|_{L^\infty}\leq M\delta^{2-\varepsilon_0},
\end{split}
\end{equation}
where $\slashed{d}$ stands for the restriction of the differential operator $d$ to the circle $\Bbb S$,
$M>0$ is a suitably chosen constant, and the meanings of $l, m$ and $Z$ are as in \eqref{bootstrap-assumptions-Y}.
To close the bootstrap assumption \eqref{bootstrap-assumptions-0} up to the $N$-th order derivatives,
similarly to the treatment in \cite{Ding5}, we first establish  the precise estimates on $\operatorname{tr}\check{\mathscr{X}}$ and $\check{\mathscr{X}}$. Subsequently,  for $t_0\leq t \leq T_{\delta}$ and
$|\alpha|\le N-2$, the following estimates are derived:
\begin{equation}\label{bootstrap-assumptions-1}
\begin{split}
	&\delta^{l+m\varepsilon_k} \| Z^{\alpha} \vp_{\gamma} \|_{L^{\infty}(\Sigma_s^u)}+\|\slashed\triangle\vp_{\gamma}\|_{L^{\infty}(\Sigma_s^u)}+\delta^{\varepsilon_k}\|\slashed d\mathring L\vp_\gamma\|_{L^\infty(\Sigma_s^u)}+\|\slashed d\vp_\gamma\|_{L^\infty(\Sigma_s^u)} \leq C\delta^{1-\varepsilon_0},\\
	&\delta^{l+m\varepsilon_k} \|Z^{\alpha}\phi\|_{L^\infty(\Sigma_s^u)}+\|\slashed d\phi\|_{L^\infty(\Sigma_s^u)}\leq C\delta^{2-\varepsilon_0},
\end{split}
\end{equation}
where the positive constant $C$ is independent of $M$.
Based on \eqref{bootstrap-assumptions-1} together with some precise  smallness or largeness orders of
$\delta$ for the related quantities,
the bootstrap assumptions for the $(N-1)$-th order and $N$-th order derivatives of $\varphi_\kappa$ can be closed by the energy estimates
for \eqref{quasi} since the resulting constants only depend on $C$ rather than $M$ (see details in Sections \ref{EE}-\ref{ert}).
In this process, a new corrector $E_0$ of the good unknown $\check{\mathscr{X}}_{XX}$ is introduced and further the resulting transport equation
of $\operatorname{tr}\check{\mathscr X}-E_0$ is derived (see Lemma \ref{lem:null-transport}).
It is emphasized that although the proof procedure for \eqref{bootstrap-assumptions-0} is somewhat analogous to that
in \cite{Ding5} or \cite{L-Y, M-Y}, all the details have to be given since the null condition structures of equation \eqref{quasi},
the essential roles of $\check{\mathscr{X}}_{XX}$, the new decomposition of
$G_{\mathring L\mathring L}^{\gamma}\slashed{\triangle}\varphi_{\gamma}$
and the optimality of the critical power $\varepsilon_k^*$ should be presented to the fullest extent.

\medskip

The paper is organized as follows.
In Section \ref{section-local}, some basic properties of the local solution $\phi$ to \eqref{quasi} are established under the
$k$-th order null condition.
In Section \ref{Section 2}, the necessary geometric preliminaries including the introduction of an optical function, the inverse foliation density $\mu$
and a null frame are stated. Meanwhile, by use of the good unknown in \eqref{FG-Y2}, we derive a series of important formulas for the covariant derivatives of the null frame and
for the deformation tensors as well as the evolution equation of $\mu$.
In Section \ref{Section 3}, the bootstrap assumptions are given,
and subsequently the uniform $L^{\infty}$ bounds for the lower order derivatives of
$\vp$ are established. These crucial uniform bounds, which are independent of the constant in the bootstrap assumptions,
will be applied to the later analysis in the next sections.
In Section \ref{YCao-3}, based on the estimates in Section \ref{Section 3},
as in  \cite{Ch1, Sp}, a detailed study on $\mu$ is conducted and then shock formation is shown.
In Sections \ref{EE}-\ref{ert}, through deriving the energy estimates for the higher-order derivatives
of $\vp$ and utilizing the Sobolev embedding theorem, the bootstrap assumptions in Section \ref{Section 3} are then closed,
which means that the proof of Theorem \ref{maintheorem} is eventually completed.

\subsection{Notation}
Throughout the paper, unless stated otherwise, Greek indices $\{\al, \beta, \cdots\}$, corresponding to the
spacetime coordinates, are chosen in $\{0, 1, 2\}$; Latin indices $\{i, j, k, \cdots\}$, corresponding to the spatial
coordinates, take values in $\{1, 2\}$; and the Einstein summation convention will be used.
In addition, the convention $f\lesssim g$ means that there exists a generic positive constant $C$ such that $f\le Cg$.

For convenience, set $g^{\al\beta}=g^{\al\beta}(\p\phi)$ and let $(g_{\al\beta})$ be the inverse of the matrix $(g^{\al\beta})$.

Introduce the following notations:
\begin{align*}
&t_0=1+2\delta,\quad \varepsilon_k^*=\f{k}{k+1}, \quad\varepsilon_k=(k+1)(\ve0-\varepsilon_k^*),\\
&\omega^i=\omega_i,\quad i=1,2,\quad \omega_\perp=(-\omega^2, \omega^1),\\
&L=\p_t+\p_r,\ \underline L=\p_t-\p_r,\ \Omega=\epsilon_i^jx^i\p_j,\\
&S=t\p_{t}+r\p_r=\f{t-r}{2}\underline L+\f{t+r}{2}L,\\
&H_i=t\p_{i}+x^i\p_{t}=\omega^i\big(\f{r-t}{2}\underline L+\f{t+r}{2}L\big)+\f{t\omega^i_\perp}{r}\Omega,\quad i=1,2,\\
&\Sigma_t=\{(t,x): x\in\mathbb R^2\},\\
&\varphi=(\varphi_0, \varphi_1, \varphi_2)=(\p_0\phi, \p_1\phi, \p_2\phi),
\end{align*}
where $\epsilon_{1}^2=1$, $\epsilon_{2}^1=-1$, $\epsilon_i^i=0$.

\section{Local solution}\label{section-local}
We now focus on the construction of a local solution to \eqref{quasi} so that the better smallness properties of $\phi$
can be established near the outermost characteristic light cone.
This approach is inspired by \cite{Ding4}.

\begin{theorem}\label{Local}
Under the assumptions \eqref{gh}--\eqref{g00}, when $\delta>0$ is small, \eqref{quasi} with \eqref{id}
admits a local smooth solution $\phi\in C^\infty([1, t_0]\times\mathbb R^2)$. Moreover,
for $m,p\in\mathbb N_0$, $\kappa\in\mathbb N_0^3$, it holds that
\begin{equation}\label{initial}
|L^m\p^\kappa\Omega^p\phi(t_0,x)|\lesssim\delta^{2-\ve0-m\varepsilon_k-|\kappa|},\quad r \in[1-2\delta, 1+2\delta].
\end{equation}
In addition, at any point $(t_0, r\omega)$, we have
\begin{align}
&\delta^{\varepsilon_0-1}\underline L\phi(t_0, r\omega) = 2\phi_1\big(\tfrac{r-t_0}{\delta}, \omega\big) + O(\delta^{(k+1)(1-\ve0)}),\label{Lphi}\\
&\delta^{\varepsilon_0}\underline L\partial_0\phi(t_0, r\omega) = -2\partial_\sigma\phi_1\big(\tfrac{r-t_0}{\delta}, \omega\big) + O(\delta^{(k+1)(1-\ve0)}).\label{y0}
\end{align}
\end{theorem}

\begin{remark}
Comparing \eqref{initial} with \cite[(2.28)]{Ding4}, for the action of $L^m$ on $\phi$, there appears an extra large factor $\delta^{-m\varepsilon_k}$ on the right-hand side of \eqref{initial} due to $\varepsilon_0\ge \varepsilon_k^*$,
while it holds that $|L^m\p^\kappa\Omega^p\phi(t_0,x)|\lesssim\delta^{2-\ve0-|\kappa|}$ for the latter because of $\varepsilon_0<\varepsilon_k^*$.
Therefore, here the $L^\infty$-norm of the solution $\phi$ grows under the action of $L$, which should be
paid special attention when the vector field $L$ acts on some related terms.
\end{remark}

\begin{proof}
Denote by $Z_g$ any fixed vector field in $\{S, H_1, H_2\}$.
Suppose that for $1\leq t\leq t_0$
and $N_0\in\mathbb N_0$ with $N_0\ge 6$, one has
	\begin{equation}\label{ba}
	|\p^\kappa\Omega^pZ_g^\nu\phi|\leq\delta^{2-\epsilon-|\nu|\varepsilon_k-|\kappa|}\quad (|\kappa|+p+|\nu|\leq N_0,\quad |\nu|\leq 1),
	\end{equation}
where $\epsilon\in(\ve0, 1)$ is a fixed constant.
Define the following energy for \eqref{quasi} and for $n\in\mathbb N_0$,
$$
M_n(t)=\sum_{\tiny\begin{array}{c}|\kappa|+p+|\nu|\leq n\\|\nu|\leq 1\end{array}}\delta^{2|\kappa|+2|\nu|\varepsilon_k}\|\p\p^\kappa\Omega^pZ_g^\nu\phi(t,\cdot)\|_{L^2(\mathbb R^2)}^2.
$$

Let $w=\delta^{|\kappa|+|\nu|\varepsilon_k}\p^\kappa\Omega^pZ_g^\nu\phi\quad(|\kappa|+p+|\nu|\leq 2N_0-2)$. It follows from \eqref{quasi}
and integration by parts that
\begin{equation}\label{Y-4}
\begin{split}
&\int_1^t\int_{\Sigma_\tau}(\p_twg^{\al\beta}\p_{\al\beta}^2w)(\tau,x) dxd\tau\\
=&\int_{\Sigma_{t}}\big(g^{0\al}\p_\al w\p_t w-\f12g^{\al\beta}\p_\al w\p_\beta w\big)(t,x)dx-\int_{\Sigma_{1}}\big(g^{0\al}\p_\al w\p_t w-\f12g^{\al\beta}\p_\al w\p_\beta w\big)(1,x)dx\\
&+\int_1^t\int_{\Sigma_\tau}\big(-(\p_\al g^{\al\beta})\p_\beta w\p_t w+\f12(\p_tg^{\al\beta})\p_\al w\p_\beta w\big)(\tau,x) dxd\tau
\end{split}
\end{equation}
with
\begin{equation}\label{w}
\begin{split}
|g^{\al\beta}&\p_{\al\beta}^2w|\lesssim\delta^{|\kappa|+|\nu|\varepsilon_k}\sum_{\tiny\begin{array}{c}
	|\kappa_1|+|\kappa_2|\leq|\kappa|,\ p_1+p_2\leq p,\\|\nu_1|+|\nu_2|\leq|\nu|,\\ |\kappa_2|+p_2+|\nu_2|<|\kappa|+p+|\nu|\end{array}}|\p^{\kappa_1}\Omega^{p_1}Z_g^{\nu_1}	(g^{\al\beta}-m^{\alpha\beta})||\p^{\kappa_2}\Omega^{p_2}Z_g^{\nu_2}\p^2\phi|.
\end{split}
\end{equation}
By \eqref{ba}-\eqref{w} and Gronwall's inequality, we obtain that for $1\leq t\leq t_0$,
$$
M_{2N_0-2}(t)\lesssim M_{2N_0-2}(1)\lesssim\delta^{3-2\varepsilon_0}.
$$
In addition, by the following Sobolev embedding theorem on the circle $\mathbb S_r$
	(with center at the origin and radius $r$), one has that
	$$
	|w(t,x)|\lesssim \f{1}{\sqrt{r}}\|\Omega^{\leq 1}w\|_{L^2(\mathbb S_r)}.
	$$
This, together with $r\sim 1$ for $t\in [1,t_0]$ and $(t,x)\in \textrm{supp}\ w$, yields that when $|\kappa|+p+|\nu|\leq N_0$,
	\begin{equation}\label{le}
	|\p^\kappa\Omega^pZ_g^\nu\phi(t,x)|\lesssim \|\Omega^{\leq 1}\p^\kappa\Omega^pZ_g^\nu\phi\|_{L^2(\mathbb S_r)}\lesssim\delta^{1/2}\|\p\Omega^{\leq 1}\p^\kappa\Omega^pZ_g^\nu\phi\|_{L^2(\Sigma_{t})}\lesssim\delta^{2-|\kappa|-\varepsilon_0-|\nu|\varepsilon_k},\quad
	\end{equation}
	where $N_0\geq 6$ has been used. Therefore, \eqref{ba} can be closed
	for suitably small $\delta>0$ and $\epsilon>\varepsilon_0$.

Next we improve the $L^\infty$ estimate of $\phi(t_0,x)$ on some special space domains. To this end,
one rewrites \eqref{quasi} as
\begin{equation}\label{me}
L\underline L\phi=\f{1}{2r}L\phi-\f{1}{2r}\underline L\phi+\f{1}{r^2}\Omega^2\phi+g^{\al\beta,
	\gamma_1,\cdots,\gamma_k}(\p_{\gamma_1}\phi)\cdots(\p_{\gamma_k}\phi)\p^2_{\al\beta}\phi+h^{\al\beta}(\p\phi)\p^2_{\al\beta}\phi.
\end{equation}

Applying the operator $L\bar\p^{\iota}\Omega^p$ to both sides of \eqref{me}, where $\bar\p\in\{\p_t,\p_r\}$ and $\iota\in\mathbb N_0^2$,
we can obtain an expression of $\underline LL^2\bar\p^{\iota}\Omega^p\phi$ by $\underline LL^2\bar\p^{\iota}\Omega^p=L\bar\p^{\iota}\Omega^pL\underline L$  and a direct computation.
Together with \eqref{le}, this implies that for $|\iota|+p\leq 2N_0-6$,
\begin{equation}\label{nonull}
	|\underline L(L^2\bar\p^{\iota}\Omega^p\phi)|\lesssim \delta^{-2\varepsilon_k-|\iota|},
\end{equation}
here  the null condition \eqref{null} is not used. Due to the vanishing property of $\phi$ on $\{(t, x): t\ge 1, t=r\}$,
integrating \eqref{nonull} along the integral curves of $\underline L$ yields that for $(t,r)\in D^1=\{(t,r): 1\le t\le t_0, 2-t\le r\le t\}$ and for $|\iota|+p\leq 2N_0-6$,
\begin{equation}\label{D}
|L^2\bar\p^{\iota}\Omega^p\phi(t,x)|\lesssim\delta^{1-2\varepsilon_k-|\iota|}.
\end{equation}
Using the equation for $\underline LL^2\bar\p^{\iota}\Omega^p\phi$ again, together with \eqref{null} and
\eqref{D}, gives $|\underline L(L^2\bar\p^{\iota}\Omega^p\phi) |\lesssim\delta^{1-\varepsilon_0-2\varepsilon_k-|\iota|}$ for $\varepsilon_k^*\leq\varepsilon_0<1$ and $|\iota|+p\leq2N_0-7$. This means that for $(t,x)\in D^1$,
\begin{equation}\label{D-1}
|L^2\bar\p^{\iota}\Omega^p\phi(t,x)|\lesssim\delta^{2-\varepsilon_0-2\varepsilon_k-|\iota|},\quad|\iota|+p\leq2N_0-7.
\end{equation}
Then it follows from an induction argument and \eqref{me} that for $(t,x)\in D^1$,
 \begin{equation}\label{oi}
|L^m\bar\p^\iota\Omega^p\phi(t,x)|\lesssim\delta^{2-\varepsilon_0-m\varepsilon_k-|\iota|},\quad 3m+|\iota|+p\leq 2N_0-4,
\end{equation}
which yields \eqref{initial} due to $\p_i=\omega^i\p_r+\frac{\omega_\perp^i}r\Omega$ ($i=1,2$).

Note that both the points $(1,(r_0-t_0+1)\omega)$ and $(t_0, r_0\omega)$ are on the line $\{(t,r\omega): r-t=r_0-t_0\}$.
Then, by \eqref{me}, \eqref{le} and \eqref{id}, we arrive at
\begin{equation*}
\begin{split}
&\underline{L}\phi(t_0,r_0\omega)
=\underline{L}\phi(1,(r_0-t_0+1)\omega)+\int_{(1,(r_0-t_0+1)\omega)}^{(t_0,r_0\omega)}L\underline{L}\phi\\
=&\delta^{1-\varepsilon_0}\big(\phi_1\big(\frac{r_0-t_0}{\delta}, \omega\big)-\partial_\sigma\phi_0\big(\frac{r_0-t_0}{\delta}, \omega\big)\big)+O(\delta^{(k+2)(1-\varepsilon_0)})\\
=&2\delta^{1-\varepsilon_0}\phi_1\big(\frac{r_0-t_0}{\delta}, \omega\big)+O(\delta^{(k+2)(1-\varepsilon_0)})
\end{split}
\end{equation*}
and
\begin{equation*}
\begin{split}
&\underline{L}\partial_{0}\phi(t_0,r_0\omega)
=\underline{L}\partial_{0}\phi(1,(r_0-t_0+1)\omega)+\int_{(1,(r_0-t_0+1)\omega)}^{(t_0,r_0\omega)}L\underline{L}\partial_{0}\phi\\
=&\delta^{-\varepsilon_0}\big(\partial_\sigma^2\phi_0\big(\frac{r_0-t_0}{\delta}, \omega\big)-\partial_\sigma\phi_1\big(\frac{r_0-t_0}{\delta}, \omega\big)\big)+O(\delta^{(k+1)(1-\varepsilon_0)-\varepsilon_0})\\
=&-2\delta^{-\varepsilon_0}\partial_\sigma\phi_1\big(\frac{r_0-t_0}{\delta}, \omega\big)+O(\delta^{(k+1)(1-\varepsilon_0)-\varepsilon_0}).
\end{split}
\end{equation*}
Hence, \eqref{Lphi} and \eqref{y0} are shown.
\end{proof}

\section{Some preliminaries}\label{Section 2}
In this section, we present some preliminary concepts from Lorentzian geometry, which are illustrated  in \cite{Ch1, Sp}.
For the convenience of readers, we still provide a detailed exposition.

\begin{definition}\label{Definition optical function}
	A $C^1$ function $u(t,x)$ is called the optical function of \eqref{quasi} if $u(t,x)$ satisfies the eikonal equation
	\begin{equation}\label{eikonal equation}
	g^{\al\be}\p_{\al}u\p_{\be}u=0
	\end{equation}
	with the initial data $u(t_0,x)=t_0-r$ and the condition $\p_{0}u>0$.
\end{definition}

According to the optical function $u$, one can define the inverse foliation density $\mu$ as in \cite{Ch1} and \cite{Sp}:
\begin{equation}\label{inverse foliation density}
\mu=-\f{1}{g^{0\al}\p_{\al}u}.
\end{equation}
It is pointed out that  the blowup of the solution $\phi$ will happen  when $\mu\rightarrow 0+$.

Note that
\begin{equation}\label{Lr}
\tilde L=-\na u=-g^{\al\be}\p_{\al}u\p_{\be}
\end{equation}
is a tangent vector field of the outgoing light cone $\{u=C\}$ and $\tilde L t=\mu^{-1}$ holds. Then we rescale $\tilde L$ as
\begin{equation}\label{L}
\Lr=\mu\tilde L.
\end{equation}
To obtain an approximate vector  field of the incoming light cone,
one sets $\tilde T=-g^{\al 0}\p_\al-\mathring L$, which is near $-\p_r$ for $t=t_0$. In order
to define a null frame, we now set
$$
T=\mu\tilde T,\quad\mathring{\underline L}=\mu\mathring L+2T,
$$
where $\mathring{L}$ and $\mathring{\underline L}$ are two vector fields in the null frame. In addition, the third vector field
$X$ in the null frame can be constructed by using $\mathring L$. Indeed, we extend the local coordinate $\theta$ on $\mathbb{S}$
$(x=(r cos\theta, r sin\theta))$ by $\mathring L\vartheta=0$ with $\vartheta|_{t=t_0}=\th$ and
subsequently set $X=\f{\p}{\p\vartheta}$.

\begin{lemma}\label{nullframe}
$\{\mathring L, \mathring{\underline L}, X\}$ constitutes a null frame with respect to the metric $(g_{\al\beta})$, and admits the following identities:
\begin{align*}
g(\mathring L,\mathring L)=g(\mathring{\underline L}, \mathring{\underline L})=g(\mathring L, X)=g(\mathring{\underline L}, X)=0,
g(\mathring L, \mathring{\underline L})=-2\mu.
\end{align*}
In addition,
\begin{align*}
\mathring L t=1,\quad\mathring L u=0,\quad
\mathring{\underline L}t=\mu,\quad \mathring{\underline L} u=2.
\end{align*}
And
\begin{align}\label{LTTT}
g(\mathring L, T)=-\mu,\quad g(T,T)=\mu^2,
\end{align}
\begin{align}\label{T}
Tt=0,\quad T u=1.
\end{align}
\end{lemma}

As in \cite{Sp,M-Y}, one can perform the change of coordinates:
$(t, x^1, x^2)\longrightarrow (s, u, \vartheta)$ near $\{(t,x) : r=t\}$ with
\begin{equation}\label{H0-7}
(s, u, \vartheta)=(t, u(t,x), \vartheta(t,x)).
\end{equation}
In the new coordinate $(s, u, \vartheta)$, the following subsets are introduced:

\begin{definition}
\begin{equation*}
\begin{split}
&\Sigma_s=\{(s',u',\vartheta): s'=s, u'>0\},\\
&\Sigma_s^{u}=\{(s',u',\vartheta): s'=s,0\leq u'\leq  u\},\quad u\in [0, 4\delta],\\
&C_{u}=\{(s',u',\vartheta): s'\geq 1+2\delta, u'=u\},\\
&C_{u}^s=\{(s',u',\vartheta): 1+2\delta\leq s'\leq s, u'=u\},\\
&S_{s, u}=\Sigma_s\cap C_{u},\\
&D^{s, u}=\{(s', u',\vartheta): 1+2\delta\leq s'<s, 0\leq u'\leq u\}.
\end{split}
\end{equation*}
\end{definition}

Next, we introduce the following geometric notations which will be used frequently.
\begin{definition} For the metric $g$ on the spacetime,
\begin{itemize}
\item $\underline g=(g_{ij})$ is defined as the induced metric of $g$ on $\Sigma_s$, i.e.,
$\underline g(U,V)=g(U,V)$ for any tangent vectors $U$ and $V$ of $\Sigma_s$;
\item $\slashed \Pi_\al^\beta=\delta_\al^\beta-\delta_\al^0\mathring L^\beta+\mathring L_\al\tilde{T}^\beta$ is the
projection tensor field on $S_{s,u}$ of type $(1,1)$, where $\delta_\al^\beta$ is the Kronecker delta;
\item For any $(m,n)$-type
spacetime tensor field
$\eta$, $\slashed\eta=\slashed\Pi\eta$ is the tensor field on $S_{s,u}$  with the components $$\slashed\eta^{\al_1\cdots\al_m}_{\beta_1\cdots\beta_n}:=(\slashed\Pi\eta)^{\al_1\cdots\al_m}_{\beta_1\cdots\beta_n}
=\slashed\Pi_{\beta_1}^{\beta_1'}\cdots\slashed\Pi_{\beta_n}^{\beta_n'}
\slashed\Pi_{\al_1'}^{\al_1}\cdots\slashed\Pi_{\al_m'}^{\al_m}\eta^{\al_1'\cdots\al_m'}_{\beta_1'\cdots\beta_n'}.$$
In particular, $\slashed g=(\slashed g_{\al\beta})$ is the induced metric of $g$ on $S_{s,u}$;
\item $\slashed g^{XX}$ is defined as the inverse of $\slashed g_{XX}$ with $\slashed g_{XX}=g(X, X)$;
\item $\mathscr D$ and $\slashed\nabla$ denote the Levi-Civita connections of $g$ and $\slashed g$, respectively;
\item $\Box_g=g^{\al\beta}\mathscr{D}^2_{\al\beta}$, \quad $\slashed\triangle=\slashed g^{XX}\slashed\nabla^2_{X}$;
\item $\mathcal L_V\eta$ is the Lie derivative of $\eta$ with respect to $V$
and $\slashed{\mathcal L}_V\eta=\slashed\Pi(\mathcal L_V\eta)$ for any tensor field $\eta$ and vector $V$;
\item For any $(m,n)$-type spacetime tensor field $\eta$,
$$
|\eta|^2=g_{\al_1\al_1'}\cdots g_{\al_m\al_m'}g^{\beta_1\beta_1'}\cdots g^{\beta_n\beta_n'}\eta_{\beta_1\cdots\beta_n}^{\al_1\cdots\al_m}\eta_{\beta_1'\cdots\beta_n'}^{\al_1'\cdots\al_m'};
$$
\item $\text{div}U=\mathscr D_\al U^\al$ for any vector field $U$; $\slashed{\text{div}}Y=\slashed{\nabla}_X Y^X$ and
    $\slashed{\text{div}}\kappa=\slashed\nabla^X\kappa_X$ are angular divergences for any vector field $Y$ and one-form $\kappa$ on $S_{s,u}$, respectively;
		\item If $\eta$ is a $(0,2)$-type spacetime tensor, $\Lambda$ is a 1-form, $U$ and $V$ are vector fields, then the contraction
		of $\eta$ with respect to $U$ and $V$ is defined as
		$$
		\eta_{UV}=\eta_{\al\beta}U^{\al}V^{\beta},
		$$
		and the contraction of $\Lambda$ with respect to $U$ is
		\[
		\Lambda_U=\Lambda_{\al}U^{\al};
		\]
		\item If $\eta$ is a $(0,2)$-type tensor on $S_{s,u}$, the trace of $\eta$ is defined as
		\[
		\text{tr}\eta=\slashed g^{XX}\eta_{XX}.
		\]
	\end{itemize}
\end{definition}

Under the frame $\{\mathring{L}, \mathring{\underline{L}}, X\}$, the \textit{second fundamental forms} $\chi$ and ${\sigma}$ can be defined as
\begin{equation}\label{chith}
\chi_{XX} = g(\mathscr{D}_X \mathring{L}, X), \quad{\sigma}_{XX} = g(\mathscr{D}_X \tilde{T}, X).
\end{equation}
Define the \textit{one-form tensors} $\zeta$ and $\xi$ as
\begin{equation}\label{zetaeta}
\zeta_X = g(\mathscr{D}_X \mathring{L}, \tilde{T}), \quad \xi_X = -g(\mathscr{D}_X T, \mathring{L}).
\end{equation}
Then $\mu \zeta_X = -X\mu + \xi_X$ holds.

For any vector field $V$, denote its associated \textit{deformation tensor} by
\begin{equation}\label{dt}
{}^{(V)}\pi_{\alpha\beta} = g(\mathscr{D}_{\alpha} V, \partial_{\beta}) + g(\mathscr{D}_{\beta} V, \partial_{\alpha}).
\end{equation}

Let
\begin{equation}\label{FG}
G_{\alpha\beta}^{\gamma}=\partial_{\varphi_{\gamma}} g_{\alpha\beta}.
\end{equation}
For any vector fields $U = U^{\alpha} \partial_{\alpha}$ and $V = V^{\alpha} \partial_{\alpha}$, set
\begin{equation*}
G_{UV}^{\gamma} = G_{\alpha\beta}^{\gamma} U^{\alpha} V^{\beta}.
\end{equation*}

On the initial hypersurface $\Sigma_{t_0}^{4\delta}$, one has that
\begin{equation}\label{errorv-0}
\begin{split}
&\tilde T^i=-\f{x^i}{r}+O(\delta^{k(1-\varepsilon_0)}),\quad
\mathring L^0=1, \quad \mathring L^i=\f{x^i}{r}+O(\delta^{k(1-\varepsilon_0)}),\\
&\chi_{XX}-\f12G_{XX}^\gamma\mathring L\vp_\gamma=\f1r\slashed g_{XX}+O(\delta^{k(1-\varepsilon_0)}).
\end{split}
\end{equation}
Note that on $\Sigma_{t_0}$, $r$ is just $t_0-u$. For $s\geq t_0$, we
define the ``{\it{error vectors}}" with the components being
\begin{equation}\label{errorv}
\begin{split}
&\check{L}^0=0,\ \check{L}^i=\mathring L^i-\f{x^i}{\varrho},\ \check{T}^i=\tilde T^i+\f{x^i}{\varrho},\ \check\varrho=\f{r}{\varrho}-1.\\
\end{split}
\end{equation}
In addition, we use the following crucial ``good unknown" which was first introduced in \cite{Ding5},
\begin{equation}\label{good-unknown}
\begin{split}
&\check{\mathscr{X}}_{XX}=\chi_{XX}- \frac{1}{2}G_{XX}^\gamma\mathring{L}\vp_\gamma-\f{1}{\varrho}\slashed g_{XX},
\end{split}
\end{equation}
where $\varrho=s-u$. Note that $\check{\mathscr{X}}_{XX}$ will appear in many important relations in the paper (see \eqref{fequation}-\eqref{H},
the expression of the last term in \eqref{EI} and other places) and further $\check{\mathscr{X}}_{XX}=O(\delta^{k(1-\varepsilon_0)})$
will be expected. However, $|G_{XX}^\gamma\mathring{L}\vp_\gamma| \lesssim\delta^{k(1-\varepsilon_0)-\varepsilon_k}$ is the optimal estimate,
which means that the introduction of the good combination $\chi_{XX}- \frac{1}{2}G_{XX}^\gamma\mathring{L}\vp_\gamma$ is very necessary due to $\varepsilon_k>0$ for $\varepsilon_0\ge\varepsilon_k^*$.

Note that $\vartheta$ denotes the coordinate on $S_{s,u}$. Then, in the new coordinate system $(s, u, \vartheta)$,
one has $\mathring{L}=\frac{\partial}{\partial s}$. Additionally, it follows from \eqref{T} that
\[ T = \frac{\partial}{\partial u} - \Theta^X X \]
for some smooth function $\Theta^X$. Furthermore, an analysis analogous to that in Lemma 3.66 of \cite{Sp} yields

\begin{lemma}
	In $D^{s,u}$, the Jacobian determinant of the map $(s, u, \vartheta)\rightarrow
(x^0, x^1, x^2)$ is given by
\begin{equation}\label{Jacobian}
\det\f{\p(x^0, x^1, x^2)}{\p(s, u, \vartheta)}=\mu(\det\underline g)^{-1/2}\sqrt{\slashed g_{XX}}.
\end{equation}
\end{lemma}

On $\mathbb{S}$, the vector field $\Omega=x^1\p_2-x^2\p_1$
is taken as the tangent derivative.
In order to project $\Omega$ on $S_{s, u}$, as in (3.39b) of \cite{Sp}, we define
\begin{equation*}
R=\slashed\Pi\Omega,\quad \slashed d=\slashed\Pi d
\end{equation*}
to be the rotation vectorfield and the differential of $S_{s, u}$, respectively.
In the frame $\{\mathring{L}, T, X\}$, the metric $g^{\alpha\beta}$ and $\partial_\alpha$ admit the following expressions:
\begin{align}
g^{\alpha\beta} &= -\mathring{L}^\alpha\mathring{L}^\beta - \tilde{T}^\alpha\mathring{L}^\beta - \mathring{L}^\alpha\tilde{T}^\beta + (\slashed{d}_X x^\alpha)(\slashed{d}^X x^\beta), \label{gab} \\
\partial_\alpha &= \delta_\alpha^0 \mathring{L} - \mu^{-1} \mathring{L}_\alpha \, T + g_{\alpha i} (\slashed{d}^X x^i) \, X. \label{pal}
\end{align}
In addition, it holds that
\begin{equation}\label{R}
R=(\delta_j^i+g_{\al j}\mathring L^\al\tilde T^i){\Omega}^j\p_i
=\Omega-g_{aj}\tilde{T}^a{\Omega}^j\tilde{T},
\end{equation}
where $\delta_j^i$ is the Kronecker delta symbol. For brevity, set
\begin{equation}\label{omega}
\upsilon=g_{ab}\tilde{T}^a{\Omega}^b=g_{aj}\check{T}^a\epsilon_i^jx^i-(g_{aj}-m_{aj})\f{\epsilon_i^j x^ix^a}\varrho.
\end{equation}
Then
$$R=\Omega-\upsilon\tilde{T}.$$

With the help of $\check{\mathscr{X}}$, $\sigma$, $\zeta$ and $\xi$ (as defined earlier), one can derive some basic equalities in the frame $\{\mathring{L}, \mathring{\underline{L}}, X\}$ or $\{T, \mathring{\underline{L}}, X\}$. Making full use of the expression of $\check{\mathscr{X}}$
and performing a computation analogous to that in \cite{Ding4}, we have

	\begin{lemma}\label{Lem3.3}
		The following identities hold:
		\begin{itemize}[leftmargin=*, labelsep=6pt, itemsep=10pt]
			\item the function $\mu$ satisfies
			\begin{equation}\label{lmu}
			\mathring{L}\mu = -\frac{1}{2}\mu G_{\mathring{L}\mathring{L}}^\gamma \mathring{L}\varphi_\gamma
			- \mu G_{\tilde{T}\mathring{L}}^\gamma \mathring{L}\varphi_\gamma + \frac{1}{2}G_{\mathring{L}\mathring{L}}^\gamma T\varphi_\gamma;
			\end{equation}
			
			\item the quantities $\zeta$ and $\sigma$ satisfy
			\begin{align}
			&\zeta_X = -\frac{1}{2}\big\{ G_{\tilde{T}\mathring{L}}^\gamma \slashed{d}_X\varphi_\gamma + G_{\tilde{T}\tilde{T}}^\gamma \slashed{d}_X\varphi_\gamma
			- G_{X\tilde{T}}^\gamma \mathring{L}\varphi_\gamma + G_{X\mathring{L}}^\gamma \tilde{T}\varphi_\gamma \big\},\label{zeta}\\
			&\sigma_{XX} = -G_{X\mathring{L}}^\gamma \slashed{d}_X\varphi_\gamma - G_{X\tilde{T}}^\gamma \slashed{d}_X\varphi_\gamma
			+ \frac{1}{2}G_{XX}^\gamma \tilde{T}\varphi_\gamma - \check{\mathscr{X}}_{XX} - \frac{\slashed{g}_{XX}}{\varrho};\label{theta}
			\end{align}
			
			\item the components of the vector fields $\mathring{L}$ and $\check{L}$ fulfill
			\begin{align}
			&\mathring{L}\mathring{L}^i = \frac{1}{2}G_{\mathring{L}\mathring{L}}^\gamma \mathring{L}\varphi_\gamma \tilde{T}^i
			- \big\{ G_{X\mathring{L}}^\gamma \mathring{L}\varphi_\gamma - \frac{1}{2}G_{\mathring{L}\mathring{L}}^\gamma \slashed{d}_X\varphi_\gamma \big\}\slashed{d}^X x^i,\label{LL}\\
			&T\check{L}^i = \big\{ \slashed{d}_X\mu - \frac{1}{2}\mu G_{\tilde{T}\tilde{T}}^\gamma \slashed{d}_X\varphi_\gamma
			- G_{X\mathring{L}}^\gamma T\varphi_\gamma \big\}\slashed{d}^X x^i + \frac{1}{2}G_{\mathring{L}\mathring{L}}^\gamma T\varphi_\gamma \tilde{T}^i \nonumber\\
			&\qquad\quad + \big\{ \frac{1}{2}\mu G_{\mathring{L}\mathring{L}}^\gamma \mathring{L}\varphi_\gamma + \mu G_{\tilde{T}\mathring{L}}^\gamma \mathring{L}\varphi_\gamma
			+ \frac{1}{2}\mu G_{\tilde{T}\tilde{T}}^\gamma \mathring{L}\varphi_\gamma \big\}\mathring{L}^i-\f{\mu-1}{\varrho}\tilde T^i-\f{1}{\varrho}\check T^i,\label{TL}\\
			&\mathring{L}\left( \varrho\check{L}^i \right) = \varrho\mathring{L}\mathring{L}^i,\label{LeL}\\
			&\slashed{d}_X\check{L}^i = \operatorname{tr}\check{\mathscr{X}} \slashed{d}_X x^i - \big\{ G_{\mathring{L}\tilde{T}}^\gamma \slashed{d}_X\varphi_\gamma
			+ \frac{1}{2}G_{\tilde{T}\tilde{T}}^\gamma \slashed{d}_X\varphi_\gamma \big\}\tilde{T}^i;\label{deL}
			\end{align}
			
			\item the covariant derivatives of the vector fields $\{T, \mathring{L}, X\}$ are given by
			\begin{equation}\label{cdf}
			\begin{split}
			&\mathscr{D}_{\mathring{L}}\mathring{L} = (\mu^{-1}\mathring{L}\mu)\mathring{L}, \quad
			\mathscr{D}_{T}\mathring{L} = -\mathring{L}\mu\mathring{L} + \xi^X X, \quad
			\mathscr{D}_X\mathring{L} = -\zeta_X\mathring{L} + \operatorname{tr}\chi \, X,\\
			&\mathscr{D}_{\mathring{L}}T = -\mathring{L}\mu\mathring{L} - \mu\zeta^X X, \quad
			\mathscr{D}_T T = \mu\mathring{L}\mu\mathring{L} + (\mu^{-1}T\mu + \mathring{L}\mu)T - \mu(\slashed{d}^X\mu) X,\\
			&\mathscr{D}_X T = \mu\zeta_X\mathring{L} + \mu^{-1}\xi_X T + \mu\operatorname{tr}\sigma \, X,\\
			&\mathscr{D}_X X = \slashed{\nabla}_X X + (\sigma_{XX} + \chi_{XX})\mathring{L} + \mu^{-1}\chi_{XX} T;
			\end{split}
			\end{equation}
			
			\item the covariant derivatives of the vector fields $\{\mathring{\underline{L}}, \mathring{L}, X\}$ are
			\begin{equation}\label{LuL}
			\begin{split}
			&\mathscr{D}_{\mathring{\underline{L}}}\mathring{L} = -\mathring{L}\mu\mathring{L} + 2\xi^X X, \quad
			\mathscr{D}_{\mathring{L}}\mathring{\underline{L}} = -2\mu\zeta^X X,\\
			&\mathscr{D}_{\mathring{\underline{L}}}\mathring{\underline{L}} = (\mu^{-1}\mathring{\underline{L}}\mu + \mathring{L}\mu)\mathring{\underline{L}} - (2\mu \slashed{d}^X\mu) X.
			\end{split}
			\end{equation}
		\end{itemize}
	\end{lemma}

Let $\leftidx{^{(V)}}{\slashed{\pi}}_{UX}= \leftidx{^{(V)}}{\pi}_{UX}$ for $U \in \{\mathring{L}, \underline{\mathring{L}}, X\}$.
By utilizing the expression of $\check{\mathscr{X}}$ and performing the computations analogous to those in \cite[Appendix A]{Ding4},
we obtain the expressions for the deformation tensors
in the frame $\{\mathring{L}, \underline{\mathring{L}}, X\}$ listed in the following lemma.
\begin{lemma}
The deformation tensors corresponding to different vector fields $V$ are given as follows:
\begin{itemize}\label{lem3.4}
	\item For $V = T$,
	\begin{equation}\label{Lpi}
	\begin{split}
	&\leftidx{^{(T)}}{\pi}_{\mathring{L}\mathring{L}} = 0,
	\quad \leftidx{^{(T)}}{\pi}_{T\tilde{T}} = 2T\mu,
	\quad \leftidx{^{(T)}}{\pi}_{\mathring{L}T} = -T\mu,
	\quad \leftidx{^{(T)}}{\slashed{\pi}}_{TX} = 0, \\
	&\leftidx{^{(T)}}{\slashed{\pi}}_{\mathring{L}X} = -2\mu\zeta_X - \slashed{d}_X\mu,
	\quad \leftidx{^{(T)}}{\slashed{\pi}}_{XX} = 2\mu\sigma_{XX};
	\end{split}
	\end{equation}
	
	\item For $V = \mathring{L}$,
	\begin{equation}\label{uLpi}
	\begin{split}
	&\leftidx{^{(\mathring{L})}}{\pi}_{\mathring{L}\mathring{L}} = 0,
	\quad \leftidx{^{(\mathring{L})}}{\pi}_{T\tilde{T}} = 2\mathring{L}\mu,
	\quad \leftidx{^{(\mathring{L})}}{\pi}_{\mathring{L}T} = -\mathring{L}\mu,
	\quad \leftidx{^{(\mathring{L})}}{\slashed{\pi}}_{\mathring{L}X} = 0, \\
	&\leftidx{^{(\mathring L)}}{\slashed\pi}_{TX}=2\mu\zeta_X+\slashed d_X\mu,
	\quad \leftidx{^{(\mathring{L})}}{\slashed{\pi}}_{XX} = 2\chi_{XX};
	\end{split}
	\end{equation}
	
	\item For $V = R$,
	\begin{equation}\label{Rpi}
	\begin{split}
	&\leftidx{^{(R)}}{\pi}_{\mathring{L}\mathring{L}} = 0,
	\quad \leftidx{^{(R)}}{\pi}_{T\tilde{T}} = 2R\mu,
	\quad \leftidx{^{(R)}}{\pi}_{\mathring{L}T} = -R\mu, \\
	&\leftidx{^{(R)}}{\slashed{\pi}}_{\mathring{L}X} = -{R}^X\check{\mathscr{X}}_{XX}
	-\upsilon\big\{G_{\mathring{L}\tilde{T}}^\gamma\slashed{d}_X\varphi_\gamma
	+ \frac{1}{2}G_{\tilde{T}\tilde{T}}^\gamma\slashed{d}_X\varphi_\gamma
	- G_{X\tilde{T}}^\gamma{\mathring{L}}\varphi_\gamma\big\}
	+ \epsilon_i^jg_{ja}\check{L}^i\slashed{d}_Xx^a, \\
	&\leftidx{^{(R)}}{\slashed{\pi}}_{TX} = \mu{R}^X\check{\mathscr{X}}_{XX}
	+ \upsilon\slashed{d}_X\mu
	+ \mu G_{X\mathring{L}}^\gamma R\varphi_{\gamma}
	+ \upsilon\big\{G_{X\tilde{T}}^\gamma T\varphi_\gamma
	- \frac{1}{2}\mu G_{\tilde{T}\tilde{T}}^\gamma\slashed{d}_X\varphi_{\gamma}\big\} \\
	&\qquad\qquad + \mu G_{X\tilde{T}}^\gamma R\varphi_{\gamma}
	+ \mu\epsilon_i^j g_{ja}\check{T}^i\slashed{d}_Xx^a, \\
	&\leftidx{^{(R)}}{\slashed{\pi}}_{XX} = 2\upsilon\big(\check{\mathscr{X}}_{XX}
	+ \frac{\slashed{g}_{XX}}{\varrho}\big)
	+ 2\upsilon\big\{G_{X\mathring{L}}^\gamma\slashed{d}_X\varphi_{\gamma}
	+ G_{X\tilde{T}}^\gamma\slashed{d}_X\varphi_{\gamma}\big\}
	+ G_{XX}^\gamma R\varphi_{\gamma} \\
	&\qquad\qquad + 2\epsilon_i^j\check{g}_{ja}(\slashed{d}_Xx^a)\slashed{d}_Xx^i,
	\end{split}
	\end{equation}
	where $\check{g}_{ia}=g_{ia}-m_{ia}$.
\end{itemize}
\end{lemma}

In addition, the following commutators among vector fields, covariant derivatives and Lie derivatives are
presented in \cite[Appendix A]{Ding4}.

\begin{lemma}\label{com}
	For any vector field $Z \in \{\mathring{L}, T, R\}$, it holds that
	
		\begin{equation}\label{c}
		\begin{split}
		[\mathring{L}, R] &= \leftidx{^{(R)}}{\slashed{\pi}}_{\mathring{L}}{}^{X}X, \quad
		[\mathring{L}, T] = \leftidx{^{(T)}}{\slashed{\pi}}_{\mathring{L}}{}^{X}X, \quad
		[T, R] = \leftidx{^{(R)}}{\slashed{\pi}}_{T}{}^{X}X.
		\end{split}
		\end{equation}
		
		For any smooth function $f$:
		\begin{align}
		\bigl([\slashed{\nabla}^2, \slashed{\mathcal{L}}_Z]f\bigr)_{XX}
		&= \frac{1}{2} \slashed{\nabla}_X\bigl(\operatorname{tr}\,\leftidx{^{(Z)}}{\slashed{\pi}}\bigr) \slashed{d}_Xf, \label{nZf}\\
		[\slashed{\triangle}, Z]f &= \leftidx{^{(Z)}}{\slashed{\pi}}^{XX} \slashed{\nabla}_X^2f
		+ \frac{1}{2} \bigl(\slashed{\nabla}_X \leftidx{^{(Z)}}{\slashed{\pi}}^{XX}\bigr) \slashed{d}_Xf; \label{LZf}
		\end{align}
		
		For any one-form $\eta$ on $S_{s, u}$:
		\begin{equation}\label{nZx}
		\bigl([\slashed{\nabla}_X, \slashed{\mathcal{L}}_Z]\eta\bigr)_X = \frac{1}{2} \slashed{\nabla}_X\bigl(\operatorname{tr}\,\leftidx{^{(Z)}}{\slashed{\pi}}\bigr) \eta_X;
		\end{equation}
		
		For any $(0,2)$-type tensor $\eta$ on $S_{s, u}$:
		\begin{align}
		\bigl([\slashed{\nabla}_X, \slashed{\mathcal{L}}_Z]\eta\bigr)_{XX}
		&= \slashed{\nabla}_X\bigl(\operatorname{tr}\,\leftidx{^{(Z)}}{\slashed{\pi}}\bigr) \eta_{XX}. \label{nZxi}
		\end{align}
	
\end{lemma}

At the end of this section, we introduce some abbreviated integrations over different domains and a norm over $\Sigma_s^u$.

\begin{definition}\label{3.4}
	For any continuous function $f$ and tensor field $\eta$, define
	\begin{align*}
	\int_{S_{s, u}} f &= \int_{S_{s, u}} f \, d\nu_{\slashed{g}}
	= \int_{\mathbb{S}} f(s, u, \vartheta) \sqrt{\slashed{g}_{XX}(s, u, \vartheta)} \, d\vartheta, \quad
	\int_{C^s_u} f = \int_{t_0}^s \int_{S_{\tau, u}} f(\tau, u, \vartheta) \, d\nu_{\slashed{g}} \, d\tau, \\
	\int_{\Sigma_s^u} f &= \int_0^u \int_{S_{s, u'}} f(s, u', \vartheta) \, d\nu_{\slashed{g}} \, du', \quad
	\int_{D^{s, u}} f = \int_{t_0}^s \int_0^u \int_{S_{\tau, u'}} f(\tau, u', \vartheta) \, d\nu_{\slashed{g}} \, du' \, d\tau, \\[6pt]
	\|\eta\|_{s,u} &= \sqrt{\int_{\Sigma_s^u} |\eta|^2}.
	\end{align*}
\end{definition}

\section{Bootstrap assumptions and  $L^{\infty}$ estimates of lower-order derivatives}\label{Section 3}

Motivated by Theorem~2.1 and the geometric definitions
in Section~3, we impose the following bootstrap
assumptions in $D^{s,u}$:
\begin{equation}\label{bootstrap-assumptions}
\begin{split}
	&\delta^{l+m\varepsilon_k} \| Z^{\alpha} \vp_{\gamma} \|_{L^{\infty}(\Sigma_s^u)}+\|\slashed\triangle\vp_{\gamma}\|_{L^{\infty}(\Sigma_s^u)}+\delta^{\varepsilon_k}\|\slashed d\mathring L\vp_\gamma\|_{L^\infty(\Sigma_s^u)}+\|\slashed d\vp_\gamma\|_{L^\infty(\Sigma_s^u)} \leq M \delta^{1-\varepsilon_0},\\
	&\delta^{l+m\varepsilon_k} \|Z^{\alpha}\phi\|_{L^\infty(\Sigma_s^u)}
+\|\slashed d\phi\|_{L^\infty(\Sigma_s^u)}\leq M\delta^{2-\varepsilon_0},
\end{split}
\end{equation}
where $\vp_{\gamma} = \partial_{\gamma} \phi$ ($\gamma = 0, 1, 2$), $t_0 \leq s <  T_\delta$,
$|\alpha| \leq N$ with $N$ being a large positive integer, $M$ is a suitably chosen positive number,
$Z \in \{\varrho\mathring{L}, T, R\}$, and $l$ (or $m$) denotes the number of $T$ (or $\mathring{L}$)
contained in $Z^{\alpha}$.

To complete the proof of \eqref{bootstrap-assumptions}, we first derive the equations for $\vp_{\gamma}$.
By using the expression of $\check{\mathscr{X}}$ and performing the computations analogous to those in
\cite[(2.24), (2.40) and (2.41)]{Ding4}, one has
\begin{equation}\label{ge}
\begin{split}
\mu\Box_g \varphi_\gamma=&\f12(g^{\al\beta}G_{\al\beta}^
{\al'})(\mu\slashed d^X\varphi_{\al'})\slashed d_X\varphi_\gamma+f(\varphi, \mathring L^1, \mathring L^2)\vp^{k-1}\left(
\begin{array}{ccc}
T\varphi\\
\mu\mathring L\varphi\\
\mu(\slashed d_Xx)\slashed d^X\varphi\\
\end{array}
\right)
\left(
\begin{array}{ccc}
\mathring L\varphi\\
(\slashed d_Xx)\slashed d^X\varphi\\
\end{array}
\right)
\end{split}
\end{equation}
and
\begin{equation}\label{fequation}
\mathring{L}\mathring{\underline{L}}\varphi_\gamma + \frac{1}{2\varrho}\mathring{\underline{L}}\varphi_\gamma
= \mu\slashed{\triangle}\varphi_\gamma + H_\gamma,
\end{equation}
where $H_\gamma$ is given by
\begin{equation}\label{H}
\begin{split}
H_\gamma =& -\left(\operatorname{tr}\check{\mathscr X}\right) T\varphi_\gamma + \frac{1}{2\varrho}\mu\mathring{L}\varphi_\gamma
+f\big(\varphi, \mathring{L}^1, \mathring{L}^2, \slashed{d}x, \slashed{g}\big)\vp^{k-1}
\begin{pmatrix}
T\varphi \\
\mu\mathring{L}\varphi \\
\mu\slashed{d}\varphi
\end{pmatrix}
\begin{pmatrix}
\mathring{L}\varphi \\
\slashed{d}\varphi
\end{pmatrix}.
\end{split}
\end{equation}
Here and below, $f$ denotes a generic smooth function of its arguments, and the product
$
\begin{pmatrix} A_1 \\ \vdots \\ A_n \end{pmatrix}
\begin{pmatrix} B_1 \\ \vdots \\ B_m \end{pmatrix}
$
represents all terms of the form $A_iB_j$ for $1 \leq i \leq n$ and $1 \leq j \leq m$.

In order to analyze $\vp_\gamma$ from \eqref{fequation}, we need to estimate $\operatorname{tr}\check{\mathscr X}$ appearing in \eqref{H}. Thus, deriving the equation of $\operatorname{tr}\check{\mathscr X}$ is necessary. In \cite[(2.44)]{Ding4}, the authors have actually provided the equation of $\check{\mathscr X}_{XX} + \frac{1}{2}G_{XX}^\gamma\mathring{L}\vp_{\gamma}$. With the help of this result, we can obtain the equation of $\operatorname{tr}\check{\mathscr X}$.

It is worth noting that in \cite{Ding4}, $\mathring{L}\vp$ and $\slashed{d}\vp$ have the same $L^\infty$ estimate and thus are not distinguished. However, in the present paper, $\mathring{L}\vp$ and $\slashed{d}\vp$ have different magnitudes, so special attention should be paid to this distinction when deriving the equation of $\operatorname{tr}\check{\mathscr X}$. Using
the expression of $\check{\mathscr X}$  and the same method as in \cite[(2.44) and (2.45)]{Ding4}, one can arrive at

\begin{lemma}\label{YHCC-555}
The quantity $\operatorname{tr}\check\chi$
satisfies the following transport identity along $\mathring L$:
	\begin{equation}\label{Lchi'}
	\begin{split}
	\mathring{L}\big(\varrho^2\operatorname{tr}\check{\mathscr X}\big)
	= &-\big(\frac{1}{2}G_{\mathring{L}\mathring{L}}^\al\mathring{L}\varphi_\al
	+ G_{\tilde{T}\mathring{L}}^\al\mathring{L}\varphi_\al
	- G_{X\mathring{L}}^\al\slashed{d}^X\varphi_\al
	+ \slashed{g}^{XX}G_{XX}^\al\mathring{L}\vp_\al\big)
	\bigl(\varrho^2\operatorname{tr}\check{\mathscr X} + \varrho\bigr) \\
	&-\varrho^2G_{X\mathring{L}}^\al\bigl(\slashed{d}^X\mathring{L}\varphi_\al\bigr)
	+ \frac{1}{2}\varrho^2G_{\mathring{L}\mathring{L}}^\al\slashed{\triangle}\varphi_\al
	- \varrho^2\bigl(\operatorname{tr}\check{\mathscr X}\bigr)^2\\
	&+ \varrho^2f\bigl(\varphi, \slashed{d}x, \mathring{L}^1, \mathring{L}^2\bigr)\begin{pmatrix} \slashed d(\vp^{k-1}) \\
\vp^{2k-2}\slashed{d}\varphi \end{pmatrix}
	\begin{pmatrix} \mathring{L}\varphi \\ \slashed{d}\varphi \end{pmatrix}.
	\end{split}
	\end{equation}
The  derivative of $\operatorname{tr}\check\chi$ in the transversal direction $T$ is given by
	\begin{equation}\label{Tchi'}
	\begin{split}
	T\bigl(&\operatorname{tr}\check{\mathscr X}\bigr)=\slashed\triangle\mu-G_{X\mathring L}^\al\slashed d^XT\vp_\al-\f12\mu G_{\tilde T\tilde T}^\al\slashed\triangle\vp_\al-\mu(\operatorname{tr}\check{\mathscr X})^2+\f{\mu-1}{\varrho^2}\\
	&+\f12\bigr(\mu G_{\mathring L\mathring L}^\al\mathring L\vp_\al+2\mu G_{\tilde T\mathring L}^\al\mathring L\vp_\al-G_{\mathring L\mathring L}^\al T\vp_\al-G_{\tilde T\mathring L}^\al T\vp_\al-2\mu G_{X\mathring L}^\al\slashed d^X\vp_\al-2\operatorname{tr}\leftidx{^{(T)}}{\slashed{\pi}}\bigr)\operatorname{tr}\check{\mathscr X}\\
	&-\slashed d^X\mu\bigr(G_{\tilde T\mathring L}^\al\slashed d_X\vp_\al+G_{\tilde T\tilde T}^\al\slashed d_X\vp_\al+\f12G_{XX}^\al\slashed d_X\vp_\al\bigr)\\
	&+f(\varphi, \slashed dx,\slashed d\vp,\slashed\triangle x, \mathring L^1,\mathring L^2)\left(
	\begin{array}{ccc}
	\vp^{k-1}\\
	\vp^{k-2}\slashed d\varphi
	\end{array}
	\right)
	\left(
	\begin{array}{ccc}
	\mu\mathring L\varphi\\
	\mu\slashed d\varphi\\
	T{\varphi}\\
	\end{array}
	\right).
	\end{split}
	\end{equation}
\end{lemma}

\begin{remark}
We point out that in the expression of \eqref{Lchi'}, the last term consists of the combinations of $(\slashed{d}\vp)(\mathring{L}\vp)$ and $\slashed{d}\vp \cdot \slashed{d}\vp$, while the troublesome term $(\mathring{L}\vp)^2$ is absent. This property
will play a crucial role in estimating  $\operatorname{tr}\check{\mathscr{X}}$.
\end{remark}

It is emphasized that different from the global existence argument in \cite{Ding4}, we will prove
$\mu\rightarrow 0+$ in this paper, which leads to the shock formation.
For this purpose, we carefully check the transport equation \eqref{lmu} of $\mu$ and
make use of the null condition \eqref{null} to get
\begin{equation}\label{GT}
G_{\mathring L\mathring L}^\gamma T\varphi_\gamma
=\mu G_{\mathring L\mathring L}^0\tilde T^i(\mathring L\varphi_i)
-\underline{G_{\mathring L\mathring L}^\gamma {\mathring L}_\gamma}\tilde T^i T\varphi_i
+\mu g_{\gamma i}G_{\mathring L\mathring L}^\gamma (\slashed d^Xx^i)\tilde T^j\slashed d_X\varphi_j
\end{equation}
with
\begin{equation}\label{gLLL}
G_{\mathring L\mathring L}^\gamma {\mathring L}_\gamma=f(\varphi, \check L^1,\check L^2,\f{x}\varrho,\check\varrho)\left(
\begin{array}{ccc}
\varphi^k\\
\check L^{1}\varphi^{k-1}\\
\check L^{2}\varphi^{k-1}\\
\check\varrho\varphi^{k-1}\\
(\mathring L\phi)\varphi^{k-2}\\
(\slashed d^Xx)(\slashed d_X\phi)\varphi^{k-2}
\end{array}
\right),
\end{equation}
which has been derived in \cite[(2.35)]{Ding4}.

With the above preparation, we now derive rough $L^{\infty}$ estimates for some lower-order derivatives.

\begin{proposition}\label{prop:lower-order-Linfty-estimates}
Under the bootstrap assumptions \eqref{bootstrap-assumptions},
for fixed $M\geq 1$ and small $\delta>0$,
the following estimates hold that for $t_0\leq s<T_\delta$
and $i=1,2$:
\begin{equation}\label{eq:lower-order-Linfty-estimates}
\begin{split}
&\|\check{\mathscr{X}}\|_{L^\infty(\Si_s^u)}
+\|\check{L}^i\|_{L^\infty(\Si_s^u)}
+\|\check{T}^i\|_{L^\infty(\Si_s^u)}
+\|\check{\varrho}\|_{L^\infty(\Si_s^u)}
+\|\upsilon\|_{L^\infty(\Si_s^u)}
\lesssim M^k\delta^{k(1-\varepsilon_0)},\\
&\|\mu\|_{L^\infty(\Si_s^u)}\lesssim M^{k+1},
\quad
\|\slashed{d}x^i\|_{L^\infty(\Si_s^u)}\lesssim 1,
\end{split}
\end{equation}
where the implicit constants are independent of $M$ and $\delta$.
\end{proposition}

\begin{proof}
	From \eqref{LTTT}, one has $g(\tilde{T}, \tilde{T}) = 1$, which implies
	$1 = g_{ij}\tilde{T}^i\tilde{T}^j = \big(1 + O(M^k\delta^{k(1-\varepsilon_0)})\big)\sum_{i=1}^2|\tilde{T}^i|^2$.
	This leads to
	\begin{equation}\label{eq:LT}
	|\tilde{T}^i|, \, |\mathring{L}^i| \leq 1 + O(M^k\delta^{k(1-\varepsilon_0)}).
	\end{equation}
	
	Moreover, using the identity $|\slashed{d}x^i|^2 = \slashed{g}^{ab}\slashed{d}_a x^i\slashed{d}_b x^i = g^{ii} + (g^{0i})^2 - (\tilde{T}^i)^2$, together with \eqref{bootstrap-assumptions} and \eqref{eq:LT}, we arrive at
	\[
	|\slashed{d}x^i| \lesssim 1.
	\]
	
	In addition, by \eqref{LeL} and \eqref{bootstrap-assumptions}, one has $|\mathring{L}(\varrho\check{L}^i)| \lesssim M^k\delta^{k(1-\varepsilon_0)}$. Integrating $\mathring{L}(\varrho\check{L}^i)$ along the integral curves of $\mathring{L}$ and using the relation $\check{T}^i = -g^{0i} - \check{L}^i$, we further obtain
	\[
	|\check{L}^i| + |\check{T}^i| \lesssim M^k\delta^{k(1-\varepsilon_0)}.
	\]
	
	From the condition $g_{ij}\big(\check{T}^i - \frac{x^i}{\varrho}\big)\big(\check{T}^j - \frac{x^j}{\varrho}\big)=1$, one gets
	\[
	\left(g_{ij}\omega^i\omega^j\right)\frac{r^2}{\varrho^2} - \left(2g_{ij}\check{T}^i\omega^j\right)\frac{r}{\varrho} + g_{ij}\check{T}^i\check{T}^j - 1 = 0.
	\]
	Thus,
	\begin{equation}\label{rrho}
	\check{\varrho}= \frac{r}{\varrho} - 1 = \frac{1 - g_{ij}\omega^i\omega^j - g_{ij}\check{T}^i\check{T}^j + 2g_{ij}\check{T}^i\omega^j}{\sqrt{g_{ij}\omega^i\omega^j - (g_{ij}\omega^i\omega^j)(g_{ab}\check{T}^a\check{T}^b) + (g_{ij}\check{T}^i\omega^j)^2} + g_{ij}\omega^i\omega^j - g_{ij}\check{T}^i\omega^j}.
	\end{equation}
	This implies $|\check{\varrho}| \lesssim M^k\delta^{k(1-\varepsilon_0)}$,
	and hence, by the definition of $\upsilon$ in \eqref{omega}, it holds that
	\[
	|\upsilon| \lesssim M^k\delta^{k(1-\varepsilon_0)}.
	\]
On the other hand, from \eqref{lmu}, \eqref{GT} and \eqref{Lchi'}, we have
	\begin{equation}\label{YCao-1}
 \begin{split}
	&|\mathring{L}\mu| \lesssim M^{k+1}\delta^{-\varepsilon_k} + M^k\delta^{2k(1-\varepsilon_0)-\varepsilon_0}\mu, \\
	&|\mathring{L}(\varrho^2\operatorname{tr}\check{\mathscr{X}})| \lesssim M^k\delta^{2k(1-\varepsilon_0)-\varepsilon_0} + (\varrho^2\operatorname{tr}\check{\mathscr{X}})^2 + M^k\delta^{2k(1-\varepsilon_0)-\varepsilon_0}|\varrho^2\operatorname{tr}\check{\mathscr{X}}|,
\end{split}
	\end{equation}
	which lead to $|\mu| \lesssim M^{k+1}$ and $|\operatorname{tr}\check{\mathscr{X}}| \lesssim M^k\delta^{k(1-\varepsilon_0)}$.
	\end{proof}

Note that $S_{s,u}$ is one-dimensional and the projected rotation
field $R$ has a uniform bounded length which is also away from zero.
By direct computation, one has

\begin{corollary}\label{cor:12form}
	Under the bootstrap assumptions \eqref{bootstrap-assumptions}, for small $\delta>0$, the following estimates hold:
	\begin{enumerate}
		\item If $\eta$ is a $1$-form on $S_{s, u}$, then
		\begin{align}
		(\eta_a R^a)^2 &\sim |\eta|^2, \label{eq:1-f}\\
		|\slashed{\mathcal L}_{R}\eta|^2 &\sim |\slashed{\nabla}\eta|^2 + O(M^k\delta^{k(1-\varepsilon_0)})|\eta|^2; \label{eq:1f}
		\end{align}
		
		\item If $\eta$ is a $2$-form on $S_{s, u}$, then
		\begin{equation}\label{eq:2-f}
		|\slashed{\mathcal L}_{R}\eta|^2 \sim |\slashed{\nabla}\eta|^2 + O(M^k\delta^{k(1-\varepsilon_0)})|\eta|^2
		\end{equation}
		and
		\begin{equation}\label{eq:2f}
		|\slashed{\nabla}^2\eta \lesssim |\slashed{\mathcal L}_R^{\leq 2}\eta|.
		\end{equation}
	\end{enumerate}
\end{corollary}

To close the bootstrap argument, we first establish the higher-order
derivative estimates for some important geometric quantities. Although the
related proofs are somewhat analogous to these in  \cite[Proposition 6.3 and Corollary 6.3]{Ding3}
and \cite[Proposition 4.1]{L-Y}, we still give the details since the precise weights of $\delta$
should be given.

\begin{proposition}\label{prop:higher-order-L-infty-estimates}
Under the bootstrap assumptions \eqref{bootstrap-assumptions},
for small $\delta>0$, the following estimates
hold that for $t_0\leq s<T_\delta$,
$|\alpha|\leq N-2$ and $|\beta|\leq N-1$:
\begin{equation}\label{eq:higher-order-L-infty-estimates}
\begin{split}
&\delta^{l+m\varepsilon_k}
\big(
|\Lies_Z^\alpha\check{\mathscr X}|
+|\Lies_Z^\alpha{}^{(R)}\pis|
\big)
\lesssim M^k\delta^{k(1-\varepsilon_0)},\\
&\delta^{l+m\varepsilon_k}
|\Lies_Z^\alpha{}^{(R)}\pis_{\mathring L}|
\lesssim M^k\delta^{k(1-\varepsilon_0)}
\big(1+M^k\delta^{2k(1-\varepsilon_0)-\varepsilon_0}\big),\\
&\delta^{l+m\varepsilon_k}
\big(
|\Lies_Z^\alpha{}^{(T)}\pis|
+|\Lies_Z^\alpha{}^{(T)}\pis_{\mathring L}|
\big)
\lesssim M^k\delta^{k(1-\varepsilon_0)-1},\\
&\delta^{l+m\varepsilon_k}
|\Lies_Z^\alpha{}^{(R)}\pis_T|
\lesssim M^{2k}\delta^{2k(1-\varepsilon_0)-1},\\
&\delta^{l+m\varepsilon_k}|Z^\beta\mu|
\lesssim M^{k+1},
\quad
\delta^{l+m\varepsilon_k}
|\Lies_Z^\beta\slashed dx^j|
\lesssim 1,\\
&\delta^{l+m\varepsilon_k}
\big(
|Z^\beta\check L^j|
+|Z^\beta\check T^j|
+|Z^\beta\upsilon|
+|Z^\beta\check\varrho|
\big)
\lesssim M^k\delta^{k(1-\varepsilon_0)}.
\end{split}
\end{equation}
In addition, for $1\leq|\beta|\leq N-1$,
\begin{equation}\label{YHCC-A20}
\begin{split}
\delta^{l+m\varepsilon_k}|\Lies_Z^\beta R|
\lesssim
M^{2k}\delta^{2k(1-\varepsilon_0)}
+M^k\delta^{\varepsilon_0},
\end{split}
\end{equation}
where $Z\in\{\mathring L,T,R\}$, $\Lies_Z$
denotes the corresponding projection of Lie derivative, $l$ and $m$
are the numbers of $T$ and $\mathring L$ in the related terms,
respectively.
\end{proposition}

\begin{proof}
\textbf{Step 1.} The case of $Z = R$ for all $Z$

\vskip 0.1 true cm

When $\alpha = 0$, the estimate $|\check{\mathscr{X}}|\lesssim M^k\de^{k(1-\varepsilon_0)}$ directly follows from \eqref{eq:lower-order-Linfty-estimates}.
Using \eqref{bootstrap-assumptions}, Lemma \ref{lem3.4}, and Proposition \ref{prop:lower-order-Linfty-estimates}, one can obtain
the estimates for ${}^{(R)}\pis$, ${}^{(R)}\pis_L$, ${}^{(R)}\pis_T$, ${}^{(T)}\pis$, and ${}^{(T)}\pis_L$.

Additionally, from \eqref{eq:1-f} and \eqref{R}, we know $|R\check{L}^i| \sim |\slashed{d}\check{L}^i|$. The estimate of $R\check{L}^i$ then follows from \eqref{deL}, which in turn allows us to treat $R\check{T}^i$, $R\upsilon$ and $R\check{\varrho}$. Hence, $\Lies_{R} \slashed{d} x^j = \slashed{d}(\Omega^i - g_{aj}\tilde{T}^a\Omega^j\tilde{T}^i)$ can be estimated.

For $R\mu$, it satisfies the equation
\begin{equation*}
\mathring{L}(R\mu) = {}^{(R)}\pis_L^X \slashed{d}_X\mu + R\big(-\frac{1}{2}\mu G_{\mathring{L}\mathring{L}}^\gamma \mathring{L}\varphi_\gamma - \mu G_{\tilde{T}\mathring{L}}^\gamma \mathring{L}\varphi_\gamma + \frac{1}{2}G_{\mathring{L}\mathring{L}}^\gamma T\varphi_\gamma\big)
\end{equation*}
by \eqref{lmu}, which implies $|\mathring{L}R\mu| \lesssim M^k\delta^{2k(1-\varepsilon_0)-\varepsilon_0}|R\mu|+M^{k+1}\delta^{-\varepsilon_k}$ with the help of \eqref{GT}. Integrating along the integral curves of $\mathring{L}$ yields $|R\mu| \lesssim M^{k+1}$.

Therefore, \eqref{eq:higher-order-L-infty-estimates} holds for $|\alpha|=0$ and $|\beta|=1$.

For general $\alpha$ and $\beta$, one can use the induction argument. Assume that \eqref{eq:higher-order-L-infty-estimates} holds 
for all vector fields $R$ and $|\alpha|=|\beta|-1=p \leq N-3$, then it follows from \eqref{Lchi'} that
\begin{equation*}
\begin{split}
|\mathring{L}R^{p+1}(\varrho^2\operatorname{tr}\check{\mathscr{X}})| &\lesssim |[\mathring{L}, R^{p+1}](\varrho^2\operatorname{tr}\check{\mathscr{X}})| + |R^{p+1}\mathring{L}(\varrho^2\operatorname{tr}\check{\mathscr{X}})| \\
&\lesssim M^k\delta^{2k(1-\varepsilon_0)-\varepsilon_0} + M^k\delta^{2k(1-\varepsilon_0)-\varepsilon_0}|R^{p+1}(\varrho^2\operatorname{tr}\check{\mathscr{X}})|,
\end{split}
\end{equation*}
which yields $|R^{p+1}(\operatorname{tr}\check{\mathscr{X}})| \lesssim M^k\delta^{k(1-\varepsilon_0)}$. Thus, the estimates of $\Lies_{R}^{p+1} {}^{(R)}\pis$, $\Lies_{R}^{p+1} {}^{(R)}\pis_{\mathring{L}}$, $\Lies_{R}^{p+1} {}^{(T)}\pis$, $R^{p+2}\check{L}^i$, $R^{p+2}\check{T}^i$, $R^{p+2}\check{\varrho}$, $R^{p+2}\upsilon$ and $\Lies_{R}^{p+1}\slashed{d}x^i$ are obtained by Lemma \ref{lem3.4}, \eqref{deL}, 
and the definitions of $\check{T}^i$, $\check{\varrho}$ and $\upsilon$.

Analogously,
$|\mathring{L}R^{p+2}\mu| \lesssim M^k\delta^{2k(1-\varepsilon_0)-\varepsilon_0}|R^{p+2}\mu|+M^{k+1}\delta^{-\varepsilon_k}$
holds by \eqref{lmu} and \eqref{GT}, which  allows us to further estimate $R^{p+2}\mu$, $\Lies_{R}^{p+1} {}^{(R)}\pis_T$ and $\Lies_{R}^{p+1} {}^{(T)}\pis_{\mathring{L}}$.

\vskip 0.1 true cm
	
\textbf{Step 2.} The general case of $Z \in \{\mathring{L}, T, R\}$

\vskip 0.1 true cm

In this case, one can use \eqref{Lchi'}, \eqref{Tchi'}, \eqref{TL}, \eqref{LeL}, \eqref{lmu} and the estimates from Step 1 to prove the results
in \eqref{eq:higher-order-L-infty-estimates}. Since the proof procedure is rather analogous that in \cite[Propositions 6.2 and 6.3]{Ding3},
we omit the details here.
\end{proof}

From the identity $T=\p_u-\Theta^XX$, one could estimate $\Theta^X$ as follows.

\begin{lemma}\label{eXi}
	Under the assumptions \eqref{bootstrap-assumptions}, for small $\delta>0$, it holds
	\begin{equation}\label{Xi}
	|\Theta|\lesssim  M^k\delta^{\varepsilon_0-1}\quad\text{for}\ t_0\leq s<T_\delta.
	\end{equation}
\end{lemma}

\begin{proof}
	Note that $\Theta$ is a vector field on $S_{s, u}$ and one has
	\begin{equation}\label{LXi}
	\mathring L\Theta^X=[T,\mathring L]^X=\big(\mathscr{D}_T\mathring L-\mathscr{D}_{\mathring L}T\big)^X=\slashed d^X\mu+2\mu\zeta^X.
	\end{equation}
	Then
	\[
	\mathring L(\varrho^{-2}|\Theta|^2)=2\varrho^{-2}|\Theta|^2|(\operatorname{tr}\check{\mathscr X}+\f12\slashed g^{XX}G_{XX}^\gamma\mathring L\vp_\gamma)+2\varrho^{-2}(\slashed d_X\mu
	+2\mu\zeta_X)\Theta^X.
	\]
	This, together with the estimates \eqref{eq:higher-order-L-infty-estimates}, \eqref{bootstrap-assumptions} and \eqref{zeta}, yields
	\begin{equation}\label{Y1}
	|\mathring L(\varrho^{-1}|\Theta|)|\lesssim  M^{k}\delta^{2k(1-\varepsilon_0)-\varepsilon_0} \varrho^{-1}|\Theta|	+ M^k\delta^{k(1-\varepsilon_0)-1}.
	\end{equation}
	Thus \eqref{Xi} is obtained through integrating \eqref{Y1} along the integral curves of $\mathring L$.
\end{proof}

Next we improve the bootstrap assumptions in \eqref{bootstrap-assumptions}
up to orders $N-2$ and study the corresponding estimates
for $\phi$.

\begin{proposition}\label{prop:bootstrap-higher-derivatives}
Under the bootstrap assumptions \eqref{bootstrap-assumptions},
for small $\delta>0$, the following estimates
hold that for $t_0\leq s<T_\delta$ and $|\alpha|\leq N-2$:
\begin{equation}\label{eq:bootstrap-higher-derivatives}
\delta^{l+m\varepsilon_k}|Z^\alpha\phi|
\lesssim\delta^{2-\varepsilon_0},
\quad
\delta^{l+m\varepsilon_k}|Z^\alpha\varphi_\gamma|
\lesssim\delta^{1-\varepsilon_0},
\quad \gamma=0,1,2,
\end{equation}
where $Z \in \{ \mathring{L}, T, R \}$, $l$ and $m$
denote the numbers of $T$  and $\mathring{L}$ in $Z^\alpha$, respectively.
Moreover,
\begin{align}
\bigl|
\varrho^{1/2}T\varphi_\gamma(s,u,\vartheta)
-\varrho_0^{1/2}T\varphi_\gamma(t_0,u,\vartheta)
\bigr|
&\lesssim M^{k+1}\delta^{-\varepsilon_k},
\label{Tvp}\\
\bigl|
\varrho^{1/2}\varphi_\gamma(s,u,\vartheta)
-\varrho_0^{1/2}\varphi_\gamma(t_0,u,\vartheta)
\bigr|
&\lesssim M^{k+1}\delta^{(k+1)(1-\varepsilon_0)}
\label{r0vp}
\end{align}
with $\varrho=s-u$ and $\varrho_0=t_0-u$.
\end{proposition}

\begin{proof}
From \eqref{fequation} and Proposition~\ref{prop:higher-order-L-infty-estimates}, we have that for all $|\alpha| \leq N-2$,
\begin{equation*}
\begin{split}
&\bigl| \mathring{L}{\bar Z}^\alpha \underline{\mathring L} \vp_\gamma + \frac{1}{2\varrho}{\bar Z}^\alpha\underline{\mathring L} \vp_\gamma \bigr| \\
&= \big|{\bar Z}^\alpha (\mu\slashed\triangle\vp_\gamma+H_\gamma )+ [\mathring{L}, {\bar Z}^\alpha] \underline{\mathring L} \vp_\gamma
- \sum_{\substack{\alpha_1 + \alpha_2 =\alpha\\ |\alpha_1| > 0}}{\bar Z}^{\alpha_1} \big( \frac{1}{2\varrho} \big){\bar Z}^{\alpha_2} \underline{\mathring L} \vp_\gamma \big| \\
&\lesssim M^{k+1}\delta^{-2\varepsilon_k-l},
\end{split}
\end{equation*}
where $\bar Z\in\{T, R\}$ and $l$ is the number of $T$ operators in $\bar Z^\alpha$.
Consequently,
\[
\bigl| \mathring{L} \bigl( \varrho^{1/2}{\bar Z}^\alpha \underline{\mathring L} \vp_\gamma \bigr) \bigr|
\lesssim M^{k+1}\delta^{-2\varepsilon_k-l},
\]
which implies
\[
\bigl| \varrho^{1/2} {\bar Z}^\alpha \underline{\mathring L} \vp_\gamma(s,u,\vartheta) -\varrho_0^{1/2}{\bar Z}^\alpha \underline{\mathring L} \vp_\gamma(t_0,u,\vartheta) \bigr|
\lesssim M^{k+1}\delta^{-\varepsilon_k-l}.
\]
From this, \eqref{Tvp} and $|{\bar Z}^\alpha T \vp_\gamma| \lesssim \delta^{-\varepsilon_0-l}$ can be obtained.
Then commuting $T$ with $\bar Z^\alpha$ yields $|T\bar{Z}^\alpha \vp_\gamma| \lesssim \delta^{-\varepsilon_0-l}$.
In addition, by \eqref{Xi} and \eqref{bootstrap-assumptions}, we deduce
\begin{equation}\label{barzvp}
\Bigl| \frac{\partial}{\partial u}{\bar Z}^\alpha \vp_\gamma \Bigr|
\lesssim \bigl| T{\bar Z}^\alpha \vp_\gamma \bigr| + \bigl| \Theta^X \slashed{d}_X {\bar Z}^\alpha \vp_\gamma \bigr|
\lesssim \delta^{-\varepsilon_0-l}
\end{equation}
and
\begin{equation}\label{YHCC-A1}
\begin{split}
	&\Bigr|\f{\p}{\p u}\bigl(\varrho^{1/2}\vp_\gamma(s, u, \vartheta)-\varrho_0^{1/2}\vp_\gamma(t_0, u, \vartheta)\bigr)\Bigr|\\
\lesssim &|\varrho^{1/2}\vp_\gamma(s, u, \vartheta)-\varrho_0^{1/2}\vp_\gamma(t_0, u, \vartheta)\Bigr|+O(M^{k+1}\delta^{-\varepsilon_k}).
\end{split}
\end{equation}
Integrating \eqref{barzvp} and \eqref{YHCC-A1} from $0$ to $u$ gives $|{\bar Z}^\alpha \vp_\gamma| \lesssim \delta^{1-\varepsilon_0-l}$ and \eqref{r0vp}, respectively.

We now prove \eqref{eq:bootstrap-higher-derivatives} for $\vp_\gamma$ by induction on the number of $\mathring{L}$ operator. 
The case of $m=0$ follows by integrating \eqref{barzvp}. For the induction argument, assume that for some integer $p\ge 0$ 
and for all $|\alpha|\leq N-2$ containing at most $p$  $\mathring{L}$'s, it holds
\[
|Z^\alpha\vp_\gamma|\lesssim\delta^{1-\varepsilon_0-l-m\varepsilon_k},
\]
where $l$ and $m$ are the numbers of $T$ and $\mathring{L}$ operators in $Z^\alpha$, respectively. Let the multi-index $\beta$ with $|\beta|\leq N-3$ contain exactly $p$ $\mathring{L}$'s. By \eqref{fequation}, we get
\begin{equation*}
\begin{split}
|Z^\beta\mathring L\underline{\mathring L}\vp_\gamma|
&=\bigl|Z^\beta\bigl(\mu\slashed\triangle\vp_\gamma+H_\gamma-\frac{1}{2\varrho}\underline{\mathring L}\vp_\gamma\bigr)\bigr|\\
&\lesssim\delta^{-l-p\varepsilon_k}\bigl(\delta^{-\varepsilon_0}+\delta^{-2\varepsilon_k}M^{k+1}\bigr),
\end{split}
\end{equation*}
which implies
\[
|Z^\beta\mathring L T\vp_\gamma|\lesssim\delta^{-l-p\varepsilon_k}\bigl(\delta^{-\varepsilon_0}+\delta^{-2\varepsilon_k}M^{k+1}\bigr).
\]
By commuting $T$ with $Z^\beta\mathring L$, one then has
\[
|T Z^\beta\mathring L\vp_\gamma|\lesssim\delta^{-l-p\varepsilon_k}\bigl(\delta^{-\varepsilon_0}+\delta^{-2\varepsilon_k}M^{k+1}\bigr).
\]
As in~\eqref{barzvp}, we can find
\[
\Bigl|\frac{\partial}{\partial u}Z^\beta\mathring L\vp_\gamma\Bigr|
\lesssim\delta^{-l-p\varepsilon_k}\bigl(\delta^{-\varepsilon_0}+\delta^{-2\varepsilon_k}M^{k+1}\bigr).
\]
Integrating this with respect to $u$ yields
\[
|Z^\beta\mathring L\vp_\gamma|\lesssim\delta^{1-\varepsilon_0 -l-(p+1)\varepsilon_k}.
\]
This completes the induction argument and obtains the second bound in~\eqref{eq:bootstrap-higher-derivatives}.

Finally, we turn to the estimate of $\phi$. From the relation $T\phi = \mu \tilde T^i \vp_i$, one has
$Z^\alpha T\phi = Z^\alpha(\mu \tilde T^i \vp_i)$. Based on the already obtained bound for $\vp_i$, we have
$|Z^\alpha T\phi| \lesssim \delta^{1-\varepsilon_0-l-m\varepsilon_k}$.
By commuting $T$ with $Z^\alpha$ and using the induction hypothesis for the lower-order terms of $\phi$, one deduces
$|T Z^\alpha \phi| \lesssim \delta^{1-\varepsilon_0-l-m\varepsilon_k}$.
Then, as in~\eqref{barzvp}, it holds that
\[
\Bigl|\frac{\partial}{\partial u} Z^\alpha \phi\Bigr|
\lesssim \delta^{1-\varepsilon_0-l-m\varepsilon_k}.
\]
Integrating this from $0$ to $u$ yields
$|Z^\alpha \phi| \lesssim \delta^{2-\varepsilon_0-l-m\varepsilon_k}$,
which is just the first estimate in~\eqref{eq:bootstrap-higher-derivatives}.
\end{proof}

\begin{remark}\label{Rem4.2}
We emphasize that under the bootstrap assumptions \eqref{bootstrap-assumptions}, the estimates in
Proposition~\ref{prop:higher-order-L-infty-estimates} involve the bounds depending on $M$.
However, Proposition~\ref{prop:bootstrap-higher-derivatives} provides the improved estimates whose bounds are independent of $M$.
Consequently, starting from Proposition~\ref{prop:bootstrap-higher-derivatives} and repeating the arguments as in Proposition~\ref{prop:higher-order-L-infty-estimates}, one can refine the estimates in \eqref{eq:higher-order-L-infty-estimates}
for all $|\alpha|\le N-4$ and $|\beta|\le N-3$ so that the related $L^\infty$ bounds do not depend on $M$.
\end{remark}

\section{Detailed analysis of $\mu$ and blowup mechanism}\label{YCao-3}
In this section, we aim to prove $\mu \to 0^+$ as $s \to T_\delta-$. Once $\mu$ approaches $0$, the coordinate transformation
from $(t, x^1, x^2)$ to $(s, u, \vartheta)$ is no longer regular. In this case, the blowup of the second order
derivatives of $\phi$ will be shown.

Set
\[
\mathscr{M}(s,u,\vartheta)
=
(-1)^k\frac{k+1}{2}
\big(
\varrho^{(k+1)/2}
h^{\alpha\beta,\gamma_1\cdots\gamma_{k+1}}
\omega_\alpha\omega_\beta
\omega_{\gamma_1}\cdots\omega_{\gamma_{k+1}}
\varphi_0^k T\varphi_0
\big)(s,u,\vartheta).
\]
We first establish the following asymptotic formula
for $\mu$ and a quantitative estimate for its decreasing rate
when $\mu$ is small.

\begin{proposition}\label{Prop.5.1}
Under the bootstrap assumptions \eqref{bootstrap-assumptions},
for small $\delta>0$ and
$t_0\le s<T_\delta$, one has
\begin{equation}\label{mathM}
\mu(s,u,\vartheta)
=
1-\frac{2\mathscr M(t_0,u,\vartheta)}{k-1}
\big(
\varrho^{-(k-1)/2}
-\varrho_0^{-(k-1)/2}
\big)
+O(\delta^{1-\varepsilon_0}),
\end{equation}
where $\varrho=s-u$ and $\varrho_0=t_0-u$.
Furthermore, at any point where $\mu(s,u,\vartheta)<1/10$,
\begin{equation}\label{blowup}
\mathring L\mu(s,u,\vartheta)
=
\varrho^{-(k+1)/2}\mathscr M(s,u,\vartheta)
+O(\delta^{1-\varepsilon_0-\varepsilon_k})
\sim -\delta^{-\varepsilon_k}.
\end{equation}
\end{proposition}

\begin{proof}It follows from \cite[(2.36)]{Ding4} that
\begin{equation}\label{kgh}
G_{\mathring L\mathring L}^{\gamma}
\mathring L_\gamma=-kg^{\alpha\beta,\gamma\gamma_2\cdots\gamma_k}\vp_{\gamma_2}\cdots\vp_{\gamma_k}\mathring L_\alpha\mathring L_\beta\mathring L_\gamma-(\p_{\vp_\gamma}h^{\alpha\beta})\mathring L_\alpha\mathring L_\beta\mathring L_\gamma.
\end{equation}
Substituting
$h^{\alpha\beta}(\partial\phi)
=h^{\alpha\beta,\gamma_1\cdots\gamma_{k+1}}
\partial_{\gamma_1}\phi\cdots
\partial_{\gamma_{k+1}}\phi
+O(|\partial\phi|^{k+2})$
into \eqref{kgh} and using
the null condition \eqref{null} give
\begin{equation}\label{hLLL}
\begin{split}
G_{\mathring L\mathring L}^{\gamma}
\mathring L_\gamma
={}&
-(-1)^k(k+1)
h^{\alpha\beta,\gamma_1\cdots\gamma_{k+1}}
\omega_\alpha\omega_\beta
\omega_{\gamma_1}\cdots\omega_{\gamma_{k+1}}
\bigl(\tilde T^i\varphi_i\bigr)^k +O(\delta^{(k+1)(1-\varepsilon_0)}).
\end{split}
\end{equation}
By the identities $\tilde T^i\varphi_i
=
-g^{0i}\varphi_i-\mathring L\phi+\varphi_0$
and $\tilde T^iT\varphi_i=-g^{0i}T\varphi_i-\mu\tilde T^i\mathring L\varphi_i+T\varphi_0$,
together with \eqref{GT}, \eqref{hLLL} and the improved
estimates from Section~\ref{Section 3}, one has
\begin{equation}\label{GTh}
\begin{split}
G_{\mathring L\mathring L}^{\gamma}T\varphi_\gamma
={}&
(-1)^k(k+1)
h^{\alpha\beta,\gamma_1\cdots\gamma_{k+1}}
\omega_\alpha\omega_\beta
\omega_{\gamma_1}\cdots\omega_{\gamma_{k+1}}
\varphi_0^kT\varphi_0
+O(\delta^{1-\varepsilon_0-\varepsilon_k}).
\end{split}
\end{equation}
Hence, by \eqref{lmu} and \eqref{GTh}, we arrive at
\begin{equation}\label{rLmu}
\varrho^{(k+1)/2}\mathring L\mu(s,u,\vartheta)
=
\mathscr M(s,u,\vartheta)
+O(\delta^{1-\varepsilon_0-\varepsilon_k}),
\end{equation}
which proves the identity in \eqref{blowup}.

In addition,
the identity $\mathring L\omega_i
=\frac{\check L^i}{r}
-\frac{x^ix^j\check L^j}{r^3}$
and the bound for $\check L$ imply
\[
	\omega_i(s,u,\vartheta)=\omega_i(t_0,u,\vartheta)+\int_{t_0}^s\mathring L\omega_i(\tau,u,\vartheta)d\tau =\omega_i(t_0,u,\vartheta)+O(\delta^{\varepsilon_0}).
	\]
	This, together with \eqref{Tvp}, \eqref{r0vp} and \eqref{rLmu}, yields
	\begin{equation}\label{rhoLmu}
	\varrho^{(k+1)/2}\mathring{L}\mu(s, u, \vartheta) = \mathscr M(t_0, u, \vartheta)+O(\delta^{1-\varepsilon_0-\varepsilon_k}).
	\end{equation}
Therefore,
	\begin{equation*}
	\begin{split}
	\mu(s,u,\vartheta)
	&= \mu(t_0,u,\vartheta)
	+ \int_{t_0}^s \frac{\varrho^{(k+1)/2}\mathring{L}\mu(\tau,u,\vartheta)}{\varrho^{(k+1)/2}}d\tau \\
	&= \mu(t_0,u,\vartheta)
	- \f{2\mathscr{M}(t_0,u,\vartheta)}{k-1}\big(\varrho^{-(k-1)/2} - \varrho_0^{-(k-1)/2}\big)
	+ O(\delta^{1-\varepsilon_0}),
	\end{split}
	\end{equation*}
	which proves \eqref{mathM}.

Suppose now that $\mu(s,u,\vartheta)<1/10$.
Retaining the error term in \eqref{mathM}, we obtain that
for small $\delta$,
\[
\mathscr{M}(t_0,u,\vartheta) \leq \frac{9(k-1)}{20\big(\varrho^{-(k-1)/2} - \varrho_0^{-(k-1)/2}\big)}.
\]
Therefore, using \eqref{rhoLmu} again yields
\begin{equation*}
\begin{split}
\varrho^{(k+1)/2}\mathring{L}\mu(s,u,\vartheta)
&\leq \frac{-9}{10\int_{t_0}^s (\tau-u)^{-(k+1)/2}d\tau} + O\left(\delta^{1-\varepsilon_0-\varepsilon_k}\right) \\
&= \frac{-9(\tilde\tau-u)^{(k+1)/2}}{10(s-t_0)} + O\left(\delta^{1-\varepsilon_0-\varepsilon_k}\right)
\lesssim -\delta^{\varepsilon_k},
\end{split}
\end{equation*}
where $\tilde\tau \in (t_0,s)$ and $\tilde\tau - u \sim 1$. On the other hand, it follows from \eqref{eq:higher-order-L-infty-estimates} that $|\mathring{L}\mu(s,u,\vartheta)| \lesssim \delta^{-\varepsilon_k}$, so \eqref{blowup} holds.
\end{proof}

It is straightforward to observe that when $\mu<\f{1}{10}$, one has by \eqref{blowup},
\[
|\vp_0^kT\vp_0|  \sim \delta^{-\varepsilon_k}.
\]
Therefore,
\[
\mu \, |\nabla \partial_t \phi| \gtrsim |T\vp_0|\gtrsim\delta^{-k(1-\varepsilon_0)}|\vp_0^kT\vp_0| \sim \delta^{-\varepsilon_0},
\]
which yields $|\partial^2\phi| \sim \mu^{-1}\delta^{-\varepsilon_0}$ and hence
\begin{equation}\label{shock}
|\partial^2\phi| \rightarrow +\infty \quad \text{as} \quad \mu \rightarrow +0.
\end{equation}

As established in Proposition \ref{Prop.5.1}, if $\mu< \frac{1}{10}$ at some time, then $\mu$ will converge to $0$ in finite time along a fixed characteristics of $\mathring L$ due to $\mathring{L}\mu \sim -\delta^{-\varepsilon_k}$. Consequently, 
the solution $\phi$ will eventually blow up (see \eqref{shock}). To characterize the behavior of $\mu$ more precisely 
near the blowup time, as in \cite{L-Y, M-Y}, we define
$$\mu_{\min}(s) = \min_{(u,\vartheta)} \mu(s,u,\vartheta).$$

\begin{proposition}
Under the assumptions of Proposition~\ref{Prop.5.1}, for
small $\delta>0$ and any $s\in[t_0, T_\delta)$
such that $\mu_{\min}(s)<1/10$, one has
\begin{equation}\label{mumin}
\mu_{\min}(s)
\gtrsim\delta^{-\varepsilon_k}(T_\delta-s).
\end{equation}
Furthermore, for each
$(u,\vartheta)\in[0,4\delta]\times\mathbb S$,
\begin{equation}\label{Tmu}
(\mu^{-1}T\mu)_+(s,u,\vartheta)
\lesssim
\delta^{-1+\varepsilon_k/2}
(T_\delta-s)^{-1/2},
\end{equation}
where $(\mu^{-1}T\mu)_+$ denotes the positive part
of $\mu^{-1}T\mu$.
\end{proposition}

\begin{proof}
Fix $s\in[t_0, T_\delta)$ such that
$\mu_{\min}(s)<1/10$.
Choose $(u_0,\vartheta_0)\in[0,4\delta]\times\mathbb S$
such that
\[
\mu(s,u_0,\vartheta_0)=\mu_{\min}(s).
\]
By Proposition~\ref{Prop.5.1} and continuity, $\mu$ remains below
$1/10$ along this fixed characteristics for all subsequent
time  and
\[
\mathring L\mu(\tau,u_0,\vartheta_0)\sim-\delta^{-\varepsilon_k},
\qquad s\le\tau<T_\delta.
\]
Thus, for $s<S< T_\delta$, the positivity of
$\mu$ gives
\[
\begin{split}
\mu_{\min}(s)
&=
\mu(S,u_0,\vartheta_0)
-\int_s^S
\mathring L\mu(\tau,u_0,\vartheta_0)\,d\tau\ge-\int_s^{S} \mathring{L}\mu(\bar{s}, u_0, \vartheta_0) \, d\bar{s} \sim \delta^{-\varepsilon_k}(S-s).
\end{split}
\]
Letting $S\to T_\delta-$ yields \eqref{mumin}.

We next estimate the positive part of $\mu^{-1}T\mu$.
For the fixed time $s$, $T=\partial_u-\Theta^X\partial_\vartheta$ can be straighten by setting
\[
\tilde u=u,\quad
\tilde\vartheta=\vartheta+g(s,u,\vartheta),
\]
where
\[
\partial_u g-\Theta^X\partial_\vartheta g=\Theta^X,
\quad
g(s,0,\vartheta)=0.
\]
Due to $T\tilde u=1$ and $T\tilde\vartheta=0$,
we have $T=\partial_{\tilde u}$.

Fix $\tilde\vartheta$ and write
\begin{equation}\label{DXY-3}
f(\tilde u)=(\mu^{-1}T\mu)(s,\tilde u,\tilde\vartheta),
\end{equation}
where the quantity on the right hand side of \eqref{DXY-3} is expressed in
the new coordinates.
If $\max_{[0,4\delta]}f\le0$, then $f_+=0$.
Otherwise, choose $u^*\in[0,4\delta]$ at which
$f$ attains its positive maximum.
If $u^* = 0$, then $(s, u^*, \tilde{\vartheta})$ is a point on the outermost characteristic curve, which implies that $\mu(u^*) = 1$ and $T\mu(u^*) = 0$. This contradicts the assumption $f(u^*) > 0$, and hence $u^* \in (0, 4\delta]$, which yields $f'(u^*) \geq 0$. Therefore, together with the fact that $|T^2\mu| \lesssim \delta^{-2}$ by \eqref{eq:higher-order-L-infty-estimates}, one has
	\begin{equation}\label{muT}
	\bigl(\mu^{-1}T\mu\bigr)(u^*) \leq \sqrt{\dfrac{T^2\mu}{\mu}(u^*)}\lesssim\delta^{-1}\bigl(\mu_{\min}(s)\bigr)^{-1/2}.
	\end{equation}
The maximality of $f(u^*)$ therefore yields
\[
f_+(\tilde u)
\lesssim
\delta^{-1}\mu_{\min}(s)^{-1/2},
\quad 0\le\tilde u\le4\delta.
\]
Note that $\tilde\vartheta$ is arbitrary, returning to the
original coordinates gives
\[
(\mu^{-1}T\mu)_+(s,u,\vartheta)
\lesssim
\delta^{-1}\mu_{\min}(s)^{-1/2}
\]
throughout $[0,4\delta]\times\mathbb S$.
Combining this estimate with \eqref{mumin}, we conclude that
\[
(\mu^{-1}T\mu)_+(s,u,\vartheta)
\lesssim
\delta^{-1+\varepsilon_k/2}
(T_\delta-s)^{-1/2}.
\]
This completes the proof.
\end{proof}

\begin{lemma}\label{lem:higher_order_L2_estimates_mu}
Under the assumptions of Proposition~\ref{Prop.5.1}, let $E=
\left\{
\tau\in[t_0,T_\delta):
\mu_{\min}(\tau)<\frac{1}{10}
\right\}$.
For small $\delta>0$, the following
statements hold.

If $E\neq\varnothing$, set $\underline s=\inf E$.
Then, for any $a>1$ and
$\underline s<s<T_\delta$,
\begin{equation}\label{eq:mu_integrability}
\int_{\underline s}^{s}\mu_{\min}^{-a}(\tau)\,d\tau
\lesssim
\frac{\delta^{\varepsilon_k}}{a-1}
\mu_{\min}^{1-a}(s).
\end{equation}

Regardless of whether $E$ is empty, for any
$s,s'\in[t_0, T_\delta)$ with $s<s'$, one has
\begin{equation}\label{eq:mu_monotonicity}
\mu_{\min}(s')\lesssim\mu_{\min}(s).
\end{equation}
\end{lemma}

\begin{proof}
At first, suppose $E\neq\varnothing$.
By Proposition~\ref{Prop.5.1}, once $\mu<1/10$ along a fixed
characteristics, it remains below $1/10$ and decreases.
It follows from the definition of $\underline s$ that $
\mu_{\min}(\tau)<1/10$ when
$\underline s<\tau<T_\delta$.
Fix $\underline s<\tau<s< T_\delta$, and choose
$(u_\tau,\vartheta_\tau)\in[0,4\delta]\times\mathbb S$
such that
\[
\mu(\tau,u_\tau,\vartheta_\tau)=\mu_{\min}(\tau).
\]
Integrating along this fixed characteristics and using
Proposition~\ref{Prop.5.1}, we obtain
\[
\begin{split}
\mu_{\min}(\tau)
&=
\mu(s,u_\tau,\vartheta_\tau)
-\int_\tau^s
\mathring L\mu(\xi,u_\tau,\vartheta_\tau)\,d\xi\ge
\mu_{\min}(s)
+c\delta^{-\varepsilon_k}(s-\tau)
\end{split}
\]
for some positive constant $c$.
By continuity of $\mu_{\min}$, the same inequality
holds at $\tau=\underline s$.
Therefore, for $a>1$,
\[
\begin{split}
\int_{\underline s}^s\mu_{\min}^{-a}(\tau)\,d\tau
&\le
\int_{\underline s}^s
\left(
\mu_{\min}(s)
+c\delta^{-\varepsilon_k}(s-\tau)
\right)^{-a}\,d\tau\\
&=
\frac{\delta^{\varepsilon_k}}{c(a-1)}
\big[
\mu_{\min}^{1-a}(s)
-
\left(
\mu_{\min}(s)
+c\delta^{-\varepsilon_k}(s-\underline s)
\right)^{1-a}
\big]\\
&\le
\frac{\delta^{\varepsilon_k}}{c(a-1)}
\mu_{\min}^{1-a}(s),
\end{split}
\]
which proves \eqref{eq:mu_integrability}.

For any $\underline s<s<s'<T_\delta$,
choose $(u_s,\vartheta_s)$ such that
$\mu(s,u_s,\vartheta_s)=\mu_{\min}(s)$.
The decrease of $\mu$ along this characteristics gives $
\mu_{\min}(s')
\le
\mu(s',u_s,\vartheta_s)
\le
\mu(s,u_s,\vartheta_s)
=
\mu_{\min}(s)$.

For $t_0\le s\le\underline s<s'<T_\delta$,
we have $
\mu_{\min}(s')<\frac{1}{10}\le\mu_{\min}(s)$.

For $t_0\le s<s'\le\underline s$, the uniform upper
bound for $\mu$ from
\eqref{eq:lower-order-Linfty-estimates} and Remark~\ref{Rem4.2},
together with $\mu_{\min}(s)\ge1/10$, yields $
\mu_{\min}(s')\lesssim1\lesssim\mu_{\min}(s)$.
Thus \eqref{eq:mu_monotonicity} holds when $E\neq\varnothing$.

If $E=\varnothing$, then $\mu_{\min}\ge1/10$
throughout $[t_0,T_\delta)$.
The same uniform upper bound therefore gives
\[
\mu_{\min}(s')\lesssim1\lesssim\mu_{\min}(s),
\quad
t_0\le s<s'< T_\delta.
\]
This completes the proof.
\end{proof}

\begin{remark}
We shall also use the following consequence of
\eqref{eq:mu_integrability}:

Set $\underline{\underline{s}}
=\inf\big\{\tau\geq t_0:
\mu_{\min}(\tau)<\frac{1}{100}\big\}$.
When $s\leq\underline{\underline{s}}$, one has
$\mu_{\min}(s)\geq\frac{1}{100}$ and then the negative powers of
$\mu_{\min}$ are bounded from above;
When $s>\underline{\underline{s}}$, decomposing $[t_0, s]=[t_0, \underline{s})\cup [\underline{s}, s]$ and applying
\eqref{eq:mu_integrability} on $[\underline{s},s]$, one can obtain that for
all sufficiently large $a$,
\begin{align}
\int_{t_0}^{s}\mu_{\min}^{-a}(\tau)\,d\tau
&\lesssim
10^a\delta^{\varepsilon_k}
+\frac{\delta^{\varepsilon_k}}{a-1}
\mu_{\min}^{1-a}(s) \nonumber\\
&\lesssim
\delta^{\varepsilon_k}\mu_{\min}^{1-a}(s)
\big(\frac{1}{10^{a-2}}+\frac{1}{a-1}\big) \nonumber\\
&\lesssim
\frac{\delta^{\varepsilon_k}}{a-1}
\mu_{\min}^{1-a}(s).
\label{eq:mu-integrability-full}
\end{align}
\end{remark}

Until now, we have obtained a series of results under the assumption that ``$\mu < \frac{1}{10}$''. A natural question arises:
does there exist some point $(s, u, \vartheta)$ such that $\mu < \frac{1}{10}$ holds? We next answer this problem.
Set $A(\sigma,\omega)=\partial_\sigma\mathcal H(\sigma,\omega)$.
Then $\max A=(k+1)\Xi$. Choosing a maximum point
$(\sigma_0,\omega_0)$ such that $A(\sigma_0,\omega_0)=\max A$. Let $u_0=-\sigma_0\delta$ and
$\vartheta_0=\omega_0\in \Sigma_{t_0}$.
By Theorem~\ref{Local}, in particular, by \eqref{Lphi} and \eqref{y0},
we have
\[
\vp_0(t_0,u_0,\vartheta_0)=\frac{L\phi+\underline L\phi}{2}(t_0,u_0,\vartheta_0)
=\delta^{1-\varepsilon_0}\phi_1(\sigma_0,\omega_0)
+O(\delta^{2-\varepsilon_0-\varepsilon_k}),
\]
\[
T\vp_0(t_0,u_0,\vartheta_0)=\frac{L\vp_0-\underline L\vp_0}{2}(t_0,u_0,\vartheta_0)
=-\delta^{-\varepsilon_0}
\partial_\sigma\phi_1(\sigma_0,\omega_0)
+O(\delta^{-\varepsilon_k}).
\]
Consequently, for $A_0=A(\sigma_0,\omega_0)=(k+1)\Xi$,
\[
\mathscr M(t_0,u_0,\vartheta_0)
=-\frac12\varrho_0^{(k+1)/2}\delta^{-\varepsilon_k}A_0
+O(\delta^{1-\varepsilon_0-\varepsilon_k}).
\]
Note that
\[
\varrho^{-(k-1)/2}-\varrho_0^{-(k-1)/2}
=-\varrho_0^{-(k-1)/2}
\big[
1-\big(1+\frac{s-t_0}{\varrho_0}\big)^{-(k-1)/2}
\big].
\]
Together with Proposition~\ref{Prop.5.1} and
$s-t_0\lesssim \delta^{\varepsilon_k}$, we obtain
that for $t_0\le s< T_\delta$,
\begin{equation}\label{YCao-2}
\mu(s,u_0,\vartheta_0)
=1-\frac{\varrho_0A_0}{k-1}\delta^{-\varepsilon_k}
\big[
1-\big(1+\frac{s-t_0}{\varrho_0}\big)^{-(k-1)/2}
\big]
+O(\delta^{1-\varepsilon_0}).
\end{equation}

When $\varepsilon_0\in(\varepsilon_k^*,1)$, one has
$0<\varepsilon_k<1$.
Set $s=t_0+\tau\delta^{\varepsilon_k}$.
For any bounded $\tau\ge0$ such that $s<T_\delta$,
it follows from the Taylor expansion and \eqref{YCao-2} that
\begin{equation}\label{supercritical-profile}
\mu(t_0+\tau\delta^{\varepsilon_k},u_0,\vartheta_0)
=1-\frac{A_0}{2}\tau
+O(\delta^{\varepsilon_k}+\delta^{1-\varepsilon_0}).
\end{equation}
Due to $A_0>0$, $\mu<\frac{1}{10}$ holds for some time $s$ and small $\delta$.

When $\varepsilon_0=\varepsilon_k^*$, $\varepsilon_k=0$ holds.
Due to $\varrho_0=t_0+\sigma_0\delta=1+O(\delta)$,
it follows from \eqref{YCao-2} that for $t_0\le s<T_\delta$,
\begin{equation}\label{critical-profile}
\mu(s,u_0,\vartheta_0)
=1-\frac{A_0}{k-1}\bigl(1-s^{-(k-1)/2}\bigr)
+O(\delta^{1/(k+1)}).
\end{equation}
Note that the condition \eqref{critical-compression} means $A_0>k-1$.
Since the leading profile in \eqref{critical-profile} is strictly decreasing for $s\ge1$,
$\mu(s,u_0,\vartheta_0)$ has a unique zero point $s_0$ which is close to $s_*=\big(1-\frac{k-1}{A_0}\big)^{-2/(k-1)}
=S_{\mathrm{crit}}$ for small $\delta$.

\section{Energy estimates and some higher-order $L^2$ estimates}\label{EE}

As stated in Remark \ref{Rem4.2}, the bootstrap assumptions in \eqref{bootstrap-assumptions} have been improved up to the index $N-4$. 
However, the closure of the bootstrap assumptions  for higher-order derivatives remains unsolved. Drawing inspiration from the techniques in \cite{Ding2,Ding3,Ding4,M-Y,Sp}, we first establish the energy estimates for $\vp_\gamma$ up to orders $2N$. 
We then utilize the embedding theorem to achieve the closure for all $L^\infty$ estimates in \eqref{bootstrap-assumptions}, 
thereby the proof on the existence of smooth solutions before $T_\delta$ is completed.

Note that $\varphi_\gamma$ satisfies \eqref{ge},
and its derivatives satisfy similar equations with
the same principal operator and additional source terms.
We now establish energy estimates for any smooth
function $\Psi$ that solves the following linear equation:
\begin{equation}\label{gel}
\mu\Box_g\Psi=\Phi,
\end{equation}
where $\Phi$ is a given smooth function, $\Psi$ and its
derivatives vanish on $u=0$.

Following \cite{L-Y,M-Y}, we use the two multipliers
$V_1\Psi=\mathring{L}\Psi$ and
$V_2\Psi=\mathring{\underline{L}}\Psi$.
Integrating $\mu(\Box_g\Psi)(V_i\Psi)$ over the domain
$D^{s,u}$ and using the integration by parts lead to the
corresponding energies $E_i[\Psi](s,u)$ and fluxes
$F_i[\Psi](s,u)$ $(i=1,2)$ which are defined as
\begin{align}
E_1[\Psi](s, u)&= \frac{1}{2} \int_{\Sigma_s^{u}} \mu
\left\{ (\mathring{L}\Psi)^2 + |\slashed{d}\Psi|^2 \right\},
\label{E1} \\
E_2[\Psi](s, u)&= \frac{1}{2} \int_{\Sigma_s^{u}}
\left\{ (\mathring{\underline{L}}\Psi)^2
+ \mu^2 |\slashed{d}\Psi|^2 \right\},
\label{E2} \\
F_1[\Psi](s, u)&= \int_{C_{u}^s} (\mathring{L}\Psi)^2,
\label{F1} \\
F_2[\Psi](s, u)&= \int_{C_{u}^s} \mu \, |\slashed{d}\Psi|^2.
\label{F2}
\end{align}
Thus,
\begin{equation}\label{EI}
E_i[\Psi](s, u)-E_i[\Psi](t_0, u)+F_i[\Psi](s, u)
=-\int_{D^{s, u}}\Phi\cdot V_i\Psi
-\int_{D^{s, u}}\f12\mu Q_{\al\beta}[\Psi]
\leftidx{^{(V_i)}}\pi^{\al\beta},
\end{equation}
where
$Q_{\al\beta}[\Psi]
=(\p_\al\Psi)(\p_\beta\Psi)
-\f12 g_{\al\beta}g^{\nu\lambda}
(\p_\nu\Psi)(\p_\lambda\Psi)$,
and $\leftidx{^{(V_i)}}\pi$ is the deformation tensor of
the vector field $V_i$, as defined in \eqref{dt}.
Using the identities
\begin{align*}
-\frac{1}{2}\mu Q_{\al\beta}[\Psi]
\leftidx{^{(V_1)}}\pi^{\al\beta}
&=\frac{1}{2}\big(
-\varrho^{-1}\mu+\mathring L\mu
-\mu(\operatorname{tr}\check{\mathscr X}
+\f12\slashed g^{XX}G_{XX}^{\gamma}\mathring L\vp_\gamma)
\big)|\slashed d\Psi|^2
-\f12\mathring L\mu|\mathring L\Psi|^2\\
&\quad +(2\mu\zeta^X+\slashed d^X\mu)
(\mathring L\Psi)(\slashed d_X\Psi)
-\f12\operatorname{tr}\chi
(\mathring L\Psi)(\underline L\Psi),\\
-\frac{1}{2}\mu Q_{\al\beta}[\Psi]
\leftidx{^{(V_2)}}\pi^{\al\beta}
&=\frac{1}{2}\left(
\underline L\mu+\mu\mathring L\mu
-\mu^2\operatorname{tr}\chi
-\mu\operatorname{tr}\leftidx{^{(T)}}{\slashed\pi}
\right)|\slashed d\Psi|^2
-\mu\slashed d^X\mu
(\mathring L\Psi)(\slashed d_X\Psi)\\
&\quad -(2\mu\zeta^X+\slashed d^X\mu)
(\underline L\Psi)(\slashed d_X\Psi)
-\f12\big(
\mu\operatorname{tr}\chi
+\operatorname{tr}\leftidx{^{(T)}}{\slashed\pi}
\big)(\mathring L\Psi)(\underline L\Psi),
\end{align*}
together with \eqref{eq:higher-order-L-infty-estimates},
\eqref{blowup} and \eqref{Tmu}, we obtain
\begin{align}
&\delta^{\varepsilon_k}\int_{D^{s, u}}
-\frac{1}{2}\mu Q_{\al\beta}[\Psi]
\leftidx{^{(V_1)}}\pi^{\al\beta}\nonumber\\
\leq{}&
\frac12\delta^{\varepsilon_k}\int_{D^{s,u}}
(\mathring L\mu)|\slashed d\Psi|^2
+C\delta^{\varepsilon_k}\int_{D^{s,u}}
\Big\{
\delta^{k(1-\varepsilon_0)}
(1+\delta^{2k(1-\varepsilon_0)-\varepsilon_0})
\mu|\slashed d\Psi|^2
+\delta^{-\varepsilon_k}|\mathring L\Psi|^2
\Big\}\nonumber\\
&+C\delta^{\varepsilon_k}\int_{D^{s,u}}
\Big\{
\delta^{k(1-\varepsilon_0)-1}
|\mathring L\Psi||\slashed d\Psi|
+(1+\delta^{3k(1-\varepsilon_0)-\varepsilon_0})
|\mathring L\Psi||\mathring{\underline L}\Psi|
\Big\}\nonumber\\
\leq{}&
C\delta^{-\varepsilon_k}\int_{t_0}^s
\delta^{\varepsilon_k}E_1[\Psi](\tau,u)\,d\tau
-c\int_{D^{s,u}\cap\{\mu<\f{1}{10}\}}
|\slashed d\Psi|^2
+C\delta^{-1}\int_0^u
\delta^{\varepsilon_k}F_1[\Psi](s,u')\,du'
\label{QV1}\\
&+C\big(
\delta^{\varepsilon_k}
+\delta^{5k(1-\varepsilon_0)-\varepsilon_0}
\big)\int_{t_0}^s
\delta E_2[\Psi](\tau,u)\,d\tau\nonumber
\end{align}
for some constants $c,C>0$, and
\begin{align}
&\delta\int_{D^{s,u}}
-\frac{1}{2}\mu Q_{\al\beta}[\Psi]
\leftidx{^{(V_2)}}\pi^{\al\beta}\nonumber\\
\lesssim{}&
\delta\int_{D^{s,u}}
\Big\{
(\mu^{-1}T\mu)_+(\mu|\slashed d\Psi|^2)
+\delta^{k(1-\varepsilon_0)-1}
(\mu|\slashed d\Psi|^2)
+\delta^{k(1-\varepsilon_0)-1}
|\underline L\Psi||\slashed d\Psi|
+\mu|\mathring L\Psi||\slashed d\Psi|\nonumber\\
&\qquad\quad
+\delta^{k(1-\varepsilon_0)-1}
|\mathring L\Psi||\underline L\Psi|
\Big\}\nonumber\\
\lesssim{}&
\delta^{\varepsilon_k/2}\int_{t_0}^s
(T_\delta-\tau)^{-1/2}
E_1[\Psi](\tau,u)\,d\tau
+\delta^{-1}\int_0^u
\delta F_2[\Psi](s,u')\,du'
+\delta^{-\varepsilon_k}\int_{t_0}^s
\delta E_2[\Psi](\tau,u)\,d\tau
\label{QV2}\\
&+\delta^{(k-1)(1-\varepsilon_0)}
\int_{D^{s,u}\cap\{\mu<\f{1}{10}\}}
|\slashed d\Psi|^2
+\delta\int_{t_0}^sE_1[\Psi](\tau,u)\,d\tau
+\delta^{(k-1)(1-\varepsilon_0)}
\int_0^uF_1[\Psi](s,u')\,du'.\nonumber
\end{align}
It follows from \eqref{EI}, \eqref{QV1}, \eqref{QV2}
and Gronwall's inequality that
\begin{equation}\label{e}
\begin{split}
&\delta E_2[\Psi](s,u)+\delta F_2[\Psi](s,u)
+\delta^{\varepsilon_k}E_1[\Psi](s,u)
+\delta^{\varepsilon_k}F_1[\Psi](s,u)
+\int_{D^{s,u}\cap\{\mu<\tfrac{1}{10}\}}
|\slashed d\Psi|^2\\
\lesssim{}&
\delta E_2[\Psi](t_0,u)
+\delta^{\varepsilon_k}E_1[\Psi](t_0,u)+\sup_{\substack{t_0\leq s'\leq s\\0\leq u'\leq u}}
\big\{
\big|
\int_{D^{s',u'}}
\Phi\cdot\delta\mathring{\underline L}\Psi
\big|
+\big|
\int_{D^{s',u'}}
\Phi\cdot\delta^{\varepsilon_k}\mathring L\Psi
\big|
\big\}.
\end{split}
\end{equation}
Choose $\Psi=\Psi_\gamma^{n+1}=Z\Psi_\gamma^n$ and
$\Phi=\Phi_\gamma^{n+1}=\mu\Box_g\Psi_\gamma^{n+1}$,
where $Z\in\{\varrho\mathring L,T,R\}$ and $0\leq n\leq 2N-6$,
so that \eqref{gel} holds. Note that
\begin{equation}\label{Psi}
\begin{split}
\Phi_\gamma^{n+1}
&=\mu[\Box_g,Z]\Psi_\gamma^n
+Z\big(\mu\Box_g\Psi_\gamma^n\big)
-(Z\mu)\Box_g\Psi_\gamma^n\\
&=\mu\mathscr D^\alpha
\big({\leftidx{^{(Z)}}C_\gamma^n}_{,\alpha}\big)
+\big(Z+\leftidx{^{(Z)}}\lambda\big)\Phi_\gamma^n,
\end{split}
\end{equation}
where
\begin{equation}\label{ClP}
\begin{split}
{\leftidx{^{(Z)}}C_\gamma^n}_{,\alpha}
&=\big(\leftidx{^{(Z)}}\pi_{\nu\alpha}
-\frac12 g_{\nu\alpha}
\big(g_{\kappa\lambda}\leftidx{^{(Z)}}\pi^{\kappa\lambda}\big)\big)
g^{\nu\beta}\partial_\beta\Psi_\gamma^n,\\
\leftidx{^{(Z)}}\lambda
&=-\mu^{-1}\leftidx{^{(Z)}}\pi_{\mathring LT}
+\frac12\operatorname{tr}\leftidx{^{(Z)}}{\slashed\pi}
-\mu^{-1}Z\mu,\\
\Psi_\gamma^0&=\varphi_\gamma,\quad
\Phi_\gamma^0=\mu\Box_g\varphi_\gamma,
\end{split}
\end{equation}
and $\Phi_\gamma^0$ equals the right-hand side of \eqref{ge}.

Consequently, for
$\Psi_\gamma^{n+1}=Z_{n+1}Z_n\cdots Z_1\varphi_\gamma$
with $Z_j\in\{\varrho\mathring L,T,R\}$ and $0\leq n\leq 2N-6$,
the induction argument using \eqref{Psi} gives
\begin{equation}\label{Phik}
\begin{split}
\Phi_\gamma^{n+1}
&= \sum_{j=1}^{n}
\big(Z_{n+1}+\leftidx{^{(Z_{n+1})}}\lambda\big)\cdots
\big(Z_{n+2-j}+\leftidx{^{(Z_{n+2-j})}}\lambda\big)
\big(\mu\mathscr{D}^\alpha
{\leftidx{^{(Z_{n+1-j})}}C_{\gamma}^{n-j}}_{,\alpha}\big)\\
&\quad +\mu\mathscr{D}^\alpha
{\leftidx{^{(Z_{n+1})}}C_\gamma^{n}}_{,\alpha}
+\big(Z_{n+1}+\leftidx{^{(Z_{n+1})}}\lambda\big)\cdots
\big(Z_1+\leftidx{^{(Z_1)}}\lambda\big)\Phi_\gamma^0\\
&=J_1^{n+1}+J_2^{n+1}, \quad n\geq1,\\
\Phi_\gamma^1
&=\big(Z_1+\leftidx{^{(Z_1)}}\lambda\big)\Phi_\gamma^0
+\mu\mathscr{D}^\alpha
{\leftidx{^{(Z_1)}}C_\gamma^0}_{,\alpha},
\end{split}
\end{equation}
where $J_1^{n+1}$ denotes the sum over $j$ in \eqref{Phik},
and $J_2^{n+1}$ denotes the remaining two terms in the
expression for $\Phi_\gamma^{n+1}$.

By \eqref{Lpi}--\eqref{Rpi}, one has
\begin{equation}\label{lamda}
\begin{split}
\leftidx{^{(T)}}\lambda
&=\frac{1}{2}\operatorname{tr}
\leftidx{^{(T)}}{\slashed{\pi}}, \quad
\leftidx{^{(\varrho\mathring L)}}\lambda
=\varrho\big\{
\operatorname{tr}_{\slashed g}\check{\mathscr X}
+\f12\slashed g^{XX}G_{XX}^\al\mathring L\vp_\al
\big\}+2, \quad
\leftidx{^{(R)}}\lambda
=\frac{1}{2}\operatorname{tr}
\leftidx{^{(R)}}{\slashed{\pi}}.
\end{split}
\end{equation}

In addition, as in \cite[(5.1)--(5.6) and (5.10)--(5.12)]{Ding4},
the term
$\mu\mathscr{D}^\alpha
{\leftidx{^{(Z)}}C_\gamma^n}_{,\alpha}$
can be written as
\begin{equation}\label{muZC}
\mu\mathscr{D}^\alpha
{\leftidx{^{(Z)}}C_\gamma^n}_{,\alpha}
=\leftidx{^{(Z)}}D_{\gamma,1}^n
+\leftidx{^{(Z)}}D_{\gamma,2}^n
+\leftidx{^{(Z)}}D_{\gamma,3}^n,
\end{equation}
where
\begin{align}
\leftidx{^{(T)}}D_{\gamma,1}^{n}
={}&(T\mu)\mathring L^2\Psi_\gamma^{n}
+\mu(\slashed d_X\mu+2\mu\zeta_X)
\slashed d^X\mathring L\Psi_\gamma^{n}
+\f12\operatorname{tr}\leftidx{^{(T)}}{\slashed\pi}
\big(\mathring L\mathring{\underline L}\Psi_\gamma^{n}
+\f12\operatorname{tr}\chi
\mathring{\underline L}\Psi_\gamma^{n}\big)\no\\
&+(\slashed d_X\mu+2\mu\zeta_X)
\slashed d^X\mathring{\underline L}\Psi_\gamma^{n}
-(T\mu)\slashed\triangle\Psi_\gamma^{n}
+\f12\mu\operatorname{tr}\leftidx{^{(T)}}{\slashed\pi}
\slashed\triangle\Psi_\gamma^{n},\label{T1}\\
\leftidx{^{(T)}}D_{\gamma,2}^{n}
={}&\big\{
\mathring L T\mu
+\f14\uwave{\mathring{\underline L}
(\operatorname{tr}\leftidx{^{(T)}}{\slashed\pi})}
+\boxed{\slashed\nabla_X\big(\f12\mu\slashed d^X\mu}
+\mu^2\zeta^X\big)
\big\}\mathring L\Psi_\gamma^n
+\big\{
\f14\mathring L
(\operatorname{tr}\leftidx{^{(T)}}{\slashed\pi})\no\\
&+\f12\boxed{\slashed\nabla_X(\slashed d^X\mu}
+2\mu\zeta^X)
\big\}\mathring{\underline L}\Psi_\gamma^n
+\big\{
\f12\slashed{\mathcal L}_{\mathring L}
(2\mu^2\zeta_X+\mu\slashed d_X\mu)
-\uline{\slashed d_XT\mu}\no\\
&+\underline{\f12\slashed{\mathcal L}_{\mathring{\underline L}}
(\slashed d_X\mu}
+2\mu\zeta_X)
+\underbrace{\f12\slashed d_X
(\mu\operatorname{tr}\leftidx{^{(T)}}{\slashed\pi})}
\big\}\slashed d^X\Psi_\gamma^n,\label{T2}\\
\leftidx{^{(T)}}D_{\gamma,3}^{n}
={}&\big\{
\operatorname{tr}\chi T\mu
+\f14\big(\mu\operatorname{tr}\chi
+\operatorname{tr}\leftidx{^{(T)}}{\slashed\pi}\big)
\operatorname{tr}\leftidx{^{(T)}}{\slashed\pi}
-\f12|\slashed d\mu|^2
-\mu\zeta_X(\slashed d^X\mu)
\big\}\mathring L\Psi_\gamma^{n}\no\\
&+\big(\f12\mathring L\mu-\mu\operatorname{tr}\chi\big)
(\slashed d_X\mu+2\mu\zeta_X)
\slashed d^X\Psi_\gamma^{n},\label{T3}\\
\leftidx{^{(\varrho\mathring L)}}D_{\gamma,1}^{n}
={}&(2-\mu+\varrho\mathring L\mu)
\mathring L^2\Psi_\gamma^{n}
-2\varrho(\slashed d_X\mu+2\mu\zeta_X)
\slashed d^X\mathring L\Psi_\gamma^{n}
+\varrho\operatorname{tr}\chi
\big(\mathring L\mathring{\underline L}\Psi_\gamma^{n}
+\f12\operatorname{tr}\chi
\mathring{\underline L}\Psi_\gamma^{n}\big)\no\\
&-\varrho\big(
\mathring L\mu-\mu\operatorname{tr}\check{\mathscr X}
-\f12\mu\slashed g^{XX}G_{XX}^\al\mathring L\vp_\al
\big)\slashed\triangle\Psi_\gamma^{n},\label{rL1}\\
\leftidx{^{(\varrho\mathring L)}}D_{\gamma,2}^{n}
={}&\big\{
\mathring L\big(\varrho\mathring L\mu-\mu\big)
+\uwave{\f12\mathring{\underline L}
\big(\varrho\operatorname{tr}\check{\mathscr X}
+\f12\varrho\slashed g^{XX}G_{XX}^\al\mathring L\vp_\al\big)}
-\varrho\boxed{\slashed\nabla_X\big(\slashed d^X\mu}
+2\mu\zeta^X\big)
\big\}\mathring L\Psi_\gamma^n\no\\
&+\f12\mathring L
\big(\varrho\operatorname{tr}\check{\mathscr X}
+\f12\varrho\slashed g^{XX}G_{XX}^\al\mathring L\vp_\al\big)
\mathring{\underline L}\Psi_\gamma^n\no\\
&-\big\{
\slashed{\mathcal L}_{\mathring L}
\big(\varrho\slashed d_X\mu+2\varrho\mu\zeta_X\big)
+\slashed d_X\big(\mu+\varrho\mathring L\mu\big)
-\varrho\underbrace{\slashed d_X(\mu\operatorname{tr}\chi)}
\big\}\slashed d^X\Psi_\gamma^n,\label{rL2}\\
\leftidx{^{(\varrho\mathring L)}}D_{\gamma,3}^{n}
={}&\operatorname{tr}\chi
\big\{
2-\mu+\varrho\mathring L\mu
+\f12\varrho\operatorname{tr}\leftidx{^{(T)}}{\slashed\pi}
+\f12\varrho\mu\operatorname{tr}\chi
\big\}\mathring L\Psi_\gamma^{n}
+2\varrho\operatorname{tr}\chi
(2\mu\zeta^X+\slashed d^X\mu)
\slashed d_X\Psi_\gamma^n,\label{rL3}\\
\leftidx{^{(R)}}D_{\gamma,1}^{n}
={}&(R\mu)\mathring L^2\Psi_\gamma^{n}
-\leftidx{^{(R)}}{\slashed\pi}_{\mathring{\underline L}X}
\slashed d^X\mathring L\Psi_\gamma^{n}
+\f12\operatorname{tr}\leftidx{^{(R)}}{\slashed\pi}
\big(\mathring L\mathring{\underline L}\Psi_\gamma^{n}
+\f12\operatorname{tr}\chi
\mathring{\underline L}\Psi_\gamma^{n}\big)
-\leftidx{^{(R)}}{\slashed\pi}_{\mathring LX}
\slashed d^X\mathring{\underline L}\Psi_\gamma^{n}\no\\
&-(R\mu)\slashed\triangle\Psi_\gamma^{n}
+\f12\mu\operatorname{tr}\leftidx{^{(R)}}{\slashed\pi}
\slashed\triangle\Psi_\gamma^{n},\label{R1}\\
\leftidx{^{(R)}}D_{\gamma,2}^n
={}&\big\{
\mathring LR\mu
-\f12\underbrace{\slashed\nabla^X
\leftidx{^{(R)}}{\slashed\pi}_{\mathring{\underline L}X}}
+\f14\uwave{\mathring{\underline L}
(\operatorname{tr}\leftidx{^{(R)}}{\slashed\pi})}
\big\}\mathring L\Psi_\gamma^n
+\big\{
\f14\mathring L
(\operatorname{tr}\leftidx{^{(R)}}{\slashed\pi})\no\\
&-\f12\underbrace{\slashed\nabla^X
\leftidx{^{(R)}}{\slashed\pi}_{\mathring LX}}
\big\}\mathring{\underline L}\Psi_\gamma^n
+\big\{
-\f12\slashed{\mathcal L}_{\mathring L}
\leftidx{^{(R)}}{\slashed\pi}_{\mathring{\underline L}X}
-\f12\uwave{\slashed{\mathcal L}_{\mathring{\underline L}}
\leftidx{^{(R)}}{\slashed\pi}_{\mathring LX}}
-\boxed{\slashed d_XR\mu}\no\\
&+\f12\underbrace{\slashed d_X
(\mu\operatorname{tr}\leftidx{^{(R)}}{\slashed\pi})}
\big\}\slashed d^X\Psi_\gamma^n,\label{R2}\\
\leftidx{^{(R)}}D_{\gamma,3}^{n}
={}&\big\{
\operatorname{tr}\chi R\mu
+\f14\big(
\operatorname{tr}\leftidx{^{(T)}}{\slashed\pi}
+\mu\operatorname{tr}\chi
\big)\operatorname{tr}\leftidx{^{(R)}}{\slashed\pi}
+\f12\slashed d^X\mu
\leftidx{^{(R)}}{\slashed\pi}_{\mathring LX}
\big\}\mathring L\Psi_\gamma^{n}
+\big\{
\operatorname{tr}\chi
\leftidx{^{(R)}}{\slashed\pi}_{TX}\no\\
&+\operatorname{tr}\leftidx{^{(R)}}{\slashed\pi}
(\mu\zeta_X+\f12\slashed d_X\mu)
+\big(
-\f12\mathring L\mu+\mu\operatorname{tr}\chi
+\f12\operatorname{tr}\leftidx{^{(T)}}{\slashed\pi}
\big)\leftidx{^{(R)}}{\slashed\pi}_{\mathring LX}
\big\}\slashed d^X\Psi_\gamma^{n}.\label{R3}
\end{align}
Note that all the terms in $\leftidx{^{(Z)}}D_{\gamma,1}^m$ are products of the deformation tensor and the second-order
derivatives of $\Psi_\gamma^m$ except the first term which contains a factor of the form
$\mathring{L}\mathring{\underline{L}}\Psi_\gamma^m + \frac{1}{2}\operatorname{tr}{\chi}\mathring{\underline{L}}\Psi_\gamma^m$
(see \eqref{T1}, \eqref{rL1} and \eqref{R1}).
The structure of $\mathring{L}\mathring{\underline{L}}\Psi_\gamma^m + \frac{1}{2}\operatorname{tr}{\chi}\mathring{\underline{L}}\Psi_\gamma^m$
is important for our analysis. By \eqref{fequation}, it holds that
\[
\mathring{L}\mathring{\underline{L}}\varphi_\gamma
+\frac{1}{2}\operatorname{tr}\chi
\mathring{\underline{L}}\varphi_\gamma
=\mu\slashed\triangle\vp_\gamma+H_\gamma
+\frac{1}{2}\big(
\operatorname{tr}\check{\mathscr X}
+\f12\slashed g^{XX}G_{XX}^\al\mathring L\vp_\al
\big)\mathring{\underline{L}}\varphi_\gamma.
\]
This identity exhibits the better smallness than either $\mathring{L}\mathring{\underline{L}}\varphi_\gamma$ or $\frac{1}{2}\operatorname{tr}{\chi}\mathring{\underline{L}}\varphi_\gamma$ individually.
The term $\leftidx{^{(Z)}}D_{\gamma,2}^m$ collects all products of the first-order derivatives of the deformation tensor and the first-order derivatives of $\Psi_\gamma^m$, while $\leftidx{^{(Z)}}D_{\gamma,3}^m$ contains the remaining terms.

The explicit expressions for $\Phi_\gamma^{n+1}$ given in
\eqref{Phik} and \eqref{muZC}--\eqref{R3} are crucial for estimating
the two source integrals in \eqref{e}.
To close our energy estimates, we introduce the following energies
and fluxes, following \cite{Ch1,M-Y,Sp}:
\begin{equation}\label{higher order energy and flux}
\begin{split}
\mathcal{E}_{i,\leq n+1}(s,u)
&=\sup_{t_0\leq s'\leq s}\big\{\sum_{\gamma=0}^2\sum_{|\alpha|\leq n}
\delta^{2l+2p\varepsilon_k}
E_i\bigl[Z^\alpha\vp_\gamma\bigr](s',u)\big\},
\quad i=1,2, \\[8pt]
\mathcal{F}_{i,\leq n+1}(s,u)
&=\sum_{\gamma=0}^2\sum_{|\alpha|\leq n}
\delta^{2l+2p\varepsilon_k}
F_i\bigl[Z^\alpha\vp_\gamma\bigr](s,u),
\quad i=1,2.\\
\widetilde{\mathcal{E}}_{i,\leq n+1}(s,u)
&=\sup_{t_0\leq s'\leq s}
\bigl\{
\mu_{\min}^{2b_{n+1}}(s')
\mathcal{E}_{i,\leq n+1}(s',u)
\bigr\},
\quad i=1,2, \\[8pt]
\widetilde{\mathcal{F}}_{i,\leq n+1}(s,u)
&=\sup_{t_0\leq s'\leq s}
\bigl\{
\mu_{\min}^{2b_{n+1}}(s')
\mathcal{F}_{i,\leq n+1}(s',u)
\bigr\},
\quad i=1,2,
\end{split}
\end{equation}
where $l$ (respectively $p$) denotes the number of $T$
(respectively $\varrho\mathring{L}$) operators in each $Z^\alpha$,
and the exponents $\{b_{n+1}\}$ will be determined subsequently.

To establish higher-order energy estimates for \eqref{ge}, we next derive the
higher-order $L^2$ estimates for certain associated quantities.
These estimates allow us to control the two source integrals in
\eqref{e} and thereby close the energy estimates.
For this purpose, we shall employ two lemmas from
\cite[Lemma 7.3]{M-Y} and \cite[Lemma 12.57]{Sp}, respectively.

\begin{lemma}\label{L2T}
Under the assumptions \eqref{bootstrap-assumptions}, let
$\psi\in C^1(\overline{D^{s,u}})$ vanish on $C_0^s$.
For small $\delta>0$, the following estimates hold:
\begin{align}
\int_{S_{s,u}}\psi^2
&\lesssim \delta\int_{\Sigma_s^u}|T\psi|^2
\lesssim \delta\int_{\Sigma_s^u}
\bigl(|\mathring{\underline L}\psi|^2
+\mu^2|\mathring L\psi|^2\bigr)
\lesssim \delta\bigl(E_2[\psi](s,u)+E_1[\psi](s,u)\bigr),
\label{SSi}\\
\int_{\Sigma_s^u}\psi^2
&\lesssim \delta^2\int_{\Sigma_s^u}|T\psi|^2
\lesssim \delta^2\int_{\Sigma_s^u}
\bigl(|\mathring{\underline L}\psi|^2
+\mu^2|\mathring L\psi|^2\bigr)
\lesssim \delta^2\bigl(E_2[\psi](s,u)+E_1[\psi](s,u)\bigr).
\label{SiSi}
\end{align}
\end{lemma}

\begin{lemma}\label{L2L}
For a given function $f\in C(\overline{D^{s,u}})$, define
\[
F(s,u,\vartheta)=\int_{t_0}^s f(\tau,u,\vartheta)\,d\tau.
\]
Under the assumptions \eqref{bootstrap-assumptions}, the following
estimate holds for small $\delta>0$:
\begin{equation}\label{Ff}
\|F\|_{s,u}\lesssim\int_{t_0}^s\|f\|_{\tau,u}\,d\tau,
\end{equation}
where $\|F\|_{s,u}$ denotes the $L^2$ norm of $F$ over
$\Sigma_s^u$, as defined in Definition~\ref{3.4}.
\end{lemma}

\subsection{$L^2$ norms of some basic quantities}

The main purpose of this section is to establish the $L^2$ estimates for
some geometric quantities related to the formation of singularities.
Comparing with the global existence problem investigated in \cite{Ding4}, where
the uniform global energy estimates are derived,
the present blowup problem requires to establish the uniform weighted energy estimates up to the
lifespan. In this process,
we have to carefully exploit the higher-order null structure to prevent
any additional regularity loss caused by negative powers of $\mu$.
On the other hand, a crucial difficulty arises from the term $
G_{\mathring L\mathring L}^{\gamma}\slashed\triangle\varphi_\gamma$
appearing in the transport equation \eqref{Lchi'} of the good unknown $\check{\mathscr X}$.
In our previous work \cite{Ding4}, this term has been treated well by solving the wave
equation divided by $\mu$ and making full use of the uniform positive lower bound of $\mu$.
However, $\mu$ becomes singular near the
blowup time in the present setting. To overcome this difficulty, we find a new decomposition of
$G_{\mathring L\mathring L}^{\gamma}\slashed\triangle\varphi_\gamma$
that preserves the good property from the higher-order null condition without
division by $\mu$. This decomposition leads to a new transport equation
of $\check{\mathscr X}$, which will  provide the essential $L^2$ control needed
for both the lower-order and top-order energy estimates in
the following subsections.

\begin{lemma}\label{lem:null-transport}
Let $\phi$ be a smooth solution of \eqref{quasi} in the region
with $\mu>0$. Then the following decomposition holds:
\begin{equation}\label{Gtrivp}
\begin{split}
G_{\mathring L\mathring L}^\gamma\slashed{\triangle}\varphi_\gamma
=&\slashed{\nabla}^X\bigl(G_{\mathring L\mathring L}^\gamma
\slashed d_X\varphi_\gamma\bigr)
-\bigl(\slashed{\nabla}^XG_{\mathring L\mathring L}^\gamma\bigr)
\slashed d_X\varphi_\gamma\\
=&G_{\mathring L\mathring L}^0\slashed d_Xx^i
\slashed d^X\mathring L\varphi_i
-\bigl(G_{\mathring L\mathring L}^\gamma\mathring L_\gamma\bigr)
\tilde T^i\slashed{\triangle}\varphi_i
-g_{\gamma j}G_{\mathring L\mathring L}^\gamma
\bigl(\mathring L^\alpha+2g^{0\alpha}\bigr)
\slashed d_Xx^j\slashed d^X\mathring L\varphi_\alpha\\
&+\slashed\nabla^X\bigl(G_{\mathring L\mathring L}^0
\slashed d_Xx^i\bigr)\mathring L\varphi_i
-\slashed\nabla^X\bigl(G_{\mathring L\mathring L}^\gamma
\mathring L_\gamma\tilde T^i\bigr)\slashed d_X\varphi_i
-\bigl(\slashed\nabla^XG_{\mathring L\mathring L}^\gamma\bigr)
\slashed d_X\varphi_\gamma\\
&-\slashed\nabla^X\bigl(g_{\gamma j}
G_{\mathring L\mathring L}^\gamma
\bigl(\mathring L^\alpha+2g^{0\alpha}\bigr)
\slashed d_Xx^j\bigr)\mathring L\varphi_\alpha.
\end{split}
\end{equation}
Set
$
E_0=-G_{X\mathring L}^\gamma\slashed d^X\varphi_\gamma
+\frac{1}{2}G_{\mathring L\mathring L}^0
\slashed d_Xx^i\slashed d^X\varphi_i
-\frac{1}{2}g_{\gamma j}G_{\mathring L\mathring L}^\gamma
\bigl(\mathring L^\alpha+2g^{0\alpha}\bigr)
\slashed d_Xx^j\slashed d^X\varphi_\alpha$.
Then $\operatorname{tr}\check{\mathscr X}-E_0$
satisfies the transport equation
\begin{equation}\label{transport}
\begin{split}
&\mathring L\bigl(\varrho^2
\bigl(\operatorname{tr}\check{\mathscr X}-E_0\bigr)\bigr)\\
=&-\frac{1}{2}\varrho^2
\bigl(G_{\mathring L\mathring L}^\gamma\mathring L_\gamma\bigr)
\tilde T^i\slashed\triangle\varphi_i
-\varrho^2\bigl(\operatorname{tr}\check{\mathscr X}\bigr)^2
+\frac{1}{2}\varrho^2\slashed\nabla^X
\bigl(G_{\mathring L\mathring L}^0\slashed d_Xx^i\bigr)
\mathring L\varphi_i
-\frac{1}{2}\varrho^2\slashed\nabla^X
\bigl(G_{\mathring L\mathring L}^\gamma
\mathring L_\gamma\tilde T^i\bigr)\slashed d_X\varphi_i\\
&-\bigl(
\frac{1}{2}G_{\mathring L\mathring L}^\alpha\mathring L\varphi_\alpha
+G_{\tilde T\mathring L}^\alpha\mathring L\varphi_\alpha
-G_{X\mathring L}^\alpha\slashed d^X\varphi_\alpha
+\slashed g^{XX}G_{XX}^\alpha\mathring L\varphi_\alpha
\bigr)
\bigl(\varrho^2\operatorname{tr}\check{\mathscr X}+\varrho\bigr)\\
&-\frac{1}{2}\varrho^2
\bigl(\slashed\nabla^XG_{\mathring L\mathring L}^\gamma\bigr)
\slashed d_X\varphi_\gamma
-\frac{1}{2}\varrho^2\slashed\nabla^X
\bigl(g_{\gamma j}G_{\mathring L\mathring L}^\gamma
\bigl(\mathring L^\alpha+2g^{0\alpha}\bigr)
\slashed d_Xx^j\bigr)\mathring L\varphi_\alpha\\
&+\mathring L\bigl(\varrho^2G_{X\mathring L}^\gamma
\slashed g^{XX}\bigr)\slashed d_X\varphi_\gamma
-\frac{1}{2}\mathring L
\bigl(\varrho^2G_{\mathring L\mathring L}^0
\slashed d^Xx^i\bigr)\slashed d_X\varphi_i+\frac{1}{2}\mathring L
\bigl(\varrho^2g_{\gamma j}G_{\mathring L\mathring L}^\gamma
\bigl(\mathring L^\alpha+2g^{0\alpha}\bigr)
\slashed d^Xx^j\bigr)\slashed d_X\varphi_\alpha\\
&+\varrho^2f\bigl(\varphi,\slashed dx,\mathring L^1,\mathring L^2\bigr)
\begin{pmatrix}
\slashed d(\vp^{k-1})\\
\vp^{2k-2}\slashed d\varphi
\end{pmatrix}
\begin{pmatrix}
\mathring L\varphi\\
\slashed d\varphi
\end{pmatrix}.
\end{split}
\end{equation}
\end{lemma}

\begin{remark}\label{DXY-4}
It is pointed out that the introduction of the corrector $E_0$ is motivated from \eqref{Gtrivp} and \eqref{Lchi'}.
In addition, the coefficient of the  second-order angular derivative $\slashed\triangle\varphi_i$
in \eqref{transport} contains the crucial factor
$G_{\mathring L\mathring L}^{\gamma}\mathring L_{\gamma}$,
which is regular due to the null condition (see \eqref{gLLL}).
\end{remark}

\begin{proof}
It follows from \eqref{gab} and \eqref{quasi} that
\[
\slashed d^Xx^i\slashed d_X\varphi_i
=-\bigl(\mathring L^\alpha+2g^{0\alpha}\bigr)
\mathring L\varphi_\alpha.
\]
Together with \eqref{pal}, this implies
\begin{equation*}
\begin{split}
G_{\mathring L\mathring L}^\gamma\slashed d_X\varphi_\gamma
=&G_{\mathring L\mathring L}^\gamma\slashed d_Xx^i
\bigl(\delta_\gamma^0\mathring L\varphi_i
-\mu^{-1}\mathring L_\gamma T\varphi_i
+g_{\gamma j}\slashed d^Xx^j\slashed d_X\varphi_i\bigr)\\
=&G_{\mathring L\mathring L}^0
\bigl(\mathring L\varphi_i\slashed d_Xx^i\bigr)
-\bigl(G_{\mathring L\mathring L}^\gamma\mathring L_\gamma\bigr)
\tilde T^i\slashed d_X\varphi_i
-g_{\gamma j}G_{\mathring L\mathring L}^\gamma
\bigl(\mathring L^\alpha+2g^{0\alpha}\bigr)
\mathring L\varphi_\alpha\slashed d_Xx^j.
\end{split}
\end{equation*}
Taking the angular divergence of this identity and applying the chain
rule yield \eqref{Gtrivp}. Substituting \eqref{Gtrivp} into
\eqref{Lchi'} together with the chain rule gives
\eqref{transport}.
\end{proof}

We next use the transport equation \eqref{transport}  to establish some preliminary $L^2$
estimates.

\begin{proposition}\label{EstimatesOnR}
Under the bootstrap assumptions \eqref{bootstrap-assumptions}, for each integer $m$ with
$0\leq m\leq 2N-6$, the following estimates hold for small $\delta>0$:
\begin{equation}\label{RmXE}
\begin{split}
\|R^m\bigl(\operatorname{tr}\check{\mathscr X}-E_0\bigr)\|_{s,u}
\lesssim&\ \delta^{k(1-\varepsilon_0)+1/2}
+\delta^{k(1-\varepsilon_0)}\Xi_{m+2}^1(s,u),
\end{split}
\end{equation}
\begin{equation}\label{RmcX}
\begin{split}
&\|R^m\operatorname{tr}\check{\mathscr X}\|_{s,u}
+\|R^{m+1}\check L^i\|_{s,u}
+\|R^{m+1}\upsilon\|_{s,u}
+\|R^{m+1}\check\varrho\|_{s,u}
+\bigl\|\slashed{\mathcal L}_R^{m}{}^{(R)}\slashed{\pi}\bigr\|_{s,u}\\
\lesssim&\ \delta^{k(1-\varepsilon_0)+1/2}
+\delta^{k(1-\varepsilon_0)}\Xi_{m+2}^1(s,u)
\end{split}
\end{equation}
and
\begin{equation}\label{RmPiL}
\begin{split}
\bigl\|\slashed{\mathcal L}_R^{m}{}^{(R)}
\slashed{\pi}_{\mathring L}\bigr\|_{s,u}
\lesssim&\ \bigl(1+\delta^{2k(1-\varepsilon_0)-\varepsilon_0}\bigr)
\bigl(\delta^{k(1-\varepsilon_0)+1/2}
+\delta^{k(1-\varepsilon_0)}\Xi_{m+2}^1(s,u)\bigr),
\end{split}
\end{equation}
where $E_0$ is defined in Lemma~\ref{lem:null-transport} and
\[
\begin{aligned}
\Xi_{m+2}^1(s,u)
={}&\int_{t_0}^s
\bigl\{\mu_{\min}^{-1/2}(\tau)
\sqrt{\mathcal E_{1,\leq m+2}(\tau,u)}
+\delta^{k(1-\varepsilon_0)}
\sqrt{\mathcal E_{2,\leq m+2}(\tau,u)}\bigr\}\,d\tau+\delta^{\varepsilon_0}
\sum_{j=1}^2\sqrt{\mathcal E_{j,\leq m+2}(s,u)}.
\end{aligned}
\]
\end{proposition}

\begin{proof}
Applying \eqref{SiSi} directly yields
\begin{equation}\label{lowerL}
\begin{split}
\|R^{\leq N-4}\check L^i\|_{s,u}
&\lesssim\delta\|TR^{\leq N-4}\check L^i\|_{s,u}\lesssim\delta\|R^{\leq N-4}T\check L^i\|_{s,u}
+\delta\bigl\|\bigl[T,R^{\leq N-4}\bigr]\check L^i\bigr\|_{s,u}.
\end{split}
\end{equation}
Combining \eqref{lowerL} with \eqref{TL} and
Proposition~\ref{prop:higher-order-L-infty-estimates},
and then applying \eqref{SiSi} again, we arrive at
\begin{equation}\label{YHC-A9}
\begin{split}
\|R^{\leq N-4}\check L^i\|_{s,u}
\lesssim&\ \delta^{3/2}
+\delta^{2k(1-\varepsilon_0)}
\|R^{\leq N-4}\check L^i\|_{s,u}+\delta^{k(1-\varepsilon_0)+\varepsilon_0}
\sum_{j=1}^2\sqrt{\mathcal E_{j,\leq N-2}(s,u)}.
\end{split}
\end{equation}
For small $\delta>0$, the second term on the right-hand
side of \eqref{YHC-A9} can be absorbed into the left-hand side,
which gives
\begin{equation}\label{RNL}
\|R^{\leq N-4}\check L^i\|_{s,u}
\lesssim\delta^{3/2}
+\delta^{k(1-\varepsilon_0)+\varepsilon_0}
\sum_{j=1}^2\sqrt{\mathcal E_{j,\leq N-2}(s,u)}.
\end{equation}
Using \eqref{omega}, one obtains the corresponding estimate for $\upsilon$:
\begin{equation}\label{RNU}
\|R^{\leq N-4}\upsilon\|_{s,u}
\lesssim\delta^{3/2}
+\delta^{k(1-\varepsilon_0)+\varepsilon_0}
\sum_{j=1}^2\sqrt{\mathcal E_{j,\leq N-2}(s,u)}.
\end{equation}

We next deal with the related higher-order derivative estimates.
For $m\geq N-4$, applying \eqref{deL} to
$R^{m+1}\check L^i=R^m\bigl(R^X\slashed d_X\check L^i\bigr)$ yields
\begin{equation}\label{RML}
\begin{split}
\|R^{m+1}\check L^i\|_{s,u}
\lesssim&\ \delta^{k(1-\varepsilon_0)}
\bigl(\sum_{j=1}^2\|R^{\leq m}\check L^j\|_{s,u}
+\|R^{\leq m}\upsilon\|_{s,u}\bigr)
+\|R^{\leq m}\operatorname{tr}\check{\mathscr X}\|_{s,u}\\
&+\delta^{(k-1)(1-\varepsilon_0)}
\|R^{\leq m+1}\varphi\|_{s,u}.
\end{split}
\end{equation}
Note that the term $\|R^{\leq m+1}x^i\|_{s,u}$ does not appear
explicitly in \eqref{RML} since we have used
\[
R^{m+1}x^i
=R^m\Omega^i
+R^m\bigl(-\upsilon g^{0i}
-\upsilon\check L^i-\upsilon\frac{x^i}{\varrho}\bigr)
\]
to derive
\[
|R^{m+1}x^i|
\lesssim 1+|R^{\leq m}\upsilon|
+\delta^{k(1-\varepsilon_0)}|R^{\leq m}\check L^i|
+\delta^{(2k-1)(1-\varepsilon_0)}|R^{\leq m}\varphi|.
\]
Thus, for any smooth function $h$, one has
\begin{equation*}
\begin{split}
\|R^{m+1}(hx^i)\|_{s,u}
\lesssim&\ \|R^{\leq m+1}h\|_{s,u}
+\|R^{\leq\lfloor m/2\rfloor}h\|_{L^\infty}
\bigl(\|R^{\leq m}\upsilon\|_{s,u}
+\delta^{k(1-\varepsilon_0)}
\|R^{\leq m}\check L^i\|_{s,u}\\
&+\delta^{(2k-1)(1-\varepsilon_0)}
\|R^{\leq m}\varphi\|_{s,u}\bigr).
\end{split}
\end{equation*}
By induction method, it follows from \eqref{RML}, \eqref{RNL} and
\eqref{RNU} that
\begin{equation}\label{RmL}
\begin{split}
\|R^{m+1}\check L^i\|_{s,u}
\lesssim&\ \delta^{k(1-\varepsilon_0)+3/2}
+\delta^{k(1-\varepsilon_0)}\|R^{\leq m}\upsilon\|_{s,u}
+\|R^{\leq m}\operatorname{tr}\check{\mathscr X}\|_{s,u}\\
&+\delta^{k(1-\varepsilon_0)+\varepsilon_0}
\sum_{j=1}^2\sqrt{\mathcal E_{j,\leq m+2}(s,u)}.
\end{split}
\end{equation}
Similarly, using \eqref{omega} and \eqref{RNU}, we have
\begin{equation}\label{RmU}
\begin{split}
\|R^{m+1}\upsilon\|_{s,u}
\lesssim&\ \delta^{k(1-\varepsilon_0)+3/2}
+\sum_{j=1}^2\|R^{\leq m+1}\check L^j\|_{s,u}
+\|R^{\leq m}\operatorname{tr}\check{\mathscr X}\|_{s,u}\\
&+\delta^{k(1-\varepsilon_0)+\varepsilon_0}
\sum_{j=1}^2\sqrt{\mathcal E_{j,\leq m+2}(s,u)}.
\end{split}
\end{equation}
It follows from \eqref{RmL}, \eqref{RmU} and the induction argument that
\begin{equation}\label{RmLU}
\begin{split}
\|R^{m+1}\check L^i\|_{s,u}
+\|R^{m+1}\upsilon\|_{s,u}
\lesssim&\ \delta^{k(1-\varepsilon_0)+3/2}
+\|R^{\leq m}\operatorname{tr}\check{\mathscr X}\|_{s,u}+\delta^{k(1-\varepsilon_0)+\varepsilon_0}
\sum_{j=1}^2\sqrt{\mathcal E_{j,\leq m+2}(s,u)}.
\end{split}
\end{equation}
By \eqref{Rpi}, \eqref{rrho} and \eqref{RmLU}, we conclude that
\begin{equation}\label{Rmrpi}
\begin{split}
\|R^{m+1}\check\varrho\|_{s,u}
+\bigl\|\slashed{\mathcal L}_R^{m}{}^{(R)}
\slashed{\pi}\bigr\|_{s,u}
\lesssim&\ \delta^{k(1-\varepsilon_0)+3/2}
+\|R^{\leq m}\operatorname{tr}\check{\mathscr X}\|_{s,u}+\delta^{k(1-\varepsilon_0)+\varepsilon_0}
\sum_{j=1}^2\sqrt{\mathcal E_{j,\leq m+2}(s,u)}
\end{split}
\end{equation}
and
\begin{equation}\label{RmpiL}
\begin{split}
&\bigl\|\slashed{\mathcal L}_R^{m}{}^{(R)}
\slashed{\pi}_{\mathring L}\bigr\|_{s,u}\\
\lesssim &
\bigl(1+\delta^{2k(1-\varepsilon_0)-\varepsilon_0}\bigr)
\bigl(\delta^{k(1-\varepsilon_0)+3/2}
+\|R^{\leq m}\operatorname{tr}\check{\mathscr X}\|_{s,u}
+\delta^{k(1-\varepsilon_0)+\varepsilon_0}
\sum_{j=1}^2\sqrt{\mathcal E_{j,\leq m+2}(s,u)}\bigr).
\end{split}
\end{equation}

Next, we estimate
$\|R^{\leq m}\operatorname{tr}\check{\mathscr X}\|_{s,u}$.
In this process, the transport equation \eqref{transport} in
Lemma~\ref{lem:null-transport} will be used.
Note that the first term on the right-hand side of \eqref{transport}
contains the factor
$G_{\mathring L\mathring L}^\gamma\mathring L_\gamma$.
By \eqref{gLLL} and
Proposition~\ref{prop:bootstrap-higher-derivatives}, we have
$\bigl|R^{\leq N-3}
\bigl(G_{\mathring L\mathring L}^\gamma\mathring L_\gamma\bigr)\bigr|
\lesssim\delta^{k(1-\varepsilon_0)}$.
Consequently,
\begin{equation}\label{RGLLLvp}
\begin{split}
&\bigl\|R^m\bigl(\varrho^2
\bigl(G_{\mathring L\mathring L}^\gamma\mathring L_\gamma\bigr)
\tilde T^i\slashed\triangle\varphi_i\bigr)\bigr\|_{s,u}\\
\lesssim&\ \delta^{k(1-\varepsilon_0)}
\bigl\{\|R^{\leq m}\slashed\triangle\varphi\|_{s,u}
+\|R^{\leq m}\upsilon\|_{s,u}
+\sum_{i=1}^2\|R^{\leq m}\check L^i\|_{s,u}\bigr\}
+\delta^{(k-1)(1-\varepsilon_0)}
\|R^{\leq m}\varphi\|_{s,u}.
\end{split}
\end{equation}
Substituting \eqref{RGLLLvp} into \eqref{transport} and using
\eqref{RmLU}--\eqref{RmpiL} yield
\begin{align*}
&\bigl\|\mathring LR^m
\bigl(\varrho^2
\bigl(\operatorname{tr}\check{\mathscr X}-E_0\bigr)\bigr)
\bigr\|_{s,u}\\
\lesssim&
\bigl\|\bigl[\mathring L,R^m\bigr]
\bigl(\varrho^2
\bigl(\operatorname{tr}\check{\mathscr X}-E_0\bigr)\bigr)
\bigr\|_{s,u}
+\bigl\|R^m\mathring L
\bigl(\varrho^2
\bigl(\operatorname{tr}\check{\mathscr X}-E_0\bigr)\bigr)
\bigr\|_{s,u}\\
\lesssim&
\delta^{k(1-\varepsilon_0)}
\bigl(\|\slashed dR^{m+1}\varphi\|_{s,u}
+\bigl\|\slashed{\mathcal L}_R^{\leq m}
\leftidx{^{(R)}}{\slashed\pi}_{\mathring L}\bigr\|_{s,u}\bigr)
+\delta^{2k(1-\varepsilon_0)-\varepsilon_0}
\bigl\|R^{\leq m}
\bigl(\varrho^2
\bigl(\operatorname{tr}\check{\mathscr X}-E_0\bigr)\bigr)
\bigr\|_{s,u}\\
&+\delta^{2k(1-\varepsilon_0)-\varepsilon_0}
\bigl\{\bigl\|\slashed{\mathcal L}_R^{\leq m}
\leftidx{^{(R)}}{\slashed\pi}\bigr\|_{s,u}
+\sum_{i=1}^2\|R^{\leq m+1}\check L^i\|_{s,u}
+\|R^{\leq m+1}\upsilon\|_{s,u}\bigr\}\\
&+\delta^{2k(1-\varepsilon_0)}
\bigl(\sqrt{\mathcal E_{1,\leq m+2}(s,u)}
+\sqrt{\mathcal E_{2,\leq m+2}(s,u)}\bigr)\\
\lesssim&
\delta^{3k(1-\varepsilon_0)-\varepsilon_0+3/2}
+\delta^{k(1-\varepsilon_0)}
\mu_{\min}^{-1/2}(s)
\sqrt{\mathcal E_{1,\leq m+2}(s,u)}
+\delta^{2k(1-\varepsilon_0)}
\sqrt{\mathcal E_{2,\leq m+2}(s,u)}\\
&+\delta^{2k(1-\varepsilon_0)-\varepsilon_0}
\bigl\|R^{\leq m}
\bigl(\varrho^2
\bigl(\operatorname{tr}\check{\mathscr X}-E_0\bigr)\bigr)
\bigr\|_{s,u}.
\end{align*}
Summing the above inequalities over the derivative orders up to $m$,
using the initial data estimates and \eqref{Ff}, and applying Gronwall's
inequality, we can obtain \eqref{RmXE}.
In addition, by the definition of $E_0$,
$\|R^{\leq m}E_0\|_{s,u}$ may be estimated. This, together with
\eqref{RmLU} and \eqref{Rmrpi}, yields \eqref{RmcX}.
Finally, \eqref{RmPiL} follows from \eqref{RmpiL}.
\end{proof}

\begin{corollary}\label{YHC-A3}
Under the assumptions of Proposition~\ref{EstimatesOnR},
for any $s$ with $t_0\leq s<T_\delta$ and $u\in[0,4\delta]$,
the following estimates hold:
\begin{equation}\label{RMMu}
\begin{split}
\|R^{m+1}\mu\|_{s,u}
\lesssim{}&
\delta^{k(1-\varepsilon_0)+\varepsilon_0-1/2}
+\delta^{(k-1)(1-\varepsilon_0)}\Xi^2_{m+2}(s,u)
\end{split}
\end{equation}
and
\begin{equation}\label{RMPIRT}
\begin{split}
\|\slashed{\mathcal L}_R^m{}^{(R)}\slashed\pi_T\|_{s,u}
\lesssim{}&
\delta^{2k(1-\varepsilon_0)-1/2}
+\delta^{2k(1-\varepsilon_0)-1}\Xi_{m+2}^1(s,u),
\end{split}
\end{equation}
where $\Xi_{m+2}^1(s,u)$ is defined in
Proposition~\ref{EstimatesOnR} and
\[
\Xi_{m+2}^2(s,u)
=\int_{t_0}^s
\bigl\{\mu_{\min}^{-1/2}(\tau)
\sqrt{\mathcal E_{1,\leq m+2}(\tau,u)}
+\sqrt{\mathcal E_{2,\leq m+2}(\tau,u)}\bigr\}\,d\tau.
\]
\end{corollary}

\begin{proof}
Note that even without the null condition \eqref{null}, one still has
\[
|G_{\mathring L\mathring L}^\gamma T\vp_\gamma|
\lesssim |\vp|^{k-1}|T\vp|.
\]
This, together with the estimate of
$[\mathring L,R^{m+1}]\mu$ and \eqref{lmu}, yields
\begin{equation}\label{RMmu}
\begin{split}
\|\mathring L R^{m+1}\mu\|_{s,u}
\lesssim{}&
\delta^{2k(1-\varepsilon_0)-\varepsilon_0}
\|R^{\leq m+1}\mu\|_{s,u}
+\delta^{(k-1)(1-\varepsilon_0)-1}
\|R^{\leq m+1}\vp\|_{s,u}\\
&+\delta^{(k-1)(1-\varepsilon_0)}
\bigl(\|R^{\leq m+1}T\vp\|_{s,u}
+\|\mu R^{m+1}\mathring L\vp\|_{s,u}
+\|R^{\leq m}\mathring L\vp\|_{s,u}\bigr)\\
&+\delta^{k(1-\varepsilon_0)-1}
\bigl(\sum_{i=1}^2\|R^{\leq m+1}\check L^i\|_{s,u}
+\|R^{\leq m+1}\upsilon\|_{s,u}\bigr)
+\|\slashed{\mathcal L}_R^m{}^{(R)}
\slashed\pi_{\mathring L}\|_{s,u}.
\end{split}
\end{equation}
In addition, we have
\begin{equation}\label{YHC-A2}
\begin{split}
R^{\leq m+1}T\vp
&=TR^{\leq m+1}\vp+[R^{\leq m+1},T]\vp,\\
\mu R^{m+1}\mathring L\vp
&=\mu\mathring L R^{m+1}\vp
+\mu[R^{m+1},\mathring L]\vp.
\end{split}
\end{equation}
It follows from \eqref{Rpi}, \eqref{YHC-A2}, \eqref{RmLU} and the estimates
in Proposition~\ref{EstimatesOnR} that
\begin{equation}\label{DXY-7}
\begin{split}
\|\slashed{\mathcal L}_R^m{}^{(R)}\slashed\pi_T\|_{s,u}
\lesssim{}&
\|R^{\leq m}\operatorname{tr}\check{\mathscr X}\|_{s,u}
+\delta^{k(1-\varepsilon_0)}\|R^{\leq m+1}\mu\|_{s,u}
+\delta^{k(1-\varepsilon_0)-1}\|R^{\leq m}\upsilon\|_{s,u}\\
&+\delta^{(2k-1)(1-\varepsilon_0)-1}
\|R^{\leq m+1}\vp\|_{s,u}
+\delta^{(2k-1)(1-\varepsilon_0)}
\|R^{\leq m}T\vp\|_{s,u}\\
&+\bigl(1+\delta^{2k(1-\varepsilon_0)-1}\bigr)
\sum_{i=1}^2\|R^{\leq m}\check L^i\|_{s,u}\\
\lesssim{}&
\delta^{2k(1-\varepsilon_0)-1/2}
+\delta^{2k(1-\varepsilon_0)-1}\Xi_{m+2}^1(s,u)
+\delta^{k(1-\varepsilon_0)}
\|R^{\leq m+1}\mu\|_{s,u}\\
&+\delta^{(2k-1)(1-\varepsilon_0)}
\sum_{i=1}^2\sqrt{\mathcal E_{i,\leq m+2}(s,u)}.
\end{split}
\end{equation}
Substituting \eqref{YHC-A2} into \eqref{RMmu}, using
Proposition~\ref{EstimatesOnR}, and applying \eqref{SiSi}
to treat $\|R^{\leq m}\mathring L\vp\|_{s,u}$, we arrive at
\[
\begin{split}
\|\mathring L R^{m+1}\mu\|_{s,u}
\lesssim{}&
\delta^{2k(1-\varepsilon_0)-1/2}
+\delta^{2k(1-\varepsilon_0)-1}\Xi_{m+2}^1(s,u)
+\delta^{2k(1-\varepsilon_0)-\varepsilon_0}
\|R^{\leq m+1}\mu\|_{s,u}\\
&+\delta^{(k-1)(1-\varepsilon_0)}
\sum_{i=1}^2\sqrt{\mathcal E_{i,\leq m+2}(s,u)}.
\end{split}
\]
By
$\mu|_{t_0}
=\frac{1}{\sqrt{(g^{0i}\omega_i)^2+g_{ij}\omega_i\omega_j}}$
and the initial data estimates, it holds that
$$\|R^{m+1}\mu\|_{t_0,u}
\lesssim\delta^{k(1-\varepsilon_0)+1/2}.$$
Therefore, applying the above estimates up to $(m+1)-$order derivatives, the uniform
bound of $\mu$, \eqref{Ff}, the definitions of
$\Xi_{m+2}^1$ and $\Xi_{m+2}^2$ and Gronwall's inequality,
\eqref{RMMu} can be achieved.

Finally, substituting \eqref{RMMu} into \eqref{DXY-7} yields \eqref{RMPIRT}.
\end{proof}

\begin{remark}\label{rem6.1}
Although the null condition \eqref{null} is not used directly
to estimate $G_{\mathring L\mathring L}^{\gamma}T\vp_\gamma$
in the proof of Corollary~\ref{YHC-A3}, it will be used
to establish the top-order $L^2$ estimates for
$\check{\mathscr X}$ and $\mu$ in Section~\ref{Section 6.2} below,
which leads to the closure of the corresponding energy estimates
when $\varepsilon_0\in[\varepsilon_k^*,1)$.
\end{remark}

To derive the highest-order $L^2$ estimates for $\check{\mathscr X}$
in Section~\ref{Section 6.2} when
$\varepsilon_0\in[\varepsilon_k^*,1)$,
we need to establish the following estimate.

\begin{corollary}\label{Cor6.2}
Under the assumptions of Proposition~\ref{EstimatesOnR},
the following estimate holds for sufficiently large $b_{m+2}$:
\begin{equation}\label{cortop}
\begin{split}
&\delta^{\varepsilon_k}
\bigl|\int_{\Sigma_s^u}
(R^m\operatorname{tr}\check{\mathscr X})
\mathring{\underline L}\vp_\gamma
(R^{m+2}\vp_\gamma)\bigr|\\
\lesssim{}&
\delta^{2-2\varepsilon_0}
\mu_{\min}^{-2b_{m+2}}(s)b_{m+2}
+\delta^{-(k+1)(1-\varepsilon_0)}
\int_0^u\mathcal F_{1,\leq m+2}(s,u')\,du'\\
&+\frac{\mu_{\min}^{-2b_{m+2}}(s)}{b_{m+2}}
\bigl(
\delta^{\varepsilon_k}
\widetilde{\mathcal E}_{1,\leq m+2}(s,u)
+\delta\widetilde{\mathcal E}_{2,\leq m+2}(s,u)
\bigr).
\end{split}
\end{equation}
\end{corollary}

\begin{proof}
By the definition of $E_0$ in
Lemma~\ref{lem:null-transport}, decompose the integral
on the left-hand side of \eqref{cortop} as
\begin{equation}\label{I+II}
\delta^{\varepsilon_k}
\int_{\Sigma_s^u}
(R^m\operatorname{tr}\check{\mathscr X})
\mathring{\underline L}\vp_\gamma
(R^{m+2}\vp_\gamma)
=I+II,
\end{equation}
where
\begin{align*}
I&=\delta^{\varepsilon_k}
\int_{\Sigma_s^u}
\bigl(R^m(\operatorname{tr}\check{\mathscr X}-E_0)\bigr)
\mathring{\underline L}\vp_\gamma
(R^{m+2}\vp_\gamma),\\
II&=\delta^{\varepsilon_k}
\int_{\Sigma_s^u}
(R^mE_0)\mathring{\underline L}\vp_\gamma
(R^{m+2}\vp_\gamma).
\end{align*}
Invoking the identity \cite[(10.1a)]{Sp},
we rewrite the term $II$ as
\begin{equation*}
\begin{split}
II={}&\delta^{\varepsilon_k}
\int_{\Sigma_{t_0}^u}
(R^mE_0)\mathring{\underline L}\vp_\gamma
(R^{m+2}\vp_\gamma)+\delta^{\varepsilon_k}
\int_{D^{s,u}}(\mathring L+\operatorname{tr}\chi)
\bigl(
(R^mE_0)\mathring{\underline L}\vp_\gamma
(R^{m+2}\vp_\gamma)
\bigr).
\end{split}
\end{equation*}
Then
\begin{equation}\label{YHC-A4}
\begin{split}
|II|\lesssim{}&
\delta^{2-\varepsilon_0}
+\delta^{\varepsilon_k}
\bigl|\int_{D^{s,u}}
(\mathring LR^mE_0)\mathring{\underline L}\vp_\gamma
(R^{m+2}\vp_\gamma)\bigr|+\delta^{\varepsilon_k}
\bigl|\int_{D^{s,u}}
(R^mE_0)\mathring L\mathring{\underline L}\vp_\gamma
(R^{m+2}\vp_\gamma)\bigr|\\
&+\delta^{\varepsilon_k}
\bigl(1+\delta^{2k(1-\varepsilon_0)-\varepsilon_0}\bigr)
\int_{D^{s,u}}
\bigl|
(R^mE_0)\mathring{\underline L}\vp_\gamma
(R^{m+2}\vp_\gamma)
\bigr|\\
&+\delta^{\varepsilon_k}
\bigl|\int_{D^{s,u}}
(R^mE_0)\mathring{\underline L}\vp_\gamma
(\mathring LR^{m+2}\vp_\gamma)\bigr|.
\end{split}
\end{equation}
Note that
\begin{equation*}
\begin{split}
&\int_{D^{s,u}}
(R^mE_0)\mathring{\underline L}\vp_\gamma
(\mathring LR^{m+2}\vp_\gamma)=\int_{D^{s,u}}
(R^mE_0)\mathring{\underline L}\vp_\gamma
(\slashed d_XR^{m+1}\vp_\gamma)
{}^{(R)}\slashed{\pi}_{\mathring L}^X\\
&\qquad-\int_{D^{s,u}}
R\bigl((R^mE_0)\mathring{\underline L}\vp_\gamma\bigr)
(\mathring LR^{m+1}\vp_\gamma)-\int_{D^{s,u}}
\tfrac12(R^mE_0)\mathring{\underline L}\vp_\gamma
(\mathring LR^{m+1}\vp_\gamma)
\operatorname{tr}{}^{(R)}\slashed{\pi}.
\end{split}
\end{equation*}
Substituting this into \eqref{YHC-A4} yields
\begin{equation*}
\begin{split}
|II|\lesssim{}&
\delta^{2-\varepsilon_0}
+\delta^{\varepsilon_0-1}
\int_{D^{s,u}}
|\mathring LR^{\leq m+1}\vp|
\Bigl(
|\slashed dR^{m+1}\vp|
+\delta^{1-\varepsilon_0}
\bigl(
\sum_{i=1}^2|R^{\leq m+1}\check L^i|
+|R^{\leq m+1}\upsilon|
\bigr)
\Bigr)\\
&+\delta^{k(1-\varepsilon_0)-1}
\int_{D^{s,u}}
\bigl(
|R^{\leq m+1}\vp|
+\delta^{\varepsilon_0}|R^{\leq m}\mathring L\vp|
\bigr)|R^{m+2}\vp|\\
&+\int_{D^{s,u}}
\Bigl(
|\slashed{\mathcal L}_R^{\leq m}
{}^{(R)}\slashed{\pi}_{\mathring L}|
+\delta^{-\varepsilon_k}
\bigl(
\sum_{i=1}^2|R^{\leq m+1}\check L^i|
+|R^{\leq m+1}\upsilon|
\bigr)
\Bigr)|R^{m+2}\vp|.
\end{split}
\end{equation*}
Together with the estimates in
Proposition~\ref{EstimatesOnR}, we further deduce
\begin{align}
|II|\lesssim{}&
\delta^{2-\varepsilon_0}
+\delta^{-(k+1)(1-\varepsilon_0)}
\int_0^u\mathcal F_{1,\leq m+2}(s,u')\,du'
+\delta^{(3k+2)(1-\varepsilon_0)-\varepsilon_0}
\int_{t_0}^s\mu_{\min}^{-2b_{m+2}}(\tau)\,d\tau
\nonumber\\
&+\delta^{(k-1)(1-\varepsilon_0)}
\int_{t_0}^s
\mu_{\min}^{-1-2b_{m+2}}(\tau)
\widetilde{\mathcal E}_{1,\leq m+2}(\tau,u)\,d\tau
\nonumber\\
&+\delta^{(k+1)(1-\varepsilon_0)}
\int_{t_0}^s
\mu_{\min}^{-2b_{m+2}}(\tau)
\widetilde{\mathcal E}_{2,\leq m+2}(\tau,u)\,d\tau
\nonumber\\
\lesssim{}&
\delta^{2-2\varepsilon_0}
\mu_{\min}^{-2b_{m+2}}(s)
+\delta^{-(k+1)(1-\varepsilon_0)}
\int_0^u\mathcal F_{1,\leq m+2}(s,u')\,du'
\label{II}\\
&+\frac{\mu_{\min}^{-2b_{m+2}}(s)}{b_{m+2}}
\bigl(
\delta^{2\varepsilon_0-1}
\widetilde{\mathcal E}_{1,\leq m+2}(s,u)
+\delta\widetilde{\mathcal E}_{2,\leq m+2}(s,u)
\bigr),
\nonumber
\end{align}
where \eqref{eq:mu-integrability-full} is used and
$b_{m+2}$ is chosen to be sufficiently large.

For the term $I$, it follows from \eqref{RmXE} that
\begin{align}
|I|\lesssim{}&
\delta^{-k(1-\varepsilon_0)}
\|R^m(\operatorname{tr}\check{\mathscr X}-E_0)\|_{s,u}
\|\slashed dR^{m+1}\vp\|_{s,u}
\nonumber\\
\lesssim{}&
\Bigl\{
\delta^{1/2}
+\frac{\delta^{\varepsilon_k}}{b_{m+2}}
\mu_{\min}^{-b_{m+2}+1/2}(s)
\sqrt{\widetilde{\mathcal E}_{1,\leq m+2}(s,u)}
\nonumber\\
&\qquad
+\frac{\delta^{\varepsilon_0}}{b_{m+2}}
\mu_{\min}^{1-b_{m+2}}(s)
\sqrt{\widetilde{\mathcal E}_{2,\leq m+2}(s,u)}
\Bigr\}
\mu_{\min}^{-b_{m+2}-1/2}(s)
\sqrt{\widetilde{\mathcal E}_{1,\leq m+2}(s,u)}
\nonumber\\
\lesssim{}&
\delta^{(k+1)(1-\varepsilon_0)}
b_{m+2}\mu_{\min}^{-1}(s)
+\frac{\mu_{\min}^{-2b_{m+2}}(s)}{b_{m+2}}
\Bigl(
\delta^{\varepsilon_k}
\widetilde{\mathcal E}_{1,\leq m+2}(s,u)
\nonumber\\
&\hspace{7cm}
+\delta^{\varepsilon_0+k(1-\varepsilon_0)}
\widetilde{\mathcal E}_{2,\leq m+2}(s,u)
\Bigr).
\label{I}
\end{align}
Combining \eqref{II} and \eqref{I} with \eqref{I+II}
yields \eqref{cortop}.
\end{proof}

Note that the $L^2$ estimates for certain quantities under the action of  $R$ have been given in Proposition~\ref{EstimatesOnR}. We
next establish the remaining $L^2$ estimates for all derivatives.

\begin{proposition}\label{L2chi}
	Under the assumptions \eqref{bootstrap-assumptions}, for all \(0\leq m\leq 2N-6\) and small \(\delta>0\), the following estimates hold:
	\begin{align*}
		&\delta^{l+p\varepsilon_k}\bigl(\|Z^m\operatorname{tr}\check{\mathscr X}\|_{s,u}+\|Z^{m+1}\check L^i\|_{s,u}+\|Z^{m+1}\upsilon\|_{s,u}+\|Z^{m+1}\check\varrho\|_{s,u}+\bigl\|\slashed{\mathcal L}_Z^{m}{}^{(R)}\slashed{\pi}\bigr\|_{s,u}\bigr)\\
		&\qquad\qquad\lesssim\ \delta^{k(1-\varepsilon_0)+1/2}+\delta^{k(1-\varepsilon_0)}\Xi^1_{m+2}(s,u),\\
		&\delta^{l+p\varepsilon_k}\|Z^{m+2}x^i\|_{s,u}\lesssim\delta^{1/2}+\delta^{k(1-\varepsilon_0)}\Xi^1_{m+2}(s,u),\\
		&\delta^{l+p\varepsilon_k}\|\slashed{\mathcal L}_Z^{m+1}\slashed g\|_{s,u}\lesssim\delta^{k(1-\varepsilon_0)}\bigl(\delta^{1/2}\bigl(1+\delta^{\varepsilon_0-2k(1-\varepsilon_0)}\bigr)+\Xi^1_{m+2}(s,u)\bigr),\\
		&\delta^{l+p\varepsilon_k}\|\slashed{\mathcal L}_Z^{m}{}^{(R)}\slashed{\pi}_{\mathring L}\|_{s,u}\lesssim\delta^{k(1-\varepsilon_0)}\bigl(1+\delta^{2k(1-\varepsilon_0)-\varepsilon_0}\bigr)\bigl(\delta^{1/2}+\Xi_{m+2}^1(s,u)\bigr),\\
		&\delta^{l+p\varepsilon_k}\|\slashed{\mathcal L}_Z^{m+1}R\|_{s,u}\lesssim \delta^{\varepsilon_0}\bigl(1+\delta^{2k(1-\varepsilon_0)-\varepsilon_0}\bigr)\bigl(\delta^{1/2}+\Xi_{m+2}^1(s,u)\bigr),\\
		&\delta^{l+p\varepsilon_k}\|\slashed{\mathcal L}_Z^m{}^{(R)}\slashed\pi_T\|_{s,u}
		\lesssim  \delta^{2k(1-\varepsilon_0)-1/2} + \delta^{2k(1-\varepsilon_0)-1}\Xi^1_{m+2}(s,u),\\
		&\delta^{l+p\varepsilon_k}\|Z^{m+1}\mu\|_{s,u} \lesssim \delta^{k(1-\varepsilon_0)+\varepsilon_0-1/2} + \delta^{(k-1)(1-\varepsilon_0)}\Xi^2_{m+2}(s,u)+\delta^{2\varepsilon_0-1}\sum_{i=1}^2\sqrt{\mathcal E_{i,\leq m+2}(s,u)}
	\end{align*}
	and
	\begin{equation*}
		\begin{split}
			&\delta^{l+p\varepsilon_k}\bigl(\|\slashed{\mathcal L}_Z^m{}^{(T)}\slashed\pi_{\mathring L}\|_{s,u}+\|\slashed{\mathcal L}_Z^m{}^{(T)}\slashed\pi\|_{s,u}\bigr)\\
			\lesssim &\delta^{k(1-\varepsilon_0)-1/2} + \delta^{2k(1-\varepsilon_0)-1}\int_{t_0}^s\mu_{\min}^{-1/2}(\tau)\sqrt{\mathcal E_{1,\leq m+2}(\tau,u)}\,d\tau \\
			&+\delta^{(k-1)(1-\varepsilon_0)}\bigl(\bigl(1+\delta^{2k(1-\varepsilon_0)-\varepsilon_0}\bigr)\int_{t_0}^s\sqrt{\mathcal E_{2,\leq m+2}(\tau,u)}\,d\tau+ \sum_{i=1}^2\sqrt{\mathcal E_{i,\leq m+2}(s,u)}\bigr),
		\end{split}
	\end{equation*}
	where \(l\) and \(p\) denote the numbers of \(T\) and \(\varrho\mathring L\) in the corresponding derivative strings, respectively.
\end{proposition}

\begin{proof}
By Lemma~\ref{Lem3.3}, the quantities \(\varrho\mathring{L}\check{L}^i\), \(T\check{L}^i\) and \(R\check{L}^i\) can be expressed as
\[
\varrho\mathring{L}\check{L}^i = \varrho\mathring{L}\mathring{L}^i - \check{L}^i,
\quad
T\check{L}^i = T\mathring{L}^i + \frac{\mu}{\varrho}g^{0i} + \frac{\mu}{\varrho}\check{L}^i + \frac{\mu-1}{\varrho^2}y^i,
\quad
R\check{L}^i = {R^X}\slashed{d}_X\check{L}^i.
\]
Then we have
\begin{equation}\label{ZkL}
\begin{split}
&\delta^{l+p\varepsilon_k} \| Z^{m+1}\check{L}^i \|_{s,u} \\
\lesssim&\delta^{1/2} \bigl( \delta^{2k(1-\varepsilon_0)} + \delta^{\varepsilon_0} \bigr)
+ \delta^{l_1+p_1\varepsilon_k} \| Z^{m_1} \operatorname{tr}\check{\mathscr{X}} \|_{s,u} + \delta^{k(1-\varepsilon_0)} \delta^{l_0+p_0\varepsilon_k} \| Z^{m_0}x \|_{s,u}\\
&+ \delta(1+\delta^{2k(1-\varepsilon_0)-\varepsilon_0}) \delta^{l_0+p_0\varepsilon_k} \| Z^{m_0}\mu \|_{s,u} + \delta^{k(1-\varepsilon_0)} \delta^{l_1+p_1\varepsilon_k} \| \slashed{\mathcal L}_Z^{m_1}\slashed{g} \|_{s,u} \\
&+ \delta^{k(1-\varepsilon_0)+\varepsilon_0} \bigl( \sqrt{\mathcal E_{1,\leq m+2}(s,u)} + \sqrt{\mathcal E_{2,\leq m+2}(s,u)} \bigr),
\end{split}
\end{equation}
where the bound \(\|1\|_{s,u}\lesssim\delta^{1/2}\) is used. Here \(l\) and \(p\) (resp.\ \(l_n\) and \(p_n\), \(n=0,1\)) denote the numbers of \(T\) and \(\varrho\mathring{L}\) in the string \(Z^{m+1}\) (resp.\ \(Z^{m_n}\)) with \(0 \leq m_n \leq m+1-n\).

Similarly, from the identities \(\varrho\mathring Lx^i=\varrho\check L^i+x^i\), \(Tx^i=\mu(-g^{0i}-\check L^i-\frac{x^i}{\varrho})\) and \(Rx^i=\Omega^i+\upsilon(g^{0i}+\check L^i+\frac{x^i}{\varrho})\), one has
\begin{equation}\label{Zkx}
\begin{split}
&\delta^{l+p\varepsilon_k}\|Z^{m+2}x^i\|_{s,u}\\
\lesssim&\delta^{1/2}+\delta^{l_0+p_0\varepsilon_k}\|Z^{m_0}\upsilon\|_{s,u}
+(\delta^{k(1-\varepsilon_0)}+\delta^{\varepsilon_k})\sum_{j=1}^2\delta^{l_0+p_0\varepsilon_k}\|Z^{m_0}\check L^j\|_{s,u}\\
&+\delta\delta^{l_{0}+p_0\varepsilon_k}\|Z^{m_{0}}\mu\|_{s,u}
+\delta^{2k(1-\varepsilon_0)+\varepsilon_0}\bigl( \sqrt{\mathcal E_{1,\leq m+2}(s,u)} + \sqrt{\mathcal E_{2,\leq m+2}(s,u)} \bigr).
\end{split}
\end{equation}
In addition, by \(\slashed{\mathcal L}_{\varrho\mathring L}R=\varrho\leftidx{^{(R)}}{\slashed\pi}_{\mathring L}\), \(\slashed{\mathcal L}_TR=\leftidx{^{(R)}}{\slashed\pi}_T\) and \eqref{Rpi}, it holds that
\begin{equation}\label{ZR}
\begin{split}
&\delta^{l+p\varepsilon_k}\|\slashed{\mathcal L}_Z^{m+1}R\|_{s,u}\\
\lesssim&\delta^{\varepsilon_0+1/2}+(\delta^{2k(1-\varepsilon_0)}+\delta^{\varepsilon_0})\delta^{l_0
+p_0\varepsilon_k}\|Z^{m_0}x\|_{s,u}+\delta^{k(1-\varepsilon_0)}\delta^{l_1+p_1\varepsilon_k}\|Z^{m_1}\upsilon\|_{s,u}\\
&+\delta^{\varepsilon_k}\delta^{l_1+p_1\varepsilon_k}\|Z^{m_1}\operatorname{tr}\check{\mathscr X}\|_{s,u}+(\delta^{2k(1-\varepsilon_0)}+\delta^{\varepsilon_k})\sum_{j=1}^2\delta^{l_1+p_1\varepsilon_k}\|Z^{m_1}\check L^j\|_{s,u}\\
&+(\delta^{\varepsilon_0}+\delta^{2k(1-\varepsilon_0)})\delta^{l_1+p_1\varepsilon_k}\|\slashed{\mathcal L}_Z^{m_1}\slashed g\|_{s,u}+\delta^{1+k(1-\varepsilon_0)}\delta^{l_{0}+p_0\varepsilon_k}\|Z^{m_{0}}\mu\|_{s,u}\\
&+\delta^{\varepsilon_0+2k(1-\varepsilon_0)}\bigl(\sqrt{\mathcal E_{1,\leq m+2}(s,u)} + \sqrt{\mathcal E_{2,\leq m+2}(s,u)}\bigr).
\end{split}
\end{equation}
On the other hand, from \eqref{omega}, \eqref{rrho} and the relations \(\slashed{\mathcal L}_{\varrho\mathring L}\slashed g=2\varrho\chi\), \(\slashed{\mathcal L}_T\slashed g=\leftidx{^{(T)}}{\slashed\pi}\), \(\slashed{\mathcal L}_R\slashed g=\leftidx{^{(R)}}{\slashed\pi}\), we have
\begin{equation}\label{Zku}
\begin{split}
\delta^{l+p\varepsilon_k}\|Z^{m+1}\upsilon\|_{s,u}
\lesssim& \delta^{k(1-\varepsilon_0)}\delta^{l_0+p_0\varepsilon_k}\|Z^{m_0}x\|_{s,u}+\sum_{i=1}^2\delta^{l_0+p_0\varepsilon_k}\|Z^{m_0}\check L^i\|_{s,u}\\
&+\delta^{k(1-\varepsilon_0)+\varepsilon_0}\bigl( \sqrt{\mathcal E_{1,\leq m+2}(s,u)} + \sqrt{\mathcal E_{2,\leq m+2}(s,u)} \bigr),
\end{split}
\end{equation}
\begin{equation}\label{Zkr}
\begin{split}
\delta^{l+p\varepsilon_k}\|Z^{m+1}\check\varrho\|_{s,u}
\lesssim & \delta^{k(1-\varepsilon_0)}\delta^{l_0+p_0\varepsilon_k}\|Z^{m_0}x\|_{s,u}
+\sum_{i=1}^2\delta^{l_0+p_0\varepsilon_k}\|Z^{m_0}\check L^i\|_{s,u}\\
&+\delta^{2k(1-\varepsilon_0)+1/2}+ \delta^{k(1-\varepsilon_0)+\varepsilon_0}\bigl(\sqrt{\mathcal E_{1,\leq m+2}(s,u)} + \sqrt{\mathcal E_{2,\leq m+2}(s,u)}\bigr),
\end{split}
\end{equation}
\begin{equation}\label{Zkg}
\begin{split}
\delta^{l+p\varepsilon_k}\|\slashed{\mathcal L}_Z^{m+1}\slashed g\|_{s,u}\lesssim&(\delta^{k(1-\varepsilon_0)}+\delta^{\varepsilon_k})\bigl(\delta^{1/2}
+\delta^{l_1+p_1\varepsilon_k}\|Z^{m_1}\operatorname{tr}\check{\mathscr X}\|_{s,u}\bigr)\\
&+\delta^{l_1+p_1\varepsilon_k}\|Z^{m_1}\upsilon\|_{s,u}+\delta\delta^{l_{0}+p_0\varepsilon_k}\|Z^{m_{0}}\mu\|_{s,u}\\
&+\delta^{2k(1-\varepsilon_0)}\sum_{i=1}^2\delta^{l_1+p_1\varepsilon_k}\|Z^{m_1}\check L^i\|_{s,u}+\delta^{k(1-\varepsilon_0)}\delta^{l_0+p_0\varepsilon_k}\|Z^{m_0}x\|_{s,u}\\
&+ \delta^{k(1-\varepsilon_0)+\varepsilon_0} \bigl(\sqrt{\mathcal E_{1,\leq m+2}(s,u)} + \sqrt{\mathcal E_{2,\leq m+2}(s,u)}\bigr).
\end{split}
\end{equation}
Combining \eqref{ZkL}--\eqref{Zkg} and using induction on the derivative orders yield
\begin{align}
\delta^{l+p\varepsilon_k}\|Z^{m+2}x^i\|_{s,u}\lesssim&\delta^{1/2}+\Pi(s,u),\label{Zk1x}\\
\delta^{l+p\varepsilon_k}\|\slashed{\mathcal L}_Z^{m+1}\slashed g\|_{s,u}\lesssim&\delta^{1/2}(\delta^{k(1-\varepsilon_0)}+\delta^{\varepsilon_k})+\Pi(s,u),\label{Zk1g}\\
\delta^{l+p\varepsilon_k}\bigl(\|Z^{m+1}\check L^i\|_{s,u}&+\|Z^{m+1}\upsilon\|_{s,u}+\|Z^{m+1}\check\varrho\|_{s,u}\bigr)\lesssim\delta^{k(1-\varepsilon_0)+1/2}+\Pi(s,u),\label{Zk1L}\\
\delta^{l+p\varepsilon_k}\|\slashed{\mathcal L}_Z^{m+1}R\|_{s,u}\lesssim& \delta^{1/2}(\delta^{\varepsilon_0}+\delta^{2k(1-\varepsilon_0)})+(\delta^{\varepsilon_k}
+\delta^{k(1-\varepsilon_0)})\delta^{l_{1}+p_1\varepsilon_k}\|Z^{m_1}\operatorname{tr}\check{\mathscr X}\|_{s,u}\no\\
&+(\delta^{(3k+1)(1-\varepsilon_0)}+\delta^{1+k(1-\varepsilon_0)}+\delta^{1+\varepsilon_k})\delta^{l_{0}+p_0\varepsilon_k}\|Z^{m_{0}}\mu\|_{s,u}\no\\
&+(\delta^{2k(1-\varepsilon_0)+\varepsilon_0}+\delta^{2\varepsilon_0})\bigl( \sqrt{\mathcal E_{1,\leq m+2}(s,u)} + \sqrt{\mathcal E_{2,\leq m+2}(s,u)} \bigr),\label{Zk1R}
\end{align}
where
\begin{align*}
\Pi(s,u)=&\delta^{l_1+p_1\varepsilon_k}\|Z^{m_1}\operatorname{tr}\check{\mathscr X}\|_{s,u}+\bigl( \delta + \delta^{(2k+1)(1-\varepsilon_0)} \bigr) \delta^{l_0+p_0\varepsilon_k} \| Z^{m_0}\mu \|_{s,u}\\
&+\delta^{k(1-\varepsilon_0)+\varepsilon_0}\bigl(\sqrt{\mathcal E_{1,\leq m+2}(s,u)} + \sqrt{\mathcal E_{2,\leq m+2}(s,u)}\bigr).
\end{align*}
Using \eqref{Lpi}--\eqref{Rpi}, one arrives at
\begin{align}
\delta^{l+p\varepsilon_k}\|\slashed{\mathcal L}_Z^{m}\leftidx{^{(R)}}{\slashed\pi}\|_{s,u}\lesssim&\delta^{k(1-\varepsilon_0)+ 1/2}+\Pi(s,u),\label{ZkRp}\\
\delta^{l+p\varepsilon_k}\|\slashed{\mathcal L}_Z^{m}\leftidx{^{(R)}}{\slashed\pi}_{\mathring L}\|_{s,u}\lesssim&(1+\delta^{2k(1-\varepsilon_0)-\varepsilon_0})\bigl(\delta^{k(1-\varepsilon_0)+1/2}+\Pi(s,u)\bigr),\label{ZkRLp}\\
\delta^{l+p\varepsilon_k}\|\slashed{\mathcal L}_Z^{m}\leftidx{^{(R)}}{\slashed\pi}_T\|_{s,u}\lesssim&\delta^{k(1-\varepsilon_0)-1}\bigl(\delta^{k(1-\varepsilon_0)+1/2}+\Pi(s,u)\bigr),\label{ZkRLpi}\\
\delta^{l+p\varepsilon_k}\|\slashed{\mathcal L}_Z^m\leftidx{^{(T)}}{\slashed\pi}\|_{s,u}\lesssim&\delta^{k(1-\varepsilon_0)-1}\delta^{l_1+p_1\varepsilon_k}\|Z^{m_1}\operatorname{tr}\check{\mathscr X}\|_{s,u}+(1+\delta^{3k(1-\varepsilon_0)-\varepsilon_0})\delta^{l_0+p_0\varepsilon_k} \| Z^{m_0}\mu \|_{s,u}\no\\
&+\delta^{k(1-\varepsilon_0)-1/ 2}
+\delta^{(k-1)(1-\varepsilon_0)}\bigl(\sqrt{\mathcal E_{1,\leq m+2}(s,u)} + \sqrt{\mathcal E_{2,\leq m+2}(s,u)}\bigr) ,\label{ZkTpi}\\
\delta^{l+p\varepsilon_k}\|\slashed{\mathcal L}_Z^m\leftidx{^{(T)}}{\slashed\pi}_{\mathring L}\|_{s,u}\lesssim&\delta^{k(1-\varepsilon_0)-1}\delta^{l_1+p_1\varepsilon_k}\|Z^{m_1}\operatorname{tr}\check{\mathscr X}\|_{s,u}+(1+\delta^{2k(1-\varepsilon_0)-\varepsilon_0})\delta^{l_0+p_0\varepsilon_k} \| Z^{m_0}\mu \|_{s,u}\no\\
&+\delta^{k(1-\varepsilon_0) -1/2}
+\delta^{(k-1)(1-\varepsilon_0)}\bigl(\sqrt{\mathcal E_{1,\leq m+2}(s,u)} + \sqrt{\mathcal E_{2,\leq m+2}(s,u)}\bigr).\label{ZkTLpi}
\end{align}

Note that all the terms on the left-hand sides of \eqref{Zk1x}--\eqref{ZkTLpi} are controlled by the \(L^2\)-norms of the derivatives of \(\operatorname{tr}{\check{\mathscr X}}\) and \(\mu\).
We now estimate \(\|Z^{m}\operatorname{tr}{\check{\mathscr X}}\|_{s,u}\) and \(\|Z^{m+1}\mu\|_{s,u}\).

It follows from \eqref{Tchi'} and \eqref{Zk1x}--\eqref{ZkTLpi} that when \(n_1+n_2=m-1\),
\begin{align}
&\delta^{1+l}\|\bar Z^{n_1} T R^{n_2}\operatorname{tr}{\check{\mathscr X}}\|_{s,u}\leq\delta^{1+l}\|\bar Z^{n_1}[T,R^{n_2}]\operatorname{tr}{\check{\mathscr X}}\|_{s,u}+\delta^{1+l}\|\bar Z^{n_1}R^{n_2}T\operatorname{tr}{\check{\mathscr X}}\|_{s,u}\nonumber\\
\lesssim & \delta^{k(1-\varepsilon_0)+1}\delta^{l_2}\bigl(\|\slashed{\mathcal L}_{\bar Z}^{m_2}\leftidx{^{(R)}}{\slashed\pi}_T\|_{s,u}+\|\slashed{\mathcal L}_{\bar Z}^{m_2}\leftidx{^{(T)}}{\slashed\pi}\|_{s,u}\bigr)+\delta^{(k-1)(1-\varepsilon_0)}\delta^{l_0}\|\bar Z^{m_0}\vp\|_{s,u}\nonumber\\
&+\delta^{k(1-\varepsilon_0)}\bigl(\delta^{l_2}\|\bar Z^{m_2}\operatorname{tr}{\check{\mathscr X}}\|_{s,u}+\delta^{l_1}\|\slashed{\mathcal L}_{\bar Z}^{m_1}\slashed g\|_{s,u}+\sum_{i=1}^2\bigl(\delta^{l_1}\|\bar Z^{m_1}\check L^i\|_{s,u}+\delta^{l_0}\|\bar Z^{m_0}x^i\|_{s,u}\bigr)\bigr)\nonumber\\
&+\delta(1+\delta^{2k(1-\varepsilon_0)-\varepsilon_0})\delta^{l_0}\|\bar Z^{m_0}\mu\|_{s,u}\nonumber\\
\lesssim & \delta^{k(1-\varepsilon_0)+1/2}+\delta^{k(1-\varepsilon_0)}\delta^{l_2}\|\bar Z^{m_2}\operatorname{tr}{\check{\mathscr X}}\|_{s,u}+\delta(1+\delta^{2k(1-\varepsilon_0)-\varepsilon_0})\delta^{l_0}\|\bar Z^{m_0}\mu\|_{s,u}\label{barZTchi}\\
&+\delta^{k(1-\varepsilon_0)+\varepsilon_0} \bigl( \sqrt{\mathcal E_{1,\leq m+2}(s,u)} + \sqrt{\mathcal E_{2,\leq m+2}(s,u)}\bigr),\nonumber
\end{align}
where \(\bar Z\in\{T,R\}\) and \(l\) is the number of \(T\) in \(\bar Z^{n_1}\), \(0\leq m_2\leq m-1\),
and \(l_j\) counts the occurrences of \(T\) in \(\bar Z^{m_j}\). In what follows, \(l\) counts the occurrences of \(T\)
in the full string on the left-hand side.

In addition, commuting \eqref{lmu} with $\bar Z^{m+1}$ and estimating the resulting commutator yield
\begin{align}
&\delta^{l}\|\mathring L\bar Z^{m+1}\mu\|_{s,u}\nonumber\\
\lesssim & \delta^{2k(1-\varepsilon_0)-\varepsilon_0}\delta^{l_0}\|\bar Z^{m_0}\mu\|_{s,u}+\delta^{(k-1)(1-\varepsilon_0)}\delta^{l_0}\bigl(\|\mu\mathring L\bar Z^{m_0}\varphi\|_{s,u}+\|T\bar Z^{m_0}\varphi\|_{s,u}\bigr)\nonumber\\
&+\delta^{l_1}\bigl(\delta^{k(1-\varepsilon_0)}\|\slashed{\mathcal L}_{\bar Z}^{m_1}\leftidx{^{(R)}}{\slashed\pi}_{T}\|_{s,u}+\|\slashed{\mathcal L}_{\bar Z}^{m_1}\leftidx{^{(R)}}{\slashed\pi}_{\mathring L}\|_{s,u}+\delta\|\slashed{\mathcal L}_{\bar Z}^{m_1}\leftidx{^{(T)}}{\slashed\pi}_{\mathring L}\|_{s,u}\bigr)\nonumber\\
&+\delta^{k(1-\varepsilon_0)-1}\delta^{l_0}\bigl(\sum_{i=1}^2\|\bar Z^{m_0}\check L^i\|_{s,u}+\|\bar Z^{m_0}\upsilon\|_{s,u}\bigr)\nonumber\\
&+\delta^{3k(1-\varepsilon_0)-1}\delta^{l_1}\|\slashed{\mathcal L}_{\bar Z}^{m_1}\slashed g\|_{s,u}
+\delta^{(k-1)(1-\varepsilon_0)-1}\delta^{l_0}\|\bar Z^{m_0}\vp\|_{s,u}\nonumber\\
\lesssim & \delta^{2k(1-\varepsilon_0)-\varepsilon_0}\delta^{l_0}\|\bar Z^{m_0}\mu\|_{s,u}+\delta^{2k(1-\varepsilon_0)-1/2}+\delta^{k(1-\varepsilon_0)-1}\delta^{l_1}\|\bar Z^{m_1}\operatorname{tr}{\check{\mathscr X}}\|_{s,u}\label{Lmu}\\
&+\delta^{(k-1)(1-\varepsilon_0)}\bigl(\sqrt{\mathcal E_{1,\leq m+2}(s,u)}+\sqrt{\mathcal E_{2,\leq m+2}(s,u)}\bigr).\nonumber
\end{align}
By the initial data estimates, applying \eqref{Ff} to \eqref{Lmu}, and then utilizing Gronwall's inequality, we obtain
\begin{equation}\label{m}
\begin{split}
    \delta^{l} \|\bar Z^{m+1} \mu \|_{s,u} \lesssim &\delta^{k(1-\varepsilon_0)+\varepsilon_0-1/2} + \delta^{(k-1)(1-\varepsilon_0)}\int_{t_0}^s \bigl(\sqrt{\mathcal E_{1,\leq m+2}(\tau,u)}+\sqrt{\mathcal E_{2,\leq m+2}(\tau,u)}\bigr)\,d\tau\\
    + & \delta^{k(1-\varepsilon_0)-1}\int_{t_0}^s\delta^{l_1}\|\bar Z^{m_1}\operatorname{tr}{\check{\mathscr X}}\|_{\tau,u}\,d\tau.
    \end{split}
\end{equation}
In addition, it follows from \eqref{barZTchi}, \eqref{m}, induction argument and Gronwall's inequality that for \(l\geq1\),
\begin{align}
    \delta^{l}\|\bar Z^{m}\operatorname{tr}{\check{\mathscr X}}\|_{s,u}\lesssim &\delta^{k(1-\varepsilon_0)+1/2}+\delta^{k(1-\varepsilon_0)+\varepsilon_0}\sum_{i=1}^2\sqrt{\mathcal E_{i,\leq m+2}(s,u)}\no\\
    &+\delta^{k(1-\varepsilon_0)+\varepsilon_0}(1+\delta^{2k(1-\varepsilon_0)-\varepsilon_0})\sum_{i=1}^2\int_{t_0}^s\sqrt{\mathcal E_{i,\leq m+2}(\tau,u)}\,d\tau,\label{1lchi}\\
    \delta^{l}\|\bar Z^{m+1}\mu\|_{s,u}\lesssim & \delta^{k(1-\varepsilon_0)+\varepsilon_0-1/2} +  \delta^{(k-1)(1-\varepsilon_0)}\sum_{i=1}^2\int_{t_0}^s\sqrt{\mathcal E_{i,\leq m+2}(\tau,u)}\,d\tau.\label{1lmu}
\end{align}

For the remaining case, the derivative string contains at least one
$\varrho\mathring L$. In order to treat this case, we will use \eqref{Lchi'} and \eqref{lmu} directly,
and then apply the induction argument on the derivative orders and on the number of
occurrences of $\varrho\mathring L$.
Based on this, together with \eqref{1lchi}--\eqref{1lmu} and the estimates for
pure $R$ derivatives, one has that for $l+p\geq1$,
\begin{align}
    \delta^{l+p\varepsilon_k}\|Z^{m}\operatorname{tr}{\check{\mathscr X}}\|_{s,u}\lesssim &\delta^{k(1-\varepsilon_0)+1/2}+\delta^{k(1-\varepsilon_0)+\varepsilon_0}\sum_{i=1}^2\sqrt{\mathcal E_{i,\leq m+2}(s,u)}\no\\
    &+\delta^{k(1-\varepsilon_0)+\varepsilon_0}(1+\delta^{2k(1-\varepsilon_0)-\varepsilon_0})\sum_{i=1}^2\int_{t_0}^s\sqrt{\mathcal E_{i,\leq m+2}(\tau,u)}\,d\tau,\label{1lpchi}\\
    \delta^{l+p\varepsilon_k}\|Z^{m+1}\mu\|_{s,u}\lesssim &\delta^{k(1-\varepsilon_0)+\varepsilon_0-1/2}
    +\delta^{(k-1)(1-\varepsilon_0)}\Xi^2_{m+2}(s,u)\no\\
    &+\delta^{2\varepsilon_0-1}\sum_{i=1}^2
    \sqrt{\mathcal E_{i,\leq m+2}(s,u)}.\label{1lpmu}
\end{align}

Collecting all the estimates from \eqref{RmcX}, \eqref{RMMu}, \eqref{1lpchi}--\eqref{1lpmu} and \eqref{Zk1x}--\eqref{ZkTLpi},
we finish the proof of Proposition \ref{L2chi}.
\end{proof}

Next, we estimate the $L^2$ norms of $\mathring L\phi$ and $\slashed d\phi$ involving all the related
derivatives that will be used in the next subsection.

\begin{corollary}\label{Cor6.3}
	Under the assumptions of Proposition~\ref{L2chi}, let
	$0\leq m\leq 2N-6$, $\bar Z\in\{T,R\}$ and $l$ denote the number of occurrences
	of $T$ in the corresponding derivative strings. Then
	\begin{align}
	\delta^l\|\bar Z^{m+1}\phi\|_{s,u}\lesssim & \delta^{5/2-\varepsilon_0}+\delta^{k(1-\varepsilon_0)+1}\Xi_{m+2}^2(s,u)+\delta^2\sum_{i=1}^2\sqrt{\mathcal E_{i,\leq m+2}(s,u)},\label{Zphi}\\
	\delta^l\|\bar Z^{m+1}\mathring L\phi\|_{s,u}\lesssim &\delta^{5/2-\varepsilon_0-\varepsilon_k}+\delta^{(2k+1)(1-\varepsilon_0)}\Xi^2_{m+2}(s,u)+\delta\sqrt{\mathcal E_{1,\leq m+2}(s,u)}\nonumber\\
	&+\delta^{(k+1)(1-\varepsilon_0)+1}\sqrt{\mathcal E_{2,\leq m+2}(s,u)}+\delta^{2-\varepsilon_0+k(1-\varepsilon_0)}\|R^m\slashed\triangle\mu\|_{s,u},\label{ZLphi}\\
	\delta^{l}\|\slashed{\mathcal L}_{\bar Z}^{m+1}\slashed{d}\phi\|_{s,u}\lesssim & \delta^{5/2-\varepsilon_0}+\delta\sqrt{\mathcal E_{1,\leq m+2}(s,u)}+\delta^2\sqrt{\mathcal E_{2,\leq m+2}(s,u)}\nonumber\\
	&+\delta^{2-\varepsilon _0}\|R^m\slashed\triangle\mu\|_{s,u}+\delta^{1+k(1-\varepsilon_0)}\Xi^2_{m+2}(s,u),\label{Zdphi}\\
	\delta^l\|\bar Z^m\mathring L^2\phi\|_{s,u}\lesssim & \delta^{5/2-\varepsilon_0-2\varepsilon_k}+\delta^{2-\varepsilon_0+2k(1-\varepsilon_0)}\|R^m\slashed\triangle\mu\|_{s,u}
+\delta^{1-\varepsilon_k}\sqrt{\mathcal E_{1,\leq m+2}(s,u)}\nonumber\\
&+\delta^{2(k+1)(1-\varepsilon_0)} \sqrt{\mathcal E_{2,\leq m+2}(s,u)}+\delta^{3k(1-\varepsilon_0)-2\varepsilon_0+1}\Xi_{m+2}^2(s,u).\label{ZLLphi}
\end{align}
Moreover,
\begin{align}
\delta^l\|\bar Z^{m+1}(G_{\mathring L\mathring L}^\gamma\mathring L_\gamma)\|_{s,u}
\lesssim&\delta^{k(1-\varepsilon_0)}\Bigl\{\delta^{1/2}+\delta^{\varepsilon_0}\|R^m\slashed\triangle\mu\|_{s,u}+\delta^{2\varepsilon_0-1}\sqrt{\mathcal E_{1,\leq m+2}(s,u)}\nonumber\\
&\qquad\qquad+\delta^{\varepsilon_0}\sqrt{\mathcal E_{2,\leq m+2}(s,u)}
+\delta^{(k-1)(1-\varepsilon_0)}\Xi^1_{m+2}(s,u)\Bigr\},\label{ZGL}\\
\delta^l\|\bar Z^m\mathring L\phi\|_{s,u}
\lesssim&\delta^{(k+2)(1-\varepsilon_0)+1/2}+\delta\sqrt{\mathcal E_{1,\leq m+2}(s,u)}\nonumber\\
&+\delta^{(k+1)(1-\varepsilon_0)+1}\sqrt{\mathcal E_{2,\leq m+2}(s,u)}+\delta^{(2k+1)(1-\varepsilon_0)}\Xi^2_{m+2}(s,u).\label{ZmLphi}
\end{align}
\end{corollary}

\begin{proof}
From \eqref{SiSi}, we have
\begin{equation}\label{lowerphi}
\begin{split}
\delta^l\|\bar Z^{m+1}\phi\|_{s,u}&\lesssim \delta^{1+l}\|T\bar Z^{m+1}\phi\|_{s,u}\\
&\lesssim \delta^{l+1}\bigl(\|[T,\bar Z^{m+1}]\phi\|_{s,u}+\|\bar Z^{m+1}(\mu\tilde T^i\varphi_i)\|_{s,u}\bigr)\\
&\lesssim \delta^{3-\varepsilon_0}\delta^{l_1}\|\slashed{\mathcal L}_{\bar Z}^{m_1}\leftidx{^{(R)}}{\slashed\pi}_T\|_{s,u}+\delta^{2k(1-\varepsilon_0)}\delta^{l_0}\|\bar Z^{m_0}\phi\|_{s,u}\\
&\quad+\delta^{(2k+1)(1-\varepsilon_0)+1}\delta^{l_1}\|\slashed{\mathcal L}_{\bar Z}^{m_1}\slashed g\|_{s,u}+\delta\delta^{l_0}\|\bar Z^{m_0}\varphi\|_{s,u}\\
&\quad+\delta^{2-\varepsilon_0}\delta^{l_0}\bigl(\|\bar Z^{m_0}\mu\|_{s,u}+\sum_{i=1}^2(\|\bar Z^{m_0}\check L^i\|_{s,u}+\|\bar Z^{m_0}x^i\|_{s,u})\bigr).
\end{split}	
\end{equation}
Absorbing the term $\delta^{2k(1-\varepsilon_0)}\delta^{l_0}\|\bar Z^{m_0}\phi\|_{s,u}$ on the right-hand side by the left-hand side
in \eqref{lowerphi}, one has
\begin{align*}
\delta^l\|\bar Z^{m+1}\phi\|_{s,u}
\lesssim & \ \delta^{3-\varepsilon_0}\delta^{l_1}\|\slashed{\mathcal L}_{\bar Z}^{m_1}\leftidx{^{(R)}}{\slashed\pi}_T\|_{s,u}
+\delta^{(2k+1)(1-\varepsilon_0)+1}\delta^{l_1}\|\slashed{\mathcal L}_{\bar Z}^{m_1}\slashed g\|_{s,u}\\
&+\delta^{2-\varepsilon_0}\delta^{l_0}\bigl(\|\bar Z^{m_0}\mu\|_{s,u}+\sum_{i=1}^2(\|\bar Z^{m_0}\check L^i\|_{s,u}+\|\bar Z^{m_0}x^i\|_{s,u})\bigr)+\delta\delta^{l_0}\|\bar Z^{m_0}\varphi\|_{s,u}\\
\lesssim & \ \delta^{5/2-\varepsilon_0}+\delta^{k(1-\varepsilon_0)+1}\Xi_{m+2}^2(s,u)+\delta^2\sum_{i=1}^2\sqrt{\mathcal E_{i,\leq m+2}(s,u)}.
\end{align*}
We next estimate $\bar Z^{m+1}\mathring L\phi$. When $\bar Z^{m+1}=R^{m+1}$, using \eqref{SiSi} gives
\begin{align}
\|R^{m+1}\mathring L\phi\|_{s,u}&\lesssim\delta\|TR^{m+1}\mathring L\phi\|_{s,u}\nonumber\\
&\lesssim \delta\|R^{m+1}T(\mathring L^\alpha\varphi_\alpha)\|_{s,u}+\delta^{1+(k+2)(1-\varepsilon_0)}\|\slashed{\mathcal L}_{R}^{m_1}\leftidx{^{(R)}}{\slashed\pi}_T\|_{s,u}\label{ZLp}\\
&\quad+\delta^{2k(1-\varepsilon_0)}\|R^{m_0}\mathring L\phi\|_{s,u}+\delta^{(3k+2)(1-\varepsilon_0)}\|\slashed{\mathcal L}_{R}^{m_1}\leftidx{^{(R)}}{\slashed\pi}\|_{s,u},\nonumber
\end{align}
where $T(\mathring L^\alpha \varphi_\alpha)=(T\mathring L^i)\varphi_i+\mu\tilde T^i(\mathring L\varphi_i)$.
Recalling that $T\mathring L^i$ satisfies \eqref{TL}, we have
\begin{equation}\label{TLp}
T(\mathring L^\alpha \varphi_\alpha)=\mu\tilde T^i(\mathring L\varphi_i)+(\slashed d_X\mu)\slashed d^X\phi-(G_{X\mathring L}^\gamma T\varphi_\gamma)\slashed d^X\phi+\frac12(G_{\mathring L\mathring L}^\gamma T\varphi_\gamma)\tilde T^i\varphi_i+\dots.
\end{equation}
With the help of \eqref{GT} and \eqref{gLLL}, substituting \eqref{TLp} into \eqref{ZLp} yields
\begin{align}
&\|R^{m+1}\mathring{L}\phi\|_{s,u}\nonumber\\
\lesssim&\ \delta^{(k+1)(1-\varepsilon_0)}\Bigl(\sum_{i=1}^2\|R^{m_0}\check{L}^i\|_{s,u}+\|R^{m_0}\check{\varrho}\|_{s,u}
+\|R^{m_0}\varphi\|_{s,u}+\delta\|TR^{m_0}\varphi\|_{s,u}\Bigr)\nonumber\\
&+\delta^{(k+2)(1-\varepsilon_0)}\bigl(\|R^{m_{-1}}x\|_{s,u}+\|R^{m_0}\mu\|_{s,u}\bigr)
+\delta^{k(1-\varepsilon_0)+1}\|\mu\slashed{d}R^{m_0}\varphi\|_{s,u}\nonumber\\
&+\delta\|\mu\mathring{L}R^{m_0}\varphi\|_{s,u}
+\delta^{(k+2)(1-\varepsilon_0)+1}\|\slashed{\mathcal{L}}_{R}^{m_1}\,{}^{(R)}\!\slashed{\pi}_T\|_{s,u}
+\delta^{k(1-\varepsilon_0)}\|\slashed{d}R^{m_0}\phi\|_{s,u}\nonumber\\
&+\delta^{3-\varepsilon_0}\|\slashed{d}R^{m+1}\mu\|_{s,u}
+\delta^{(k+1)(1-\varepsilon_0)}(\delta+\delta^{(2k+1)(1-\varepsilon_0)})\|\slashed{\mathcal L}_{R}^{m_1}\leftidx{^{(R)}}{\slashed\pi}\|_{s,u}\nonumber\\
&+\delta^{2-\varepsilon_0}\|\slashed{\mathcal{L}}_{R}^{m_1}\,{}^{(R)}\!\slashed{\pi}_{\mathring{L}}\|_{s,u}
+\delta^{3-\varepsilon_0}\|\slashed{\mathcal{L}}_{R}^{m_1}\,{}^{(T)}\!\slashed{\pi}_{\mathring{L}}\|_{s,u}\nonumber\\
\lesssim&\ \delta^{(k+2)(1-\varepsilon_0)+1/2}+\delta^{(2k+1)(1-\varepsilon_0)}\Xi^2_{m+2}(s,u)+\delta\sqrt{\mathcal E_{1,\leq m+2}(s,u)}\nonumber\\
&+\delta^{(k+1)(1-\varepsilon_0)+1}\sqrt{\mathcal E_{2,\leq m+2}(s,u)}+\delta^{3-\varepsilon_0}\|R^m\slashed\triangle\mu\|_{s,u}+\delta^{k(1-\varepsilon_0)}\|\slashed{d}R^{m_0}\phi\|_{s,u},\nonumber
\end{align}
where $1\leq m_p\leq m+1-p$ with $p=-1,0,1$. If $\bar Z^{m+1}=\bar Z^{p_1}TR^{p_2}$ with $p_1+p_2=m$, then $\bar Z^{m+1}\mathring L\phi=\bar Z^{p_1}R^{p_2}T(\mathring L^\alpha\varphi_\alpha)+\bar Z^{p_1}[T,R^{p_2}]\mathring L\phi$. Hence, for any $\bar Z^{m+1}$ containing $l$ occurrences of the vector field $T$, one has
\begin{equation}\label{bZLp}
\begin{split}
\delta^l\|\bar Z^{m+1}\mathring L\phi\|_{s,u}
\lesssim & \ \delta^{(k+2)(1-\varepsilon_0)+1/2}+\delta^{(2k+1)(1-\varepsilon_0)}\Xi^2_{m+2}(s,u)+\delta\sqrt{\mathcal E_{1,\leq m+2}(s,u)}\\
&+\delta^{(k+1)(1-\varepsilon_0)+1}\sqrt{\mathcal E_{2,\leq m+2}(s,u)}+\delta^{3-\varepsilon_0}\|R^m\slashed\triangle\mu\|_{s,u}\\
&+\delta^{k(1-\varepsilon_0)}\delta^{l_0}\|\slashed{d}\bar Z^{m_0}\phi\|_{s,u}.
\end{split}
\end{equation}

In addition, due to $|\slashed{\mathcal L}_{\bar Z}^{m+1}\slashed d\phi|\lesssim |R\bar Z^{m+1}\phi|$, it suffices to
treat the $L^2$ norm of $R\bar Z^{m+1}\phi$. As before, for $\bar Z^{m+1}=R^{m+1}$, we have
\begin{equation}\label{Rp}
\begin{split}
\|R^{m+2}\phi\|_{s,u}
\lesssim & \ \delta\|R^{m+1}TR\phi\|_{s,u}+\delta^{3-\varepsilon_0}\|\slashed{\mathcal L}_{R}^{m_1}\leftidx{^{(R)}}{\slashed\pi}_T\|_{s,u}+\delta^{2k(1-\varepsilon_0)}\|R^{m_0}R\phi\|_{s,u}\\
&+\delta^{2k(1-\varepsilon_0)+2-\varepsilon_0}\|\slashed{\mathcal L}_{R}^{m_1}\leftidx{^{(R)}}{\slashed\pi}\|_{s,u}.
\end{split}
\end{equation}
Note that $TR\phi=\leftidx{^{(R)}}{\slashed\pi}_{TX}\slashed d^X\phi+(R\mu)\tilde T^i\varphi_i-\mu(Rg^{0i})\varphi_i-\mu(R\mathring L^i)\varphi_i+\mu \tilde T^iR\varphi_i$. It follows from \eqref{Rpi} and \eqref{deL} that
\begin{equation*}
\begin{split}
TR\phi=&\ \upsilon\slashed d^X\mu\slashed d_X\phi+\mu(G_{X\mathring L}^\gamma R\varphi_\gamma)\slashed d^X\phi+\mu(G_{X\tilde T}^\gamma R\varphi_\gamma)\slashed d^X\phi+\mu\epsilon_i^a\check T^ig_{ab}(\slashed d_Xx^b\slashed d^X\phi)\\
&-\frac{\mu}{\varrho}R\phi+\upsilon(G_{X\tilde T}^\gamma T\varphi_\gamma)\slashed d^X\phi-\frac12\upsilon\mu(G_{\tilde T\tilde T}^\gamma\slashed d^X\varphi_\gamma)\slashed d_X\phi+\mu(G_{\mathring L\tilde T}^\gamma R\varphi_\gamma)\tilde T^i\varphi_i\\
&+\frac12\mu(G_{\tilde T\tilde T}^\gamma R\varphi_\gamma)\tilde T^i\varphi_i+(R\mu)\tilde T^i\varphi_i-\mu(Rg^{0i})\varphi_i+\mu \tilde T^iR\varphi_i.
\end{split}
\end{equation*}
Thus,
\begin{align}
\|R^{m+2}\phi\|_{s,u}
\lesssim&\ \delta^{2-\varepsilon_0}\bigl\{\|\slashed dR^{m+1}\mu\|_{s,u}+\|R^{m_0}\mu\|_{s,u}\bigr\}+\delta^{3-\varepsilon_0}\|\slashed{\mathcal L}_{R}^{m_1}\leftidx{^{(R)}}{\slashed\pi}_T\|_{s,u}\nonumber\\
&+\delta\|R^{m_0}\varphi\|_{s,u}+\delta \|\mu\slashed dR^{m+1}\varphi\|_{s,u}+\delta^{2k(1-\varepsilon_0)+2}\|TR^{m+1}\varphi\|_{s,u}\nonumber\\
&+\delta^{(k+1)(1-\varepsilon_0)+1}\|R^{m_0}\upsilon\|_{s,u}+\delta^{(2k+1)(1-\varepsilon_0)+1}\|\slashed{\mathcal L}_{R}^{m_1}\leftidx{^{(R)}}{\slashed\pi}\|_{s,u}\nonumber\\
&+\delta^{2-\varepsilon_0}\sum_{i=1}^2\bigl(\|R^{m_0}\check L^i\|_{s,u}+\|R^{m_{-1}}x^i\|_{s,u}\bigr)\nonumber\\
\lesssim & \ \delta^{5/2-\varepsilon_0}+\delta\sqrt{\mathcal E_{1,\leq m+2}(s,u)}+\delta^2\sqrt{\mathcal E_{2,\leq m+2}(s,u)}\label{Rmp}\\
&+\delta^{2-\varepsilon _0}\|R^m\slashed\triangle\mu\|_{s,u}+\delta^{1+k(1-\varepsilon_0)}\Xi^2_{m+2}(s,u).\nonumber
\end{align}
For $\bar Z^{m+1}=\bar Z^{p_1}TR^{p_2}$ with $p_1+p_2=m$, due to $R\bar Z^{m+1}\phi=R\bar Z^{p_1}R^{p_2}(\mu \tilde T^i\varphi_i)+R\bar Z^{p_1}[T,R^{p_2}]\phi$, one has
\begin{equation}\label{DXY-8}
\begin{split}
\delta^l\|R\bar Z^{m+1}\phi\|_{s,u}
\lesssim & \ \delta^{3-\varepsilon_0}\delta^{l_1}\|\slashed{\mathcal L}_{\bar Z}^{m_1}\leftidx{^{(R)}}{\slashed\pi}_T\|_{s,u}+\delta^{2k(1-\varepsilon_0)}\delta^{l_0}\|\bar Z^{m_0}\phi\|_{s,u}+\delta^{(2k+1)(1-\varepsilon_0)+1}\delta^{l_1}\|\slashed{\mathcal L}_{\bar Z}^{m_1}\slashed g\|_{s,u}\\
&+\delta^{2-\varepsilon_0}\delta^{l_0}\bigl(\|\bar Z^{m_0}\mu\|_{s,u}+\sum_{i=1}^2(\|\bar Z^{m_0}\check L^i\|_{s,u}+\|\bar Z^{m_0}x^i\|_{s,u})\bigr)+\delta\delta^{l_0}\|\bar Z^{m_0}\varphi\|_{s,u}.
\end{split}
\end{equation}
Note that the terms on the right-hand side in \eqref{DXY-8} are exactly the same as the ones in \eqref{lowerphi}. Therefore,
we can obtain \eqref{Zdphi}, and then \eqref{ZLphi} follows by substituting \eqref{Zdphi} into \eqref{bZLp}.

Analogously, owing to
\[
\|R^m\mathring L^2\phi\|_{s,u}
\lesssim \delta\bigl\|[T,R^m\mathring L]\mathring L\phi\bigr\|_{s,u}
+\delta\bigl\|R^m\mathring L T\bigl(\mathring L^\alpha\varphi_\alpha\bigr)\bigr\|_{s,u},
\]
one deduces from \eqref{TLp} that
\begin{align*}
&\|R^m\mathring L^2\phi\|_{s,u}\\
\lesssim &\delta^{1-\varepsilon_k}\biggl(
\sum_{i=1}^2\bigl(\delta^{-\varepsilon_k}\|R^{m_0}\check L^i\|_{s,u}
+\|R^{m_1}\mathring L\check L^i\|_{s,u}\bigr)
+\delta^{-\varepsilon_k}\|R^{m_0}\check\varrho\|_{s,u}
+\|R^{m_1}\mathring L\check\varrho\|_{s,u}
\biggr)\\
&+\delta^{(2k+2)(1-\varepsilon_0)-\varepsilon_0}\|R^{m_{0}}x\|_{s,u}
+\delta^{3-\varepsilon_0-\varepsilon_k}
\bigl\|\slashed{\mathcal L}_{R}^{m_1} \leftidx{^{(T)}}{\slashed\pi}_{\mathring L}\bigr\|_{s,u}+\delta^{3-\varepsilon_0-2\varepsilon_k}
\bigl\|\slashed{\mathcal L}_{R}^{m_1} \leftidx{^{(R)}}{\slashed\pi}_T\bigr\|_{s,u}\\
&+\delta^{k(1-\varepsilon_0)}\|R^{m_0}\mathring L\phi\|_{s,u}+\delta^{k(1-\varepsilon_0)-\varepsilon_k}\|R^{m_0}\phi\|_{s,u}
+\delta^{3-\varepsilon_0}\bigl\|\slashed d R^{m}\mathring L\mu\bigr\|_{s,u}\\
&+\delta^{2-\varepsilon_0-\varepsilon_k}\|R^{m_1}\mathring L\mu\|_{s,u}
+\delta^{2-\varepsilon_0-2\varepsilon_k}\|R^{m_0}\mu\|_{s,u}+\delta^{2-\varepsilon_0-\varepsilon_k}
\bigl\|\slashed{\mathcal L}_{R}^{m_1} \leftidx{^{(R)}}{\slashed\pi}_{\mathring L}\bigr\|_{s,u}\\
&+\delta^{1-\varepsilon_k}\sqrt{\mathcal E_{1,\leq m+2}(s,u)}+\delta^{2(k+1)(1-\varepsilon_0)} \sqrt{\mathcal E_{2,\leq m+2}(s,u)}\\
\lesssim & \delta^{5/2-\varepsilon_0-2\varepsilon_k}
+\delta^{2-\varepsilon_0+2k(1-\varepsilon_0)}\|R^m\slashed\triangle\mu\|_{s,u}+\delta^{1-\varepsilon_k}\sqrt{\mathcal E_{1,\leq m+2}(s,u)}\\
&+\delta^{2(k+1)(1-\varepsilon_0)} \sqrt{\mathcal E_{2,\leq m+2}(s,u)}+\delta^{3k(1-\varepsilon_0)-2\varepsilon_0+1}\Xi_{m+2}^2(s,u).
\end{align*}

Meanwhile, it follows from \eqref{gLLL} that
\begin{align*}
\delta^l\bigl\lVert\bar{Z}^{m+1}\bigl(G_{\mathring{L}\mathring{L}}^\gamma \mathring{L}\gamma\bigr)\bigr\rVert
&\lesssim \delta^{(k-1)(1-\varepsilon_0)}\delta^{l_0}\bigl(
\lVert\bar{Z}^{m_0}\varphi\rVert{s,u}
\sum_{i=1}^2\lVert\bar{Z}^{m_0}\check{L}^i\rVert_{s,u}
\lVert\bar{Z}^{m_0}\check{\varrho}\rVert_{s,u}
\bigr) \\
&\quad + \delta^{k(1-\varepsilon_0)}\delta^{l_{-1}}\lVert\bar{Z}^{m_{-1}}x\rVert_{s,u}
\delta^{(k-2)(1-\varepsilon_0)}\delta^{l_0}\bigl(
\lVert\bar{Z}^{m_0}\mathring{L}\phi\rVert_{s,u}
\lVert\slashed{d}\bar{Z}^{m_0}\phi\rVert_{s,u}
\bigr).
\end{align*}
Together with Proposition \ref{L2chi} and the estimates \eqref{ZLphi}--\eqref{Zdphi}, this yields \eqref{ZGL}.
In addition, \eqref{ZmLphi} follows from \eqref{ZLphi} and Proposition \ref{L2chi}.\end{proof}

\subsection{Top-order $L^2$ estimates for the derivatives of
$\check{\mathscr X}$ and $\mu$}
\label{Section 6.2}

As stated in Remark~\ref{rem6.1}, to close the energy estimates, we need the precise
bounds for $\|R^{m+1}\mathring L\phi\|_{s,u}$ and
$\|\slashed dR^{m+1}\phi\|_{s,u}$ when \eqref{GT} and \eqref{gLLL} are used to estimate $\|R^{m+1}(G_{\mathring L\mathring L}^\gamma T\vp_\gamma)\|_{s,u}$. For this purpose, it is crucial
to control the term $\|R^m\slashed\triangle\mu\|_{s,u}$ appeared
in Corollary~\ref{Cor6.3}. Note that Proposition~\ref{L2chi}
cannot be applied directly to bound \(\|R^{m}\slashed\triangle\mu\|_{s,u}\) since the resulting term \(\Xi_{m+3}^2(s,u)\)
can not be absorbed by the energies of \((m+2)\) orders.

On the other hand, as in \cite{Ch1,Ding4,M-Y,Sp}, the energy
estimate \eqref{e} and the equation \eqref{Psi} require the control
of derivatives of $\varphi$ up to orders $2N-4$.
By the expressions for $\leftidx{^{(Z)}}D_{\gamma,2}^n$
in \eqref{T2}, \eqref{rL2} and \eqref{R2}, the source terms in \eqref{e}
involve the derivatives of $\check{\mathscr X}$ and $\mu$ up to
orders $2N-5$ and $2N-4$, respectively.
It is worth noticing that the $L^2$ estimates provided by
Proposition~\ref{L2chi} involve the energies of orders $2N-3$, however,
the available energies only control the derivatives of $\varphi$
up to orders $2N-4$.
Thus, in this section, we require to take the top-order $L^2$ estimates
for the derivatives of $\check{\mathscr X}$ and $\mu$
without applying the energies of orders $2N-3$.

\subsubsection{Estimates on the derivatives of $\slashed d\operatorname{tr}\check{\mathscr{X}}$}\label{trchi}

If there is at least one vector field $\varrho\mathring L$ in $Z^m$, that is, $Z^m=Z^{n_1}(\varrho\mathring L)\bar Z^{n_2}$ and $n_1+n_2=m-1$ with $\bar Z\in\{T,R\}$, then it follows from \eqref{Lchi'} and Proposition \ref{L2chi} that
\begin{align}
&\delta^{l+(p+1)\varepsilon_k}\|\mu\slashed dZ^m(\operatorname{tr}\check{\mathscr{X}})\|_{s,u}\nonumber\\
=&\delta^{l+(p+1)\varepsilon_k}\|\mu\slashed dZ^{n_1}\bar Z^{n_2}(\varrho\mathring L) (\operatorname{tr}\check{\mathscr{X}})
+\mu\slashed dZ^{n_1}[\varrho\mathring L,\bar Z^{n_2}](\operatorname{tr}\check{\mathscr{X}})\|_{s, u}\nonumber\\
\lesssim & (\delta^{\varepsilon_k}+\delta^{k(1-\varepsilon_0)})\delta^{l_1+p_1\varepsilon_k}\|Z^{m_1}(\operatorname{tr}\check{\mathscr{X}})\|_{s, u}+\delta^{(k-1)(1-\varepsilon_0)}\delta^{l_0+p_0\varepsilon_k}\|\mu\slashed dZ^{m_0}\varphi\|_{s, u}\nonumber\\
&+\delta^{k(1-\varepsilon_0)}\Big\{\delta^{l_1+p_1\varepsilon_k}\Big(\|\slashed{\mathcal L}_Z^{m_1}\slashed g\|_{s, u}+\sum_{i=1}^2\|Z^{m_1}\check L^i\|_{s, u}\Big)+\delta^{l_0+p_0\varepsilon_k}\|Z^{m_0}x\|_{s, u}\Big\}\nonumber\\
&+\delta^{(k-1)(1-\varepsilon_0)}\delta^{l_0+p_0\varepsilon_k}\|Z^{m_0}\varphi\|_{s, u}+\delta^{\varepsilon_0}\delta^{l_1+p_1\varepsilon_k}(\|\slashed{\mathcal L}_Z^{m_1}\leftidx^{{(R)}}\slashed\pi_{\mathring L}\|_{s, u}+\delta\|\slashed{\mathcal L}_Z^{m_1}\leftidx^{{(T)}}\slashed\pi_{\mathring L}\|_{s, u})\nonumber\\
\lesssim & \delta^{k(1-\varepsilon_0)+1/2}+\delta^{(k-1)(1-\varepsilon_0)}\sqrt{{\mathcal E}_{1,\leq m+2}(s, u)}+\delta^{k(1-\varepsilon_0)+\varepsilon_0}\sqrt{{\mathcal E}_{2,\leq m+2}(s, u)}\label{ZmX}\\
&+\delta^{\varepsilon_0}(1+\delta^{2k(1-\varepsilon_0)-\varepsilon_0})\int_{t_0}^s\bigl\{\mu_{\min}^{-1/2}(\tau)\sqrt{\mathcal E_{1,\leq m+2}(\tau, u)}+\delta^{k(1-\varepsilon_0)}\sqrt{\mathcal E_{2,\leq m+2}(\tau,u)}\bigr\}d\tau,\nonumber
\end{align}
where $l$ (resp. $p$) is the number of $T$ (resp. $\varrho\mathring L$) appearing in $Z^{n_1}\bar Z^{n_2}$.

For the case \(Z^m=\bar Z^m\), we do not apply \eqref{fequation} and \eqref{Lchi'} to derive the top-order \(L^2\) 
norm as in \cite{Ch1,M-Y,Sp,Ding4}. Instead, we will rely on the more precise transport equation \eqref{transport} 
in Lemma~\ref{lem:null-transport},
whose structure is crucial for establishing the top-order estimates below.

It follows from \eqref{transport} that
\begin{align}
&\mathring{L}\bigl(\operatorname{tr}{\mathscr{X}} - E_0\bigr) \nonumber\\
&\quad = \frac{1}{\varrho^2} - \frac{2}{\varrho}\operatorname{tr}\mathscr{X}
- \frac{1}{2}\bigl(G_{\mathring{L}\mathring{L}}^\gamma \mathring{L}_\gamma\bigr)\tilde{T}^i\slashed{\triangle}\varphi_i
- \bigl(\operatorname{tr}\check{\mathscr{X}}\bigr)^2 \nonumber\\
&\quad\quad - \bigl(
\frac{1}{2}G_{\mathring{L}\mathring{L}}^\alpha\mathring{L}\varphi_\alpha
+ G_{\tilde{T}\mathring{L}}^\alpha\mathring{L}\varphi_\alpha
- G_{X\mathring{L}}^\alpha\slashed{d}^X\varphi_\alpha
+ \slashed{g}^{XX}G_{XX}^\alpha\mathring{L}\varphi_\alpha
\bigr)\operatorname{tr}\mathscr{X} \nonumber\\
&\quad\quad + \mathring{L}\bigl(G_{X\mathring{L}}^\gamma\slashed{g}^{XX}\bigr)\slashed{d}_X\varphi_\gamma
- \frac{1}{2}\mathring{L}\bigl(G_{\mathring{L}\mathring{L}}^0\slashed{d}^X x^i\bigr)\slashed{d}_X\varphi_i \label{transport2}\\
&\quad\quad + \frac{1}{2}\mathring{L}\bigl(
g_{\gamma j}G_{\mathring{L}\mathring{L}}^\gamma\bigl(\mathring{L}^\alpha + 2g^{0\alpha}\bigr)\slashed{d}^X x^j
\bigr)\slashed{d}_X\varphi_\alpha\nonumber\\
&\quad\quad + \frac{1}{2}\slashed{d}^X\bigl(G_{\mathring{L}\mathring{L}}^0\bigr)\slashed{d}_X x^i \mathring{L}\varphi_i
- \frac{1}{2}\slashed{d}^X\bigl(G_{\mathring{L}\mathring{L}}^\gamma \mathring{L}_\gamma \tilde{T}^i\bigr)\slashed{d}_X\varphi_i \nonumber\\
&\quad\quad - \frac{1}{2}\slashed{d}^X\bigl(G_{\mathring{L}\mathring{L}}^\gamma\bigr)\slashed{d}_X\varphi_\gamma
- \frac{1}{2}\slashed{d}^X\bigl(
g_{\gamma j}G_{\mathring{L}\mathring{L}}^\gamma\bigl(\mathring{L}^\alpha + 2g^{0\alpha}\bigr)
\bigr)\slashed{d}_X x^j \mathring{L}\varphi_\alpha \nonumber\\
&\quad\quad + \frac{1}{2}\bigl(G_{\mathring{L}\mathring{L}}^0 \mathring{L}\varphi_i\bigr)\slashed{\triangle} x^i
- \frac{1}{2}g_{\gamma j}G_{\mathring{L}\mathring{L}}^\gamma\bigl(\mathring{L}^\alpha + 2g^{0\alpha}\bigr)\mathring{L}\varphi_\alpha \slashed{\triangle} x^j \nonumber\\
&\quad\quad + f\bigl(\varphi, \slashed{d}x, \mathring{L}^1, \mathring{L}^2\bigr)
\begin{pmatrix}
\slashed{d}(\varphi^{k-1}) \\
\varphi^{2k-2}\slashed{d}\varphi
\end{pmatrix}
\begin{pmatrix}
\mathring{L}\varphi \\
\slashed{d}\varphi
\end{pmatrix}.\nonumber
\end{align}
Note that in view of \eqref{deL} and \eqref{LL}, we have
\begin{equation}\label{YHC-A5}
\mathring{L}\bigl(G_{X\mathring{L}}^\gamma\slashed{g}^{XX}\bigr)\slashed{d}_X\varphi_\gamma
= -\bigl(G_{X\mathring{L}}^\gamma\slashed{d}^X\varphi_\gamma\bigr)\operatorname{tr}\mathscr{X} + \dots,
\end{equation}
where and below the terms represented by the ellipsis in \eqref{YHC-A5} can be absorbed into the remainder
\[
f\bigl(\varphi, \slashed{d}x, \mathring{L}^1, \mathring{L}^2\bigr)
\begin{pmatrix}
\slashed{d}(\varphi^{k-1}) \\
\varphi^{2k-2}\slashed{d}\varphi
\end{pmatrix}
\begin{pmatrix}
\mathring{L}\varphi \\
\slashed{d}\varphi
\end{pmatrix}.
\]
Similarly, one can arrive at
\begin{align*}
&-\frac{1}{2}\mathring{L}\bigl(G_{\mathring{L}\mathring{L}}^0\slashed{d}^X x^i\bigr)\slashed{d}_X\varphi_i
=\frac{1}{2}\bigl(G_{\mathring{L}\mathring{L}}^0\slashed{d}^X x^i \slashed{d}_X\varphi_i\bigr)\operatorname{tr}\mathscr{X} + \dots, \\
&\frac{1}{2}\mathring{L}\bigl(
g_{\gamma j}G_{\mathring{L}\mathring{L}}^\gamma\bigl(\mathring{L}^\alpha + 2g^{0\alpha}\bigr)\slashed{d}^X x^j
\bigr)\slashed{d}_X\varphi_\alpha
=-\frac{1}{2}g_{\gamma j}G_{\mathring{L}\mathring{L}}^\gamma\bigl(\mathring{L}^\alpha + 2g^{0\alpha}\bigr)\bigl(\slashed{d}^X x^j \slashed{d}_X\varphi_\alpha\bigr)\operatorname{tr}\mathscr{X} + \dots, \\[6pt]
&\frac{1}{2}\slashed{d}^X\bigl(G_{\mathring{L}\mathring{L}}^0\bigr)\slashed{d}_X x^i \mathring{L}\varphi_i
= \bigl(G_{X\mathring{L}}^0 \mathring{L}^\alpha \slashed{d}^X\varphi_\alpha\bigr)\operatorname{tr}\mathscr{X} + \dots, \\
&-\frac{1}{2}\slashed{d}^X\bigl(G_{\mathring{L}\mathring{L}}^\gamma \mathring{L}_\gamma \tilde{T}^i\bigr)\slashed{d}_X\varphi_i
=-\bigl( G_{X\mathring{L}}^\gamma \mathring{L}_\gamma \tilde{T}^i \slashed{d}^X\varphi_i+\frac{1}{2}G_{\mathring{L}\mathring{L}}^\gamma g_{\gamma i}\slashed{d}^X x^i \tilde{T}^j \slashed{d}_X\varphi_j\\
&\qquad\qquad\qquad\qquad \qquad\quad  - \frac{1}{2}G_{\mathring{L}\mathring{L}}^\gamma \mathring{L}_\gamma \slashed{d}^X x^j \slashed{d}_X\varphi_j
\bigr)\operatorname{tr}\mathscr{X} + \dots, \\
&-\frac{1}{2}\slashed{d}^X\bigl(G_{\mathring{L}\mathring{L}}^\gamma\bigr)\slashed{d}_X\varphi_\gamma
= -\bigl(G_{X\mathring{L}}^\gamma\slashed{d}^X\varphi_\gamma\bigr)\operatorname{tr}\mathscr{X} + \dots, \\
&-\frac{1}{2}\slashed{d}^X\bigl(
g_{\gamma j}G_{\mathring{L}\mathring{L}}^\gamma\bigl(\mathring{L}^\alpha + 2g^{0\alpha}\bigr)
\bigr)\slashed{d}_X x^j \mathring{L}\varphi_\alpha
= -g_{\gamma j}\slashed{d}^X x^j \bigl(G_{X\mathring{L}}^\gamma\bigl(\mathring{L}^\alpha + 2g^{0\alpha}\bigr)\mathring{L}\varphi_\alpha\\
&\qquad \qquad \qquad\qquad \qquad \quad + \frac{1}{2}G_{\mathring{L}\mathring{L}}^\gamma \mathring{L}^\alpha \slashed{d}_X\varphi_\alpha
\bigr)\operatorname{tr}\mathscr{X} + \dots.
\end{align*}
Therefore, collecting the above identities yields
\begin{align*}
&\mathring{L}\bigl(\operatorname{tr}{\mathscr{X}} - E_0\bigr) \\
&\quad = \frac{1}{\varrho^2}
- \frac{1}{2}\bigl(G_{\mathring{L}\mathring{L}}^\gamma \mathring{L}_\gamma\bigr)\tilde{T}^i\slashed{\triangle}\varphi_i
- \bigl(\operatorname{tr}\check{\mathscr{X}}\bigr)^2
+ \frac{1}{2}\bigl(G_{\mathring{L}\mathring{L}}^0 \mathring{L}\varphi_i\bigr)\slashed{\triangle} x^i
- \frac{1}{2}g_{\gamma j}G_{\mathring{L}\mathring{L}}^\gamma\bigl(\mathring{L}^\alpha + 2g^{0\alpha}\bigr)\mathring{L}\varphi_\alpha \slashed{\triangle} x^j \\
&\quad\quad + \operatorname{tr}\mathscr{X} \left\{
-\frac{2}{\varrho}
- \frac{1}{2}G_{\mathring{L}\mathring{L}}^\gamma \mathring{L}\varphi_\gamma
- G_{\tilde{T}\mathring{L}}^\gamma \mathring{L}\varphi_\gamma
- \slashed{g}^{XX}G_{XX}^\gamma \mathring{L}\varphi_\gamma
+ \frac{1}{2}G_{\mathring{L}\mathring{L}}^0\slashed{d}^X x^i \slashed{d}_X\varphi_i
\right. \\
&\quad\quad\qquad \left.
+ \frac{1}{2}G_{\mathring{L}\mathring{L}}^\gamma g_{\gamma i}\bigl(\slashed{d}^X x^i \slashed{d}_X\varphi_j\bigr)\tilde{T}^j
+ G_{X\mathring{L}}^0 \mathring{L}^\alpha \slashed{d}^X\varphi_\alpha
- G_{X\mathring{L}}^\gamma \mathring{L}_\gamma \tilde{T}^j \slashed{d}^X\varphi_j
\right. \\
&\quad\quad\qquad \left.
+ \frac{1}{2}G_{\mathring{L}\mathring{L}}^\gamma \mathring{L}_\gamma \slashed{d}^X x^j \slashed{d}_X\varphi_j
- g_{\gamma j}\slashed{d}^X x^j G_{X\mathring{L}}^\gamma\bigl(\mathring{L}^\alpha + 2g^{0\alpha}\bigr)\mathring{L}\varphi_\alpha
- G_{X\mathring{L}}^\gamma\slashed{d}^X\varphi_\gamma
\right\} \\
&\quad\quad + f\bigl(\varphi, \slashed{d}x, \mathring{L}^1, \mathring{L}^2\bigr)
\begin{pmatrix}
\slashed{d}(\varphi^{k-1}) \\
\varphi^{2k-2}\slashed{d}\varphi
\end{pmatrix}
\begin{pmatrix}
\mathring{L}\varphi \\
\slashed{d}\varphi
\end{pmatrix},
\end{align*}
which immediately gives that by \eqref{GT},
\begin{equation}\label{Ltr}
\begin{split}
\mathring{L}\bigl(\mu\operatorname{tr}\mathscr{X} - \mu E_0\bigr)
= \biggl(-\frac{2\mu}{\varrho} + \mathcal{E}\biggr)\operatorname{tr}\mathscr{X}
+ \frac{\mu}{\varrho^2}
- \mu\bigl(\operatorname{tr}\check{\mathscr{X}}\bigr)^2
- \frac{1}{2}\bigl(G_{\mathring{L}\mathring{L}}^\gamma \mathring{L}_\gamma\bigr)\tilde{T}^i\bigl(\mu\slashed{\triangle}\varphi_i\bigr)
+ e,
\end{split}
\end{equation}
where
\begin{align}
\mathcal{E} =& \mathring{L}\mu + \mu \left\{
-\frac{1}{2}G_{\mathring{L}\mathring{L}}^\gamma \mathring{L}\varphi_\gamma
- G_{\tilde{T}\mathring{L}}^\gamma \mathring{L}\varphi_\gamma
- \slashed{g}^{XX}G_{XX}^\gamma \mathring{L}\varphi_\gamma
+ \frac{1}{2}G_{\mathring{L}\mathring{L}}^0\slashed{d}^X x^i \slashed{d}_X\varphi_i
\right. \nonumber\\
&\quad\qquad \left.
+ \frac{1}{2}G_{\mathring{L}\mathring{L}}^\gamma g_{\gamma i}\bigl(\slashed{d}^X x^i \slashed{d}_X\varphi_j\bigr)\tilde{T}^j
+ G_{X\mathring{L}}^0 \mathring{L}^\alpha \slashed{d}^X\varphi_\alpha
- G_{X\mathring{L}}^\gamma \mathring{L}_\gamma \tilde{T}^j \slashed{d}^X\varphi_j
\right. \nonumber\\
&\quad\qquad \left.
+ \frac{1}{2}G_{\mathring{L}\mathring{L}}^\gamma \mathring{L}_\gamma \slashed{d}^X x^j \slashed{d}_X\varphi_j
- g_{\gamma j}\slashed{d}^X x^j G_{X\mathring{L}}^\gamma\bigl(\mathring{L}^\alpha + 2g^{0\alpha}\bigr)\mathring{L}\varphi_\alpha
- G_{X\mathring{L}}^\gamma\slashed{d}^X\varphi_\gamma
\right\}, \label{Ee}\\[2mm]
e =& \frac{1}{2}\bigl(G_{\mathring{L}\mathring{L}}^\gamma \mathring{L}_\gamma\bigr)\bigl(\tilde{T}^i T\varphi_i\bigr)E_0
+ \frac{1}{2}\bigl(G_{\mathring{L}\mathring{L}}^0 \mathring{L}\varphi_i\bigr)\mu\slashed{\triangle} x^i - \frac{1}{2}g_{\gamma j}G_{\mathring{L}\mathring{L}}^\gamma\bigl(\mathring{L}^\alpha + 2g^{0\alpha}\bigr)\mathring{L}\varphi_\alpha\bigl(\mu\slashed{\triangle} x^j\bigr) \nonumber\\
&+ \mu f\bigl(\varphi, \slashed{d}x, \mathring{L}^1, \mathring{L}^2\bigr)
\begin{pmatrix}
\slashed{d}(\varphi^{k-1}) \\
\varphi^{2k-2}\slashed{d}\varphi
\end{pmatrix}
\begin{pmatrix}
\mathring{L}\varphi \\
\slashed{d}\varphi\nonumber
\end{pmatrix}.
\end{align}

Let $F^m=\mu\slashed d\bar Z^m\operatorname{tr}{\mathscr X}-\slashed d\bar Z^m (\mu E_0)$ with $\bar Z\in\{T,R\}$. Then by an induction
argument on \eqref{Ltr}, one can get
\begin{equation}\label{LFal}
\begin{split}
\slashed{\mathcal L}_{\mathring L}F^m=&(-\f{2}{\varrho}+\mu^{-1}{\mathcal E})(F^m+\slashed d\bar Z^m (\mu E_0))-\mu\slashed d\bar Z^{m}\bigl((\operatorname{tr}\check{\mathscr X})^2\bigr)\\
&-\f12\slashed d\bar Z^m\bigl((G_{\mathring{L}\mathring{L}}^\gamma \mathring{L}_\gamma)\tilde{T}^i(\mu\slashed{\triangle}\varphi_i)\bigr) +e^m,
\end{split}
\end{equation}
where
\begin{equation}\label{eal}
\begin{split}
e^m=&\slashed{\mathcal L}_{\bar Z}^m e^0+\sum_{p_1+p_2=m-1}{\tiny {\tiny {\tiny {\tiny }}}}\slashed{\mathcal L}_{\bar Z}^{p_1}\slashed{\mathcal L}_{[\mathring L,\bar Z]}F^{p_2}\\
&-\sum_{p_1+p_2=m-1}
\slashed{\mathcal L}_{\bar Z}^{p_1}
\Bigl[
\bigl\{
{}^{(\bar Z)}\slashed\pi_{\mathring L}^X\slashed d_X\mu
+\bar Z\mathring L\mu
-\frac{2\mu}{\varrho^2}\bar Z\varrho
-\bar Z\mathcal E
\bigr\}
\slashed d\bar Z^{p_2}\operatorname{tr}\mathscr X
\Bigr]\\
&-\sum_{p_1+p_2=m-1}{\tiny {\tiny {\tiny {\tiny }}}}\slashed{\mathcal L}_{\bar Z}^{p_1}\Bigl\{(\bar Z\mu)\Bigl(\slashed d\bar Z^{p_2}\bigl(\frac{1}{\varrho^2}\mathring L(\varrho^2\operatorname{tr}\check{\mathscr X})+(\operatorname{tr}\check{\mathscr X})^2\bigr)+\slashed d[\mathring L, \bar Z^{p_2}]\operatorname{tr}\check{\mathscr X}\\
&\qquad\qquad\qquad\qquad\qquad-\underbrace{\sum\limits_{\substack{q_1+q_2=p_2 \\ q_1\ge 1}}\bar Z^{q_1}(\frac{2}{\varrho})\slashed d \bar Z^{q_2}\operatorname{tr}\check{\mathscr X}}_{\text{vanish when $p_2=0$}}\Bigr)\Bigr\}
\end{split}
\end{equation}
and
\begin{equation}\label{e0}
\begin{split}
e^0=&\slashed d e+(\slashed d{\mathcal E}-\slashed d\mathring L\mu)(\operatorname{tr}\check{\mathscr X}+\varrho^{-1})-\slashed d\mu(\operatorname{tr}\check{\mathscr X})^2-\f{\slashed d\mu}{\varrho^2}\mathring L(\varrho^2\operatorname{tr}\check{\mathscr X}).
\end{split}
\end{equation}
Furthermore, utilizing identity \eqref{fequation}, we have
\[
\begin{split}
	\mu\slashed{\triangle}\varphi_i = &\mathring{L}\mathring{\underline{L}}\varphi_i + \frac{1}{2\varrho}\mathring{\underline{L}}\varphi_i - H_i\\
	=&\mathring{L}\mathring{\underline{L}}\varphi_i+\operatorname{tr}\mathscr X T\vp_i+ f\big(\varphi, \mathring{L}^1, \mathring{L}^2, \slashed{d}x, \slashed{g}\big)\vp^{k-1}
\begin{pmatrix}
T\varphi \\
\mu\mathring{L}\varphi \\
\mu\slashed{d}\varphi
\end{pmatrix}
\begin{pmatrix}
\mathring{L}\varphi \\
\slashed{d}\varphi
\end{pmatrix}.
\end{split}
\]
Hence,
\begin{align*}
	&-\f12\slashed d\bar Z^m\bigl((G_{\mathring{L}\mathring{L}}^\gamma \mathring{L}_\gamma)\tilde{T}^i(\mu\slashed{\triangle}\varphi_i)\bigr)\\
	=&-\f12\slashed{\mathcal L}_{\mathring L}\bigl(G_{\mathring L\mathring L}^\gamma\mathring L_\gamma\tilde T^i\slashed d\bar Z^m\mathring{\underline L}\vp_i\bigr)-\f12\sum_{m_1+m_2=m}\slashed d\bar Z^{m_1}(G_{\mathring L\mathring L}^\gamma\mathring L_\gamma\tilde T^i)\bar Z^{m_2}(\mu\slashed\triangle\vp_i)\\
	&-\f12\sum\limits_{\substack{m_1+m_2=m \\ m_1\ge 1}}\bar Z^{m_1}(G_{\mathring L\mathring L}^\gamma\mathring L_\gamma\tilde T^i)\slashed d\bar Z^{m_2}(\mu\slashed\triangle\vp_i)-\f12(G_{\mathring L\mathring L}^\gamma\mathring L_\gamma\tilde T^iT\vp_i)\slashed d\bar Z^m\operatorname{tr}\mathscr X\\
	&-\f12(G_{\mathring L\mathring L}^\gamma\mathring L_\gamma)\tilde T^i\bigl(\slashed d\bar Z^m(\operatorname{tr}\mathscr X T\vp_i)-(\slashed d\bar Z^m\operatorname{tr}\mathscr X)T\vp_i\bigr)\\
	&+\f12\mathring L(G_{\mathring L\mathring L}^\gamma\mathring L_\gamma\tilde T^i)\slashed d\bar Z^m\mathring{\underline L}\vp_i+\f12(G_{\mathring L\mathring L}^\gamma\mathring L_\gamma\tilde T^i)\slashed d[\mathring L,\bar Z^m]\mathring{\underline L}\vp_i\\
	&+(G_{\mathring L\mathring L}^\gamma\mathring L_\gamma)\slashed d\bar Z^m\Bigl(f\big(\varphi, \mathring{L}^1, \mathring{L}^2, \slashed{d}x, \slashed{g}\big)\vp^{k-1}
\begin{pmatrix}
T\varphi \\
\mu\mathring{L}\varphi \\
\mu\slashed{d}\varphi
\end{pmatrix}
\begin{pmatrix}
\mathring{L}\varphi \\
\slashed{d}\varphi
\end{pmatrix}\Bigr)\\
\bar{\mathcal E}=&\mathcal E-\f12\bigl(G_{\mathring{L}\mathring{L}}^\gamma \mathring{L}_\gamma\bigr)\tilde T^iT\vp_i.
	\end{align*}
Then, if we define
\begin{align*}
\bar{e}^m =& e^m -\f12\sum_{m_1+m_2=m}\slashed d\bar Z^{m_1}(G_{\mathring L\mathring L}^\gamma\mathring L_\gamma\tilde T^i)\bar Z^{m_2}(\mu\slashed\triangle\vp_i)\\
	&-\f12\sum\limits_{\substack{m_1+m_2=m \\ m_1\ge 1}}\bar Z^{m_1}(G_{\mathring L\mathring L}^\gamma\mathring L_\gamma\tilde T^i)\slashed d\bar Z^{m_2}(\mu\slashed\triangle\vp_i)\\
	&-\f12(G_{\mathring L\mathring L}^\gamma\mathring L_\gamma)\tilde T^i\bigl(\slashed d\bar Z^m(\operatorname{tr}\mathscr X T\vp_i)-(\slashed d\bar Z^m\operatorname{tr}\mathscr X)T\vp_i\bigr)\\
	&+\f12\mathring L(G_{\mathring L\mathring L}^\gamma\mathring L_\gamma\tilde T^i)\slashed d\bar Z^m\mathring{\underline L}\vp_i+\f12(G_{\mathring L\mathring L}^\gamma\mathring L_\gamma\tilde T^i)\slashed d[\mathring L,\bar Z^m]\mathring{\underline L}\vp_i\\
	&+(G_{\mathring L\mathring L}^\gamma\mathring L_\gamma)\slashed d\bar Z^m\Bigl(f\big(\varphi, \mathring{L}^1, \mathring{L}^2, \slashed{d}x, \slashed{g}\big)\vp^{k-1}
\begin{pmatrix}
T\varphi \\
\mu\mathring{L}\varphi \\
\mu\slashed{d}\varphi
\end{pmatrix}
\begin{pmatrix}
\mathring{L}\varphi \\
\slashed{d}\varphi
\end{pmatrix}\Bigr),\\
\bar{\mathcal E}=&\mathcal E-\f12\bigl(G_{\mathring{L}\mathring{L}}^\gamma \mathring{L}_\gamma\bigr)\tilde T^iT\vp_i
\end{align*}
and
\[
\bar F^m=F^m + \frac{1}{2}\bigl(G_{\mathring{L}\mathring{L}}^\gamma \mathring{L}_\gamma\tilde{T}^i\bigr)\slashed d\bar Z^m \mathring{\underline{L}}\varphi_i,
\]
then it follows from \eqref{LFal} that
\begin{equation}\label{LXE}
\begin{split}
\slashed{\mathcal L}_{\mathring L}\bar F^m=&(-\f{2}{\varrho}+\mu^{-1}{\bar{\mathcal E}})\Bigl(\bar F^m+\slashed d\bar Z^m (\mu E_0)-\f12G_{\mathring{L}\mathring{L}}^\gamma \mathring{L}_\gamma\tilde{T}^i\slashed d\bar Z^m \mathring{\underline{L}}\varphi_i\Bigr)\\
&-\mu\slashed d\bar Z^{m}\bigl((\operatorname{tr}\check{\mathscr X})^2\bigr) +\bar e^m.
\end{split}
\end{equation}

Note that for any one-form $\xi$ on $S_{s,u}$, it holds that
\begin{equation}\label{Lxi}
\mathring{L}\bigl(\varrho^2|\xi|^2\bigr)
= -2\varrho^2\bigl(\operatorname{tr}\check{\mathscr{X}}
+ \frac{1}{2}\slashed{g}^{XX}G_{XX}^\gamma\mathring{L}\vp_\gamma\bigr)|\xi|^2
+ 2\varrho^2\slashed{g}^{XX}\bigl(\slashed{\mathcal{L}}_{\mathring{L}}\xi_X\bigr)\xi_X.
\end{equation}
By taking $\xi = \varrho^2 \bar F^m$ in \eqref{Lxi} and using \eqref{LXE}, one can obtain
\begin{align*}
\mathring{L}\bigl(\varrho^6|\bar F^m|^2\bigr)
=& 2\varrho^6\bigl\{ \bigl(-\operatorname{tr}\check{\mathscr{X}}
-\frac{1}{2}\slashed{g}^{XX}G_{XX}^\gamma\mathring{L}\vp_\gamma
+\mu^{-1}\bar{\mathcal E}\bigr)|\bar F^m|^2 +\bar e^m \cdot \bar F^m \\
&\qquad + \bigl(-\frac{2}{\varrho} + \mu^{-1}\bar{\mathcal{E}}\bigr)\bigl(\slashed{d}\bar{Z}^m (\mu E_0)-\f12G_{\mathring{L}\mathring{L}}^\gamma \mathring{L}_\gamma\tilde{T}^i\slashed d\bar Z^m \mathring{\underline{L}}\varphi_i\bigr) \cdot\bar F^m \\
&\qquad
- \mu\slashed{d}\bar{Z}^m\bigl((\operatorname{tr}\check{\mathscr{X}})^2\bigr) \cdot\bar F^m \bigr\}.
\end{align*}
Due to $L\mu \sim -\delta^{-\varepsilon_k}$ when $\mu < \frac{1}{10}$ (see \eqref{blowup}), one has $\mu^{-1}\bar{\mathcal E}=2\mu^{-1}\mathring{L}\mu+O(\delta^{2k(1-\varepsilon_0)-\varepsilon_0})\lesssim \delta^{-\varepsilon_k}$. This implies
\[
\begin{split}
\mathring{L}\bigl(\varrho^3|\bar F^m|\bigr) \lesssim &\delta^{-\varepsilon_k}|\bar F^m|
+ \mu^{-1}\delta^{-\varepsilon_k}\bigl|\slashed{d}\bar{Z}^m (\mu E_0)-\f12G_{\mathring{L}\mathring{L}}^\gamma \mathring{L}_\gamma\tilde{T}^i\slashed d\bar Z^m \mathring{\underline{L}}\varphi_i\bigr| \\
&
+ \mu\bigl|\slashed{d}\bar{Z}^m\bigl((\operatorname{tr}\check{\mathscr{X}})^2\bigr)\bigr|
+ |\bar e^m|,
\end{split}
\]
and hence
\begin{equation}\label{LrhoF}
\begin{split}
\varrho^3|\bar F^m(s,u,\vartheta)|\lesssim\varrho_0^3|\bar F^m(t_0,u,\vartheta)|+\int_{t_0}^s\big\{&\delta^{-\varepsilon_k}|\bar F^m|+ \mu^{-1}\delta^{-\varepsilon_k}|\slashed{d}\bar{Z}^m (\mu E_0)
-\f12G_{\mathring{L}\mathring{L}}^\gamma \mathring{L}_\gamma\tilde{T}^i\slashed d\bar Z^m \mathring{\underline{L}}\varphi_i| \\
&+ \mu|\slashed{d}\bar{Z}^m\bigl((\operatorname{tr}\check{\mathscr{X}})^2\bigr)|
+ |\bar e^m|\}(\tau,u,\vartheta)d\tau.
\end{split}
\end{equation}
By the initial data estimate
$\delta^l\|\bar F^m\|_{t_0,u}
\lesssim\delta^{1/2-\varepsilon_k}$,
it follows from \eqref{Ff} and \eqref{LrhoF} that
\begin{equation*}
\begin{split}
\delta^l\|\bar F^m\|_{s, u}\lesssim\delta^{\f12-\varepsilon_k}+\delta^l\int_{t_0}^s\big\{&\delta^{-\varepsilon_k}\|\bar F^m\|_{\tau,u}+\mu_{\min}^{-1}(\tau)\delta^{-\varepsilon_k}\|\slashed{d}\bar{Z}^m (\mu E_0)-\f12G_{\mathring{L}\mathring{L}}^\gamma \mathring{L}_\gamma\tilde{T}^i\slashed d\bar Z^m \mathring{\underline{L}}\varphi_i\|_{\tau,u}\\
&+\|\mu\slashed d\bar Z^m\bigl((\operatorname{tr}\check{\mathscr{X}})^2\bigr)\|_{\tau,u}+\|\bar e^m\|_{\tau,u}\big\}d\tau,
\end{split}
\end{equation*}
where $l$ is the number of $T$ in $F^m$. This, together with Gronwall's inequality, yields
\begin{equation}\label{Eal}
\begin{split}
\delta^l\|\bar F^m\|_{s, u}\lesssim\delta^{-\varepsilon_k+1/2}+\delta^l\int_{t_0}^s\big\{&\mu_{\min}^{-1}(\tau)\delta^{-\varepsilon_k}\|\slashed{d}\bar{Z}^m (\mu E_0)-\f12G_{\mathring{L}\mathring{L}}^\gamma \mathring{L}_\gamma\tilde{T}^i\slashed d\bar Z^m \mathring{\underline{L}}\varphi_i\|_{\tau,u}\\
&+\|\mu\slashed d\bar Z^m\bigl((\operatorname{tr}\check{\mathscr{X}})^2\bigr)\|_{\tau,u}+\|\bar e^m\|_{\tau,u}\big\}d\tau.
\end{split}
\end{equation}
Each term in the integrand of \eqref{Eal} will be estimated as follows.

\vskip 0.1 true cm

{\bfseries (1) The estimate of $\|\slashed{d}\bar{Z}^m (\mu E_0)-\f12G_{\mathring{L}\mathring{L}}^\gamma \mathring{L}_\gamma\tilde{T}^i\slashed d\bar Z^m \mathring{\underline{L}}\varphi_i\|_{s, u}$}

\vskip 0.1 true cm

Due to $E_0=-G_{X\mathring L}^\gamma\slashed d^X\varphi_\gamma+\frac{1}{2}G_{\mathring L\mathring L}^0\slashed d_Xx^i\slashed d^X\varphi_i-\frac{1}{2}g_{\gamma j}G_{\mathring L\mathring L}^\gamma\bigl(\mathring L^\alpha+2g^{0\alpha}\bigr)\slashed d_Xx^j \slashed d^X\varphi_\alpha$, it holds that
\begin{align}
&\delta^l\|\slashed{d}\bar{Z}^m (\mu E_0)-\f12G_{\mathring{L}\mathring{L}}^\gamma \mathring{L}_\gamma\tilde{T}^i\slashed d\bar Z^m \mathring{\underline{L}}\varphi_i\|_{s, u}\nonumber\\
\lesssim&\delta^{k(1-\varepsilon_0)}\sum_{i=1}^2\big\{\delta^{l_0}\|\bar Z^{m_0}\check L^i\|_{s, u}+\delta^{l_{-1}}\|\bar Z^{m_{-1}}x^i\|_{s,u}\big\}+\delta^{(k-1)(1-\varepsilon_0)}\delta^{l_0}\|\bar Z^{m_0}\vp\|_{s,u}\nonumber\\
&+\delta^{k(1-\varepsilon_0)}\bigl\{\delta^{l_1}\|\bar Z^{m_1}\mathring L\vp\|_{s, u}+\delta^l\|\mu\mathring L\bar Z^{m+1}\vp\|_{s,u}+\delta^{l_0}\|T\bar Z^{m_0}\varphi\|_{s, u}\bigr\}\nonumber\\
&+\delta^{k(1-\varepsilon_0)}(1+\delta^{2k(1-\varepsilon_0)-\varepsilon_0})\delta^{l_0}\|\slashed{\mathcal L}_{\bar Z}^{m_0}\slashed g\|_{s, u}
+\delta^{(k-1)(1-\varepsilon_0)}\delta^{l}\|\mu\slashed d\bar Z^{m+1}\varphi\|_{s, u}\nonumber\\
&+\delta^{k(1-\varepsilon_0)}(1+\delta^{(k+1)(1-\varepsilon_0)-\varepsilon_0})\delta^{l_0}\|\bar Z^{m_0}\mu\|_{s, u}\nonumber\\
&+\de^{(k+1)(1-\varepsilon_0)}\de^{l_1}\big\{\|\slashed{\mathcal L}_{\bar Z}^{m_1}\leftidx{^{(R)}}{\slashed\pi}_T\|_{s, u}
+\|\slashed{\mathcal L}_{\bar Z}^{m_1}\leftidx{^{(R)}}{\slashed\pi}_{\mathring L}\|_{s, u}+\delta\|\slashed{\mathcal L}_{\bar Z}^{m_1}\leftidx{^{(T)}}{\slashed\pi}_{\mathring L}\|_{s, u}\big\}\nonumber\\
\lesssim&\delta^{k(1-\varepsilon_0)+1/2}(1+\delta^{2k(1-\varepsilon_0)-\varepsilon_0})+\delta^{(k-1)(1-\varepsilon_0)}\sqrt{\mathcal E_{1,\leq m+2}(s,u)}\label{dZE}\\
			&+\delta^{k(1-\varepsilon_0)} \sqrt{\mathcal E_{2,\leq m+2}(s,u)}+\delta^{(2k-1)(1-\varepsilon_0)}(1+\delta^{(k+1)(1-\varepsilon_0)-\varepsilon_0})\Xi_{m+2}^2(s,u).\nonumber
\end{align}
This, together with \eqref{eq:mu_integrability}, yields
\begin{equation}\label{sdE}
\begin{split}
&\int_{t_0}^s\mu_{\min}^{-1}(\tau)\delta^{-\varepsilon_k}\delta^l\|\slashed{d}\bar{Z}^m (\mu E_0)-\f12G_{\mathring{L}\mathring{L}}^\gamma \mathring{L}_\gamma\tilde{T}^i\slashed d\bar Z^m \mathring{\underline{L}}\varphi_i\|_{\tau,u}d\tau\\
\lesssim&\delta^{2k(1-\varepsilon_0)-\varepsilon_0 +1/2}(1+\delta^{2k(1-\varepsilon_0)-\varepsilon_0})\int_{t_0}^s\mu_{\min}^{-1}(\tau)d\tau\\
&+\f{\mu_{\min}^{-b_{m+2}}(s)}{b_{m+2}}\Big\{\delta^{(k-1)(1-\varepsilon_0)}\sqrt{\widetilde{\mathcal{E}}_{1,\leq m+2}(s, u)}+\delta^{k(1-\varepsilon_0)}\sqrt{\widetilde{\mathcal{E}}_{2,\leq m+2}(s, u)}\Big\}.
\end{split}
\end{equation}

\vskip 0.1 true cm

{\bfseries (2) The estimate of $\|\mu\slashed d\bar Z^m\bigl((\operatorname{tr}\check{\mathscr{X}})^2\bigr)\|_{s, u}$}

\vskip 0.1 true cm

The identity
\[
\begin{split}
\mu\slashed d\bar Z^m((\operatorname{tr}\check{\mathscr{X}})^2)
=&2\operatorname{tr}\check{\mathscr{X}}(\bar F^m+\slashed{d}\bar{Z}^m (\mu E_0)-\f12G_{\mathring{L}\mathring{L}}^\gamma \mathring{L}_\gamma\tilde{T}^i\slashed d\bar Z^m \mathring{\underline{L}}\varphi_i)\\
&+2\sum_{p\leq m-1}\mu\bar Z^{m-p}(\operatorname{tr}\check{\mathscr{X}})(\slashed d\bar Z^p\operatorname{tr}\check{\mathscr{X}}),
\end{split}
\]
together with Proposition \ref{L2chi} and \eqref{dZE}, yields
\begin{equation*}
\begin{split}
&\delta^l\|\mu\slashed d\bar Z^m((\operatorname{tr}\check{\mathscr{X}})^2)\|_{s, u}\\
\lesssim&\delta^{k(1-\varepsilon_0)}\big\{\delta^l\|\bar F^m\|_{s, u}+\delta^l\|\slashed{d}\bar{Z}^m (\mu E_0)-\f12G_{\mathring{L}\mathring{L}}^\gamma \mathring{L}_\gamma\tilde{T}^i\slashed d\bar Z^m \mathring{\underline{L}}\varphi_i\|_{s, u}+\delta^{l_1}\|\bar Z^{m_1}\operatorname{tr}\check{\mathscr{X}}\|_{s, u}\big\}\\
\lesssim&\delta^{k(1-\varepsilon_0)}\delta^l\|\bar F^m\|_{s, u}+\delta^{2k(1-\varepsilon_0)+1/2}(1+\delta^{2k(1-\varepsilon_0)-\varepsilon_0}) +\delta^{(2k-1)(1-\varepsilon_0)}\sqrt{\mathcal E_{1,\leq m+2}(s, u)}\\
&+\delta^{2k(1-\varepsilon_0)}\sqrt{\mathcal E_{2,\leq m+2}(s, u)}+\delta^{2k(1-\varepsilon_0)}(1+\delta^{2k(1-\varepsilon_0)-\varepsilon_0})\int_{t_0}^s\mu_{\min}^{-\f12}(\tau)\sqrt{\mathcal E_{1,\leq m+2}(\tau,u)}d\tau\\
&+\delta^{(3k-1)(1-\varepsilon_0)}(1+\delta^{(k+1)(1-\varepsilon_0)-\varepsilon_0})\int_{t_0}^s\sqrt{\mathcal E_{2,\leq m+2}(\tau,u)}d\tau,
\end{split}
\end{equation*}
and hence,
\begin{equation}\label{Y-19}
\begin{split}
&\int_{t_0}^s\delta^l\|\mu\slashed d\bar Z^m((\operatorname{tr}\check{\mathscr{X}})^2)\|_{\tau,u}d\tau\\
\lesssim&\int_{t_0}^s\delta^{k(1-\varepsilon_0)}\delta^l\|\bar F^m\|_{\tau,u}d\tau+\delta^{k(1-\varepsilon_0)+\varepsilon_0+1/2}(1+\delta^{2k(1-\varepsilon_0)-\varepsilon_0})\\
&+\delta^{(k-1)(1-\varepsilon_0)+\varepsilon_0}\mu_{\min}^{-b_{m+2}}\Big\{\sqrt{\widetilde{\mathcal{E}}_{1,\leq m+2}(s, u)}+\delta^{1-\varepsilon_0}\sqrt{\widetilde{\mathcal{E}}_{2,\leq m+2}(s, u)}\Big\}.
\end{split}
\end{equation}

\vskip 0.1 true cm

{\bfseries (3) The estimate of $\|\bar e^m\|_{s, u}$}

\vskip 0.1 true cm

It follows from \eqref{Ee} that $\slashed d\bar Z^m e$ contains the terms
\[
\tfrac12\mu \bigl(G_{\mathring L\mathring L}^0\mathring L\vp_i\bigr)\slashed d\bar Z^m\slashed\triangle x^i,\quad -\tfrac12\mu g_{\gamma j}G_{\mathring L\mathring L}^\gamma\bigl(\mathring L^\alpha+2g^{0\alpha}\bigr)\mathring L\vp_\alpha\slashed d\bar Z^m\slashed\triangle x^j.
\]
Note that the $L^2$ norms of the above two terms cannot be directly bounded by virtue of Proposition \ref{L2chi} since the
related derivative orders with respect to $x^i$ are $m+3$. The resulted derivative loss prevents the closure of the highest order energy estimates.
To overcome this difficulty, we reformulate the terms $R\slashed \triangle x^i$ and $T\slashed\triangle x^i$ by Lemma \ref{lem3.4},
which are presented as
\begin{align*}
	R\slashed\triangle x^i=&[R,\slashed\triangle]x^i+\slashed\triangle R^i\\
	=&-\operatorname{tr}\leftidx{^{(R)}}{\slashed\pi}(\slashed\triangle x^i)-\f12(\slashed\nabla_X\operatorname{tr}\leftidx{^{(R)}}{\slashed\pi})\slashed d^Xx^i+\slashed\triangle(\Omega^i-\upsilon\tilde T^i)\\
	=&-\operatorname{tr}\leftidx{^{(R)}}{\slashed\pi}(\slashed\triangle x^i)-(\slashed d^Xx^i\slashed d_X\upsilon)(\operatorname{tr}\mathscr X+G_{X\mathring L}^\gamma\slashed d^X\vp_\gamma+G_{X\tilde{T}}^\gamma\slashed d^X\vp_\gamma)\\
	&-\upsilon\slashed d^Xx^i\bigl(\slashed d_X\operatorname{tr}\mathscr X+\slashed\nabla_X(G_{X\mathring L}^\gamma\slashed d^X\vp_\gamma+G_{X\tilde{T}}^\gamma\slashed d^X\vp_\gamma)\bigr)\\
	&-\f12\slashed d^Xx^i\slashed\nabla_X(g^{XX}G_{XX}^\gamma R\vp_\gamma+2\epsilon_b^{\,j}\check g_{ja}
(\slashed d_Xx^a)(\slashed d_Xx^b))\\
	&+\slashed\triangle\Omega^i-\slashed\nabla^X(\tilde T^i\boxed{\slashed d_X\upsilon}+\upsilon\boxed{\slashed d_X\tilde T^i})
\end{align*}
and
\begin{align*}
	T\slashed\triangle x^i=&-\operatorname{tr}\leftidx{^{(T)}}{\slashed\pi}(\slashed\triangle x^i)-\f12(\slashed\nabla_X\operatorname{tr}\leftidx{^{(T)}}{\slashed\pi})\slashed d^Xx^i+\slashed\triangle(\mu\tilde T^i)\\
	=&-\operatorname{tr}\leftidx{^{(T)}}{\slashed\pi}(\slashed\triangle x^i)+(\slashed d^Xx^i)\slashed d_X\mu(\operatorname{tr}\mathscr X+G_{X\mathring L}^\gamma\slashed d^X\vp_\gamma+G_{X\tilde{T}}^\gamma\slashed d^X\vp_\gamma)\\
	&+\mu(\slashed d^Xx^i)\slashed d_X(\operatorname{tr}\mathscr X+G_{X\mathring L}^\gamma\slashed d^X\vp_\gamma+G_{X\tilde{T}}^\gamma\slashed d^X\vp_\gamma)+\tilde T^i\slashed\triangle\mu\\
	&-\f12\slashed d^Xx^i\slashed\nabla_X(g^{XX}G_{XX}^\gamma T\vp_\gamma)+2(\slashed d_X\mu)\slashed d^X\tilde T^i-\mu\slashed\triangle g^{0i}-\mu\slashed\nabla^X\boxed{\slashed d_X\mathring L^i}.
\end{align*}
Based on \eqref{omega} and \eqref{deL}, the boxed terms above can be expressed in terms of $\operatorname{tr}\mathscr X$ and
other good auxiliary quantities. Therefore,
\begin{align*}
	&\delta^l\|\slashed d\bar Z^m\bigl(\mu(G_{\mathring L\mathring L}^0\mathring L\vp_i)\slashed\triangle x^i\bigr)\|_{s,u}+\delta^l\|\slashed d\bar Z^m\bigl(\mu g_{\gamma j}G_{\mathring L\mathring L}^\gamma(\mathring L^\alpha+2g^{0\alpha})\mathring L\vp_\alpha\slashed\triangle x^j\bigr)\|_{s,u}\\
	\lesssim&\delta^{2k(1-\varepsilon_0)-\varepsilon_0}\bigl\{\delta^{l_0}\|\bar Z^{m_0}\mu\|_{s,u}+\delta^{l_{-1}}\|\bar Z^{l_{-1}}x\|_{s,u}+\delta^{l_1}\|\slashed{\mathcal L}_{\bar Z}^{m_1}\leftidx{^{(R)}}{\slashed\pi}\|_{s,u}+\delta^{l_1}\|\bar Z^{m_1}\operatorname{tr}\check{\mathscr X}\|_{s,u}\\
	&\quad+\delta^{l_0}\|\bar Z^{m_0}\upsilon\|_{s,u}+\sum_{i=1}^2\delta^{l_0}\|\bar Z^{m_0}\check L^i\|_{s,u}+\delta^{l}\|\mu\slashed d\bar Z^m\operatorname{tr}\check{\mathscr X}\|_{s,u}\bigr\}\\
	&+\delta^{(2k+1)(1-\varepsilon_0)}\bigl(\delta^{l_1}\|\slashed{\mathcal L}_{\bar Z}^{m_1}\leftidx{^{(T)}}{\slashed\pi}\|_{s,u}+\delta^l\|\mu\bar Z^m\slashed\triangle\mu\|_{s,u}\bigr)+\delta^{2k(1-\varepsilon_0)}\sum_{i=1}^2\sqrt{\mathcal E_{i,\leq m+2}(s,u)}\\
	&+\delta^{(k-1)(1-\varepsilon_0)}\delta^l\|\mathring L\bar Z^{m+1}\vp\|_{s,u}+\delta^{k(1-\varepsilon_0)}\delta^{l_1}\bigl(\|\slashed{\mathcal L}_{\bar Z}^{m_1}\leftidx{^{(R)}}{\slashed\pi}_{\mathring L}\|_{s, u}+\delta\|\slashed{\mathcal L}_{\bar Z}^{m_1}\leftidx{^{(T)}}{\slashed\pi}_{\mathring L}\|_{s, u}\bigr)\\
	&+\delta^{3k(1-\varepsilon_0)-1}\delta^l\|\mu\slashed d\bar Z^{m+1}\vp\|_{s,u}+\delta^{3k(1-\varepsilon_0)}\delta^l\|T\bar Z^{m+1}\vp\|_{s,u}+\delta^{(3k+1)(1-\varepsilon_0)}\delta^{l_1}\|\slashed{\mathcal L}_{\bar Z}^{m_1}\leftidx{^{(R)}}{\slashed\pi}_{T}\|_{s, u}\\
	\lesssim&\delta^{2k(1-\varepsilon_0)-\varepsilon_0+1/2}+\delta^{3k(1-\varepsilon_0)-1}\sqrt{\mathcal E_{1,\leq m+2}(s,u)}+\delta^{2k(1-\varepsilon_0)}\sqrt{\mathcal E_{2,\leq m+2}(s,u)}\\
	&+\delta^{2k(1-\varepsilon_0)-\varepsilon_0}\delta^{l}\|\mu\slashed d\bar Z^m\operatorname{tr}\check{\mathscr X}\|_{s,u}+\delta^{(2k+1)(1-\varepsilon_0)}\delta^l\|\mu\bar Z^m\slashed\triangle\mu\|_{s,u}+\delta^{(k-1)(1-\varepsilon_0)}\delta^l\|\mathring L\bar Z^{m+1}\vp\|_{s,u}\\
	&+\delta^{3k(1-\varepsilon_0)-1}\Xi^2_{m+2}(s,u).
\end{align*}
This, together with \eqref{ZGL}, yields that
\begin{align}
&\delta^l\|\slashed{d}\bar{Z}^m e\|_{s, u} \nonumber\\
\lesssim&\delta^l\|\slashed d\bar Z^m\bigl(\mu(G_{\mathring L\mathring L}^0\mathring L\vp_i)\slashed\triangle x^i\bigr)\|_{s,u}+\delta^l\|\slashed d\bar Z^m\bigl(\mu g_{\gamma j}G_{\mathring L\mathring L}^\gamma(\mathring L^\alpha+2g^{0\alpha})\mathring L\vp_\alpha\slashed\triangle x^j\bigr)\|_{s,u}\nonumber\\
&+\delta^{-\varepsilon_k}\delta^{l_0}\|\bar Z^{m_0}(G_{\mathring L\mathring L}^\gamma\mathring L_\gamma)\|_{s,u}+\delta^{2k(1-\varepsilon_0)}\delta^{l}\|T\bar Z^{m+1}\vp\|_{s,u}+\delta^{(2k+1)(1-\varepsilon_0)}\delta^{l_1}\|\slashed{\mathcal L}_{\bar Z}^{m_1}\leftidx{^{(R)}}{\slashed\pi}_{T}\|_{s, u}\nonumber\\
&+\delta^{2k(1-\varepsilon_0)-1}\delta^{l}\|\slashed d\bar Z^{m+1}\vp\|_{s,u}+\delta^{2k(1-\varepsilon_0)}\sum_{i=1}^2\sqrt{\mathcal E_{i,\leq m+2}(s,u)}+\delta^{(k-1)(1-\varepsilon_0)}\delta^{l}\|\slashed d\bar Z^m\mathring L\vp\|_{s,u}\nonumber\\
&+\delta^{2k(1-\varepsilon_0)-\varepsilon_0}\delta^{l_0}\bigl(\|\bar Z^{m_0}\mu\|_{s,u}+\|\slashed{\mathcal L}_{\bar Z}^{m_0}\slashed g\|_{s,u}+\sum_{i=1}^2\|\bar Z^{m_0}\check L^i\|_{s,u}\bigr)\nonumber\\
&+\sum_{i=1}^2\delta^{2k(1-\varepsilon_0)-\varepsilon_0}\delta^{l_{-1}}\|\bar Z^{m_{-1}}x^i\|_{s,u}\nonumber\\
	\lesssim&\delta^{2k(1-\varepsilon_0)-\varepsilon_0+1/2}+\delta^{2k(1-\varepsilon_0)-1}\mu_{\min}^{-1/2}(s)\sqrt{\mathcal E_{1,\leq m+2}(s,u)}+\delta^{2k(1-\varepsilon_0)}\sqrt{\mathcal E_{2,\leq m+2}(s,u)}\nonumber\\
	&+\delta^{2k(1-\varepsilon_0)-\varepsilon_0}\delta^{l}\|\mu\slashed d\bar Z^m\operatorname{tr}\check{\mathscr X}\|_{s,u}+\delta^{2k(1-\varepsilon_0)}\delta^l\|\bar Z^m\slashed\triangle\mu\|_{s,u}+\delta^{3k(1-\varepsilon_0)-1}\Xi^2_{m+2}(s,u)\nonumber\\
	\lesssim&\delta^{2k(1-\varepsilon_0)-\varepsilon_0}\delta^l\|\bar F^m\|_{s,u}+\delta^{2k(1-\varepsilon_0)-\varepsilon_0+1/2}+\delta^{5k(1-\varepsilon_0)-2\varepsilon_0+1/2}\nonumber\\
	&+\delta^{2k(1-\varepsilon_0)-1}\mu_{\min}^{-1/2}(s)\sqrt{\mathcal E_{1,\leq m+2}(s,u)}+\delta^{3k(1-\varepsilon_0)-\varepsilon_0}\sqrt{\mathcal E_{2,\leq m+2}(s,u)}\label{de}\\
	&+\delta^{2k(1-\varepsilon_0)}\delta^l\|\bar Z^m\slashed\triangle\mu\|_{s,u}+(\delta^{3k(1-\varepsilon_0)-1}+\delta^{5k(1-\varepsilon_0)-2\varepsilon_0})\Xi^2_{m+2}(s,u).\nonumber
\end{align}

In addition, \eqref{lmu}, \eqref{GT}, \eqref{Lchi'} and \eqref{Ee} give
\begin{align}
&\delta^l\Bigl\|\slashed{\mathcal{L}}_{\bar{Z}}^m\Bigl\{(\slashed{d}\mathcal{E}-\slashed{d}\mathring{L}\mu)
(\operatorname{tr}\check{\mathscr{X}}+\varrho^{-1})
-\slashed{d}\mu(\operatorname{tr}\check{\mathscr{X}})^2
-\f{\slashed{d}\mu}{\varrho^2}\mathring{L}(\varrho^2\operatorname{tr}\check{\mathscr{X}})\Bigr\}\Bigr\|_{s, u} \nonumber\\
\lesssim&\, \delta^{(k-1)(1-\varepsilon_0)}\delta^l\bigl(\|\mathring{L}R\bar{Z}^m\vp\|_{s, u}
+\|\slashed{d}\bar{Z}^{m+1}\vp\|_{s, u}\bigr)
+\delta^{2k(1-\varepsilon_0)-\varepsilon_0}\delta^{l_0}\|\bar{Z}^{m_0}\mu\|_{s, u} \nonumber\\
&+\delta^{k(1-\varepsilon_0)}\delta^{l_1}\bigl(\|\slashed{\mathcal{L}}_{\bar{Z}}^{m_1}\leftidx{^{(R)}}{\slashed{\pi}}_{\mathring{L}}\|_{s, u}
+\delta\|\slashed{\mathcal{L}}_{\bar{Z}}^{m_1}\leftidx{^{(T)}}{\slashed{\pi}}_{\mathring{L}}\|_{s, u}
+\|\slashed{\mathcal{L}}_{\bar{Z}}^{m_1}\leftidx{^{(R)}}{\slashed{\pi}}\|_{s, u}+\delta\|\slashed{\mathcal{L}}_{\bar{Z}}^{m_1}\leftidx{^{(T)}}{\slashed{\pi}}\|_{s, u} \bigr) \nonumber\\
&+\delta^{2k(1-\varepsilon_0)-\varepsilon_0}\sum_{i=1}^2\Bigl(\delta^{l_0}\|\bar{Z}^{m_0}\check{L}^i\|_{s, u} +\delta^{l_{-1}}\|\bar Z^{m_{-1}}x^i\|_{s,u}+\delta^{l_0}\|\slashed{\mathcal{L}}_{\bar{Z}}^{m_0}\slashed{g}\|_{s, u}
\Bigr) \nonumber\\
&+\delta^{2k(1-\varepsilon_0)-\varepsilon_0}\delta^{l_1}\|\bar{Z}^{m_1}\operatorname{tr}\check{\mathscr{X}}\|_{s, u} +\delta^{2k(1-\varepsilon_0)}\sum_{i=1}^2\sqrt{\mathcal E_{i,\leq m+2}(s,u)}\nonumber\\
\lesssim&\delta^{2k(1-\varepsilon_0)-\varepsilon_0+1/2}+\delta^{(k-1)(1-\varepsilon_0)}\mu_{\min}^{-1/2}(s)\sqrt{\mathcal E_{1,\leq m+2}(s,u)}\label{add-eq}\\
&+\delta^{2k(1-\varepsilon_0)}\sqrt{\mathcal E_{2,\leq m+2}(s,u)}+\delta^{3k(1-\varepsilon_0)-1}\Xi^2_{m+2}(s,u).\nonumber
\end{align}

Due to $\slashed{\mathcal L}_{[\mathring L,\bar Z]}F^{p_2}=\leftidx^{{(\bar Z)}}{\slashed\pi_{\mathring L}}^X\slashed\nabla_XF^{p_2}+F^{p_2}_X\slashed\nabla\leftidx^{{(\bar Z)}}{\slashed\pi_{\mathring L}}^X$,
it holds that
\begin{align}
&\sum_{p_1+p_2=m-1}\delta^l\|\slashed{\mathcal L}_{\bar Z}^{p_1}\slashed{\mathcal L}_{[\mathring L, \bar Z]}F^{p_2}\|_{s,u}\nonumber\\
\lesssim&\delta^{k(1-\varepsilon_0)}(1+\delta^{2k(1-\varepsilon_0)-\varepsilon_0})\bigl(\delta^l\|\bar F^m\|_{s, u}+\delta^{l_1}\|\bar{Z}^{m_1}\operatorname{tr}\check{\mathscr{X}}\|_{s, u}+\delta^{k(1-\varepsilon_0)}\delta^l\|T\bar Z^{m+1}\vp\|_{s,u}\nonumber\\
 &\quad+\delta^{k(1-\varepsilon_0)}\delta^l\|\mu\mathring L\bar Z^{m+1}\vp\|_{s,u}+(\delta^{k(1-\varepsilon_0)} +\delta^{(2k+1)(1-\varepsilon_0)-\varepsilon_0})\delta^{l_1}\|\bar Z^{m_1}\mu\|_{s, u}\nonumber\\
&\quad+\delta^{k(1-\varepsilon_0)}\delta^{l_1}\|\slashed{\mathcal{L}}_{\bar{Z}}^{m_1}\slashed{g}\|_{s, u}+\delta^{k(1-\varepsilon_0)}\sum_{i=1}^2\bigl(\delta^{l_1}\|\bar{Z}^{m_1}\check{L}^i\|_{s, u}+\delta^{l_0}\|\bar{Z}^{m_0}{x}^i\|_{s, u}\bigr)\bigr)\nonumber\\
&+\delta^{k(1-\varepsilon_0)}(1+\delta^{(3k+1)(1-\varepsilon_0)-\varepsilon_0})\delta^{l_1}
\bigl(\|\slashed{\mathcal{L}}_{\bar{Z}}^{m_1}\leftidx{^{(R)}}{\slashed{\pi}}_{\mathring{L}}\|_{s, u}
+\delta\|\slashed{\mathcal{L}}_{\bar{Z}}^{m_1}\leftidx{^{(T)}}{\slashed{\pi}}_{\mathring{L}}\|_{s, u} \bigr)\nonumber\\
&+\delta^{(2k+1)(1-\varepsilon_0)}(1+\delta^{2k(1-\varepsilon_0)-\varepsilon_0})
\delta^{l_1}\|\slashed{\mathcal{L}}_{\bar{Z}}^{m_1}\leftidx{^{(R)}}{\slashed{\pi}}_{T}\|_{s, u} \nonumber\\
&+\delta^{2k(1-\varepsilon_0)}(1+\delta^{2k(1-\varepsilon_0)-\varepsilon_0})\sum_{i=1}^2\sqrt{\mathcal E_{i,\leq m+2}(s,u)}\nonumber\\
\lesssim&\delta^{k(1-\varepsilon_0)}(1+\delta^{2k(1-\varepsilon_0)-\varepsilon_0})\big\{\delta^l\|\bar F^m\|_{s, u}+\delta^{k(1-\varepsilon_0)+1/2}(1+\delta^{2k(1-\varepsilon_0)-\varepsilon_0})\label{LF}\\
&\quad+\delta^{k(1-\varepsilon_0)}\sum_{i=1}^2 \sqrt{\mathcal E_{i,\leq m+2}(s, u)}+\delta^{k(1-\varepsilon_0)}(1+\delta^{2k(1-\varepsilon_0)-\varepsilon_0})\Xi^2_{m+2}(s,u)\big\}.\nonumber
\end{align}

For the remainder terms in $e^m$, it follows from Propositions \ref{prop:higher-order-L-infty-estimates}, \ref{L2chi} and \eqref{Lchi'} that
\begin{align}
&\sum_{p_1+p_2=m-1}\|\slashed{\mathcal L}_{\bar Z}^{p_1}
\bigl[
\bigl\{
{}^{(\bar Z)}\slashed\pi_{\mathring L}^X\slashed d_X\mu
+\bar Z\mathring L\mu
-\frac{2\mu}{\varrho^2}\bar Z\varrho
-\bar Z\mathcal E
\bigr\}
\slashed d\bar Z^{p_2}\operatorname{tr}\mathscr X
\bigr]\|_{s, u}\nonumber\\
&+\sum_{p_1+p_2=m-1}\|\slashed{\mathcal L}_{\bar Z}^{p_1}\bigl\{(\bar Z\mu)\bigl(\slashed d\bar Z^{p_2}\bigl(\frac{1}{\varrho^2}\mathring L(\varrho^2\operatorname{tr}\check{\mathscr X})+(\operatorname{tr}\check{\mathscr X})^2\bigr)+\slashed d[\mathring L, \bar Z^{p_2}]\operatorname{tr}\check{\mathscr X}\nonumber\\
&\qquad\qquad\qquad\qquad\qquad-\sum\limits_{\substack{q_1+q_2=p_2 \\ q_1\ge 1}}\bar Z^{q_1}(\frac{2}{\varrho})\slashed d \bar Z^{q_2}\operatorname{tr}\check{\mathscr X}\bigr)\bigr\}\|_{s, u}\nonumber\\
\lesssim &\delta^{k(1-\varepsilon_0)}\delta^{l_1}\bigl(\|\slashed{\mathcal{L}}_{\bar{Z}}^{m_1}\leftidx{^{(R)}}{\slashed{\pi}}_{\mathring{L}}\|_{s, u}
+\delta\|\slashed{\mathcal{L}}_{\bar{Z}}^{m_1}\leftidx{^{(T)}}{\slashed{\pi}}_{\mathring{L}}\|_{s, u} \bigr)+\delta^{2k(1-\varepsilon_0)}\sum_{i=1}^2 \sqrt{\mathcal E_{i,\leq m+2}(s, u)}\nonumber\\
&+\delta^{2k(1-\varepsilon_0)-\varepsilon_0}\delta^{l_1}\bigl(\|\bar Z^{m_1}\mu\|_{s, u}+\|\bar{Z}^{m_1}\operatorname{tr}\check{\mathscr{X}}\|_{s, u}+\|\slashed{\mathcal{L}}_{\bar{Z}}^{m_1}\slashed{g}\|_{s, u}+\sum_{i=1}^2\|\bar{Z}^{m_1}\check{L}^i\|_{s, u} \bigr)\nonumber\\
&+\delta^{2k(1-\varepsilon_0)-\varepsilon_0}\sum_{i=1}^2\delta^{l_0}\|\bar{Z}^{m_0}x^i\|_{s, u}+\delta^{(k-1)(1-\varepsilon_0)}\delta^{l}\big(\|\mathring LR\bar Z^m\vp\|_{s, u}+\|\bar Z^m\slashed\triangle\vp\|_{s, u}\big)\nonumber\\
\lesssim&\delta^{2k(1-\varepsilon_0)-\varepsilon_0+1/2}
+\delta^{(k-1)(1-\varepsilon_0)}\mu_{\min}^{-1/2}\sqrt{\mathcal{E}_{1,\leq m+2}(s, u)}\label{remainder}
\\
&+\delta^{2k(1-\varepsilon_0)}\sqrt{\mathcal{E}_{2,\leq m+2}(s, u)}
+\delta^{3k(1-\varepsilon_0)-1}\Xi^2_{m+2}(s,u).\nonumber
\end{align}
Collecting \eqref{de}--\eqref{remainder} and using \eqref{eal}, we obtain
\begin{align*}
\delta^l \|e^m\|_{s, u}
\lesssim& \delta^{2k(1-\varepsilon_0)-\varepsilon_0}\delta^l\|\bar F^m\|_{s,u}+\delta^{2k(1-\varepsilon_0)-\varepsilon_0+1/2}+\delta^{5k(1-\varepsilon_0)-2\varepsilon_0+1/2}\\
	&+\delta^{2k(1-\varepsilon_0)-1}\mu_{\min}^{-1/2}(s)\sqrt{\mathcal E_{1,\leq m+2}(s,u)}+\delta^{2k(1-\varepsilon_0)}(1+\delta^{2k(1-\varepsilon_0)-\varepsilon_0})\sqrt{\mathcal E_{2,\leq m+2}(s,u)}\\
	&+\delta^{2k(1-\varepsilon_0)}\delta^l\|\bar Z^m\slashed\triangle\mu\|_{s,u}+(\delta^{3k(1-\varepsilon_0)-1}+\delta^{5k(1-\varepsilon_0)-2\varepsilon_0})\Xi^2_{m+2}(s,u),
\end{align*}
and then,
\begin{align*}
	\delta^l\|\bar e^m\|_{s,u}\lesssim & \delta^l\|e^m\|_{s,u}+\delta^{1-\varepsilon_0}\delta^{l_0}\|\bar Z^{m_0}(G_{\mathring L\mathring L}^\gamma\mathring L_\gamma)\|_{s,u}+\delta^{k(1-\varepsilon_0)}\delta^l\|\bar Z^m\slashed\triangle\vp\|_{s,u}\\
	&+\delta^{-\varepsilon_k}\delta^{l_1}\|\bar Z^{m_1}\operatorname{tr}\check{\mathscr X}\|_{s,u}+\delta^{2k(1-\varepsilon_0)-\varepsilon_0}\delta^l\bigl(\|T\bar Z^{m+1}\vp\|_{s,u}+\|\mu\mathring L\bar Z^{m+1}\vp\|_{s,u}\bigr)\\
	&+\delta^{-\varepsilon_k}\delta^{l_1}\bigl(\|\slashed{\mathcal{L}}_{\bar{Z}}^{m_1}\leftidx{^{(R)}}{\slashed{\pi}}_{\mathring{L}}\|_{s, u}
+\delta\|\slashed{\mathcal{L}}_{\bar{Z}}^{m_1}\leftidx{^{(T)}}{\slashed{\pi}}_{\mathring{L}}\|_{s, u}\bigr)+ \delta^{(2k+1)(1-\varepsilon_0)-\varepsilon_0}\delta^{l_1}\|\slashed{\mathcal{L}}_{\bar{Z}}^{m_1}\leftidx{^{(R)}}{\slashed{\pi}}_{\mathring{T}}\|_{s, u} \\
&+\delta^{(3k+1)(1-\varepsilon_0)-2\varepsilon_0}\delta^{l_0}\|\bar Z^{m_0}\mu\|_{s,u}+\delta^{2k(1-\varepsilon_0)-\varepsilon_0}\sum_{i=1}\sqrt{\mathcal E_{i,\leq m+2}(s, u)}\\
&+\delta^{3k(1-\varepsilon_0)-2\varepsilon_0}\sum_{i=1}^2\bigl(\delta^{l_0}\|\bar Z^{m_0}\check L^i\|_{s,u}+\delta^{l_{-1}}\|\bar Z^{m_{-1}}x^i\|_{s,u}\bigr)\\
\lesssim&\delta^{2k(1-\varepsilon_0)-\varepsilon_0}\delta^l\|\bar F^m\|_{s,u}+\delta^{3k(1-\varepsilon_0)-2\varepsilon_0+1/2}+\delta^{2k(1-\varepsilon_0)-1}\mu_{\min}^{-1/2}(s)\sqrt{\mathcal E_{1,\leq m+2}(s,u)}\\
&+\delta^{2k(1-\varepsilon_0)-\varepsilon_0}\sqrt{\mathcal E_{2,\leq m+2}(s,u)}+\delta^{2k(1-\varepsilon_0)}\delta^l\|\bar Z^m\slashed\triangle\mu\|_{s,u}\\
&+\delta^{2k(1-\varepsilon_0)-\varepsilon_0}(1+\delta^{2k(1-\varepsilon_0)-\varepsilon_0})\int_{t_0}^s\mu_{\min}^{-1/2}\sqrt{\mathcal E_{1,\leq m+2}(\tau,u)}d\tau\\
&+(\delta^{3k(1-\varepsilon_0)-1}+\delta^{4k(1-\varepsilon_0)-2\varepsilon_0})\int_{t_0}^s\sqrt{\mathcal E_{2,\leq m+2}(\tau,u)}d\tau.
\end{align*}
And hence, with the help of \eqref{eq:mu_integrability}, one has
\begin{equation}\label{Y-21}
\begin{split}
&\int_{t_0}^s \delta^l \|\bar e^m\|_{\tau,u} \,\mathrm{d}\tau\\
\lesssim &\delta^{2k(1-\varepsilon_0)-\varepsilon_0} \int_{t_0}^s \delta^l \|\bar F^m\|_{\tau,u} \,\mathrm{d}\tau
+ \delta^{2k(1-\varepsilon_0)-\varepsilon_0+1/2}+\delta^{2k(1-\varepsilon_0)}\int_{t_0}^s\delta^l\|\bar Z^m\slashed\triangle\mu\|_{\tau,u}d\tau\\
& + \frac{\delta^{(k-1)(1-\varepsilon_0)}\mu_{\min}^{\frac{1}{2} - b_{m+2}}(s)}{b_{m+2}} \sqrt{\widetilde{\mathcal{E}}_{1,\leq m+2}(s, u)} + \frac{\delta^{k(1-\varepsilon_0)}}{b_{m+2}} \mu_{\min}^{1 - b_{m+2}}(s) \sqrt{\widetilde{\mathcal{E}}_{2,\leq m+2}(s, u)}.
\end{split}
\end{equation}
Inserting \eqref{sdE}, \eqref{Y-19} and \eqref{Y-21} into \eqref{Eal}, and then applying Gronwall's inequality, we arrive at
\begin{equation}\label{Fal}
\begin{split}
\delta^l \|\bar F^m \|_{s, u} \lesssim \,&\delta^{k(1-\varepsilon_0)-\varepsilon_0+1/2} + \delta^{2k(1-\varepsilon_0)-\varepsilon_0 +1/2}(1+\delta^{2k(1-\varepsilon_0)-\varepsilon_0})\int_{t_0}^s\mu_{\min}^{-1}(\tau)d\tau\\
&+ \frac{\mu_{\min}^{-b_{m+2}}(s)}{b_{m+2}} \big\{ \delta^{(k-1)(1-\varepsilon_0)}\sqrt{\widetilde{\mathcal{E}}_{1,\leq m+2}(s, u)} + \delta^{k(1-\varepsilon_0)}\sqrt{\widetilde{\mathcal{E}}_{2,\leq m+2}(s, u)} \big\}\\
&+\delta^{2k(1-\varepsilon_0)}\int_{t_0}^s\delta^l\|\bar Z^m\slashed\triangle\mu\|_{\tau,u}d\tau.
\end{split}
\end{equation}
In addition, by the definition of $\bar F^m$, the identity
\[
\mu \slashed{d} \bar{Z}^m \, \mathrm{tr}\check{\mathscr{X}} =\bar F^m + \slashed{d} \bar{Z}^m (\mu E_0)-\frac{1}{2}\bigl(G_{\mathring{L}\mathring{L}}^\gamma \mathring{L}_\gamma\tilde{T}^i\bigr)\slashed d\bar Z^m \mathring{\underline{L}}\varphi_i
\]
holds. Combining inequalities \eqref{Fal} and \eqref{dZE}, one then has
\begin{align}
&\delta^l \| \mu \slashed{d} \bar{Z}^m \, \mathrm{tr}\check{\mathscr{X}} \|_{s, u}\nonumber\\
 \lesssim \,&\delta^{k(1-\varepsilon_0)-\varepsilon_0+1/2} + \delta^{2k(1-\varepsilon_0)-\varepsilon_0 +1/2}(1+\delta^{2k(1-\varepsilon_0)-\varepsilon_0})\int_{t_0}^s\mu_{\min}^{-1}(\tau)d\tau\nonumber\\
&+ \frac{\mu_{\min}^{-b_{m+2}}(s)}{b_{m+2}} \big\{ \delta^{(k-1)(1-\varepsilon_0)}\sqrt{\widetilde{\mathcal{E}}_{1,\leq m+2}(s, u)} + \delta^{k(1-\varepsilon_0)}\sqrt{\widetilde{\mathcal{E}}_{2,\leq m+2}(s, u)} \big\}\label{d}\\
&+\delta^{2k(1-\varepsilon_0)}\int_{t_0}^s\delta^l\|\bar Z^m\slashed\triangle\mu\|_{\tau,u}d\tau.\nonumber
\end{align}

In summary, for any string $Z^m$ generated by
$\{\varrho\mathring L,T,R\}$, let $l$ and $p$ denote the total numbers of $T$ and
$\varrho\mathring L$ occurring in $Z^m$,
respectively, then we have  that by \eqref{ZmX} and \eqref{d},
\begin{equation}\label{dnnchi}
\begin{split}
&\delta^{l+p\varepsilon_k} \| \mu \slashed{d}{Z}^m \, \mathrm{tr}\check{\mathscr{X}} \|_{s, u}\\
 \lesssim \,&\delta^{k(1-\varepsilon_0)-\varepsilon_0+1/2} + \delta^{2k(1-\varepsilon_0)-\varepsilon_0 +1/2}(1+\delta^{2k(1-\varepsilon_0)-\varepsilon_0})\int_{t_0}^s\mu_{\min}^{-1}(\tau)d\tau\\
&+ \mu_{\min}^{-b_{m+2}}(s) \big\{ \delta^{(k-1)(1-\varepsilon_0)}\sqrt{\widetilde{\mathcal{E}}_{1,\leq m+2}(s, u)} + \delta^{k(1-\varepsilon_0)}\sqrt{\widetilde{\mathcal{E}}_{2,\leq m+2}(s, u)} \big\}\\
&+\delta^{2k(1-\varepsilon_0)}\int_{t_0}^s\delta^l\|\bar Z^m\slashed\triangle\mu\|_{\tau,u}d\tau,
\end{split}
\end{equation}
where the last term on the right-hand side of \eqref{dnnchi}  is only needed in the
case of $Z^m=\bar Z^m$ with $\bar Z\in\{T,R\}$.

\subsubsection{Estimates on the derivatives of $\slashed\triangle\mu$}\label{Sformu}

Inspired by \cite{M-Y, Sp}, we now derive a transport equation for $\mu\slashed\triangle\mu$
along the vector field $\mathring L$. Following the computations analogous to those in \cite[Section~4.2]{Ding4},
one has
\begin{equation*}
\begin{split}
\mathring{L}\tilde{F} &= \bigl(-2\mu\,\mathrm{tr}\check{\mathscr{X}} - \frac{2\mu}{\varrho} + \tilde{\mathcal{E}}\bigr)\slashed{\triangle}\mu+ \mu\bigl(\slashed{d}^X\mathrm{tr}\check{\mathscr{X}}\bigr)\mathcal{Q}_X + \tilde{e},
\end{split}
\end{equation*}
where
\begin{equation}\label{EEF}
\begin{split}
\tilde{F} &= \mu\slashed{\triangle}\mu - \tilde{E}, \\
\tilde{E} &= -\frac{1}{2}\mu^2 \bigl(G_{\mathring{L}\mathring{L}}^\gamma + 2G_{\tilde{T}\mathring{L}}^\gamma\bigr)\slashed{\triangle}\varphi_\gamma
+ \frac{1}{2}G_{\mathring{L}\mathring{L}}^\gamma T\underline{\mathring{L}}\varphi_\gamma, \\
\tilde{\mathcal{E}} &= -\mu G^\gamma_{\mathring{L}\mathring{L}}\mathring{L}\varphi_\gamma
- 2\mu G^\gamma_{\tilde{T}\mathring{L}}\mathring{L}\varphi_\gamma
- \mu\slashed{g}^{XX}G_{XX}^\gamma\mathring{L}\varphi_{\gamma}
+\f12 G_{\mathring{L}\mathring{L}}^\gamma T\varphi_\gamma, \\
\mathcal{Q}_X &= -\slashed{d}_X\mu
- \mu G^\gamma_{X\tilde{T}}\mathring{L}\varphi_\gamma
- \mu G^{\gamma}_{\tilde{T}\mathring{L}}\slashed{d}_X\varphi_\gamma
- \mu G^\gamma_{\mathring{L}\mathring{L}}\slashed{d}_X\varphi_\gamma
+ G^\gamma_{X\mathring{L}}T\varphi_\gamma
\end{split}
\end{equation}
and
\begin{align}\label{te}
\tilde e=&-\mu\slashed d^X\mu\big\{\slashed d_X\big(\f12\slashed g^{XX}G_{XX}^\gamma\mathring L\vp_\gamma+G_{\mathring L\mathring L}^\gamma\mathring L\vp_\gamma+2G_{\tilde T\mathring L}^\gamma\mathring L\vp_\gamma\big)\no\\
&\qquad+\f12G_{\mathring L\mathring L}^\gamma\slashed d_X\vp_\gamma\big(G_{X\mathring L}^\al\slashed d^X\vp_\al+G_{X\tilde T}^\al\slashed d^X\vp_\al+\mathrm{tr}\check{\mathscr{X}}+\varrho^{-1}
\big)
\big\}-\f12\mathring L(G_{\mathring L\mathring L}^\gamma)T\underline{\mathring L}\varphi_\gamma\no\\
&-\f12G_{\mathring L\mathring L}^\gamma\big\{T\mu(\slashed\triangle\varphi_\gamma)+\leftidx^{{(T)}}{\slashed\pi}_{\mathring L}^X\slashed d_X\underline{\mathring L}\varphi_\gamma-\f{1}{2\varrho^2}\underline{\mathring L}\varphi_\gamma
-\f{1}{2\varrho}T\underline{\mathring L}\varphi_\gamma+T H_\gamma\big\}\no\\
&+f_1(\varphi, \mathring L^1, \mathring L^2,\slashed dx,\varrho,\slashed\triangle x,\operatorname{tr}\check{\mathscr X})\mu\slashed\triangle\varphi\left(
\begin{array}{ccc}
\mu\mathring L(\varphi^{k-1})\\
\mu\slashed d(\varphi^{k-1})\\
T(\varphi^{k-1})\\
\mu\varphi^{k-1}
\end{array}
\right)\\
&+f_2(\varphi,\mathring L^1,\mathring L^2,\slashed d\mathring L^1,\slashed d\mathring L^2,\varrho,\slashed\triangle x)\mu\vp^{(k-3)_+}\left(
\begin{array}{ccc}
\vp^{k+1}\\
\slashed d\varphi\\
\vp\slashed d\mathring L^1\\
\vp\slashed d\mathring L^2
\end{array}
\right)
\left(
\begin{array}{ccc}
\slashed d\varphi\\
\vp\slashed d\mathring L^1\\
\vp\slashed d\mathring L^2
\end{array}
\right)
\left(
\begin{array}{ccc}
\mu\mathring L\varphi\\
\mu\slashed d\varphi\\
T\varphi
\end{array}
\right)\no\\
&+f_3(\varphi, \mathring L^1, \mathring L^2,\slashed dx,\varrho,\operatorname{tr}\check{\mathscr X})\mu\vp^{k-1}\slashed\triangle x\left(
\begin{array}{ccc}
\mu\mathring L\varphi\\
\mu\slashed d\varphi\\
T\varphi\\
\mu\vp
\end{array}
\right)+f_4(\varphi,\mathring L^1,\mathring L^2,\slashed dx)\left(
\begin{array}{ccc}
\slashed dT\varphi\\
\mu\slashed d\mathring L\varphi
\end{array}
\right)
\left(
\begin{array}{ccc}
\mu\slashed d(\varphi^{k-1})\\
\mu\varphi^{k-1}\slashed d\mathring L^1\\
\mu\varphi^{k-1}\slashed d\mathring L^2\\
\mu\varphi^{k-1}\slashed d\mathring x^1\\
\mu\varphi^{k-1}\slashed d\mathring x^2
\end{array}
\right).\no
\end{align}

Analogously to the introduction of $F^m$ in Subsection \ref{trchi}, we set
\[
\tilde{F}^m \;=\; \mu \bar{Z}^m \slashed{\triangle} \mu \;-\; \bar{Z}^m \tilde{E}
\]
with $\bar{Z}$ being any vector field in $\{T, R\}$. Then an induction argument yields that for $m \geq 1$,
\begin{equation}\label{LbarF}
\begin{split}
\mathring{L} \tilde{F}^m
&= \bigl( -2\,\mathrm{tr}\check{\mathscr{X}} \;-\; \frac{2}{\varrho} \;+\; \mu^{-1}\tilde{\mathcal{E}} \bigr)
\bigl( \tilde{F}^m \;+\; \bar{Z}^m \tilde{E} \bigr) + \mu \bigl( \slashed{d}^X \bar{Z}^m \,\mathrm{tr}\check{\mathscr{X}} \bigr) \mathcal{Q}_X
\;+\; \tilde{e}^m,
\end{split}
\end{equation}
where
\begin{equation}\label{bare}
\begin{split}
\tilde e^m=&\bar Z^{m}\tilde e+\sum_{m_1+m_2=m-1}{\bar Z}^{m_1}\big\{-(\bar Z\mu)\mathring L\bar Z^{m_2}\slashed\triangle\mu+(\slashed d^X{\bar Z}^{m_2}\mathrm{tr}\check{\mathscr{X}})\slashed{\mathcal L}_{\bar Z}(\mu\mathcal Q_X)\big\}\\
&+\sum_{m_1+m_2=m-1}{\bar Z}^{m_1}\big\{\leftidx^{{(\bar Z)}}{\slashed\pi}_{\mathring L}^X(\mu\slashed d_X\bar Z^{m_2}\slashed\triangle\mu-\slashed d_X\bar Z^{m_2}\tilde E)\\
&\qquad-\bar Z\bigl(\f12\mu G_{\mathring L\mathring L}^\gamma\mathring L\vp_\gamma+\mu G_{\tilde T\mathring L}^\gamma\mathring L\vp_\gamma+\mu\slashed g^{XX}G_{XX}^\gamma\mathring L\vp_\gamma+2\mu\,\mathrm{tr}\check{\mathscr{X}}+\f{2\mu}\varrho\bigr)\bar Z^{m_2}\slashed \triangle\mu\big\}.
\end{split}
\end{equation}
Since
\[\mu^{-1}\tilde{\mathcal E}=\mu^{-1}\mathring L\mu-\f12G_{\mathring L\mathring L}^\gamma\mathring L\vp_\gamma-G_{\tilde T\mathring L}^\gamma\mathring L\vp_\gamma-\slashed g^{XX}G_{XX}^\gamma\mathring L\vp_\gamma\lesssim\delta^{-\varepsilon_k}\]
holds by \eqref{blowup} and $|G_{X\mathring L}^\gamma T\vp_\gamma|\lesssim\delta^{k(1-\varepsilon_0)-1}$, it follows from \eqref{LbarF} that
\begin{equation*}
\begin{split}
\mathring{L}\bigl( \varrho^4 |\tilde{F}^m|^2 \bigr)
&= 2\varrho^2 \tilde{F}^m \mathring{L}\bigl( \varrho^2 \tilde{F}^m \bigr) \\
&\lesssim \delta^{-\varepsilon_k} \varrho^4 |\tilde{F}^m|^2 + |\varrho^2 \tilde{F}^m| \bigl(
\mu^{-1} \delta^{-\varepsilon_k} |\bar{Z}^m \tilde{E}| + \delta^{k(1-\varepsilon_0)-1} |\mu \slashed{d} \bar{Z}^m \mathrm{tr}\check{\mathscr{X}}| + |\tilde{e}^m|
\bigr).
\end{split}
\end{equation*}
Integrating this inequality along the integral curves of
$\mathring L$ and applying Gronwall's inequality, together with
$s-t_0\lesssim\delta^{\varepsilon_k}$, we obtain
\begin{equation}\label{r2F}
\begin{split}
\delta^{l} \varrho^2 |\tilde{F}^m(s)|
&\lesssim \delta^{l} \varrho_0^2 |\tilde{F}^m(t_0)| + \delta^{l} \int_{t_0}^s \bigl\{
\mu_{\min}^{-1} \delta^{-\varepsilon_k} |\bar{Z}^m \tilde{E}| + \delta^{k(1-\varepsilon_0)-1} |\mu \slashed{d} \bar{Z}^m \mathrm{tr}\check{\mathscr{X}}| + |\tilde{e}^m|
\bigr\} d\tau.
\end{split}
\end{equation}
Applying \eqref{Ff} to \eqref{r2F} yields
\begin{equation}\label{F}
\begin{split}
\delta^{l} \|\tilde{F}^m\|_{s, u}
&\lesssim \delta^{-\varepsilon_k-1/2} + \delta^{-\varepsilon_k} \int_{t_0}^s \mu_{\min}^{-1}(\tau) \delta^{l} \|\bar{Z}^m \tilde{E}\|_{\tau,u} d\tau \\
&\quad + \delta^{k(1-\varepsilon_0)-1}\int_{t_0}^s \delta^{l}  \|\mu \slashed{d} \bar{Z}^m \mathrm{tr}\check{\mathscr{X}}\|_{\tau,u} d\tau +  \int_{t_0}^s \delta^{l}\|\tilde{e}^m\|_{\tau,u} d\tau,
\end{split}
\end{equation}
where
\begin{align}
 &\delta^{k(1-\varepsilon_0)-1}\int_{t_0}^s  \delta^{l}\|\mu \slashed{d} \bar{Z}^m \mathrm{tr}\check{\mathscr{X}}\|_{\tau,u} d\tau\nonumber\\
\lesssim & \delta^{k(1-\varepsilon_0)-1} \Big\{\delta^{1/2}+\delta^{k(1-\varepsilon_0)+ 1/2}(1+\delta^{2k(1-\varepsilon_0)-\varepsilon_0})\int_{t_0}^s\mu_{\min}^{-1}(\tau)d\tau\nonumber\\
&\qquad\qquad+\f{\mu_{\min}^{1-b_{m+2}}}{b_{m+2}}\bigl(\delta^{2\varepsilon_0-1}\sqrt{\widetilde{\mathcal{E}}_{1,\leq m+2}(s, u)} + \delta^{\varepsilon_0}\sqrt{\widetilde{\mathcal{E}}_{2,\leq m+2}(s, u)} \bigr)
\Big\}\label{chi}\\
&+\delta^{3k(1-\varepsilon_0)-1}\int_{t_0}^s\int_{t_0}^\tau\delta^l\|\bar Z^m\slashed\triangle\mu\|_{\bar\tau, u}d\bar\tau d\tau.\nonumber
\end{align}
In addition, since the term $G_{\mathring L\mathring L}^\gamma T^2\vp_\gamma$ appears in $\tilde E$, it follows from \cite[(4.46)]{Ding4} that
\begin{equation*}
G_{\mathring L\mathring L}^\gamma  T^2\varphi_\gamma=f(\varphi, \mathring L^1,\mathring L^2,\f{x}\varrho,\check\varrho)\left(
\begin{array}{ccc}
\varphi^kT^2\varphi\\
\varphi^{k-1}T(\mu\tilde T^a\mathring L\varphi_a)\\
\varphi^{k-1}T(\mathring L_\gamma\tilde T^a)T\varphi_a\\
\varphi^{k-1}T(\mu\tilde T^a g_{\gamma j}\slashed d^Xx^j\slashed d_X\varphi_a)\\
\check L^{1}\varphi^{k-1}T^2\varphi\\
\check L^{2}\varphi^{k-1}T^2\varphi\\
\check\varrho\varphi^{k-1}T^2\varphi\\
\varphi^{k-2}(\mathring L\phi)T^2\varphi\\
\varphi^{k-2}(\slashed d^Xx)(\slashed d_X\phi)T^2\varphi
\end{array}
\right),
\end{equation*}
and hence, with the help of Proposition \ref{L2chi} and Corollary \ref{Cor6.3}, we have
\begin{align*}
	&\delta^l\|\bar Z^m(G_{\mathring L\mathring L}^\gamma  T^2\varphi_\gamma)\|_{s,u}\\
	\lesssim & \delta^{k(1-\varepsilon_0)-1}\sum_{i=1}^2\sqrt{\mathcal E_{i,\leq m+2}(s,u)}+\delta^{2k(1-\varepsilon_0)-\varepsilon_0-1}\delta^{l_0}\|\bar Z^{m_0}\mu\|_{s,u}+\delta^{-\varepsilon_k}\delta^{l_1}\|\slashed{\mathcal{L}}_{\bar{Z}}^{m_1}\leftidx{^{(R)}}{\slashed{\pi}}_{{T}}\|_{s, u}\\
	&+\delta^{k(1-\varepsilon_0)-2}(\sum_{i=1}^2\delta^{l_0}\|\bar Z^{m_0}\check L^i\|_{s,u}+\delta^{l_0}\|\bar Z^{m_0}\check\varrho\|_{s,u})+\delta^{k(1-\varepsilon_0)-1-\varepsilon_0}\delta^{l_0}\|\bar Z^{m_0}x\|_{s,u}\\
	&+\delta^{(k-1)(1-\varepsilon_0)-2}\delta^{l_1}(\|\bar Z^{m_1}\mathring L\phi\|_{s,u}+\|\slashed d\bar Z^{m_1}\phi\|_{s,u})+\delta^{k(1-\varepsilon_0)-1}(1+\delta^{2k(1-\varepsilon_0)-\varepsilon_0})\delta^{l_1}\|\slashed{\mathcal L}_{\bar Z}^{m_1}\slashed g\|_{s,u}\\
	\lesssim & \delta^{-\varepsilon_k-1/2}+\delta^{(k-1)(1-\varepsilon_0)-1}\sqrt{\mathcal{E}_{1,\leq m+2}(s, u)}+\delta^{k(1-\varepsilon_0)-1}\sqrt{\mathcal{E}_{2,\leq m+2}(s, u)}\\
	&+\delta^{2k(1-\varepsilon_0)-2}\int_{t_0}^s\bigl\{\mu_{\min}^{-1/2}(\tau)\sqrt{\mathcal E_{1,\leq m+2}(\tau, u)}+\delta^{k(1-\varepsilon_0)}\sqrt{\mathcal E_{2,\leq m+2}(\tau,u)}\bigr\}d\tau,
\end{align*}
which gives
\begin{align}
&\delta^l\|\bar Z^m\tilde E\|_{s, u}\nonumber\\
\lesssim & \delta^{(k-1)(1-\varepsilon_0)}\de^{l}\bigl(\|\mu\slashed\triangle\bar Z^{m}\vp\|_{s, u}+\|T\bar Z^{m}\mathring{L}\vp\|_{s, u}\bigr)+\delta^l\|\bar Z^m(G_{\mathring L\mathring L}^\gamma  T^2\varphi_\gamma)\|_{s,u}\nonumber\\
&+\delta^{k(1-\varepsilon_0)}\delta^{l_1}\bigl(\|\slashed{\mathcal{L}}_{\bar{Z}}^{m_1}\leftidx{^{(R)}}{\slashed{\pi}}\|_{s, u}
+\delta\|\slashed{\mathcal{L}}_{\bar{Z}}^{m_1}\leftidx{^{(T)}}{\slashed{\pi}}\|_{s, u}+\delta^{-\varepsilon_k}\|\slashed{\mathcal{L}}_{\bar{Z}}^{m_1}\leftidx{^{(R)}}{\slashed{\pi}}_T\|_{s, u}\bigr)\nonumber\\
&+\delta^{2k(1-\varepsilon_0)-\varepsilon_0-1}\Bigl(\sum_{i=1}^2\|\bar\delta^{l_1} \bar Z^{m_1}\check L^i\|_{s, u}+\delta^{l_1}\|\bar Z^{m_1}x\|_{s, u}+\delta^{l_0}\|\bar Z^{m_0}\mu\|_{s, u}\Bigr)\nonumber\\
&+\delta^{2k(1-\varepsilon_0)-1}\sum_{i=1}^2\sqrt{\mathcal{E}_{i,\leq m+2}(s, u)}\nonumber \\
\lesssim &\delta^{-\varepsilon_k-1/2}+\delta^{(k-1)(1-\varepsilon_0)-1}\sqrt{\mathcal{E}_{1,\leq m+2}(s, u)}+\delta^{k(1-\varepsilon_0)-1}\sqrt{\mathcal{E}_{2,\leq m+2}(s, u)}\label{ZtE}\\
&+\delta^{2k(1-\varepsilon_0)-2}\int_{t_0}^s\bigl\{\mu_{\min}^{-1/2}(\tau)\sqrt{\mathcal E_{1,\leq m+2}(\tau, u)}+\delta^{k(1-\varepsilon_0)}\sqrt{\mathcal E_{2,\leq m+2}(\tau,u)}\bigr\}d\tau.\nonumber
\end{align}
Therefore, it holds that
\begin{equation}\label{ZmE}
\begin{split}
\delta^{-\varepsilon_k}\int_{t_0}^s\mu_{\min}^{-1}&(\tau)\delta^l\|\bar Z^m\tilde E\|_{\tau,u}\,d\tau
\lesssim \delta^{-2\varepsilon_k-\frac12}\int_{t_0}^s\mu_{\min}^{-1}(\tau)\,d\tau \\
& + \frac{\delta^{(k-1)(1-\varepsilon_0)-1}}{b_{m+2}}\mu_{\min}^{-b_{m+2}}(s)
\bigl(\sqrt{\widetilde{\mathcal{E}}_{1,\leq m+2}(s, u)}
+ \delta^{1-\varepsilon_0}\sqrt{\widetilde{\mathcal{E}}_{2,\leq m+2}(s, u)}\bigr).
\end{split}
\end{equation}

Next, we estimate $\tilde e^m$. By \eqref{bare}, we start to deal with $\bar Z^m\tilde e$. Similarly to \cite[(4.52), (4.53)]{Ding4}, $\tilde e$ in \eqref{te} contains
the following term
\begin{equation}\label{LGT}
\begin{split}
&\mathring L(G_{\mathring L\mathring L}^\gamma)T^2\varphi_\gamma
=2G_{i\beta}^\gamma(\underbrace{\mathring L\mathring L^i}_{\eqref{LL}})\mathring L^\beta(T^2\varphi_\gamma)+(\p_{\varphi_\nu}G_{\al\beta}^\gamma)(\mathring L\varphi_\nu)\mathring L^\al\mathring L^\beta(T^2\varphi_\gamma)
\end{split}
\end{equation}
with
\begin{equation}\label{pGLT}
(\p_{\varphi_\nu}G_{\al\beta}^\g)(\mathring L\varphi_\nu)\mathring L^\al\mathring L^\beta(T^2\varphi_\g)=f(\varphi, \check L^1,\check L^2,\f{x}\varrho,\check\varrho)\left(
\begin{array}{ccc}
\varphi^{k-1}(\mathring L\varphi)T^2\varphi\\
\varphi^{k-2}(\mathring L\varphi)T(\mu\tilde T^a\mathring L\varphi_a)\\
\varphi^{k-2}(\mathring L\varphi)T(\mathring L_\g\tilde T^a)T\varphi_a\\
\varphi^{k-2}(\mathring L\varphi)T(\mu\tilde T^a g_{\g j}\slashed d^Xx^j\slashed d_X\varphi_a)\\
\check L^{1}\varphi^{k-2}(\mathring L\varphi)T^2\varphi\\
\check L^{2}\varphi^{k-2}(\mathring L\varphi)T^2\varphi\\
\check\varrho\varphi^{k-2}(\mathring L\varphi)T^2\varphi\\
\varphi^{(k-3)_+}(\mathring L\phi)(\mathring L\varphi)T^2\varphi\\
\varphi^{(k-3)_+}(\slashed d^Xx)(\slashed d_X\phi)(\mathring L\varphi)T^2\varphi\\
\varphi^{k-2}(\mathring L^\al\mathring L\varphi_\al)T^2\varphi\\
\varphi^{k-2}\slashed d^Xx(\mathring L^\al\slashed d_X\varphi_\al)T^2\varphi
\end{array}
\right).
\end{equation}
By the facts $\mathring L^\al\mathring L\varphi_\al=\mathring L^2\phi+f(\varphi,\mathring L^1,\mathring L^2)\varphi^{k}\left(
\begin{array}{ccc}\mathring L\varphi\\
\slashed d^X\varphi\slashed d_Xx\end{array}
\right)$ and $\mathring L^\al\slashed d_X\varphi_\al=\slashed d_X\mathring L\phi-\textrm{tr}\mathscr X\slashed d_X\phi+f(\varphi,\mathring L^1,\mathring L^2)\varphi^{k}\slashed d_X\varphi$, one has
\begin{align*}
&\delta^l\|\bar Z^m\big(\mathring L(G_{\mathring L\mathring L}^\g)T\underline{\mathring L}\varphi_\g\big)\|_{s,u}\lesssim\delta^l\|\bar Z^m\big(\mathring L(G_{\mathring L\mathring L}^\g)T(\mu{\mathring L}\varphi_\g)\big)\|_{s,u}+\delta^l\|\bar Z^m\big(\mathring L(G_{\mathring L\mathring L}^\g)T^2\varphi_\g\big)\|_{s,u}\\
\lesssim & \delta^{3k(1-\ve0)-1-2\ve0}\delta^{l_0}\|\bar Z^{m_0}\mu\|_{s,u}+\delta^{2k(1-\ve0)-1-\varepsilon_0}\delta^l\|T\bar Z^{m+1}\varphi\|_{s,u}+\delta^{2k(1-\varepsilon_0)-1}\delta^l\|T\bar Z^m\mathring L\vp\|_{s,u}\\
&+\delta^{-2\varepsilon_k}\delta^{l_1}\|\slashed{\mathcal L}_{\bar Z}^{m_1}\leftidx{^{(R)}}{\slashed\pi}_T\|_{s,u}+\delta^{2k(1-\ve0)-1-\varepsilon_0}\sum_{i=1}^2\sqrt{\mathcal{E}_{i,\leq m+2}(s, u)}\\
&+\delta^{k(1-\ve0)-2-\varepsilon_k}\delta^{l_0}\big(\sum_{i=1}^2\|\bar Z^{m_0}\check L^i\|_{s,u}+\|\bar Z^{m_0}\check\varrho\|_{s,u}\big)+\delta^{k(1-\ve0)-1}\delta^{l_1}\|\slashed{\mathcal L}_{\bar Z}^{m_1}\check{\mathscr X}\|_{s,u}\\
&+\delta^{(k-1)(1-\ve0)-2}\bigl(\delta^{l_0}\|Z^{m_0}\phi\|_{s,u}+\delta^{l_1}\|\bar Z^{m_1}\mathring L\phi\|_{s,u}\bigr)+\delta^{(k+1)(1-\ve0)-2-\varepsilon_k}\delta^{l_{-1}}\|\bar Z^{m_{-1}}x\|_{s,u}\\
&+\delta^{(k-1)(1-\ve0)-2}\delta^{l_1}\big(\|\bar Z^{m_1}\mathring L^2\phi\|_{s,u}+\|\slashed d\bar Z^{m_1}\mathring L\phi\|_{s,u}\big)+\delta^{2k(1-\ve0)-2}\delta^{l_0}\|\slashed{\mathcal L}_{\bar Z}^{m_0}\slashed g\|_{s,u}\\
\lesssim & \delta^{2k(1-\ve0)-1/2-2\ve0}+\delta^{2k(1-\ve0)-2}\bigl(\sqrt{\mathcal E_{1,\leq m+2}}+\delta^{1-\varepsilon_0}\sqrt{\mathcal E_{2,\leq m+2}}\bigr)+\delta^{3k(1-\varepsilon_0)-1}\|R^m\slashed\triangle\mu\|_{s,u}\\
&+\delta^{3k(1-\ve0)-2-\varepsilon_0}\int_{t_0}^s\bigl(\mu_{\min}^{-1/2}\sqrt{\mathcal E_{1,\leq m+2}(\tau,u)}+\delta^{k(1-\varepsilon_0)}\sqrt{\mathcal E_{2,\leq m+2}(\tau,u)}\bigr)d\tau\no\\
\lesssim & \delta^{3k(1-\varepsilon_0)-1}\mu_{\min}^{-1}(s)\delta^l\|\tilde{F}^m\|_{s,u}+\delta^{2k(1-\ve0)-1/2-2\ve0}\mu_{\min}^{-1}(s)\\
&+\delta^{2k(1-\ve0)-2}\mu_{\min}^{-1}(s)\bigl(\sqrt{\tilde{\mathcal E}_{1,\leq m+2}(s,u)}+\delta^{1-\varepsilon_0}\sqrt{\tilde{\mathcal E}_{2,\leq m+2}(s,u)}\bigr).
\end{align*}

Note that \eqref{te} and \eqref{H} imply that $\bar{Z}^m \tilde{e}$ contains the term
$\frac{1}{2} \big( G_{\mathring{L}\mathring{L}}^\gamma T\varphi_\gamma \big)\bar{Z}^m T\operatorname{tr}\check{\mathscr{X}}$,
which equals
\[
\frac{1}{2} \big(G_{\mathring{L}\mathring{L}}^\gamma T\varphi_\gamma \big) \bar{Z}^m \slashed{\triangle}\mu
+ \cdots=\frac{1}{2} \mu^{-1}\big( G_{\mathring{L}\mathring{L}}^\gamma T\varphi_\gamma \big)(\tilde F^m+\bar Z^m\tilde E)+\cdots
\]
by \eqref{Tchi'}.
Thus, it follows from Proposition \ref{L2chi} and \eqref{ZtE} that
\begin{align}
&\delta^l\|\bar Z^m\tilde e\|_{s, u}\nonumber\\
\lesssim & \delta^{k(1-\varepsilon_0)-1}(1+\delta^{2k(1-\varepsilon_0)-\varepsilon_0})\delta^{l_0}\|\bar Z^{m_0}\mu\|_{s, u}+\delta^{2k(1-\varepsilon_0)-2+\varepsilon_0}\delta^l\|\mu\slashed d\bar Z^m\mathring L\vp\|_{s, u}\no\\
&+\delta^{2k(1-\varepsilon_0)-2}(1+\delta^{(k+1)(1-\varepsilon_0)-\varepsilon_0})\delta^{l_1}\|\slashed{\mathcal{L}}_{\bar{Z}}^{m_1}\slashed g\|_{s, u}+\delta^{3k(1-\varepsilon_0)-1-2\varepsilon_0}\delta^{l_{-1}}\|\bar Z^{m_{-1}}x\|_{s,u}\nonumber\\
&+\delta^{(3k-1)(1-\varepsilon_0)-2\varepsilon_0}\sum_{i=1}^2\sqrt{\mathcal E_{i,\leq m+2}(s,u)}+\delta^l\|\bar Z^m\bigl(L(G_{\mathring L\mathring L}^\gamma)T\mathring{\underline L}\vp_\gamma\bigr)\|_{s,u}\nonumber\\
&+\delta^{(k-1)(1-\varepsilon_0) -1}\delta^{l_0}\|\slashed d\bar Z^{m_0}\vp\|_{s, u}+\delta^{k(1-\varepsilon_0)-1}\delta^{l_1}\bigl(\|\slashed{\mathcal{L}}_{\bar{Z}}^{m_1}\leftidx{^{(T)}}{\slashed{\pi}}_{\mathring{L}}\|_{s, u}+\|\slashed{\mathcal{L}}_{\bar{Z}}^{m_1}\leftidx{^{(R)}}{\slashed{\pi}}\|_{s, u}\bigr)\nonumber\\
&+\delta^{(2k-1)(1-\varepsilon_0)-1}\delta^{l_0}\|\bar Z^{m_0}T\vp\|_{s, u}+\delta^{2k(1-\varepsilon_0)-1}\delta^{l_1}\|\bar Z^{m_1}T\mathring{L}\vp\|_{s, u}\nonumber\\
&+\delta^{3k(1-\varepsilon_0)-2\varepsilon_0}\delta^{l_1}\|\slashed{\mathcal{L}}_{\bar{Z}}^{m_1}\leftidx{^{(T)}}{\slashed{\pi}}\|_{s, u}+\delta^{-\varepsilon_k-1}\delta^{l_1}\|\bar Z^{m_1}\operatorname{tr}\check{\mathscr X}\|_{s, u} \nonumber\\
&+\delta^{2k(1-\varepsilon_0)-1-\varepsilon_0}\delta^{l_0}\sum_{i=1}^2\|\bar Z^{m_0}\check L^i\|_{s,u}+(1+\delta^{2k(1-\varepsilon_0)-\varepsilon_0})\delta^{l_1}\|\bar Z^{m_1}(G_{\mathring L\mathring L}^\gamma T^2\vp_\gamma)\|_{s,u}\nonumber\\
&+\delta^{-\varepsilon_k}\delta^l\|
\mu^{-1}\bigl(\widetilde F^m+\bar Z^m\widetilde E\bigr)
\|_{s,u}+\delta^{k(1-\varepsilon_0)-1}\delta^{l_1}\|\bar Z^{m_1}(G_{\mathring L\mathring L}^\gamma T\vp_\gamma)\|_{s,u}\nonumber\\
\lesssim &\delta^{-\varepsilon_k}\mu_{\min}^{-1}(s)\de^l\|\tilde F^m\|_{s, u}+\delta^{2k(1-\ve0)-1/2-2\ve0}\mu_{\min}^{-1}(s)\label{Zme}\\
&+\delta^{2k(1-\ve0)-2}\mu_{\min}^{-1}(s)\bigl(\sqrt{\tilde{\mathcal E}_{1,\leq m+2}(s,u)}+\delta^{1-\varepsilon_0}\sqrt{\tilde{\mathcal E}_{2,\leq m+2}(s,u)}\bigr).\nonumber
\end{align}
Furthermore, when $m_1+m_2=m-1$,
\begin{align}
&\delta^l\|{\bar Z}^{m_1}\Big\{-(\bar Z\mu)\mathring L\bar Z^{m_2}\slashed\triangle\mu+(\slashed d^X{\bar Z}^{m_2}\mathrm{tr}\check{\mathscr{X}})\slashed{\mathcal L}_{\bar Z}(\mu\mathcal Q_X)\Big\}\|_{s, u}\nonumber\\
&+\delta^l\|{\bar Z}^{m_1}\Big\{\leftidx^{{(\bar Z)}}{\slashed\pi}_{\mathring L}^X(\mu\slashed d_X\bar Z^{m_2}\slashed\triangle\mu-\slashed d_X\bar Z^{m_2}\tilde E)\nonumber\\
&\qquad-\bar Z\bigl(\f12\mu G_{\mathring L\mathring L}^\gamma\mathring L\vp_\gamma+\mu G_{\tilde T\mathring L}^\gamma\mathring L\vp_\gamma+\mu\slashed g^{XX}G_{XX}^\gamma\mathring L\vp_\gamma+2\mu\,\mathrm{tr}\check{\mathscr{X}}+\f{2\mu}\varrho\bigr)\bar Z^{m_2}\slashed \triangle\mu\Big\}\|_{s, u}\nonumber\\
\lesssim & \delta^{-1-\varepsilon_k}\delta^{l_1}\bigl(\|\slashed{\mathcal{L}}_{\bar{Z}}^{m_1}\leftidx{^{(R)}}{\slashed{\pi}}_{\mathring L}\|_{s, u}
+\delta\|\slashed{\mathcal{L}}_{\bar{Z}}^{m_1}\leftidx{^{(T)}}{\slashed{\pi}}_{\mathring{L}}\|_{s, u}
+\delta^2\|\slashed{\mathcal{L}}_{\bar{Z}}^{m_1}\leftidx{^{(R)}}{\slashed{\pi}}_T\|_{s, u}\bigr)\nonumber\\
&+\delta^{k(1-\varepsilon_0)}(1+\delta^{2k(1-\varepsilon_0)-\varepsilon_0})\bigl(\delta^{l}\|\tilde F^m\|_{s, u}+\delta^{l_2}\|\bar Z^{m_2}\tilde E\|_{s, u}+\delta^{-1-\varepsilon_k}\delta^{l_1}\|\slashed{\mathcal{L}}_{\bar{Z}}^{m_1}\slashed g\|_{s, u}\bigr)\nonumber\\
&+\delta^{-\varepsilon_k}\bigl(\delta^{l_0}\|\bar Z^{m_0}\mu\|_{s, u}+\delta^{l_{-1}}\|\bar Z^{m_{-1}}x\|_{s,u}\bigr)+\delta^{k(1-\varepsilon_0)}\delta^{l_0}\|T\bar Z^{m_0}\vp\|_{s, u}\nonumber\\
& +\delta^{k(1-\varepsilon_0)-1}\bigl(\sum_{i=1}^2\delta^{l_0}\|\bar Z^{m_0}\check L^i\|_{s, u}+\delta^{l_0}\|\bar Z^{m_0}\check\varrho\|_{s, u}+\delta^{l_1}\|\bar Z^{m_1}\operatorname{tr}\check{\mathscr X}\|_{s, u}+\delta^{l_0}\|\bar Z^{m_0}\vp\|_{s, u}\bigr)\nonumber\\
&+\delta^{(k-1)(1-\varepsilon_0)-1}\delta^{l_0}\bigl(\|\bar Z^{m_0}\mathring L\phi\|_{s,u}+\|\slashed d\bar Z^{m_0}\phi\|_{s,u}+\delta\|\mu\mathring L\bar Z^{m_0}\vp\|_{s, u}+\delta\|\mu\slashed d\bar Z^{m_0}\vp\|_{s,u}\bigr)\nonumber\\
\lesssim & \delta^{k(1-\varepsilon_0)}(1+\delta^{2k(1-\varepsilon_0)-\varepsilon_0})\delta^{l}\|\tilde F^m\|_{s, u}+\delta^{k(1-\varepsilon_0)}\|R^m\slashed\triangle\mu\|_{s,u}\nonumber\\
&+(1+\delta^{2k(1-\varepsilon_0)-\varepsilon_0})\Bigl\{\delta^{2k(1-\varepsilon_0)-\varepsilon_0-1/2}
+\delta^{k(1-\varepsilon_0)-1}\sqrt{\tilde{\mathcal E}_{1,\leq m+2}}+\delta^{2k(1-\varepsilon_0)-1}\sqrt{\tilde{\mathcal E}_{2,\leq m+2}}\Bigr\}\nonumber\\
\lesssim & \delta^{k(1-\varepsilon_0)}(1+\delta^{2k(1-\varepsilon_0)-\varepsilon_0})\mu_{\min}^{-1}(s)\Bigl\{\delta^{l}\|\tilde F^m\|_{s, u}+\delta^{-\varepsilon_k -1/2}\label{eleft} +\delta^{-1}\sqrt{\tilde{\mathcal E}_{1,\leq m+2}(s,u)}\\
&\qquad\qquad\qquad\qquad\qquad\qquad\qquad+\delta^{k(1-\varepsilon_0)-1}\sqrt{\tilde{\mathcal E}_{2,\leq m+2}(s,u)}\Bigr\}.\nonumber
\end{align}
By substituting \eqref{Zme} and \eqref{eleft} into \eqref{bare}, one has
\begin{equation}\label{intem}
\begin{split}
\delta^l \int_{t_0}^s \|\tilde{e}^m\|_{\tau,u} \, d\tau
\lesssim \,&\delta^{-\varepsilon_k} \int_{t_0}^s \mu_{\min}^{-1}(\tau) \, \delta^l \|\tilde{F}^m\|_{\tau,u} \, d\tau
+ \delta^{2k(1-\varepsilon_0)-2\varepsilon_0-1/2} \int_{t_0}^s \mu_{\min}^{-1}(\tau) \, d\tau \\
&+ \frac{\delta^{k(1-\varepsilon_0)-1}}{b_{m+2}} \, \mu_{\min}^{-b_{m+2}}(s)
\bigl(\delta^{\varepsilon_0-1} \sqrt{\widetilde{\mathcal{E}}_{1,\leq m+2}(s, u)}
+ \sqrt{\widetilde{\mathcal{E}}_{2,\leq m+2}(s, u)} \bigr).
\end{split}
\end{equation}
Combining \eqref{chi}, \eqref{ZmE}, \eqref{intem}
and \eqref{F}, and then applying the weighted Gronwall's
inequality together with \eqref{eq:mu-integrability-full},
when $b_{m+2}$ is chosen sufficiently large, we can obtain
\begin{equation*}
\mu_{\min}^{b_{m+2}} \, \delta^l \|\tilde{F}^m\|_{s, u}
\lesssim \delta^{-\varepsilon_k-1/2}
+ \frac{\delta^{k(1-\varepsilon_0)-1}}{b_{m+2}}
\bigl(\delta^{\varepsilon_0-1} \sqrt{\widetilde{\mathcal{E}}_{1,\leq m+2}(s, u)}
+ \sqrt{\widetilde{\mathcal{E}}_{2,\leq m+2}(s, u)} \bigr).
\end{equation*}
This, together with \eqref{ZtE}, yields
\begin{equation*}
\mu_{\min}^{b_{m+2}} \, \delta^l \|\mu \bar{Z}^m \slashed{\triangle} \mu\|_{s, u}
\lesssim \delta^{-\varepsilon_k-1/2}
+ \delta^{k(1-\varepsilon_0)-1}
\bigl(\delta^{\varepsilon_0-1} \sqrt{\widetilde{\mathcal{E}}_{1,\leq m+2}(s, u)}
+ \sqrt{\widetilde{\mathcal{E}}_{2,\leq m+2}(s, u)}\bigr).
\end{equation*}

When $Z^m$ contains at least one $\varrho\mathring L$, similarly to \eqref{ZmX}, one 
can use \eqref{lmu} to estimate $\|\mu{Z}^m \slashed{\triangle} \mu\|_{s, u}$ directly.

In the end, for any string $Z^m$ generated by
$\{\varrho\mathring L,T,R\}$, where $l$ and $p$ denote
the total numbers of $T$ and
$\varrho\mathring L$ occurring in $Z^m$, respectively, it holds that
\begin{equation}\label{Zmu}
\mu_{\min}^{b_{m+2}}(s) \, \delta^{l+p\varepsilon_k} \bigl\| \mu Z^m \slashed{\triangle} \mu \bigr\|_{s, u}
\lesssim \delta^{-\varepsilon_k-1/2}
+ \delta^{k(1-\varepsilon_0)-1}
\bigl(\delta^{\varepsilon_0-1} \sqrt{\widetilde{\mathcal{E}}_{1,\leq m+2}}
+ \sqrt{\widetilde{\mathcal{E}}_{2,\leq m+2}}\bigr).
\end{equation}
Therefore, substituting \eqref{Zmu} into \eqref{dnnchi}, we have
\begin{equation}\label{mucheckX}
\mu_{\min}^{b_{m+2}}(s)\delta^{l+p\varepsilon_k} \bigl\| \mu \slashed{d} Z^m \, \mathrm{tr}\check{\mathscr{X}} \bigr\|_{s, u}
\lesssim \delta^{-\varepsilon_k+1/2}
+ \delta^{k(1-\varepsilon_0)}
\bigl(\delta^{\varepsilon_0-1} \sqrt{\widetilde{\mathcal{E}}_{1,\leq m+2}}
+ \sqrt{\widetilde{\mathcal{E}}_{2,\leq m+2}}\bigr).
\end{equation}

\section{Estimates for the error terms}\label{ert}

Having established the optimal $L^2$ estimates for the relevant
quantities in Section~\ref{EE}, we are now ready to treat the error
terms
$\delta|\int_{D^{s,u}}\Phi\cdot\mathring{\underline L}\Psi|$
and
$\delta^{\varepsilon_k}|\int_{D^{s,u}}\Phi\cdot\mathring L\Psi|$
in~\eqref{e} so that the final energy estimates for
$\varphi$ can be derived. Following~\cite{M-Y,L-Y,Sp}, we examine all the terms in
\eqref{Phik} whose explicit expressions are given in
\eqref{T1}--\eqref{R3}.

\subsection{Treatment of $J_1^{m+1}$ in \eqref{Phik}}\label{tot}

\begin{enumerate}
\item
From the explicit expression of $J_1^{m+1}$, it follows directly
that the highest derivative order occurring in
${}^{(Z)}D_{\gamma,i}^{n}$ ($i=1,3$) does not exceed $m-n$.
Since neither the corresponding $L^\infty$ estimates in Section~\ref{Section 3}
nor the related $L^2$ estimates in Proposition~\ref{L2chi} depend on the
top-order derivative estimates in Subsection~\ref{Section 6.2}, the
$L^2$ norms of all the terms in $J_1^{m+1}$ can be estimated by
the aforementioned results. Therefore,
\begin{align}
&\delta\cdot\delta^{2l+2p\varepsilon_k}
\bigl|
\int_{D^{s,u}}
\sum_{j=1}^{m}
\bigl(Z_{m+1}+\leftidx{^{(Z_{m+1})}}\lambda\bigr)
\cdots
\bigl(Z_{m+2-j}+\leftidx{^{(Z_{m+2-j})}}\lambda\bigr)
\leftidx{^{(Z_{m+1-j})}}D_{\gamma,1}^{m-j}
\cdot\mathring{\underline L}\varphi_\gamma^{m+1}
\bigr|
\notag\\
\lesssim&\,
\delta\cdot\delta^{2l+2p\varepsilon_k}
\int_{t_0}^s
\sum_{j=1}^{m}
\bigl\|
\bigl(Z_{m+1}+\leftidx{^{(Z_{m+1})}}\lambda\bigr)
\cdots
\bigl(Z_{m+2-j}+\leftidx{^{(Z_{m+2-j})}}\lambda\bigr)
\leftidx{^{(Z_{m+1-j})}}D_{\gamma,1}^{m-j}
\bigr\|_{\tau,u}
\notag\\
&\qquad\qquad\qquad\qquad
\cdot\bigl\|\mathring{\underline L}Z^{m+1}\varphi_\gamma\bigr\|_{\tau,u}\,d\tau
\notag\\
\lesssim&\,
\delta
\int_{t_0}^s
\sqrt{\mathcal{E}_{2,\leq m+2}(\tau,u)}
\bigl\{
(1+\delta^{2k(1-\varepsilon_0)-\varepsilon_0})\bigl(\delta^{-\varepsilon_k+1/2}
+ \mu_{\min}^{-1/2}(\tau)\sqrt{\widetilde{\mathcal{E}}_{1,\leq m+2}(\tau,u)}\bigr)
\notag\\
&\qquad\qquad\qquad\qquad\qquad
+ \delta^{2k(1-\varepsilon_0)-\varepsilon_0}\sqrt{\widetilde{\mathcal{E}}_{2,\leq m+2}(\tau,u)}+\delta^{-\varepsilon_k}\delta^{l+p\varepsilon_k}\|\mathring LZ^{m+1}\vp\|_{\tau,u}
\bigr\}\!d\tau
\notag\\
\lesssim & \f{\mu_{\min}^{-2b_{m+2}}(s)}{b_{m+2}}\big\{
(1+\delta^{4k(1-\varepsilon_0)-2\varepsilon_0})\bigl(\delta^{2}+\delta^{1+2\varepsilon_k}{\widetilde{\mathcal{E}}_{1,\leq m+2}(s, u)} \bigr)
+\delta {\widetilde{\mathcal{E}}_{2,\leq m+2}(s, u)} \big\}\label{D11}\\
&+\delta^{(k+1)(1-\varepsilon_0)}\int_0^u\mathcal F_{1,\leq m+2}(s,u')du'.\nonumber
\end{align}

Similarly,
\begin{align}
&\delta\cdot\delta^{2l+2p\varepsilon_k}\bigl|\int_{D^{s,u}}\sum_{j=1}^{m}\big(Z_{m+1}+\leftidx{^{(Z_{m+1})}}\lambda\big)\dots\big(Z_{m+2-j}
+\leftidx{^{(Z_{m+2-j})}}\lambda\big)\leftidx{^{(Z_{m+1-j})}}D_{\gamma,3}^{m-j}\cdot\mathring{\underline L}Z^{m+1}\varphi_\gamma\bigr|\notag\\
\lesssim&\delta^{2}(1+\delta^{4k(1-\varepsilon_0)-2\varepsilon_0})+\delta^{1+2k(1-\varepsilon_0)}{\mathcal{E}}_{1,\leq m+2}(s, u)+\f{\mu_{\min}^{-2b_{m+2}}(s)}{b_{m+2}}\delta {\widetilde{\mathcal{E}}_{2,\leq m+2}(s, u)}.\label{D33}
\end{align}

The terms in $\delta^{2l+2p\varepsilon_k}|\int_{D^{s,u}}\Phi\cdot \delta^{\varepsilon_k}\mathring L\Psi|$
related to the integrand factors $\leftidx{^{(Z)}}D_{\gamma,i}^{n}$ $(i=1,3)$  can also be treated as
\begin{align}
&\delta^{2l+2p\varepsilon_k}\int_{D^{s,u}}\bigl|\sum_{j=1}^{m}\big(Z_{m+1}+\leftidx{^{(Z_{m+1})}}\lambda\big)\dots\big(Z_{m+2-j}
+\leftidx{^{(Z_{m+2-j})}}\lambda\big)\big(\leftidx{^{(Z_{m+1-j})}}D_{\gamma,1}^{m-j}\no\\
&\qquad\qquad\qquad+\leftidx{^{(Z_{m+1-j})}}D_{\gamma,3}^{m-j}\big)\delta^{\varepsilon_k}\mathring L\varphi_\gamma^{m+1}\bigr|\no\\
\lesssim&\delta^{2l+2p\varepsilon_k +1+\varepsilon_k}\int_{D^{s,u}}\Big\{\sum_{j=1}^{m}\big(Z_{m+1}+\leftidx{^{(Z_{m+1})}}\lambda\big)\dots\big(Z_{m+2-j}
+\leftidx{^{(Z_{m+2-j})}}\lambda\big)\nonumber\\
&\qquad\qquad\cdot\big(\leftidx{^{(Z_{m+1-j})}}D_{\gamma,1}^{m-j}+\leftidx{^{(Z_{m+1-j})}}D_{\gamma,3}^{m-j}\big)\Big\}^2
+\delta^{2l+2p\varepsilon_k}\int_{D^{s,u}}\delta^{\varepsilon_k-1} |\mathring LZ^{m+1}\varphi_\gamma|^2\no\\
\lesssim &
(1+\delta^{4k(1-\varepsilon_0)-2\varepsilon_0})\Bigl(\delta^{2}+\delta^{1+2\varepsilon_k}{\mu_{\min}^{-2b_{m+2}}(s)\widetilde{\mathcal{E}}_{1,\leq m+2}(s, u)}\Bigr)\label{D13}\\
&+\f{\mu_{\min}^{-2b_{m+2}}(s)}{b_{m+2}}\delta {\widetilde{\mathcal{E}}_{2,\leq m+2}(s, u)}+\delta^{-1}\int_0^u\delta^{\varepsilon_k}\mathcal F_{1,\leq m+2}(s,u')du'.\nonumber
\end{align}

\item As in \cite[Section 5]{Ding4}, we now proceed to estimate the terms in $J_1^{m+1}$ that involve $\leftidx{^{(Z)}}D_{\gamma,2}^{n}$ for $0 \leq n \leq m-1$. Specifically, in the case $n=0$, the index $j$ in $J_1^{m+1}$ equals $m$, and the top-order derivatives in $\leftidx{^{(Z)}}D_{\gamma,2}^{0}$ are of order $m$. Consequently, $\leftidx{^{(Z)}}D_{\gamma,2}^{0}$ contains the terms involving $(m+1)^\text{th}$-order derivatives of the deformation tensor.

This structure precludes the direct application of Proposition \ref{L2chi} to estimate the $L^2$ norm of $\leftidx{^{(Z)}}D_{\gamma,2}^0$, as the $L^2$ norm of the $(m+1)^\text{th}$-order derivatives of the deformation tensor can only be controlled by $\mathcal{E}_{1,\leq m+3}(s, u)$ and $\mathcal{E}_{2,\leq m+3}(s, u)$ in the energy estimates. Notably, by \eqref{e}, $\mathcal{E}_{1,\leq m+3}(s, u)$ and $\mathcal{E}_{2,\leq m+3}(s, u)$ cannot be directly absorbed by the energies $\mathcal{E}_{1,\leq m+2}(s, u)$ and $\mathcal{E}_{2,\leq m+2}(s, u)$ on the left-hand side of the resulting energy inequality.

To solve this, we will carefully analyze the expression of $\leftidx{^{(Z)}}D_{\gamma,2}^n$ and apply the estimates from Subsections \ref{trchi} and \ref{Sformu} to handle the top-order derivatives of $\operatorname{tr}{\check{\mathscr X}}$ and $\mu$.

According to the expressions given in \eqref{T2}, \eqref{rL2} and \eqref{R2}, special attention should be devoted to handling the terms marked with underlines, wavy lines, boxes or braces. For $\leftidx{^{(T)}}D_{\gamma,2}^{n}$, we first utilize the relation $\mathring{\underline L} = \mu\mathring L + 2T$ to obtain the identity $\frac{1}{2}\slashed{\mathcal L}_{\mathring{\underline L}}\slashed d_X\mu = \slashed d_XT\mu + \frac{1}{2}\mu\slashed d_X\mathring L\mu$. By substituting this identity into the underlined terms in $\leftidx{^{(T)}}D_{\gamma,2}^{n}$, one has
\begin{equation}\label{Y-25}
-\slashed d_XT\mu + \frac{1}{2}\slashed{\mathcal L}_{\mathring{\underline L}}\slashed d_X\mu = \frac{1}{2}\mu\slashed d_X\mathring L\mu.
\end{equation}
Based on this, the underlined terms in $\leftidx{^{(T)}}D_{\gamma,2}^{n}$ will be addressed by means of \eqref{lmu}.

It follows from \eqref{Rpi} and \eqref{theta} that all the terms marked by
wavy underlines in \eqref{T2}, \eqref{rL2} and \eqref{R2} involve the quantity
$\slashed{\mathcal L}_{\mathring{\underline L}}\check{\mathscr X}$.
 As indicated in \eqref{Tchi'}, this quantity can be expanded as $2\slashed\nabla^2\mu + \mu\slashed{\mathcal L}_{\mathring{L}}\check{\mathscr X} + \cdots$. Specifically, the terms contained in $\mu\slashed{\mathcal L}_{\mathring{L}}\check{\mathscr X} + \cdots$ can be estimated based on \eqref{Lchi'} and Proposition \ref{L2chi}, whereas the $L^2$-norm of the terms corresponding to $\slashed\triangle\mu$ may be handled by \eqref{Zmu}. Meanwhile, the boxed terms and braced terms are estimated through \eqref{Zmu} and \eqref{dnnchi}, respectively.

On the other hand, some terms contain $\mathring L$-derivatives of deformation tensors. A typical example is $\frac{1}{4}\mathring L(\text{tr}\leftidx{^{(T)}}{\slashed\pi})\underline{\mathring L}\varphi_\gamma^n$ appeared in \eqref{T2}. These terms do not cause a derivative loss: a detailed examination for each term in \eqref{Lpi} and \eqref{Rpi} shows that
it is found that the $\mathring L$-derivatives of the involved deformation tensors are associated with ``good" quantities, which can be
treated by \eqref{lmu} and \eqref{Lchi'}.

Therefore, it follows from Proposition \ref{L2chi} and \eqref{Zmu}-\eqref{mucheckX} that
		\begin{align}
&\delta^{l+p\varepsilon_k}\|\sum_{j=1}^{m}\big(Z_{m+1}+\leftidx{^{(Z_{m+1})}}\lambda\big)\dots\big(Z_{m+2-j}
+\leftidx{^{(Z_{m+2-j})}}\lambda\big)\leftidx{^{(Z_{m+1-j})}}D_{\gamma,2}^{m-j}\|_{s, u}\nonumber\\
\lesssim&\delta^{l_0+p_0\varepsilon_k}\big(\delta^{(k+1)(1-\varepsilon_0)-\varepsilon_0}(1+\delta^{2k(1-\varepsilon_0)
-\varepsilon_0})\|Z^{m_0}\mu\|_{s, u}+\delta^{2k(1-\varepsilon_0)-\varepsilon_0}\|TZ^{m_0}\vp\|_{s, u}\nonumber\\
&\qquad\qquad+\delta^{k(1-\varepsilon_0)-1}\|\mathring LZ^{m_0}\vp\|_{s, u}+\delta^{2k(1-\varepsilon_0)-1}\|\slashed dZ^{m_0}\vp\|_{s, u}+\delta^{2k(1-\varepsilon_0)-\varepsilon_0-1}\|Z^{m_0}\vp\|_{s, u}\nonumber\\
&\qquad\qquad+\delta^{-\varepsilon_0}(1+\delta^{2k(1-\varepsilon_0)-\varepsilon_0})\sum_{i=1}^2\|Z^{m_0}\check L^i\|_{s, u}+\delta^{k(1-\varepsilon_0)-2\varepsilon_0}\|Z^{m_0}\upsilon\|_{s, u}\nonumber\\
 &\qquad\qquad+\delta^{-2\varepsilon_k}\|Z^{m_0}\check\varrho\|_{s,u}+\delta^{-\varepsilon_k}\|\slashed{\mathcal L}_Z^{m_0}R\|_{s,u}+\delta^{(k+1)(1-\varepsilon_0)-\varepsilon_0}\|Z^{m_0}\mathring L\mu\|_{s,u}\big)\nonumber\\
&+\delta^{l_1+p_1\varepsilon_k}\big(\delta^{-\varepsilon_0}\|\slashed dZ^{m_1}\operatorname{tr}\check{\mathscr X}\|_{s, u}+\delta^{1-\varepsilon_0}\|Z^{m_0}\slashed\triangle\mu\|_{s, u}+(\delta^{-2\varepsilon_k} +\delta^{-k(1-\varepsilon_0)})\|Z^{m_1}\operatorname{tr}\check{\mathscr X}\|_{s, u}\nonumber\\
&\qquad\qquad\quad+\delta^{(2k+1)(1-\varepsilon_0)-\varepsilon_0}\|\slashed{\mathcal{L}}_{{Z}}^{m_1}\leftidx{^{(R)}}{\slashed{\pi}}_T\|_{s, u}+\delta^{(k+1)(1-\varepsilon_0)}\|\slashed{\mathcal{L}}_{{Z}}^{m_1}\leftidx{^{(T)}}{\slashed{\pi}}_{\mathring L}\|_{s, u}+\delta^{-\varepsilon_k}\|\slashed{\mathcal{L}}_{{Z}}^{m_1}\leftidx{^{(R)}}{\slashed{\pi}}_{\mathring L}\|_{s, u}\nonumber\\
&\qquad\qquad\quad+\delta^{(k+1)(1-\varepsilon_0)}(1+\delta^{2k(1-\varepsilon_0)-\varepsilon_0})
\|\slashed{\mathcal{L}}_{{Z}}^{m_1}\leftidx{^{(T)}}{\slashed{\pi}}\|_{s, u}\big)\nonumber\\
&+\delta^{-2\varepsilon_k}\delta^{l_{-1}+p_{-1}\varepsilon_k}\|Z^{m_{-1}}x\|_{s,u}\nonumber\\
\lesssim & \mu_{\min}^{-1-b_{m+2}}(s)\delta^{k(1-\varepsilon_0)}\Big\{\delta^{-2\varepsilon_0+1/2} +\delta^{-1}{\widetilde{\mathcal{E}}_{1,\leq m+2}(s, u)}
+\delta^{-\varepsilon_0}
{\widetilde{\mathcal{E}}_{2,\leq m+2}(s, u)} \Bigr\},\label{Dg2}
\end{align}
which gives
\begin{align}
&\delta^{2l+2p\varepsilon_k}
\bigl|
\int_{D^{s,u}}
\sum_{j=1}^{m}
\bigl(Z_{m+1}+\leftidx{^{(Z_{m+1})}}\lambda\bigr)
\cdots
\bigl(Z_{m+2-j}+\leftidx{^{(Z_{m+2-j})}}\lambda\bigr)
\leftidx{^{(Z_{m+1-j})}}D_{\gamma,2}^{m-j}
\cdot\delta\mathring{\underline L}\varphi_\gamma^{m+1}
\bigr|\notag\\
&\lesssim\f{\mu_{\min}^{-2b_{m+2}}(s)}{b_{m+2}}\bigl(\delta^{2-2\varepsilon_0}
+\delta^{2\varepsilon_0-1} {\widetilde{\mathcal{E}}_{1,\leq m+2}(s, u)}
+\delta {\widetilde{\mathcal{E}}_{2,\leq m+2}(s, u)}
\bigr).\label{iDg2}
\end{align}

But it should be noted that \eqref{Dg2} is not sufficient to close the energy estimate in $\delta^{2l+2p\varepsilon_k}|\int_{D^{s,u}}\Phi \cdot \delta^{\varepsilon_k}\mathring{L}\Psi|$ if one adopts the same strategy as in \eqref{D13}. Indeed, after examining all terms in $\leftidx{^{(Z)}}D_{\gamma,2}^{0}$, we find that the worst terms which cannot be controlled as in \eqref{D13} are $(R^{m+1}\operatorname{tr}\check{\mathscr{X}})\underline{\mathring L}\vp_\gamma$ (arising from $R^m\big(\leftidx{^{(R)}}D_{\gamma,2}^{0}\big)$) and $(\bar{Z}^m\slashed{\triangle}\mu) \underline{\mathring L}\vp_\gamma$ (appearing in $\bar{Z}^m\big(\leftidx{^{(T)}}D_{\gamma,2}^{0}\big)$). 
These two terms are not suitable for applying \eqref{mucheckX} and \eqref{Zmu} due to the lack of the factor $\mu$.
Consequently, the final energy estimates cannot be closed by
directly applying the same argument as in \eqref{D13}.

Inspired by \cite[Section 9.2]{M-Y} and \cite[Section 7.2]{L-Y}, we have
\begin{align}
&\int_{D^{s,u}}\bigl(R^{m+1}\operatorname{tr}\check{\mathscr{X}}\bigr)\underline{\mathring L}\vp_\gamma\bigl(\delta^{\varepsilon_k}\mathring L R^{m+1}\vp_\gamma\bigr)\notag\\
=&-\delta^{\varepsilon_k}\int_{D^{s,u}}\bigl(R^{m+1}\operatorname{tr}\check{\mathscr{X}}\bigr)R^{m+1}\vp_\gamma
\big(\mathring L\underline{\mathring L}\vp_\gamma
+\big(\operatorname{tr}\check{\mathscr{X}}+\frac12\slashed g^{XX}G_{XX}^\al\mathring L\vp_\alpha+\frac1\varrho\big)\underline{\mathring L}\vp_\gamma\big)\notag\\
&+\delta^{\varepsilon_k}\int_{D^{s,u}}\bigl([R^{m+1},\mathring L]\operatorname{tr}\check{\mathscr{X}}\bigr)\underline{\mathring L}\vp_\gamma\bigl(R^{m+1}\vp_\gamma\bigr)\notag\\
&+\delta^{\varepsilon_k}\int_{D^{s,u}}\bigl(R^m\mathring L\operatorname{tr}\check{\mathscr{X}}\bigr)
\big(\frac12\underline{\mathring L}\vp_\gamma\bigl(R^{m+1}\vp_\gamma\bigr)\operatorname{tr}\leftidx{^{(R)}}{\slashed{\pi}}
+R\underline{\mathring L}\vp_\gamma\bigl(R^{m+1}\vp_{\gamma}\bigr)
+\underline{\mathring L}\vp_\gamma\bigl(R^{m+2}\vp_\gamma\bigr)\big)\nonumber\\
&-\delta^{\varepsilon_k}\int_{\Sigma_{t_0}^u}\bigl(R^m\operatorname{tr}\check{\mathscr{X}}\bigr)\underline{\mathring L}\vp_\gamma\bigl(R^{m+1}\vp_\gamma\bigr)\label{RmX}\\
&-\delta^{\varepsilon_k}\int_{\Sigma_s^u}\bigl(R^m\operatorname{tr}\check{\mathscr{X}}\bigr)R^{m+1}\vp_\gamma
\Big(\frac12\underline{\mathring L}\vp_\gamma\operatorname{tr}\leftidx{^{(R)}}{\slashed{\pi}}
+R\underline{\mathring L}\vp_\gamma\Big)\notag\\
&-\delta^{\varepsilon_k}\int_{\Sigma_s^u}\bigl(R^m\operatorname{tr}\check{\mathscr{X}}\bigr)\underline{\mathring L}\vp_\gamma\bigl(R^{m+2}\vp_\gamma\bigr).\nonumber
\end{align}
With the help of \eqref{fequation}, the first integral on the right-hand side of \eqref{RmX} can be bounded by
\begin{equation}\label{YHC-A8}
\begin{split}
\bigl(\delta^{-k(1-\varepsilon_0)}+\delta^{-\varepsilon_k}\bigr) \int_{t_0}^s \mu_{\min}^{-1}(\tau) \bigl\| \mu \slashed{d} R^m \operatorname{tr}\check{\mathscr{X}} \bigr\|_{\tau,u} \bigl\| R^{m+1} \vp \bigr\|_{\tau,u} \,d\tau.
\end{split}
\end{equation}
Note that from \eqref{mucheckX},  \eqref{YHC-A8} is estimated by
\[
\frac{\mu_{\min}^{-2b_{m+2}}(s)}{b_{m+2}}
\bigl(\delta^{2-2\varepsilon_0}
+\delta^{\varepsilon_k} \widetilde{\mathcal{E}}_{1,\leq m+2}(s, u)
+\delta \widetilde{\mathcal{E}}_{2,\leq m+2}(s, u)
\bigr).
\]
The estimate for the last integral on the right-hand side of \eqref{RmX} has been treated separately in Corollary \ref{Cor6.2}. More precisely,
\[
\bigl|\delta^{\varepsilon_k}\int_{\Sigma_s^u}\bigl(R^m\operatorname{tr}\check{\mathscr{X}}\bigr)\underline{\mathring L}\vp_\gamma\bigl(R^{m+2}\vp_\gamma\bigr)\bigr|
\]
can be controlled by
\begin{align*}
&\delta^{2-2\varepsilon_0}\mu_{\min}^{-2b_{m+2}}(s)b_{m+2}
+\delta^{-(k+1)(1-\varepsilon_0)}\int_0^u \mathcal{F}_{1,\leq m+2}(s,u')\,du'\\
&+\frac{\mu_{\min}^{-2b_{m+2}}(s)}{b_{m+2}}
\bigl(\delta^{\varepsilon_k}\widetilde{\mathcal{E}}_{1,\leq m+2}(s,u)
+\delta\widetilde{\mathcal{E}}_{2,\leq m+2}(s,u)\bigr).
\end{align*}
For the remaining integrals on the right-hand side of \eqref{RmX}, we employ \eqref{Lchi'} and \eqref{Gtrivp} to show that they are bounded by
\begin{align*}
&\delta^{2-\varepsilon_0}
+\delta^{-k(1-\varepsilon_0)}\bigl\|R^{\leq m}\operatorname{tr}\check{\mathscr{X}}\bigr\|_{s,u}\bigl\|R^{m+1}\vp_\gamma\bigr\|_{s,u}
+\int_{t_0}^s\bigl\|R^{m+1}\vp_\gamma\bigr\|_{\tau,u}
\bigl\|\slashed{\mathcal{L}}_R^{\leq m}\leftidx{^{(R)}}{\slashed{\pi}}_{\mathring L}\bigr\|_{\tau, u}\,d\tau\\
&+\int_{t_0}^s\delta^{-k(1-\varepsilon_0)}\bigl(\bigl\|R^{m+2}\vp_\gamma\bigr\|_{\tau,u}+\bigl\|R^{m+1}\vp_\gamma\bigr\|_{\tau,u}\bigr)
\big\{
\bigl(1+\delta^{2k(1-\varepsilon_0)-\varepsilon_0}\bigr)\bigl\|R^{\leq m}\operatorname{tr}\check{\mathscr{X}}\bigr\|_{\tau,u}\\
&+\delta^{(k-1)(1-\varepsilon_0)}\bigl\|\slashed{d}R^{m}\mathring L\vp\bigr\|_{\tau,u}
+\delta^{1-\varepsilon_0}\bigl\|R^{m}\bigl(G_{\mathring L\mathring L}^\gamma\mathring L_\gamma\bigr)\bigr\|_{\tau,u}
+\delta^{k(1-\varepsilon_0)}\bigl\|R^m\slashed\triangle\vp\bigr\|_{\tau,u}\\
&+\delta^{2k(1-\varepsilon_0)}\bigl(\sqrt{\mathcal{E}_{1,\leq m+2}(\tau,u)}+\sqrt{\mathcal{E}_{2,\leq m+2}(\tau,u)}\bigr)\\
&+\delta^{2k(1-\varepsilon_0)-\varepsilon_0}
\bigl(\sum_{i=1}^2\bigl\|R^{\leq m+1}\check L^i\bigr\|_{\tau,u}
+\bigl\|R^{\leq m+2}x\bigr\|_{\tau,u}
+\bigl\|\slashed{\mathcal{L}}_R^{\leq m}\slashed g\bigr\|_{\tau,u}\bigr)
\big\}\,d\tau.
\end{align*}
This expression can be directly estimated by Proposition \ref{L2chi} and Corollary \ref{Cor6.3}. Consequently, one has
\begin{align}
&\bigl|\int_{D^{s,u}}\bigl(R^{m+1}\operatorname{tr}\check{\mathscr{X}}\bigr)\underline{\mathring L}\vp_\gamma\bigl(\delta^{\varepsilon_k}\mathring L R^{m+1}\vp_\gamma\bigr)\bigr|\nonumber\\
\lesssim & \delta^{2-2\varepsilon_0}\mu_{\min}^{-2b_{m+2}}(s)b_{m+2}
+\delta^{-(k+1)(1-\varepsilon_0)}\int_0^u \mathcal{F}_{1,\leq m+2}(s,u')\,du'\label{topX}\\
&+\frac{\mu_{\min}^{-2b_{m+2}}(s)}{b_{m+2}}
\bigl(\delta^{\varepsilon_k}\widetilde{\mathcal{E}}_{1,\leq m+2}(s,u)
+\delta\widetilde{\mathcal{E}}_{2,\leq m+2}(s,u)\bigr).\nonumber
\end{align}
	
In addition, due to
\begin{align*}
&\delta^{2l+2}\int_{D^{s,u}}(\bar Z^m\slashed\triangle\mu)\underline{\mathring L}\vp_\gamma(\delta^{\varepsilon_k}\mathring L\bar Z^mT\vp_\gamma)\\
=&\delta^{2l+2}\delta^{\varepsilon_k}\int_{D^{s,u}}(\bar Z^m\slashed\triangle\mu)\underline{\mathring L}\vp_\gamma\big(\mathring L[\bar Z^m, T]\vp_\gamma+[\mathring L,T]\bar Z^m\vp_\gamma+T\mathring L\bar Z^m\vp_{\gamma}\big),
\end{align*}
it follows from \eqref{Zmu} and Proposition \ref{L2chi} that
\begin{equation}\label{topmu}
\begin{split}
&\bigl|\delta^{2l+2}\int_{D^{s,u}}(\bar Z^m\slashed\triangle\mu)\underline{\mathring L}\vp_\gamma(\delta^{\varepsilon_k}\mathring L\bar Z^mT\vp_\gamma)\bigr|\\
\lesssim & \frac{\mu_{\min}^{-2b_{m+2}}(s)}{b_{m+2}}
\bigl(\delta^{2-2\varepsilon_0}
+\delta^{\varepsilon_k} \widetilde{\mathcal{E}}_{1,\leq m+2}(s, u)
+\delta \widetilde{\mathcal{E}}_{2,\leq m+2}(s, u)
\bigr).
\end{split}
\end{equation}

On the other hand, the factors $\bar Z^m\slashed\triangle\mu$ and $\slashed{d} R^{m}\operatorname{tr}\check{\mathscr X}$ also emerge in the expression of $\bar{Z}^m\big(\leftidx{^{(\varrho\mathring L)}}D_{\gamma,2}^{0}\big)$. These factors can be handled analogously by virtue of the estimates \eqref{Zmu} and \eqref{mucheckX} as above.
Therefore, with the delicate estimates \eqref{topX} and \eqref{topmu}, we have
\begin{align}\label{DL2}
&\delta^{2l+2p\varepsilon_k}\bigl|\int_{D^{s,u}}\sum_{j=1}^{m}\big(Z_{m+1}+\leftidx{^{(Z_{m+1})}}\lambda\big)\dots\big(Z_{m+2-j}
+\leftidx{^{(Z_{m+2-j})}}\lambda\big)\leftidx{^{(Z_{m+1-j})}}D_{\gamma,2}^{m-j}(\delta^{\varepsilon_k}\mathring L\varphi_\gamma^{m+1})\bigr|\notag\\
\lesssim & \delta^{2-2\varepsilon_0}\mu_{\min}^{-2b_{m+2}}(s)b_{m+2}
+\delta^{-1}\int_0^u \delta^{\varepsilon_k}\mathcal{F}_{1,\leq m+2}(s,u')\,du'\\
&+\frac{\mu_{\min}^{-2b_{m+2}}(s)}{b_{m+2}}
\bigl(\delta^{\varepsilon_k}\widetilde{\mathcal{E}}_{1,\leq m+2}(s,u)
+\delta\widetilde{\mathcal{E}}_{2,\leq m+2}(s,u)\bigr).\notag
\end{align}
\end{enumerate}

\subsection{Treatment of $J_2^{m+1}$ and $\Phi_\gamma^1$ in \eqref{Phik}}\label{l}

In this subsection, we analyze the structures of $J_2^{m+1}$ and
$\Phi_\gamma^1$ appearing in \eqref{Phik}. We observe that neither
$J_2^{m+1}$ nor $\Phi_\gamma^1$ contains the highest-order
derivatives of $\operatorname{tr}\check{\mathscr X}$ or $\Delta\mu$.
Therefore, Proposition~6.2 and the expressions for
${}^{(Z)}D_{\gamma,j}^n$ in \eqref{T1}--\eqref{R3} yield

\begin{align}
&\delta^{2l+2p\varepsilon_k}\int_{D^{s,u}}\bigl\{|\sum_{j=1}^3\leftidx{^{(Z_{m+1})}}D_{\gamma,j}^m\cdot \delta\mathring{\underline L}{\varphi}_\gamma^{m+1}|+|\sum_{j=1}^3\leftidx{^{(Z_{m+1})}}D_{\gamma,j}^m\cdot \delta^{\varepsilon_k}\mathring L{\varphi}_\gamma^{m+1}|\big\}\no\\
\lesssim &
\delta^{2-2\varepsilon_0}
+\delta(\delta^{2\varepsilon_k}+	\delta^{2k(1-\varepsilon_0)}){\mu_{\min}^{-2b_{m+2}}(s)\widetilde{\mathcal{E}}_{1,\leq m+2}(s, u)}\label{ZL}\\
&+\f{\mu_{\min}^{-2b_{m+2}}(s)}{b_{m+2}}\delta {\widetilde{\mathcal{E}}_{2,\leq m+2}(s, u)}
+\delta^{-1}\int_0^u\delta^{\varepsilon_k}\mathcal F_{1,\leq m+2}(s,u')du'.\nonumber
\end{align}

Using the explicit expression
$\Phi_\gamma^0=\mu\Box_g\varphi_\gamma$ in \eqref{ge}, the $L^2$
estimate for $\operatorname{tr}\check{\mathscr X}$ in Proposition~\ref{L2chi}, and the
asymptotic behavior given in \eqref{lamda}, one arrives at
\begin{align}
&\delta^{2l+2p\varepsilon_k}\int_{D^{s,u}}\Big\{|(Z_{m+1}+\leftidx{^{(Z_{m+1})}}\lambda)\cdots(Z_{1}
+\leftidx{^{(Z_{1})}}\lambda)\Phi_\gamma^0\cdot\delta\mathring{\underline L}\varphi_\gamma^{m+1}|\no\\
&\qquad\qquad\qquad\quad +|(Z_{m+1}+\leftidx{^{(Z_{m+1})}}\lambda)\cdots(Z_{1}
+\leftidx{^{(Z_{1})}}\lambda)\Phi_\gamma^0\cdot\delta^{\varepsilon_k}\mathring L\varphi_\gamma^{m+1}|\Big\}\nonumber\\
\lesssim &
\delta^{2-2\varepsilon_0}+\f{\mu_{\min}^{-2b_{m+2}}(s)}{b_{m+2}}\bigl\{\delta {\widetilde{\mathcal{E}}_{2,\leq m+2}(s, u)}+\delta^{2\varepsilon_0-1}\widetilde{\mathcal{E}}_{1,\leq m+2}(s, u)\bigr\}\label{Phi0}\\
&+\delta^{-1}\int_0^u\delta^{\varepsilon_k}\mathcal F_{1,\leq m+2}(s,u')du'.\nonumber
\end{align}

\subsection{Closure of the bootstrap assumptions \eqref{bootstrap-assumptions}}

In Subsections~\ref{tot} and~\ref{l}, we have established some basic estimates for all the terms comprising the expressions
of $\delta|\int_{D^{s,u}}\Phi \cdot \mathring{\underline{L}}\Psi|$ and $\delta^{\varepsilon_k}|\int_{D^{s,u}}\Phi \cdot \mathring{L}\Psi|$
in \eqref{e}.
Substituting the bounds from the inequalities \eqref{D11}--\eqref{D13}, \eqref{iDg2}, \eqref{DL2}--\eqref{Phi0} into \eqref{e},
one can obtain
\begin{equation}\label{energy}
\begin{split}
\delta\widetilde{\mathcal{E}}_{2,\leq m+2}(s, u) + \delta \widetilde{\mathcal{F}}_{2,\leq m+2}(s, u) + \delta^{\varepsilon_k}\widetilde{\mathcal{E}}_{1,\leq m+2}(s, u) + \delta^{\varepsilon_k}\widetilde{\mathcal{F}}_{1,\leq m+2}(s, u)
\lesssim_b \delta^{2-2\varepsilon_0},
\end{split}
\end{equation}
where $ f\lesssim_b g$ means $f\leq C_bg$ with $C_b$ depending on $b_{m+2}$.

In \eqref{energy}, we choose $b_{m+2}>0$ sufficiently large, with
fractional part $1/4$; that is,
$b_{m+2}=B+\frac14$
for some sufficiently large integer $B$. Under this choice, the top-order energies
$\widetilde{\mathcal E}_{1,\leq m+2}$ and
$\widetilde{\mathcal E}_{2,\leq m+2}$, together with the corresponding
fluxes, are degenerate in the sense of the definitions in \eqref{higher order energy and flux}.
 To close the bootstrap assumption \eqref{bootstrap-assumptions}, it is necessary to construct the non-degenerate energies for the lower-order derivatives of the solutions. For this purpose, we next estimate $\tilde{\mathcal{E}}_{1,\leq m+1}$ 
 and $\tilde{\mathcal{E}}_{2,\leq m+1}(s, u)$ by \eqref{energy} with a fixed constant $b_{m+2}$.

For any $Z\in\{\varrho\mathring L, T, R\}$, it follows from \eqref{fequation} that
\begin{align*}
 \mathring{L}\bigl(\varrho^{1/2}\underline{L}{Z}^m \vp_\gamma \bigr)
&=\varrho^{1/2}\mathring L[\mathring{\underline L},Z^m]\vp_\gamma- \f12\varrho^{-1/2}[Z^m,\mathring {\underline L}]\vp_{\gamma}+
\varrho^{1/2}{Z}^m (\mu\slashed\triangle\vp_\gamma+H_\gamma )\\
&+ \varrho^{1/2}[\mathring{L}, {Z}^m] \underline{L} \vp_\gamma
- \varrho^{1/2}\sum_{\substack{m_1 + m_2 = m \\ m_1 > 0}}{Z}^{m_1} \big(\frac{1}{2\varrho}\big){Z}^{m_2} \underline{L} \vp_\gamma.
\end{align*}
This, together with Proposition \ref{L2chi} and \eqref{energy}, yields
\begin{align*}
&\delta^{l+p\varepsilon_k}\|\mathring{L}\bigl(\varrho^{1/2}\underline{L}{Z}^m \vp\bigr)\|_{s, u}\\
\lesssim & \,\delta^{l_0+p_0\varepsilon_k}\bigl(\delta^{1-\varepsilon_0-2\varepsilon_k}\|Z^{m_0}\mu\|_{s, u}+\delta^{-1}(1+\delta^{2k(1-\varepsilon_0)-\varepsilon_0})\|Z^{m_0}\vp\|_{s, u}\bigr)
+\delta^{l+p\varepsilon_k}\|\mu\slashed\triangle Z^m\vp\|_{s, u}\\
&+\delta^{-\varepsilon_0}\delta^{l_1+p_1\varepsilon_k}\bigl(\|Z^{m_1}\operatorname{tr}\check{\mathscr X}\|_{s, u}+\delta^{k(1-\varepsilon_0)}\|\slashed{\mathcal{L}}_{{Z}}^{m_1}\leftidx{^{(R)}}{\slashed{\pi}}_T\|_{s, u}+\delta\|\slashed{\mathcal{L}}_{{Z}}^{m_1}\leftidx{^{(T)}}{\slashed{\pi}}_{\mathring L}\|_{s, u}\\
&\qquad+\|\slashed{\mathcal{L}}_{{Z}}^{m_1}\leftidx{^{(R)}}{\slashed{\pi}}_{\mathring L}\|_{s, u}+\delta^{2}\|\slashed{\mathcal{L}}_{{Z}}^{m_1}\leftidx{^{(T)}}{\slashed{\pi}}\|_{s, u}+\delta\|\slashed{\mathcal{L}}_{{Z}}^{m_1}\leftidx{^{(R)}}{\slashed{\pi}}\|_{s, u}+\delta^{2k(1-\varepsilon_0)-\varepsilon_0}\sum_{i=1}^2\|Z^{m_0}\check L^i\|_{s,u}\bigr)\\
&+\delta^{-\varepsilon_k}(1+\delta^{2k(1-\varepsilon_0)-\varepsilon_0})\delta^{l_1+p_1\varepsilon_k}\|\slashed{\mathcal{L}}_{{Z}}^{m_1}\slashed g\|_{s, u}+\delta^{-2\varepsilon_k}\delta^{l_{-1}+p_{-1}\varepsilon_k}\|Z^{m_{-1}}x\|_{s,u}\\
\lesssim &\,(1+\delta^{2k(1-\varepsilon_0)-\varepsilon_0})\bigl\{\delta^{-\varepsilon_k+1/2}+\mu_{\min}^{-b_{m+2}}(s)\bigl( \sqrt{\widetilde{\mathcal{E}}_{1,\leq m+2}(s, u) }+ \sqrt{\widetilde{\mathcal{E}}_{2,\leq m+2}(s, u) }\bigr)\bigr\}\\
\lesssim_b &\,\delta^{-\varepsilon_0+1/2}(1+\delta^{2k(1-\varepsilon_0)-\varepsilon_0})\mu_{\min}^{-b_{m+2}}(s).
\end{align*}
Therefore, by \eqref{Ff}, we have
\begin{align}
\delta^{l+p\varepsilon_k}\|\varrho^{1/2}\underline{L}{Z}^m \vp\|_{s, u}
\lesssim&\,\delta^{-\varepsilon_0+1/2}+\delta^{l+p\varepsilon_k}\int_{t_0}^s\|\mathring{L}\bigl(\varrho^{1/2}\underline{L}{Z}^m \vp\bigr)\|_{\tau,u}d\tau\nonumber\\
\lesssim_b&\, \delta^{-\varepsilon_0+1/2}\mu_{\min}^{1-b_{m+2}}(s).\label{E2m+1}
\end{align}
In addition, due to $\mu^{-1}\mathring L\mu\lesssim\delta^{-\varepsilon_k}$ by \eqref{blowup} and \eqref{eq:higher-order-L-infty-estimates},
it holds that
\begin{align*}
&\delta^{2l+2p\varepsilon_k}\bigl(\mu|RZ^m\vp_\gamma|^2+\mu|\mathring LZ^m\vp_\gamma|^2\bigr)(s,u,\vartheta)\\
\lesssim&\delta^{2-2\varepsilon_0-2\varepsilon_k}+\delta^{2l+2p\varepsilon_k}\int_{t_0}^s\mu^{-1}\mathring L\mu\bigl(\mu|RZ^m\vp_\gamma|^2+\mu|\mathring LZ^m\vp_\gamma|^2\bigr)(\tau,u,\vartheta)d\tau\\
&+\delta^{2l+2p\varepsilon_k}\int_{t_0}^s2\mu\bigl(RZ^m\vp_\gamma\cdot\mathring LRZ^m\vp_\gamma+\mathring LZ^m\vp_\gamma\cdot\mathring L^2Z^m\vp_\gamma\bigr)(\tau,u,\vartheta)d\tau\\
\lesssim & \delta^{2-2\varepsilon_0-2\varepsilon_k}+\delta^{2l+2p\varepsilon_k}\int_{t_0}^s\delta^{-\varepsilon_k}\bigl(\mu|RZ^m\vp_\gamma|^2+\mu|\mathring LZ^m\vp_\gamma|^2\bigr)(\tau,u,\vartheta)d\tau\\
&+\delta^{2l+2p\varepsilon_k}\int_{t_0}^s\mu\bigl(|RZ^m\vp_\gamma|\cdot|\mathring LRZ^m\vp_\gamma|+|\mathring LZ^m\vp_\gamma|\cdot|\mathring L^2Z^m\vp_\gamma|\bigr)(\tau,u,\vartheta)d\tau,
\end{align*}
and hence,
\begin{equation*}
\begin{split}
&\delta^{l+p\varepsilon_k}\bigl(\mu^{1/2}|RZ^m\vp_\gamma|+\mu^{1/2}|\mathring LZ^m\vp_\gamma|\bigr)(s,u,\vartheta)\\
\lesssim&\,\delta^{1-\varepsilon_0-\varepsilon_k}
+\delta^{l+p\varepsilon_k}\int_{t_0}^s\bigl(|\mu^{1/2}\mathring LRZ^m\vp_\gamma|
+|\mu^{1/2}\mathring L^2Z^m\vp_\gamma|\bigr)(\tau,u,\vartheta)d\tau.
\end{split}
\end{equation*}
Together with \eqref{Ff}, this implies
\begin{align}
&\delta^{l+p\varepsilon_k}\bigl(\|\mu^{1/2}RZ^m\vp\|_{s, u}+\|\mu^{1/2}\mathring LZ^m\vp\|_{s, u}\bigr)\nonumber\\
\lesssim&\, \delta^{-\varepsilon_0-\varepsilon_k+3/2}+\int_{t_0}^s\delta^{-\varepsilon_k}\mu_{\min}^{-b_{m+2}}(\tau)\sqrt{\tilde{\mathcal E}_{1,\leq m+2}(\tau,u)}d\tau\nonumber\\
\lesssim_b &\, \delta^{1-\f32\varepsilon_0+k(1-\varepsilon_0)}\mu_{\min}^{1-b_{m+2}}(s).\label{E1m+1}
\end{align}
Set $b_{m+1} = b_{m+2} - 1$. Then from \eqref{E2m+1} and \eqref{E1m+1}, one has
\[
\delta \widetilde{\mathcal{E}}_{2,\leq m+1}(s, u) + \delta^{\varepsilon_k}\widetilde{\mathcal{E}}_{1,\leq m+1}(s, u) \lesssim_b\, \delta^{2-2\varepsilon_0}.
\]
Note that we can define $b_{m+2-n} = b_{m+2} - n$ for $1 \leq n \leq [b_{m+2}]$ and $b_{m+2-n} = 0$ for $[b_{m+2}] + 1 \leq n \leq m + 2$. Repeating the argument as in deriving \eqref{E2m+1} and \eqref{E1m+1} step by step, we eventually arrive at
\[
\delta\widetilde{\mathcal{E}}_{2,\leq n}(s, u) + \delta^{\varepsilon_k}\widetilde{\mathcal{E}}_{1,\leq n}(s, u) \lesssim_b\, \delta^{2-2\varepsilon_0}, \quad 0 \leq n \leq m+2.
\]
In particular, it holds that
\begin{align}
\mathcal E_{2,\leq m + 1 - [b_{m+2}]}(s, u) &\lesssim_b\, \delta^{1-2\varepsilon_0}, \label{E2low}\\
\mathcal E_{1,\leq m + 1 - [b_{m+2}]}(s, u) &\lesssim_b\, \delta^{2-3\varepsilon_0+k(1-\varepsilon_0)}.\label{E1low}
\end{align}
which are the non-degenerate energy estimates,

We next apply the Sobolev inequality on $S_{s,u}$ and \eqref{SSi} to close the first bootstrap assumption in \eqref{bootstrap-assumptions}.
For any $n\leq m-1-[b_{m+2}]$, one has
\begin{equation}\label{improvement}
\begin{split}
&\delta^{l+p\varepsilon_k}|Z^n\vp_\gamma(s,u,\vartheta)|
\lesssim\delta^{l+p\varepsilon_k}\sum_{i\leq 1}\|R^iZ^n\vp_\gamma\|_{L^2(S_{s,u})}\\
\lesssim&\delta^{1/2}\bigl(\sqrt{\mathcal E_{1,\leq n+2}(s, u)}+\sqrt{\mathcal E_{2,\leq n+2}(s, u)}\bigr)\lesssim_b\,\delta^{1-\varepsilon_0},
\end{split}
\end{equation}
where $l$ and $p$ denote the numbers of occurrences of $T$
and $\varrho\mathring L$ in $Z^n$, respectively.
Therefore, \eqref{improvement} obviously improves the first inequality in \eqref{bootstrap-assumptions}.

To verify the second assumption in \eqref{bootstrap-assumptions}, we now estimate the Sobolev norm $\|Z^{n+1}\phi\|_{s,u}$.
Following the same procedure as in \eqref{lowerphi}, one has
\begin{align*}
&\delta^{l+p\varepsilon_k}\|Z^{n+1}\phi\|_{s,u}\\
\lesssim & \ \delta^{3-\varepsilon_0}\delta^{l_1+p_1\varepsilon_k}\bigl(\|\slashed{\mathcal L}_{Z}^{n_1}\leftidx{^{(R)}}{\slashed\pi}_T\|_{s,u}+\delta^{\varepsilon_k}\|\slashed{\mathcal L}_{Z}^{n_1}\leftidx{^{(T)}}{\slashed\pi}_{\mathring L}\|_{s,u}\bigr)+\delta^2(1+\delta^{2k(1-\varepsilon_0)-\varepsilon_0})\delta^{l_1+p_1\varepsilon_k}\|\slashed{\mathcal L}_{\bar Z}^{n_1}\slashed g\|_{s,u}\\
&+\delta^{2-\varepsilon_0}\delta^{l_0+p_0\varepsilon_k}\bigl(\| Z^{n_0}\mu\|_{s,u}+\sum_{i=1}^2\big(\| Z^{n_0}\check L^i\|_{s,u}+\| Z^{n_0}x^i\|_{s,u}\big)\big)+\delta\cdot\delta^{l_0+p_0\varepsilon_k}\| Z^{n_0}\varphi\|_{s,u}\\
\lesssim & \ \delta^{5/2-\varepsilon_0}+\delta^{1+\varepsilon_0}\mu_{\min}^{-b_{n+2}}\sqrt{\widetilde {\mathcal E}_{1,\leq n+2}(s,u)}+\delta^2\sqrt{\widetilde {\mathcal E}_{2,\leq n+2}(s,u)}.
\end{align*}
Under the condition of $n+1\leq m-[b_{m+2}]$, it follows from the estimates \eqref{E2low}--\eqref{E1low} that
\[
\delta^{l+p\varepsilon_k}\|Z^{n+1}\phi\|_{s,u}\lesssim_b \delta^{5/2-\varepsilon_0}.
\]
Furthermore, for all $n\leq m-2-[b_{m+2}]$, it holds that
\[
\delta^{l+p\varepsilon_k}|Z^{n}\phi(s,u,\vartheta)|\lesssim_b\delta^{2-\varepsilon_0}.
\]
This verifies the second assumption in \eqref{bootstrap-assumptions} and ultimately completes the proof of Theorem \ref{maintheorem}.

\vskip 0.2 true cm

{\bf \color{blue}{Conflict of Interest Statement:}}

\vskip 0.2 true cm

{\bf The authors declare that there is no conflict of interest in relation to this article.}

\vskip 0.2 true cm
{\bf \color{blue}{Data availability statement:}}

\vskip 0.2 true cm

{\bf  Data sharing is not applicable to this article as no data sets are generated
during the current study.}

\end{document}